\documentclass[12pt,reqno]{amsart}
\usepackage{amsmath}
\usepackage{latexsym,amssymb,eucal}
\usepackage{mathrsfs}
\usepackage{enumerate}
\usepackage{bm}
\allowdisplaybreaks[1]

\newtheorem{thm}{Theorem}[section]
\newtheorem{lem}[thm]{Lemma}
\newtheorem{cor}[thm]{Corollary}
\newtheorem{prop}[thm]{Proposition}
\newtheorem{prob}[thm]{Problem}

\newtheorem{conj}[thm]{Conjecture}

\newtheorem{thme}[]{Theorem}
\newtheorem{coro}[thme]{Corollary}

\newtheorem{clm}[]{Claim}
\newtheorem*{clam}{Claim}

\theoremstyle{definition}

\newtheorem{defn}[thm]{Definition}
\newtheorem{rem}[thm]{Remark}
\newtheorem{exam}[thm]{Example}

\numberwithin{equation}{section}

\renewcommand{\labelenumi}{$(\arabic{enumi})$}

\newcommand{\id}{\mathrm{id}}
\newcommand{\Ad}{\mathrm{Ad}\,}

\newcommand{\Tr}{\mathrm{Tr}}

\newcommand{\C}{{\mathbb{C}}}
\newcommand{\R}{{\mathbb{R}}}
\newcommand{\Z}{{\mathbb{Z}}}
\newcommand{\N}{{\mathbb{N}}}
\newcommand{\Q}{{\mathbb{Q}}}
\newcommand{\T}{{\mathbb{T}}}

\DeclareMathOperator{\Aut}{Aut}

\DeclareMathOperator{\Cnt}{Cnt}

\DeclareMathOperator{\Int}{Int}
\DeclareMathOperator{\oInt}{\overline{Int}}
\DeclareMathOperator{\Log}{Log}
\DeclareMathOperator{\mo}{mod}
\DeclareMathOperator{\pr}{pr}

\DeclareMathOperator{\supp}{supp}
\DeclareMathOperator{\Sp}{Sp}

\DeclareMathOperator{\vN}{vN}

\def\cA{\mathcal{A}}

\def\cF{\mathcal{F}}

\def\cH{\mathcal{H}}
\def\cK{\mathcal{K}}
\def\cM{\mathcal{M}}

\def\tcM{\widetilde{\mathcal{M}}}
\def\tcN{\widetilde{\mathcal{N}}}
\def\cN{\mathcal{N}}
\def\cO{\mathcal{O}}
\def\cP{\mathcal{P}}
\def\tcP{\widetilde{\mathcal{P}}}
\def\cQ{\mathcal{Q}}
\def\tcQ{\widetilde{\mathcal{Q}}}
\def\cR{\mathcal{R}}

\def\sB{\mathscr{B}}
\def\sC{\mathscr{C}}
\def\sE{\mathscr{E}}
\def\sG{\mathscr{G}}

\def\sJ{\mathscr{J}}
\def\sN{\mathscr{N}}

\def\sT{\mathscr{T}}

\def\tJ{\widetilde{J}}

\def\al{\alpha}
\def\cal{{\check{\al}}}
\def\hal{{\hat{\al}}}
\def\tal{{\widetilde{\alpha}}}
\def\be{\beta}
\def\hbe{\hat{\be}}
\def\tbe{{\widetilde{\beta}}}
\def\ga{\gamma}
\def\hga{\hat{\gamma}}

\def\de{\delta}
\def\ka{\kappa}
\def\la{\lambda}

\def\vep{\varepsilon}
\def\ph{{\phi}}
\def\tph{\widetilde{\phi}}
\def\ps{{\psi}}

\def\vph{\varphi}
\def\hvph{{\hat{\varphi}}}

\def\om{\omega}
\def\tpi{\widetilde{\pi}}
\def\trho{\widetilde{\rho}}
\def\si{\sigma}
\def\ta{\tau}

\def\th{\theta}

\def\De{\Delta}
\def\Ga{\Gamma}
\def\La{\Lambda}
\def\Om{\Omega}
\def\Ph{\Phi}
\def\Ps{\Psi}
\def\Th{\Theta}

\def\lG{L^\infty(G)}
\def\Meq{\cM_{\om,\al}}
\def\Mequ{\cM_\al^\om}
\def\tcN{\widetilde{\cN}}

\def\col{\colon}
\def\nin{\notin}
\def\ra{\rightarrow}

\def\subs{\subset}
\def\ovl{\overline}
\def\oti{\otimes}
\def\rti{\rtimes}

\def\btr{{\bm 1}}

\title{Rohlin flows on von Neumann algebras}
\author[T. Masuda]{Toshihiko Masuda$^1$}
\address{$^1$ 
Graduate School of Mathematics, Kyushu University,
Fukuoka\\ \indent \mbox{819-0395},
JAPAN}
\email{masuda@math.kyushu-u.ac.jp}

\author[R. Tomatsu]{Reiji Tomatsu$^2$}
\address{$^2$
Department of Mathematics, Hokkaido University,
Hokkaido\\ \indent \mbox{060-0810},
JAPAN}
\email{tomatsu@math.sci.hokudai.ac.jp}


\begin{document}
\maketitle

\begin{abstract}
We will introduce the Rohlin property
for flows on von Neumann algebras
and classify them up to strong cocycle conjugacy.
This result provides alternative approaches
to some preceding results
such as Kawahigashi's classification of
flows on the injective type II$_1$ factor,
the classification of injective type III factors
due to Connes, Krieger and Haagerup
and
the non-fullness of type III$_0$ factors.
Several concrete examples are also studied.
\end{abstract}

\section{Introduction}
In this paper,
we study flows on von Neumann algebras.
Our purpose is to classify
highly outer flows called Rohlin flows.

A flow, that is,
a one-parameter automorphism group,
appears in many scenes
in the theory of operator algebras,
and it has
attracted attention
among operator algebraists.
We have known some examples of
classification of non-periodic flows
on injective factors.
In \cite{Ha-III1},
Haagerup has solved the Connes' bicentralizer
problem for injective type III$_1$ factors.
As an important consequence,
the uniqueness of the injective type III$_1$
factor follows.
In other words,
trace scaling flows on the injective type II$_\infty$
factor are (cocycle) conjugate to one another
if their Connes-Takesaki modules are equal.
In the type II$_1$ setting,
Kawahigashi has studied several kinds of flows
on the injective type II$_1$ factor
\cite{Kw-cent,Kw-irrat,Kw-cartan,Kw-Cotriv}.
Among them,
he has obtained the classification
of flows on the injective type II$_1$ factor
such that
they have the full Connes spectrum
and
fix a Cartan subalgebra.

We can expect that
these examples may possess some sort of right ``outerness'',
and consequently
they are classifiable.
Thus it is a natural attempt to give a comprehensive
method of classifying flows on von Neumann algebras.
In classification of group actions,
``outerness'',
which, to be precise, includes the central freeness,
is considered as an essentially important notion.
In this point,
the usual pointwise outerness
is known to be not so sufficiently strong
that we can classify flows up to cocycle conjugacy.
Indeed,
Kawahigashi has found 
a family of non-cocycle conjugate outer flows 
on the injective
factor of type II$_1$
\cite{Kw-Cotriv}.
Thus
it is conceivable
that
both
pointwise outerness
and pointwise central non-triviality
are not right notions of ``outerness'' for flows.

One formulation of ``outerness''
is
to observe how non-trivially
a given group is acting on
a central sequence algebra.
This is the case
for
actions of discrete amenable groups
\cite{Co-outer,J-act,Kat-S-T,Ocn-act}
or
duals of compact groups \cite{Ma-T,Ma-TIII}.
A flow, however,
causes a serious problem
concerning discontinuity
on
a central sequence algebra $\cM_\om$.
One prescription of that
is to focus on
the much smaller subalgebra $\cM_{\om,\al}$,
which
consists of
$(\al,\om)$-equicontinuous sequences
(see Definition \ref{defn:al-om}).
Then the Rohlin property,
which
has been introduced by Kishimoto
to flows on C$^*$-algebras \cite{Kishi-CMP}
and later by Kawamuro
to flows on 
finite von Neumann algebras \cite{Kawamuro-RIMS},
can be a candidate of
``outerness''.
This property means
that
we can find out a unitary eigenvector
in $\cM_{\om,\al}$ with the eigenvalue $p$
for any $p\in\R$.

Assuming the Rohlin property,
we will prove the following
main theorem of this paper
(Theorem \ref{thm:class2}).

\begin{thme}
\label{thme:main}
Let $\al,\be$ be Rohlin flows
on a von Neumann algebra with separable predual.
Then
$\al$ and $\be$ are strongly cocycle conjugate
if and only if
$\al_t\be_{-t}$ is approximately inner for all $t\in\R$.
\end{thme}

We emphasize that
either the factoriality or the injectivity
are not required in our assumption.
For injective factors,
we obtain the following result
in terms of the Connes-Takesaki module
(Corollary \ref{cor:class-inj}).

\begin{coro}
Let $\al,\be$ be Rohlin flows
on an injective factor.
Then
$\al$ and $\be$ are strongly cocycle
conjugate
if and only if
$\mo(\al_t)=\mo(\be_t)$
for all $t\in\R$.
\end{coro}

It turns out that
if
a flow $\al$
on the injective type II$_1$ factor
fixes a Cartan subalgebra
and the Connes spectrum
$\Ga(\al)$ equals $\R$,
then
$\al$ has the Rohlin property.
Thus Theorem \ref{thme:main} implies
the following Kawahigashi's result
(Theorem \ref{thm:Kaw-Cartan}).

\begin{thme}[Kawahigashi]
Let $\al$ be a flow on the injective
type II$_1$ factor $\cM$.
If $\al$ pointwise
fixes a Cartan subalgebra of $\cM$
and $\Ga(\al)=\R$,
then $\al$ is cocycle conjugate
to a product type flow,
and absorbs any product type flows.
Thus such action $\al$ is unique
up to cocycle conjugacy.
\end{thme}

Thanks to works due to Connes and Haagerup,
a modular automorphism group
on any injective factor
is an approximately inner flow,
and hence
the dual flow has the Rohlin property
(Theorem \ref{thm:dual}, Proposition \ref{prop:ext-dual}).
Then Theorem \ref{thme:main}
implies the following result
(Theorem \ref{thm:inj-class}).

\begin{thme}[Connes, Haagerup, Krieger]
Let $\cM_1$ and $\cM_2$ be injective factors
of type III.
Then
they are isomorphic
if and only if
their flows of weights are isomorphic.
\end{thme}

This paper is organized as follows.
In Section 2,
the basic notions such
as the core of a von Neumann algebra
and
an ultraproduct
von Neumann algebra
are reviewed.

In Section 3,
to a Borel map $\al\col\R\ra\Aut(\cM)$,
we introduce the notion of 
$(\al,\om)$-equicontinuity
and the $(\al,\om)$-equicontinuous
parts $\cM_\al^\om$ and $\cM_{\om,\al}$
of $\cM^\om$ and $\cM_\om$,
respectively.

In Section 4,
the Rohlin property and
the invariant approximate innerness
are introduced.
We show they are dual notions to each other.

Section 5 is
devoted to proving the main classification
result.
We first prove the 
2-cohomology vanishing
for Borel cocycle actions of $\R$
with Rohlin property.
We next obtain the approximate
vanishing of the 1-cohomology of a Rohlin flow.
We show that by disintegration,
it suffices to prove the main theorem
for centrally ergodic flows.
Then the Bratteli-Elliott-Evans-Kishimoto intertwining argument
achieves strong cocycle conjugacy

In Section 6,
we apply the main result
to give alternative proofs
of some known results:
Kawahigashi's results about
flows on the injective type II$_1$ factor,
the classification of
injective type III factors
(assuming Haagerup's work on a bicentralizer)
and
the non-fullness of
an arbitrary type III$_0$ factor,
more precisely,
the approximate innerness
of a modular automorphism group.
We also discuss results obtained
by Hui and Aoi-Yamanouchi
in \cite{AY,Hui}.
Some concrete examples of Rohlin flows
are given.
In particular,
we will classify product type flows
and quasi-free flows coming from a Cuntz algebra
up to cocycle conjugacy.

In Section 7,
we will give a characterization
of the Rohlin property
which states that
a flow $\al$ on a factor $\cM$
has the Rohlin property
if and only if
$\al$ is faithful on $\cM_{\om,\al}$.

In Section 8,
we will pose a plausible conjecture
on a characterization of the Rohlin property.
Some unsolved problems are also mentioned.

We will close this paper
with appendix in Section 9,
where
basic results on measure theory
and
a disintegration of automorphisms are studied.
Also, with some assumptions on a factor,
we will show that
the condition of Theorem \ref{thme:main}
derives an approximation of $\al_t$
by $\Ad v(t)\circ\be_t$
with $v$ being a continuous unitary path.

\vspace{10pt}
\noindent
{\bf Acknowledgements.}
The authors are grateful
to Masamichi Takesaki for various comments on our work.
The second named author would like to
thank Akitaka Kishimoto for stimulating discussions.
We also thank George Elliott and Masaki Izumi for helpful advice.
The authors
are supported by 
Grant-in-Aid for Scientific Research (C)
and
Grant-in-Aid for Young Scientists (B)
of
Japan Society
for the Promotion of Science,
respectively.

\tableofcontents

\section{Preliminary}
Throughout this paper,
we mainly treat a von Neumann algebra with separable predual
unless otherwise noted.

\subsection{Notation}
Let $\cM$ be a not necessarily separable von Neumann algebra.
Let us denote by
$\cM^{\rm U}$, $\cM^{\rm P}$,
$\cM^{\rm PI}$ and $\cM_1$
the set of
unitaries, projections,
partial isometries
and contractions in $\cM$,
respectively.
The center of $\cM$
is denoted by $Z(\cM)$.
The set of faithful normal semifinite weights
is denoted by $W(\cM)$.

For $\alpha\in \Aut(\cM)$,
$a\in \cM$ and $\varphi\in \cM_*$,
let $\alpha(\varphi)$, $a\varphi$,
$\varphi a$, $[a,\varphi]\in \cM_*$ be 
\[
\alpha(\varphi)
:=\varphi\circ \alpha^{-1},\
\varphi a(x)
:=\varphi(ax),\
a\varphi(x)
:=\varphi(xa),\
[a,\varphi]
:=a\varphi-\varphi a,
\]
respectively.
For $a\in\cM$ and $\vph\in(\cM_*)_+$,
we define the following seminorms:
\[
\|a\|_\vph:=\vph(a^*a)^{1/2},
\quad
\|a\|_\vph^\sharp:=2^{-1/2}(\vph(a^*a)+\vph(aa^*))^{1/2}.
\]

In this paper,
$\{\cH,J,\cP\}$ denotes the standard Hilbert space of $\cM$
(see \cite{Ha-stand} for the notations).
We regard $\cH$ as an  $\cM$-$\cM$-bimodule
as follows:
\[
x\xi y:=x Jy^*J\xi,
\quad
x,y\in\cM,
\
\xi\in\cH.
\]

For $\al\in\Aut(\cM)$,
there uniquely exists a unitary
$U(\al)$ on $\cH$
such that
$\al(x)=\Ad U(\al)(x)$
for $x\in\cM$,
$JU(\al)=U(\al) J$ and $U(\al)\cP=\cP$.
We use the notation as
$\al(\xi):=U(\al)\xi$.
Then we have
$\al(x\xi y)=\al(x)\al(\xi)\al(y)$.
Since $U(\Ad u)=uJuJ$ for $u\in\cM^{\rm U}$,
we have $\Ad u(\xi)=u\xi u^*$.

We equip $\Aut(\cM)$
with the $u$-topology as usual.
Namely,
a net $\al_\la\in\Aut(\cM)$
converges to $\al\in\Aut(\cM)$
if
$\al_\la(\vph)\to\al(\vph)$
for all $\vph\in\cM_*$.
Then the map $\Aut(\cM)\ni\al\mapsto U(\al)$
is strongly continuous.
If $\cM$ is separable, that is, $\cM_*$ is norm separable,
then
$\Aut(\cM)$ is a Polish group.

Let us denote by $\Int(\cM)$
the set of inner automorphisms.
An automorphism which belongs to
the closure $\oInt(\cM)$ of $\Int(\cM)$
is said to be
{\it approximately inner}.

Throughout this paper,
we always equip $\R^n$ with the usual Lebesgue measure.

\subsection{Actions and cocycle actions}
\label{subsect:action-cocycle}
In this paper,
we mean by a {\it flow}
a one-parameter automorphism group on
a von Neumann algebra,
that is,
a group homomorphism $\al\col \R\to\Aut(\cM)$
with the following continuity:
\[
\lim_{t\to0}\|\al_t(\vph)-\vph\|=0
\quad
\mbox{for all }\vph\in\cM_*,
\]
or equivalently,
\[
\lim_{t\to0}
\|\al_t(\xi)-\xi\|=0
\quad
\mbox{for all }\xi\in\cH.
\]

By $\cM^\al$,
we denote the fixed point algebra of $\al$.
We say that
$\al$ is
\emph{ergodic}
if $\cM^\al=\C$,
and \emph{centrally ergodic}
if $Z(\cM)^\al=\C$.

A flow $\al$ is said to be
{\it inner} if $\al_t\in\Int(\cM)$
for all $t\in\R$,
and
{\it outer} if $\al_t\nin\Int(\cM)$
for all $t\in\R\setminus\{0\}$.
Thanks to \cite[Theorem 0.1]{Kal} or \cite[Theorem 5]{MooreIV},
if $\cM$ is separable,
then an inner flow $\al$ is implemented
by a one-parameter unitary group
$u\col\R\ra\cM^{\rm U}$.
See also Corollary \ref{cor:ptwiseinner}.

An $\al$-\emph{cocycle} means a strongly continuous
unitary path $v$ in $\cM$
such that $v(s)\al_s(v(t))=v(s+t)$.
The perturbed flow is defined by $\al_t^v:=\Ad v(t)\circ\al_t$.

Let $\al$ and $\be$ be flows on
von Neumann algebras $\cM$ and $\cN$, respectively.
They are said to be
\begin{itemize}
\item 
{\it conjugate}
if there exists an isomorphism
$\th\col\cN\ra\cM$
such that $\al_t=\th\circ\be_t\circ\th^{-1}$.
We write $\al\approx \be$;

\item
{\it cocycle conjugate}
if there exist an isomorphism
$\th\col\cN\ra\cM$
and an $\al$-cocycle $v$
such that $\al_t^v=\th\circ\be_t\circ\th^{-1}$.
We write $\al\sim\be$;

\item
{\it stably conjugate}
if $\al\oti\id_{B(\ell^2)}$
and $\be\oti\id_{B(\ell^2)}$
are cocycle conjugate.
\end{itemize}

When $\cM=\cN$,
$\al$ and $\be$ are said to be
\emph{strongly cocycle conjugate}
if there exist $\th\in\oInt(\cM)$
and an $\al$-cocycle $v$
such that $\al_t^v=\th\circ\be_t\circ\th^{-1}$.

A {\it Borel cocycle action}
means a pair $(\al,c)$
of
Borel maps
$\al\col\R\ra\Aut(\cM)$
and $c\col\R^2\ra\cM^{\rm U}$
such that
for all $r,s,t\in\R$,
$c(s,0)=1=c(0,s)$,
$\al_0=\id$
and
\[
\al_s\circ\al_t=\Ad c(s,t)\circ\al_{s+t},
\]
\[
c(r,s)c(r+s,t)=\al_r(c(s,t))c(r,s+t).
\]

The perturbation of $(\al,c)$
by a Borel unitary path $v\col\R\ra\cM^{\rm U}$
is the Borel cocycle action $(\al^v,c^v)$
defined by
\[
\al_t^v:=\Ad v(t)\circ\al_t,
\quad
c^v(s,t):=v(s)\al_s(v(t))c(s,t)v(s+t)^*
\quad
\mbox{for all }
s,t\in\R.
\]

As is well-known,
if $\cM$ is properly infinite,
then any $2$-cocycle is a coboundary.
However, the solution presented below
is always ``big'' even if a given
2-cocycle is close to 1.

\begin{lem}
\label{lem:propinf2coho}
Let $\cM$ be a properly infinite von Neumann algebra
and 
$(\alpha, c)$ a Borel cocycle action of $\R$ on $\cM$.
Then there exists a Borel unitary path $u(t)\in\cM$
such that 
$u(t)\alpha_t(u(s))c(t,s)u(t+s)^*=1$
for all $(t,s)\in\R^2$.
\end{lem}
\begin{proof}
Let $H$ be a separable infinite dimensional Hilbert space.
Regard $B(H)$ as a von Neumann subalgebra of $\cM$
such that $B(H)'\cap\cM$ is properly infinite.
Let $\{e_{ij}\}_{i,j=1}^\infty$
be a system of matrix units of $B(H)$
such that $\sum_i e_{ii}=1$
and
$e_{11}$ is minimal in $B(H)$.
Take an isometry $v$
with $vv^*=e_{11}$. 
Set $w(t):=\sum_{i}e_{i1}v\alpha_t(v^*e_{1i})$.
It is easy to see $w(t)$ is a Borel unitary path,
and $w(t)\alpha_t(e_{ij})w(t)^*=e_{ij}$. 

Hence we may and do assume that
$(\alpha,c)$ is of the form
$(\beta\otimes \id, d\otimes 1)$
on $\cM=\mathcal{N}\otimes B(L^2(\mathbb{R}))$
for a von Neumann algebra $\cN\subs B(K)$
and a Hilbert space $K$.
As given in the proof of \cite[Proposition 2.1.3]{Su-homoII},
the following $u(t)$ does the job:
\[
(u(t)\xi)(s)=d(t,s)\xi(t+s)
\quad
\mbox{for all }
\xi\in K\oti L^2(\R),
\ s,t\in\R.
\]
\end{proof}

\begin{rem}
\label{rem:alid}
In the proof above,
it turns out that
the unitary path $w$ is in fact an $\al$-cocycle
when $\al$ is a flow.
Thus if $\al$ is a flow on a properly infinite
von Neumann algebra
$\cM$,
then $\al\sim\al\oti\id_{B(H)}$.
Indeed,
\[\al\sim\be\oti\id_{B(H)}
\approx
\be\oti\id_{B(H)}\oti\id_{B(H)}
\sim\al\oti\id_{B(H)}.
\]
Hence the stable conjugacy implies
the cocycle conjugacy if $\cM$ is properly infinite.
When $\cM$ is finite,
this is not true in general
(see \cite[Theorem 2.9]{Kw-Cotriv}).
\end{rem}

Let $\al$ be a flow on $\cM$.
we define
$\pi_\al(x),\la^\al(t)\in\cM\oti B(L^2(\R))$
for $x\in\cM$
and $t\in\R$
as follows:
\[
(\pi_\al(x)\xi)(s)
=
\al_{-s}(x)\xi(s),
\quad
(\la^\al(t)\xi)(s)
=\xi(s-t)
\quad
\mbox{for }
\xi\in\cH\oti L^2(\R),
\ s\in\R.
\]
Then the crossed product
$\cM\rti_\al\R$
is the von Neumann algebra
generated by $\pi_\al(\cM)$
and $\la^\al(\R)$.
Note that $\la^\al(t)=1\oti \la(t)$,
where $\la(t)$ denotes
the left regular representation.
Let us denote by $\rho(t)$
the right regular representation.

The dual flow $\hal$ on $\cM\rti_\al\R$
is defined as
\[
\hal_p(\pi_\al(x))=\pi_\al(x),
\quad
\hal_p(\la^\al(t))=e^{-ipt}\la^\al(t)
\quad\mbox{for }
x\in\cM,\ p,t\in\R.
\]

\subsection{Core and canonical extension}
The \emph{core} $\tcM$
of a von Neumann algebra $\cM$
is introduced in \cite{FT},
and that is generated by
a copy of $\cM$
and a one-parameter unitary group
$\{\la^\vph(t)\}_{t\in\R}$,
$\vph\in W(\cM)$.
Their relations are
described as follows:
for $x\in\cM$, $t\in\R$
and
$\vph,\ps\in W(\cM)$,
\[
\la^\vph(t)x=\si_t^\vph(x)\la^\vph(t),
\quad
\la^\vph(t)=[D\vph:D\ps]_t\la^\ps(t).
\]
Then the core $\tcM$
is naturally isomorphic to
$\cM\rti_{\si^\vph}\R$.

The restriction
of the dual flow $\th$ of $\si^\vph$
on $Z(\tcM)$ is
called the
\emph{$($smooth$)$ flow of weights}
of $\cM$
\cite{CT}.
Note that $Z(\tcM)^\th=Z(\cM)$.

It is known that
the flow of weights 
is a complete invariant
for isomorphic classes
among
injective type III factors.
We will present a proof of this fact
in Theorem \ref{thm:inj-class}
as an application of our classification of Rohlin flows.

Let $\cN$ be another von Neumann algebra.
Any isomorphism $\pi$ from $\cM$ onto $\cN$
extends to the isomorphism
$\tpi\col\tcM\ra\tcN$
such that for $x\in\cM$ and $t\in\R$,
\[
\tpi(x)
:=\pi(x),
\quad
\tpi(\la^\vph(t))
:
=\la^{\pi(\vph)}(t),
\]
where $\vph\in W(\cM)$ and $\pi(\vph):=\vph\circ\pi^{-1}$.
We call $\tpi$
the \emph{canonical extension} of $\pi$
(see \cite[Theorem 2.4]{FT} and \cite[Proposition 12.1]{HS}).
Let $\th^\cM$ and $\th^\cN$ be the dual flows
on $\tcM$ and $\tcN$.
Then $\tpi$ intertwines them,
that is,
$\th_t^{\cN}\circ\tpi=\tpi\circ\th_t^{\cM}$.
The restriction
$\tpi|_{Z(\tcM)}\col Z(\tcM)\ra Z(\tcN)$
is
called
the \emph{Connes-Takesaki module}
of $\pi$
\cite{CT}.

When $\cN=\cM$,
we note that
the canonical extension
$\Aut(\cM)\ni\al\ra \tal\in\Aut(\tcM)$
is a continuous group homomorphism.

Let $G$ be a locally compact group
and $\al\col G\ra\Aut(\cM)$ an action.
For $\vph\in W(\cM)$,
we denote by $\hvph$ the dual weight on $\cM\rti_\al G$.
Then we have
\begin{equation}
\label{eq:dual-modular}
\si_t^\hvph(\pi_\al(x))=\pi_\al(\si_t^\vph(x)),
\quad
\si_t^\hvph(\la^\al(g))
=
\de_G(g)^{it}\la^\al(g)
\pi_\al([D\vph\circ\al_g:D\vph]_t),
\end{equation}
where $\de_G$ denotes the modular function of $G$.
See \cite[Theorem 3.2]{Ha-dualI}.

We introduce the action $\check{\al}\col G\ra\Aut(\tcM)$
defined by
$\cal_g:=\tal_g\circ\th_{\log\de_G(g)}$.
Then the core of $\cM\rti_\al G$
is canonically isomorphic to
$\tcM\rti_\cal G$ as shown below.

\begin{lem}
\label{lem:can-natural}
One has the isomorphism
$\rho\col(\cM\rti_\al G)\,\widetilde{}
\,\ra\tcM\rti_\cal G$
such that
\begin{itemize}
\item
$\rho(\pi_\al(x))=\pi_\cal(x)$
for all $x\in\cM$;

\item
$\rho(\la^\al(g))=\la^{\cal}(g)$
for all $g\in G$;

\item
$\rho(\la^\hvph(t))=\pi_\cal(\la^\vph(t))$
for all $\vph\in W(\cM)$
and $t\in\R$.
\end{itemize}
In particular,
when $G$ is abelian,
we have
$\rho\circ\widetilde{\hal}_p=\widehat{\tal}_p\circ\rho$
for all $p\in\hat{G}$.
\end{lem}
\begin{proof}
Let $\cN:=\cM\rti_\al G$ and $\vph\in W(\cM)$.
Then we have the canonical isomorphisms
$\Xi_\hvph\col\tcN\ra \cN\rti_{\si^\hvph}\R$
and
$\La_\vph\col\tcM\rti_\cal G
\ra
(\cM\rti_{\si^\vph}\R)\rti_\cal G$.
Let $U\col L^2(G\times \R)\ra L^2(\R\times G)$
be the flip unitary.
Set the unitary
$V\in \cM\oti L^\infty(\R\times G)$
defined by $V(t,g):=\de_G(g)^{it}[D\vph:D\vph\circ\al_g]_{-t}$.

We will show
$\rho:=\La_\vph^{-1}\circ\Ad V(1\oti U)\circ\Xi_\hvph$
is the well-defined isomorphism
from $\tcN$ onto
$\tcM\rti_\cal\R$
satisfying the required conditions.
Let $x\in\cM$.
Then
$\Xi_\hvph(\pi_\al(x))=\pi_{\si^\hvph}(\pi_\al(x))$.
By (\ref{eq:dual-modular}),
we have
$\Ad V(1\oti U)(\Xi_\hvph(\pi_\al(x)))
=
\pi_\cal(\pi_{\si^\vph}(x))
$.
Thus we get
$\rho(\pi_\al(x))=\pi_\cal(x)$.
For $g\in G$,
we have
$\Xi_\hvph(\la^\al(g))=\pi_\hvph(\la^\al(g))$.
Recall that
$\si_{-t}^\hvph(\la^\al(g))
=
\la^\al(g)\pi_\al(V_{t,g}^*)
$.
Thus for $\xi,\eta\in \cH\oti L^2(\R\times G)$,
we have
\begin{align*}
&\langle
\Ad V(1\oti U)(\Xi_\hvph(\la^\al(g)))
\xi,
\eta
\rangle
\\
&=
\int_\R dt\int_G dh
\,
\big{\langle}
V_{t,h}
\cdot
\big{(}
\si_{-t}^\hvph(\la^\al(g))
V^*\xi
\big{)}
(t,h)
,
\eta(t,h)
\big{\rangle}
\\
&=
\int_\R dt\int_G dh
\,
\big{\langle}
V_{t,h}
\cdot
\left(
\la^\al(g)
\pi_\al(V_{t,g}^*)V^*\xi\right)
(t,h)
,
\eta(t,h)
\big{\rangle}
\\
&=
\int_\R dt\int_G dh
\,
\big{\langle}
V_{t,h}
\al_{h^{-1}g}(V_{t,g}^*)
V_{t,g^{-1}h}^*\xi(t,g^{-1}h)
,
\eta(t,h)
\big{\rangle}
\\
&=
\int_\R dt\int_G dh
\,
\big{\langle}
\xi(t,g^{-1}h)
,
\eta(t,h)
\big{\rangle}
\\
&=
\langle\La_\vph(\la^{\cal}(g))\xi,\eta\rangle.
\end{align*}
Thus $\rho(\la^\al(g))=\la^\cal(g)$.

Since $\Xi_\hvph(\la^\hvph(t))=\la^{\si^\hvph}(t)$
and
$\cal_{g^{-1}}(\la^{\si^\vph}(t))
=
\de_G(g)^{it}
\pi_{\si^\vph}([D\vph\circ\al_g:D\vph]_t)\la^{\si^\vph}(t)$,
we have
\[
\pi_\tal(\la^{\si^\vph}(t))
=
\pi_{\si^\vph}(V(-t,\cdot)^*)\la^{\si^\vph}(t).
\]
Then
\[
V\la^{\si^\vph}(t)V^*\pi_\tal(\la^{\si^\vph}(t)^*)
=
V\la^{\si^\vph}(t)V^*\la^{\si^\vph}(t)^*
\pi_{\si^\vph}(V(-t,\cdot)).
\]
For $(s,g)\in\R\times G$,
we have
\begin{align*}
\left(V\la^{\si^\vph}(t)V^*\la^{\si^\vph}(t)^*
\right)(s,g)
&=V(s,g)V(-t+s,g)^*
\\
&=
\de_G(g)^{is}
[D\vph:D\vph\circ\al_g]_{-s}
\cdot
\de_G(g)^{it-is}
[D\vph:D\vph\circ\al_g]_{t-s}^*
\\
&=
\de_G(g)^{it}
\si_{-s}^\vph([D\vph\circ\al_g:D\vph]_{t})
\\
&=
\pi_{\si^\vph}(V(-t,\cdot)^*)(s,g),
\end{align*}
and
$V\la^{\si^\vph}(t)V^*\pi_\tal(\la^{\si^\vph}(t)^*)=1$.
\end{proof}

\begin{rem}
The previous lemma
shows that
$\tcM$ is regarded as a von Neumann subalgebra of
$(\cM\rti_\al G)\,\widetilde{}\ $.
This is generalized as follows.
Let $\cN\subs \cM$ be an inclusion of von Neumann algebras.
When there exists an operator valued weight $T$
from $\cM$ onto $\cN$,
we can regard $\tcN$ as a von Neumann subalgebra of $\tcM$
in such a way that
$\la^\vph(t)=\la^{\vph\circ T}(t)$
for $\vph\in W(\cN)$ and $t\in \R$.
Note that this identification depends on the choice
of $T$.
If we take $T$ as the canonical operator valued weight
$T_{\hal}\col \cM\rti_\al G\ra \pi_\al(\cM)$,
which is given by $T_{\hal}(x)=\int_{\hat{G}} \hal_p(x)\,dp$
when $G$ is abelian,
then the associated map is nothing but
$\pi_\cal\col\tcM\ra\tcM\rti_{\check{\al}} G$.
\end{rem}

\subsection{Ultraproduct von Neumann algebras}
\label{subsect:ultra}
Our standard reference is \cite[Chapter 5]{Ocn-act}.
Let $\cM$ be a von Neumann algebra.
We denote by $\ell^\infty(\cM)$ the C$^*$-algebra
of norm bounded sequences in $\cM$.
Let $\om$ be a free ultrafilter over $\N$.

An element
$(x^\nu)_\nu$ of $\ell^\infty(\cM)$
is said to be
\begin{itemize}
\item
\emph{trivial}
if
$x^\nu\to0$ as $\nu\to\infty$
in the strong$*$ topology;
\item
\emph{$\om$-trivial}
if $x^\nu\to0$ as $\nu\to\om$
in the strong$*$ topology;

\item
\emph{central} if for all $\vph\in\cM_*$,
$\|[\vph,x^\nu]\|_{\cM_*}\to0$ as $\nu\to\infty$;

\item
\emph{$\om$-central} if for all $\vph\in\cM_*$,
$\|[\vph,x^\nu]\|_{\cM_*}\to0$ as $\nu\to\om$.
\end{itemize}

Let $\sT_\om(\cM)$ and $\sC_\om(\cM)$ be the collections of
$\om$-trivial and $\om$-central sequences in $\cM$,
respectively,
which are unital C$^*$-subalgebras of $\ell^\infty(\cM)$.
Let $\sN_\om(\cM)$ be the normalizer of
$\sT_\om(\cM)$ in $\ell^\infty(\cM)$.
Then
$\sT_\om(\cM)\subs \sC_\om(\cM)\subs\sN_\om(\cM)$.
We often simply write
$\sT_\om$, $\sC_\om$ and $\sN_\om$
for them unless otherwise confused.

The quotient C$^*$-algebras
$\cM^\om:=\sN_\om/\sT_\om$
and $\cM_\om:=\sC_\om/\sT_\om$
are in fact von Neumann algebras.
We call them \emph{ultraproduct von Neumann algebras}.
The quotient map from $\sN_\om$ onto $\cM^\om$
is denoted by $\pi_\om$.
Each $x\in \cM$ is mapped to the constant sequence
$(x,x,\dots)\in\sN_\om$.
Then $\cM$ is regarded as a von Neumann subalgebra of $\cM^\om$.

Let $\ta^\om\col\cM^\om\ra\cM$ be the map
defined by
$\ta^\om(\pi_\om((x^\nu)_\nu)):=\lim_{\nu\to\om}x^\nu$,
where the limit is taken in the $\si$-weak topology
in $\cM$.
Then $\ta^\om$ is a faithful normal conditional expectation.
For $\vph\in\cM_*$,
we denote by $\vph^\om$ the functional $\vph\circ\ta^\om$.
Any element $a\in\cM_\om$
commutes with 
$\vph^\om$,
that is, $\vph^\om a=a\vph^\om$.

If $\cM$ is a factor,
then $\ta^\om$ gives a faithful normal tracial state
on $\cM_\om$.
We often denote the trace by $\ta_\om$.
Note that in this case,
$\vph^\om=\vph(1)\ta_\om$ on $\cM_\om$ for all $\vph\in\cM_*$.

Each $\al\in \Aut(\cM)$ extends to the automorphism
$\al^\om\in\Aut(\cM^\om)$
by putting $\al^\om(\pi_\om((x^\nu)_\nu))=\pi_\om((\al(x^\nu))_\nu)$.
Then $\al^\om(\cM_\om)=\cM_\om$.
We often simply write $\al$ for $\al^\om$.

When $\al^\om$ is trivial on $\cM_\om$,
$\al$ is said to be \emph{centrally trivial}.
Denote by $\Cnt(\cM)$ the set of centrally trivial automorphisms
that is a Borel subgroup of $\Aut(\cM)$
as shown in Lemma \ref{lem:cnt-borel}.

In this paper,
the compactness of a subset of $\cH$ or $\cM_*$
means the norm compactness.

\begin{lem}
\label{lem:cpcttriv}
If $(x^\nu)_\nu\in\sT_\om$
and
$\Ps\subs \cH$ is compact,
then
$\sup_{\eta\in \Ps}(\|x^\nu \eta\|+\|\eta x^\nu\|)$
converges to 0
as $\nu\to\om$.
\end{lem}
\begin{proof}
Take $C>0$ with $C>\sup_\nu\|x^\nu\|$.
Let $\vep>0$ and take $\eta_1,\dots,\eta_n\in \Ps$
so that
any $\eta\in \Ps$ has some $\eta_i$ with
$\|\eta-\eta_i\|<\vep/4C$.
Using such $\eta_i$,
we have
\begin{align*}
\|x^\nu \eta\|+\|\eta x^\nu\|
&\leq
\|x^\nu (\eta-\eta_i)\|
+
\|x^\nu \eta_i\|+\|\eta_i x^\nu\|
+
\|(\eta_i-\eta) x^\nu\|
\\
&\leq
C\vep/4C+\|x^\nu \eta_i\|+\|\eta_i x^\nu\|+C\vep/4C
\\
&\leq
\vep/2
+
\max_i(\|x^\nu \eta_i\|+\|\eta_i x^\nu\|).
\end{align*}
Thus
\[
\sup_{\eta\in \Ps}(\|x^\nu \eta\|+\|\eta x^\nu\|)
\leq
\vep/2
+
\max_i(\|x^\nu \eta_i\|+\|\eta_i x^\nu\|)
\quad
\mbox{for all }
\nu\in\N.
\]
If $\nu$ is sufficiently close to $\om$,
the second term in the right hand side
becomes less than $\vep/2$.
Hence we obtain
$\lim_{\nu\to\om}\sup_{\eta\in\Ps}
(\|x^\nu \eta\|+\|\eta x^\nu\|)\leq\vep$.
\end{proof}

In a similar way,
we can prove the following.

\begin{lem}
\label{lem:centralcpct}
Let $(x^\nu)_\nu\in\sC_\om$.
Then
for any compact set $\Ps\subs\cH$,
$\sup_{\eta\in\Ps}\|[x^\nu,\eta]\|$
converges to 0
as $\nu\to\om$.
\end{lem}

\begin{lem}
\label{lem:normalize}
Let $(x^\nu)_\nu\in \ell^\infty(\cM)$
and $\xi\in\cH$ a cyclic and separating vector for $\cM$.
Then the following statements are equivalent:
\begin{enumerate}
\item 
$(x^\nu)_\nu\in \sN_\om$;
\item
For any $\vep>0$ and compact set $\Ps\subs \cH$,
there exist $\de>0$ and $W\in \om$
such that
if $y\in\cM_1$ and $\|y\xi\|+\|\xi y\|<\de$,
then
$\sup_{\eta\in\Ps}(\|x^\nu y\eta\|+\|\eta x^\nu y\|)<\vep$,
and
$\sup_{\eta\in\Ps}(\|yx^\nu\eta\|+\|\eta yx^\nu\|<\vep$
for all $\nu\in W$.
\end{enumerate}
\end{lem}
\begin{proof}
(1)$\Rightarrow$(2).
Suppose on the contrary
that
there exist $\vep>0$ and a compact set
$\Ps\subs\cH$
such that
for any $n\in\N$,
there exists $y_n\in\cM_1$ with
$\|y_n\xi\|+\|\xi y_n\|<1/n$,
but
the following set belongs to $\om$:
\[
A_n:=
\left\{\nu\in\N\mid
\sup_{\eta\in\Ps}(\|x^\nu y_n\eta\|+\|\eta x^\nu y_n\|)
+
\sup_{\eta\in\Ps}(\|y_nx^\nu\eta\|+\|\eta y_nx^\nu\|)
\geq
\vep
\right\}.
\]
Let $W_0:=\N$ and $W_n:=A_1\cap\cdots\cap A_n\cap[n,\infty)$
for $n\geq1$.
We may and do assume that $W_n\supsetneq W_{n+1}$.
For $\nu\in W_n\setminus W_{n+1}$,
we set $z^\nu:=y_n$.
Then $(z^\nu)_\nu\in\sT_\om$,
and $(x^\nu z^\nu)_\nu, (z^\nu x^\nu)_\nu\in\sT_\om$
since $(x^\nu)_\nu\in\sN_\om$.
Nevertheless, we have
\[
\sup_{\eta\in\Ps}(\|x^\nu z^\nu\eta\|+\|\eta x^\nu z^\nu\|)
+
\sup_{\eta\in\Ps}(\|z^\nu x^\nu\eta\|+\|\eta z^\nu x^\nu\|)
\geq
\vep
\quad\mbox{for all }\nu\in\N,
\]
which is a contradiction to Lemma \ref{lem:cpcttriv}.

(2)$\Rightarrow$(1).
This implication is trivial.
\end{proof}

The following result is probably well-known for experts
(see \cite[Lemma 2.11]{Co-almost} for example),
but we give a proof for readers' convenience.

\begin{lem}
\label{lem:tensorBH}
Let $\cM$ be a separable von Neumann algebra
and $\cQ$ a separable type I factor.
Put $\cN:=\cM\oti\cQ$.
Then $\cN^\om\cong\cM^\om\oti\cQ$ and $\cN_\om\cong\cM_\om\oti\C$,
naturally.
\end{lem}
\begin{proof}
Let us use the notations
$\sT_\om(\cN)$, $\sN_\om(\cN)$, $\sT_\om(\cM)$ and $\sN_\om(\cM)$
to distinguish $\sT_\om$ and $\sN_\om$ of $\cN$ and $\cM$.
We prove the following claim.

\begin{clam}
$(x^\nu)_\nu\in\sN_\om(\cM)$
if and only if
$(x^\nu\oti1)_\nu\in\sN_\om(\cN)$.
\end{clam}
\begin{proof}[Proof of Claim.]
The ``if'' part is trivial.
We show the ``only if'' part.
Suppose that
$(x^\nu)_\nu\in\sN_\om(\cM)$.
Let $(y^\nu)_\nu\in\sT_\om(\cN)$.
Let $\th_1\in(\cM_*)_+$ and $\th_2\in(\cQ_*)_+$
be faithful states.
Then
$\|y^\nu (x^\nu\oti1)\|_{\th_1\oti \th_2}^2
=
\th_1((x^\nu)^* (\id\oti\th_2)((y^\nu)^*y^\nu)x^\nu)
$.
Since $(\id\oti\th_2)((y^\nu)^*y^\nu)^{1/2}\in\sT_\om(\cM)$,
it turns out that
$\|y^\nu (x^\nu\oti1))\|_{\th_1\oti \th_2}\to0$
as $\nu\to\om$.
Thus $y^\nu (x^\nu\oti1)\to0$ strongly as $\nu\to\om$.
In a similar way,
we can show that $y^\nu(x^\nu\oti1)$ and $(x^\nu\oti1)y^\nu$
converges to 0 in the strong* topology as $\nu\to\om$.
Thus $(x^\nu\oti1)_\nu\in\sN_\om(\cN)$.
\end{proof}

Let us consider the inclusion
$\cQ\subs\cN^\om$.
Since $\cQ$ is a type I factor,
we have the tensor product decomposition
$\cN^\om=(\cQ'\cap \cN^\om)\vee\cQ\cong(\cQ'\cap \cN^\om)\oti\cQ$.
Let $p$ be a minimal projection of $\cQ$.
Let $x\in \cQ'\cap \cN^\om$ and
$(a^\nu)_\nu$ its representing sequence.
Take $x^\nu\in\cM$ with $(1\oti p)a^\nu(1\oti p)=x^\nu\oti p$.
Then $x(1\oti p)=(1\oti p)x(1\oti p)=\pi_\om((x^\nu\oti p)_\nu)$.
Since $(x^\nu\oti p)_\nu\in\sN_\om(\cN)$,
$(x^\nu)_\nu\in\sN_\om(\cM)$.

By the claim above,
we can consider the element $\pi_\om((x^\nu\oti1)_\nu)\in\cN^\om$.
Hence we have
$x(1\oti p)=\pi_\om((x^\nu\oti1)_\nu)(1\oti p)$.
Since the normal $*$-homomorphism
$\cQ'\cap \cN^\om\ni y\mapsto y(1\oti p)\in (\cQ'\cap\cN^\om)_p$
is faithful,
we have
$x=\pi_\om((x^\nu\oti1)_\nu)$.

Thus we obtain the natural $*$-homomorphism
$\Ph\col\cQ'\cap \cN^\om\to\cM^\om$
defined by
$\Ph(x)=\pi_\om((x^\nu)_\nu)$.
The faithfulness of $\Ph$ is trivial.
The claim above implies the surjectivity of $\Ph$.
Hence $\Ph$ is an isomorphism,
and we obtain an isomorphism
$\Ps\col \cN^\om\ra\cM^\om\oti\cQ$.

Since $\cN_\om\subs \cQ'\cap \cN^\om$,
$\Ph$ maps $\cN_\om$ into $\cM^\om\oti\C$.
Then it is immediately verified that
the image is precisely equal to $\cM_\om\oti\C$.

We verify the naturality of $\Ps$ as follows.
Let $\{e_{ij}\}_{i,j\in I}$ be a system of matrix units
of $\cQ$
such that $e_{ii}$ are minimal projections
and $\sum_i e_{ii}=1$.
Let $x=\pi_\om((x^\nu)_\nu)\in\cN^\om$.
For each $\nu$,
we have the decomposition
$x^\nu=\sum_{i,j}x_{ij}^\nu\oti e_{ij}$
with $x_{ij}^\nu\in\cM$.
It is easy to see that
$(x_{ij}^\nu)_\nu\in\sN_\om(\cM)$
for all $i,j\in I$.

Then
$\Ps\left(
\pi_\om((x_{ij}^\nu\oti e_{ij})_\nu)
\right)
=
\pi_\om((x_{ij}^\nu)_\nu)\oti e_{ij}$.
For a diagonal
finite rank projection $q\in \cQ$,
we obtain
\[
\Ps((1\oti q)x(1\oti q))
=(1\oti q)
\cdot
\big{(}\sum_{i,j}\pi_\om((x_{ij}^\nu)_\nu)\oti e_{ij}
\big{)}
\cdot(1\oti q)
\]
Since $\|\Ps((1\oti q)x(1\oti q))\|\leq\|x\|$ for any $q$,
the operator
$\sum_{i,j}\pi_\om((x_{ij}^\nu)_\nu)\oti e_{ij}$
is norm bounded.
Hence letting $q\to1$,
we have
\[
\Ps(x)=\sum_{i,j}\pi_\om((x_{ij}^\nu)_\nu)\oti e_{ij}.
\]
\end{proof}

\subsection{Ultraproduct of reduced von Neumann algebras}

Let $\cM$ be a von Neumann algebra.
For a projection $p\in\cM$,
we denote by $\cM_p$ the reduced von Neumann algebra.

\begin{lem}
\label{lem:reduced}
The following properties hold:
\begin{enumerate}
\item
$\sT_\om(\cM_p)
=\{(px^\nu p)_\nu\mid (x^\nu)_\nu\in\sT_\om(\cM)\}
\subs\sT_\om(\cM)$;

\item
$\sN_\om(\cM_p)
=\{(px^\nu p)_\nu\mid (x^\nu)_\nu\in\sN_\om(\cM)\}
\subs\sN_\om(\cM)$;

\item
$\sC_\om(\cM_p)=\{(px^\nu p)_\nu\mid (x^\nu)_\nu\in\sC_\om(\cM)\}$.
\end{enumerate}
\end{lem}
\begin{proof}
(1).
It is trivial.

(2).
It suffices to show that
the inclusion
$\sN_\om(\cM_p)\subs\sN_\om(\cM)$
because the others are clear.
Let $(x^\nu)_\nu\in\sN_\om(\cM_p)$
and $(y^\nu)_\nu\in\sT_\om(\cM)$.
Let $\xi\in\cH$.
Since $p|y^\nu|^2 p\to0$ in the strong$*$
topology as $\nu\to\om$,
we obtain
\[
\|y^\nu x^\nu\xi\|^2
=
\langle p|y^\nu|^2 px^\nu\xi,x^\nu\xi\rangle
\to0.
\]
Hence $(x^\nu)_\nu\in\sN_\om$.

(3).
It is clear that
the right hand side is contained in the left.
We will show the converse inclusion.
Let $(x^\nu)_\nu\in\sC_\om(\cM_p)$
and $\vph\in\cM_*$.
In the following,
we assume that $\|x^\nu\|\leq1$.
Put $q:=1-p$.
For any $y\in\cM_q$,
we have
\[
[\vph,x^\nu+y]
=
[p\vph p,x^\nu]
+
q\vph x^\nu-x^\nu\vph q
+
[q\vph q,y]
+
p\vph y-y\vph p.
\]
Thus it suffices to
show that
there exists $(y^\nu)_\nu\in\sC_\om(\cM_q)$
such that
\begin{equation}
\label{eq:qxnuy}
\lim_{\nu\to\om}
\|q\vph x^\nu-y^\nu\vph p\|=0,
\quad
\lim_{\nu\to\om}
\|x^\nu\vph q-p\vph y^\nu\|=0
\quad
\mbox{for all }
\vph\in\cM_*.
\end{equation}
If this is the case,
then indeed we obtain
$(x^\nu+y^\nu)_\nu\in\sC_\om(\cM)$
and
$x^\nu=p(x^\nu+y^\nu)p$.

Take a maximal orthogonal family
of projections
$\{q_i\}_{i\in I}$
such that $q_i\leq q$
and $q_i\preceq p$.
We let $r:=q-\sum_i q_i$.
Then there exist
$z_1,z_2\in Z(\cM)^{\rm P}$
such that
$z_1+z_2=1$
and
$rz_1\succeq pz_1$
$rz_2\preceq pz_2$.
By maximality,
we deduce $pz_1=0$ and $rz_2=0$.
Hence $pz_2=p$.
Since
$z_2 x^\nu=x^\nu=x^\nu z_2$,
we get
$q\vph x^\nu= q\vph z_2 x^\nu=qz_2 \vph x^\nu$,
and likewise,
$x^\nu\vph q=x^\nu\vph qz_2$
for all $\nu$ and $\vph$.
Hence
we may and do assume that
$z_2=1$, that is, $r=0$.
Then $q=\sum_i q_i$.

For each $i$,
take $v_i\in\cM^{\rm PI}$
such that
$q_i=v_i^*v_i$
and
$v_iv_i^*\leq p$.
Set
$y^\nu:=\sum_i v_i^* x^\nu v_i$.
Then $\|y^\nu\|\leq \|x^\nu\|\leq1$,
and $(y^\nu)_\nu\in \ell^\infty(\cM_q)$.
We will check that
$(y^\nu)_\nu\in\sC_\om(\cM_q)$.
Let $\vep>0$ and $\vph\in \cM_*$ a state.
Take a finite subset $J\subs I$
such that
$\vph(q-q_J)<\vep$
with $q_J:=\sum_{i\in J}q_i$.
Then
\begin{align*}
\|q\vph q-q_J \vph q_J\|
&\leq
\|q_J\vph\cdot(q-q_J)\|
+
\|(q-q_J)\vph q\|
\\
&\leq
2\vph(q-q_J)^{1/2}<2\vep^{1/2}.
\end{align*}
Using this,
we obtain
\begin{align*}
\|y^\nu q\vph q-q\vph q y^\nu\|
&\leq
2\|y^\nu\|\|q\vph q-q_J\vph q_J\|
+
\|y^\nu q_J\vph q_J-q_J\vph q_J y^\nu\|
\\
&\leq
4\vep^{1/2}
+
\sum_{i,j\in J}
\|v_i^* x^\nu v_i\vph q_j
-q_j\vph v_i^* x^\nu v_i\|
\\
&=
4\vep^{1/2}
+
\sum_{i,j\in J}
\|v_i^* [x^\nu, v_i\vph v_j^*]v_j\|
\\
&\leq
4\vep^{1/2}
+
\sum_{i,j\in J}
\|[x^\nu, v_i\vph v_j^*]\|.
\end{align*}
Since $\|[x^\nu, v_i\vph v_j^*]\|\to0$
as $\nu\to\om$,
we have
$\|y^\nu q\vph q-q\vph q y^\nu\|<5\vep^{1/2}$
for $\nu$ being close to $\om$.
Thus $(y^\nu)_\nu\in\sC_\om(\cM_q)$.

Now we will check (\ref{eq:qxnuy}).
Using $\vph(q-q_J)<\vep$,
we have
\begin{align*}
\|q\vph x^\nu-y^\nu\vph p\|
&\leq
\|q\vph x^\nu-q_J\vph x^\nu\|
+
\|y^\nu(q_J-q)\vph p\|
+
\sum_{i\in J}
\|q_i \vph x^\nu-v_i^* x^\nu v_i\vph p\|
\\
&<
2\vep^{1/2}
+
\sum_{i\in J}
\|v_i^*[v_i\vph p,x^\nu]\|.
\end{align*}
Hence for $\nu$ being close to $\om$,
we have
$\|q\vph x^\nu-y^\nu\vph p\|<3\vep^{1/2}$.
Similarly,
$\|x^\nu\vph q-p\vph y^\nu\|<3\vep^{1/2}$
Thus (\ref{eq:qxnuy}) holds.
\end{proof}

The previous lemma implies the following
result.

\begin{prop}
Let $\cM$ be a von Neumann algebra
and $p\in\cM^{\rm P}$.
Then the ultraproduct von Neumann algebras
$(\cM_p)^\om$ and $(\cM_p)_\om$
are realized in $\cM^\om$ as follows:
\[
(\cM_p)^\om=(\cM^\om)_p,
\quad
(\cM_p)_\om=(\cM_\om)_p.
\]
\end{prop}

\section{Flows on ultraproduct von Neumann algebras}

Let $\al,\be$ be flows on a von Neumann algebra $\cM$.
Assume that $\al_t\be_t^{-1}\in\oInt(\cM)$ for each $t\in\R$.
Then, as will be shown in Lemma \ref{lem:adBorel},
we can take a Borel unitary path
$u$ such that $\Ad u(t)\circ\al_t$ is close to $\be_t$
on a closed interval.
The path may be arranged to be strongly continuous
with a certain assumption on $\cM$
(see Proposition \ref{prop:appro}).
However, we do not know whether this is true
for a general von Neumann algebra.
Therefore, we have to treat a Borel unitary path,
and a Borel cocycle action.

When one classifies flows,
an analysis of them on an ultraproduct von Neumann algebra
shall be inevitable.
Nevertheless,
a flow
is usually acting on $\cM_\om$ discontinuously,
which is the most significant
difference
from discrete group actions.
One way to treat a flow or a Borel map on $\cM^\om$
is to collect elements
which behave continuously by the given flow.
However,
the continuity is insufficient
in lifting a continuous or
Borel path from $\cM^\om$ to $\cM$.
As a result,
we have to think of a much smaller von Neumann
subalgebra in $\cM^\om$
that is called the $(\al,\om)$-equicontinuous part
(see Definition \ref{defn:equicont}).

\subsection{$\om$-equicontinuity}

\begin{defn}
\label{defn:metequicont}
Let $(\Omega,d)$ be a metric space
and $\{x^\nu\col \Omega\ra \cM\}_{\nu\in\N}$
a family of maps.
We will say that
$\{x^\nu\}_\nu$
is \emph{$\om$-equicontinuous}
if
for any $\vep>0$ and finite set $\Ph\subset \cH$,
there exist $\de>0$
and $W\in\om$
such that
for
all $s,t\in \Omega$
with $d(s,t)<\de$,
$\nu\in W$
and
$\xi\in\Ph$,
we have
\[
\|(x^\nu(s)-x^\nu(t))\xi\|<\vep,
\quad
\|\xi(x^\nu(s)-x^\nu(t))\|<\vep.
\]
\end{defn}

Several statements in this paper
can be replaced
with normal functionals for vectors
in a standard Hilbert space.
We should note that the $\om$-equicontinuity
does not necessarily require the continuity of each $x^\nu$.

\begin{lem}
\label{lem:eqctcpt}
Let $(\Om,d)$ be a metric space
and $\{x^\nu\col \Om\ra\cM\}_\nu$
a family of uniformly bounded maps,
that is,
$\sup_{t\in \Om,\nu\in\N}\|x^\nu(t)\|<\infty$.
Then the following statements are equivalent:
\begin{enumerate}
\item
$\{x^\nu\col \Om\ra\cM\}_\nu$ is $\om$-equicontinuous;

\item
For any $\vep>0$
and compact set $\Ps\subs\cH$,
there exist $\de>0$ and $W\in\om$
such that
for all $s,t\in \Om$
with $d(s,t)<\de$,
$\nu\in W$ and $\xi\in \Ps$,
we have
\[
\|(x^\nu(s)-x^\nu(t))\xi\|<\vep,
\quad
\|\xi(x^\nu(s)-x^\nu(t))\|<\vep;
\]

\item
Let $\xi_0\in\cH$ be a cyclic and separating
vector for $\cM$.
For any $\vep>0$,
there exist $\de>0$ and $W\in\om$
such that
for all $s,t\in \Om$
with $d(s,t)<\de$
and
$\nu\in W$,
we have
\[
\|(x^\nu(s)-x^\nu(t))\xi_0\|<\vep,
\quad
\|\xi_0(x^\nu(s)-x^\nu(t))\|<\vep.
\]
\end{enumerate}
\end{lem}
\begin{proof}
(1)$\Rightarrow$(2).
Take $C>0$
with $C\geq \sup_{t,\nu}\|x(t)^\nu\|$.
Let $\Ps$ and $\vep$ be given.
Choose
$\{\xi_i\}_{i=1}^N$ in $\Ps$
such that
for any $\xi\in \Ps$,
there exists $\xi_i$
such that $\|\xi-\xi_i\|<\vep/4C$.
Using the $\om$-equicontinuity of $x^\nu$,
we can take $\de>0$ and $W\in\om$
such that
for any $s,t\in \Om$ with $d(s,t)<\de$,
$\nu\in W$ and $i=1,\dots,N$,
we have
\[
\|(x^\nu(s)-x^\nu(t))\xi_i\|<\vep/2,
\quad
\|\xi_i(x^\nu(s)-x^\nu(t))\|<\vep/2.
\]
Then it is clear that
these $\de$ and $W$ are desired ones.

(2)$\Rightarrow$(3).
This implication is trivial.

(3)$\Rightarrow$(1).
Let $\vep>0$
and $\Ph:=\{\xi_i\mid i=1,\dots,N\}\subs\cH$.
Take $C>0$ as $C\geq\sup_{t,\nu}\|x^\nu(t)\|$.
Let $a_i, b_i\in\cM$
such that $\|\xi_i-\xi_0 a_i\|<\vep/4C$
and $\|\xi_i-b_i\xi_0\|<\vep/4C$.
Set $M:=\max\{\|a_i\|,\|b_i\|,1\mid i=1,\dots,N\}$.
By our assumption,
there exist $\de>0$ and $W\in\om$
such that
for all $s,t\in \Om$ with $d(s,t)<\de$
and $\nu\in W$,
we have
\[
\|(x^\nu(s)-x^\nu(t))\xi_0\|<\vep/2M,
\quad
\|\xi_0(x^\nu(s)-x^\nu(t))\|<\vep/2M.
\]
Then
\begin{align*}
\|(x^\nu(s)-x^\nu(t))\xi_i\|
&\leq
\|(x^\nu(s)-x^\nu(t))(\xi_i-\xi_0 a_i)\|
+
\|(x^\nu(s)-x^\nu(t))\xi_0 a_i\|
\\
&\leq
2C\cdot \vep/4C+M\cdot \vep/2M=\vep.
\end{align*}
Similarly, we obtain
$\|\xi_i(x^\nu(s)-x^\nu(t))\|<\vep$.
Thus $\{x^\nu\}_\nu$ is $\om$-equicontinuous.
\end{proof}

The following result is frequently used in this paper.

\begin{lem}
\label{lem:uniconverge}
Let $(\Om,d)$ be a metric space and $E\subs \Om$
a relatively compact set.
Let $\{f^\nu\col \Om\ra \C\}_\nu$ be
a family of functions.
Suppose that
$\{f^\nu\col E\ra \C\}_\nu$ is $\om$-equicontinuous.
Then the convergence
$\lim_{\nu\to\omega}f^\nu(t)$
is uniform on $E$.
\end{lem}
\begin{proof}
Put $F(t):=\lim_{\nu\rightarrow\omega}f^\nu(t)$
for $t\in E$.
For $\vep>0$,
take $\de>0$ and $W_1\in\om$
such that for all $s,t\in E$ with $d(s,t)<\de$
and $\nu\in W_1$,
we have
$|f^\nu(s)-f^\nu(t)|<\vep$.
Letting $\nu\to\om$,
we obtain $|F(s)-F(t)|\leq \vep$
which shows the uniform continuity of $F$ on $E$.

Let us keep $\vep$, $\de$ and $W_1$
introduced above.
Since $E$ is relatively compact,
there exists a finite set $E_0\subset E$
such that 
$E\subs\bigcup_{t\in E_0}B(t,\delta)$,
where
$B(t,\delta):=\{s\in \Om\mid d(s,t)<\delta\}$. 
Then
there exists $W_2\in \omega$ such that
for all $s\in E_0$ and $\nu\in W_2$,
we have
$|f^\nu(s)-F(s)|
<\vep$.
Let $t\in E$,
and take $s\in E_0$ with $d(s,t)<\delta$,
Then for all $\nu\in W_1\cap W_2$, we have
\begin{align*}
|f^\nu(t)-F(t)|
&\leq
|f^\nu(t)-f^\nu(s)|
+
|f^\nu(s)-F(s)|
+|F(s)-F(t)|
\\
&<
3\vep.
\end{align*}
Thus we are done.
\end{proof}

\subsection{Borel maps and flows}
\begin{defn}
\label{defn:al-om}
Let $\al\col\R\ra\Aut(\cM)$
be a Borel map.
An element $(x^\nu)_\nu\in\ell^\infty(\cM)$
is said to be
\emph{$(\al,\om)$-equicontinuous}
if
for any Borel set $E\subs\R$
with $\mu(E)<\infty$
and $\vep>0$,
there exists
a compact $K\subs E$ such that
\begin{itemize}
\item 
$\al|_K$ is continuous;

\item
$\mu(E\setminus K)<\vep$;

\item
the family
$\{K\ni t\mapsto \al_t(x^\nu)\in\cM\}_\nu$
is $\om$-equicontinuous.
\end{itemize}
We denote by $\mathscr{E}_\al^\om$
the set of $(\al,\om)$-equicontinuous
sequences.
\end{defn}

In the definition above,
$\mu$ denotes the Lebesgue measure on $\R$
with $\mu([0,1])=1$.
Lemma \ref{lem:Lusin}
implies that
for any Borel set $E\subs \R$,
we can take a compact $K\subs E$
satisfying
the first and second conditions above.
The reason why we must consider the third is
to make the stability holds
with respect to a perturbation of a cocycle action
(Lemma \ref{lem:omeqstab}).
To flows,
the following characterization
is useful.

\begin{prop}
\label{prop:flow-equi}
Let $\al$ be a flow on $\cM$
and $(x^\nu)_\nu\in\ell^\infty(\cM)$.
Let $\xi\in\cH$ be a cyclic and separating
vector for $\cM$.
Then the following statements are
equivalent:
\begin{enumerate}

\item
$(x^\nu)_\nu$ is $(\al,\om)$-equicontinuous;

\item
For any $\vep>0$,
there exist $\de>0$ and $W\in\om$
such that
if $|t|<\de$ and $\nu\in W$,
we have
\[
\|(\al_t(x^\nu)-x^\nu)\xi\|
+
\|\xi(\al_t(x^\nu)-x^\nu)\|
<\vep.
\]

\item
For any $\vep>0$
and compact set $\Ps\subs\cH$,
there exist $\de>0$ and $W\in\om$
such that
if $|t|<\de$ and $\nu\in W$,
we have
\[
\sup_{\eta\in\Ps}
(\|(\al_t(x^\nu)-x^\nu)\eta\|
+
\|\eta(\al_t(x^\nu)-x^\nu)\|)
<\vep.
\]

\item
For any $T>0$,
$\{[-T,T]\ni t\mapsto \al_t(x^\nu)\in\cM\}_\nu$
is $\om$-equicontinuous.
\end{enumerate}
\end{prop}
\begin{proof}
(4)$\Rightarrow$(1)
is trivial.

(1)$\Rightarrow$(2).
Suppose that
$(x^\nu)_\nu$ is $(\al,\om)$-equicontinuous.
For $E:=[0,1]$,
there exists a compact set $K\subs E$
such that $\mu(K)\geq1/2$,
$\al|_K$ is continuous,
and
$\{K\ni t\mapsto \al_t(x^\nu)\in\cM\}_\nu$
is $\om$-equicontinuous.
Since $\mu(K)>0$,
we can find $\de>0$ with
$(-\de,\de)\subs K-K$.

Set $\Ps:=\{\al_s(\xi)\mid s\in K\}$
that is compact.
Then by Lemma \ref{lem:eqctcpt},
for any $\vep>0$,
there exist $\de'>0$ and $W\in\om$
such that
for all $s,t\in K$ with $|s-t|<\de'$
and $\nu\in W$,
we have
\[
\sup_{\eta\in\Ps}
(
\|(\al_s(x^\nu)-\al_t(x^\nu))\eta\|
+
\|\eta(\al_s(x^\nu)-\al_t(x^\nu))\|
)
<\vep.
\]

Then for $s,t\in K$ with $|s-t|<\de'$
and $\nu\in W$,
we have
\begin{align*}
&\|(\al_{s-t}(x^\nu)-x^\nu)\xi\|
+
\|\xi(\al_{s-t}(x^\nu)-x^\nu)\|
\\
&=
\|(\al_s(x^\nu)-\al_t(x^\nu))\al_t(\xi)\|
+
\|\al_t(\xi)(\al_s(x^\nu)-\al_t(x^\nu))\|
\\
&\leq
\sup_{\eta\in\Ps}
(
\|(\al_s(x^\nu)-\al_t(x^\nu))\eta\|
+
\|\eta(\al_s(x^\nu)-\al_t(x^\nu))\|
)
\\
&<\vep.
\end{align*}
Hence
if $|t|<\min(\de,\de')$
and $\nu\in W$,
then
\[
\|(\al_{t}(x^\nu)-x^\nu)\xi\|
+
\|\xi(\al_{t}(x^\nu)-x^\nu)\|
<\vep.
\]

(2)$\Rightarrow$(3).
By compactness,
it suffices to prove (3)
for a finite $\Ps$.
We may and do assume that $\|x^\nu\|\leq1/2$
for all $\nu\in\N$.
Let $\vep>0$.
Take $\de>0$ so that
if $a\in\cM$ satisfies
$\|a\|\leq 1$ and $\|a\xi\|+\|\xi a\|<\de$,
then $\sup_{\eta\in\Ps}(\|a\eta\|+\|\eta a\|)<\vep$.
By (2),
there exist $\de'>0$ and $W\in\om$
such that
if $|t|\leq \de'$
and $\nu\in W$,
then
\[
\|(\al_{t}(x^\nu)-x^\nu)\xi\|
+
\|\xi(\al_{t}(x^\nu)-x^\nu)\|
<\de.
\]
Hence it implies
\[
\sup_{\eta\in\Ps}
(\|(\al_{t}(x^\nu)-x^\nu)\eta\|
+
\|\eta(\al_{t}(x^\nu)-x^\nu)\|)
<\vep.
\]

(3)$\Rightarrow$(4).
Let $\vep>0$,
$T>0$
and $\Ps:=\{\al_t(\xi)\mid |t|\leq T\}$.
By (3),
there exist $\de>0$ and $W\in\om$
such that for all $t$ with $|t|<\de$ and $\nu\in W$,
\[
\sup_{\eta\in\Ps}
(\|(\al_{t}(x^\nu)-x^\nu)\eta\|
+
\|\eta(\al_{t}(x^\nu)-x^\nu)\|)
<\vep.
\]
This implies the following:
\[
\|(\al_{t+s}(x^\nu)-\al_s(x^\nu))\xi\|
+
\|\xi(\al_{t+s}(x^\nu)-\al_s(x^\nu))\|
<\vep
\quad
\mbox{if }
|t|<\de,
\
|s|\leq T,
\
\nu\in W.
\]
Hence we are done.
\end{proof}

\begin{lem}
\label{lem:multi-equicont}
Let $(\Om,d)$ be a compact metric space,
$\{x^\nu\col \Om\ra\cM\}_\nu$
and
$\{y^\nu\col \Om\ra\cM\}_\nu$ families of maps.
Suppose that the following conditions hold:
\begin{itemize}
\item
They are uniformly bounded
and $\om$-equicontinuous;

\item
For each $t\in\Om$,
$(x(t)^\nu)_\nu$ and $(y(t)^\nu)_\nu$
belong to $\sN_\om(\cM)$.
\end{itemize}
Then the family of their multiplications
$\{x^\nu y^\nu\col \Om\ra\cM\}_\nu$
is also $\om$-equicontinuous.
\end{lem}
\begin{proof}
We may and do assume that
$\|x(t)^\nu\|,\|y(t)^\nu\|\leq1$
for all $t\in\Om$ and $\nu\in\N$.
Let $\vep>0$ and $\xi\in\cH$ a cyclic separating vector for $\cM$.
Then there exists $\de>0$ and $W_1\in\om$
such that
if $s,t\in\Om$ satisfies $d(s,t)<\de$ and $\nu\in W_1$,
then
\begin{equation}
\label{eq:ysnu}
\|(y(s)^\nu-y(t)^\nu)\xi\|<\vep.
\end{equation}

Since $\Om$ is compact,
there exists a finite subset $F\subs\Om$
such that each $s\in\Om$ has $t\in F$
with $d(s,t)<\de$.
Then by Lemma \ref{lem:normalize},
we can take
$\de'>0$ and $W_2\in\om$
such that
if $a\in\cM$ with $\|a\|\leq2$
satisfies $\|a\xi\|+\|\xi a\|\leq\de'$,
then $\|ay(s)^\nu\xi\|<\vep$
for $s\in F$.

By $\om$-equicontinuity of $\{x^\nu\}_\nu$,
we take $\de''>0$ and $W_3\in\om$
such that
if $s,t\in\Om$ satisfies $d(s,t)<\de''$ and $\nu\in W_2$,
then
\[
\|(x(s)^\nu-x(t)^\nu)\xi\|
+
\|\xi(x(s)^\nu-x(t)^\nu)\|<\de'.
\]
This implies
\begin{equation}
\label{eq:xsnu}
\|(x(s)^\nu-x(t)^\nu)y(t_0)^\nu\xi\|<\vep
\quad\mbox{for all }t_0\in F.
\end{equation}

Let $s,t\in\Om$ with $d(s,t)<\min(\de,\de'')$
and $\nu\in W_1\cap W_2\cap W_3$.
Take $t_0\in F$ with $d(t,t_0)<\de$.
Then we have
\begin{align*}
&\|(x(s)^\nu y(s)^\nu-x(t)^\nu y(t)^\nu)\xi\|
\\
&\leq
\|x(s)^\nu (y(s)^\nu-y(t)^\nu)\xi\|
+
\|x(s)^\nu (y(t)^\nu-y(t_0)^\nu)\xi\|
\\
&\quad
+
\|(x(s)^\nu-x(t)^\nu) y(t_0)^\nu)\xi\|
+
\|x(t)^\nu (y(t_0)^\nu-y(t)^\nu)\xi\|
\\
&\leq
\|(y(s)^\nu-y(t)^\nu)\xi\|
+
\|(y(t)^\nu-y(t_0)^\nu)\xi\|
\\
&\quad
+
\|(x(s)^\nu-x(t)^\nu) y(t_0)^\nu)\xi\|
+
\|(y(t_0)^\nu-y(t)^\nu)\xi\|
\\
&<
\vep+\vep+\vep+\vep
=4\vep
\quad
\mbox{by }
(\ref{eq:ysnu}), (\ref{eq:xsnu}).
\end{align*}

Likewise,
we can show that there exist $\de'''>0$ and $W_4\in\om$
such that
if $s,t\in\Om$ satisfies $d(s,t)<\de''$
and $\nu\in W_4$,
then
\[
\|\xi(x(s)^\nu y(s)^\nu-x(t)^\nu y(t)^\nu)\|
<4\vep.
\]
Hence we are done.
\end{proof}

\begin{lem}
\label{lem:Ealom}
Let $\al\col\R\ra\Aut(\cM)$ be a Borel map.
Then the following hold:
\begin{enumerate}

\item
If $(x^\nu)_\nu\in\sE_\al^\om$
and $(y^\nu)_\nu\in\ell^\infty(\cM)$
satisfy $(x^\nu-y^\nu)_\nu\in\sT_\om$,
then $(y^\nu)_\nu\in\sE_\al^\om$;

\item
$\sE_\al^\om$ contains $\sT_\om$;

\item
$\sE_\al^\om\cap\sN_\om$
is a C$^{\,*}$-subalgebra of $\ell^\infty(\cM)$.
\end{enumerate}
\end{lem}
\begin{proof}
(1).
Let $E$ be a Borel set
and $K$ a compact set in $E$
such that
$\al|_K$ is continuous
and
$\{K\ni t\mapsto \al_t(x^\nu)\}_\nu$
is $\om$-equicontinuous.
Then for $\vep>0$ and a finite set $\Ph\subset\cH$,
there exist $\de>0$
and $W\in \om$
such that
for $s,t\in K$ with $|s-t|<\de$,
$\nu\in W$
and $\xi\in\Ph$,
we have
\[
\|(\al_s(x^\nu)-\al_t(x^\nu))\xi\|
+
\|\xi(\al_s(x^\nu)-\al_t(x^\nu))\|<\vep/2.
\]
Set
$\Ps:=\{\al_s^{-1}(\xi)\mid \xi\in\Ph,s\in K\}$
that is a compact subset of $\cH$.
We let
$s^\nu
:=\sup_{\eta\in \Ps}
(\|(x^\nu-y^\nu)\eta\|
+\|\eta(x^\nu-y^\nu)\|)$.
By Lemma \ref{lem:cpcttriv},
we have $\lim_{\nu\to\om}s^\nu=0$.
Thus we may and do assume that $s^\nu<\vep/4$
for $\nu\in W$.

Then for $\xi\in \Ph$, $s,t\in K$
with
$|s-t|<\de$ and $\nu\in W$,
we have
\begin{align*}
&\|(\al_s(y^\nu)-\al_t(y^\nu))\xi\|
+
\|\xi(\al_s(y^\nu)-\al_t(y^\nu))\|
\\	
&\leq
\|\al_s(y^\nu-x^\nu)\xi\|
+
\|(\al_s(x^\nu)-\al_t(x^\nu))\xi\|
+
\|\al_t(x^\nu-y^\nu)\xi\|
\\
&\quad
+
\|\xi\al_s(y^\nu-x^\nu)\|
+
\|\xi(\al_s(x^\nu)-\al_t(x^\nu))\|
+
\|\xi\al_t(x^\nu-y^\nu)\|
\\
&=
\|(y^\nu-x^\nu)\al_s^{-1}(\xi)\|
+
\|(\al_s(x^\nu)-\al_t(x^\nu))\xi\|
+
\|(x^\nu-y^\nu)\al_t^{-1}(\xi)\|
\\
&\quad
+
\|\al_s^{-1}(\xi)(y^\nu-x^\nu)\|
+
\|\xi(\al_s(x^\nu)-\al_t(x^\nu))\|
+
\|\al_t^{-1}(\xi)(x^\nu-y^\nu)\|
\\
&<
2s^\nu
+
\vep/2
<
\vep/2+\vep/2=\vep.
\end{align*}

Hence $(y^\nu)_\nu$ is $(\al,\om)$-equicontinuous.

(2).
Let $K\subs\R$ be a compact set
on which $\al$ is continuous.
Then $\Ps:=\{\al_s^{-1}(\xi)\mid s\in K\}$
is compact in $\cH$.
Thus if $(x^\nu)_\nu\in\sT_\om$,
then $s^\nu
:=\sup_{\eta\in\Ps}(\|x^\nu\eta\|+\|\eta x^\nu\|)\to0$
as $\nu\to\om$ by Lemma \ref{lem:cpcttriv}.
Then the statement is clear
because of the inequalities
$\|(\al_s(x^\nu)-\al_t(x^\nu))\xi\|
\leq
2s^\nu$
and
$\|\xi(\al_s(x^\nu)-\al_t(x^\nu))\|
\leq
2s^\nu$.

(3).
It is easy to see that
$\sE_\al^\om$ is a norm closed operator system
in $\ell^\infty(\cM)$.
We show that $\sE_\al^\om\cap\sN_\om$
is closed under multiplication.

Let $E\subs \R$ be a Borel set
with $0<\mu(E)<\infty$
and $0<\ka<1/2$.
Let
$(x^\nu)_\nu, (y^\nu)_\nu\in \sE_\al^\om\cap\sN_\om$.
Take a compact set $K_0\subs E$
such that $\mu(E\setminus K_0)<\ka$,
$\al|_K$ is continuous
and
the maps
$\{K_0\ni t\mapsto\al_t(x^\nu)\}_\nu$,
$\{K_0\ni t\mapsto\al_t(y^\nu)\}_\nu$
are $\om$-equicontinuous.
Hence
$\{K_0\ni t\mapsto\al_t(x^\nu y^\nu)\}_\nu$
is $\om$-equicontinuous
by the previous lemma.
\end{proof}

\begin{lem}
\label{lem:omeqstab}
Let $(\al,c)$ be a Borel cocycle action of $\R$.
Then the following statements hold:
\begin{enumerate}

\item
$\sE_\al^\om\cap\sN_\om$ is $\al$-invariant;

\item
Let $(\al^v,c^v)$ be the perturbation
by a Borel unitary path $v\col\R\ra\cM^{\rm U}$.
Then $\sE_\al^\om\cap\sN_\om
=\sE_{\al^v}^\om\cap\sN_\om$.
\end{enumerate}
\end{lem}
\begin{proof}
(1).
Let $E\subs\R$ be a Borel set
with $0<\mu(E)<\infty$.
Let $\vep,\ka>0$ and $(x^\nu)_\nu\in\sE_\al^\om\cap\sN_\om$.
We may and do assume $\|x^\nu\|\leq1$ for all $\nu\in\N$.
Fix $s\in\R$.
Then we can take a compact set $K_1\subs E+s$
such that
$\mu((E+s)\setminus K_1)<\ka$,
$\al|_{K_1}$ is continuous,
and $\{K_1\ni t\mapsto \al_t(x^\nu)_\nu\}_\nu$
is $\om$-equicontinuous.

Next, we take a compact set $K_2\subs E$
such that $\mu(E\setminus K_2)<\ka$,
and the map
$K_2\ni t\mapsto c(t,s)$ is continuous.
Set $K:=(K_1-s)\cap K_2$, which satisfies
$K\subs E$
and
$\mu(E\setminus K)<2\ka$.
Let $\xi\in\cH$ be a cyclic and separating
vector for $\cM$.
We set the following compact set
\[
\Ps:=\{c(t,s)^*\xi \mid t\in K\}
\cup
\{\al_{t+s}^{-1}(c(t,s)^*\xi)
\mid
t\in K
\}.
\]
Take $\de>0$ and $W_1\in\om$
such that
for all $t,t'\in K_1$ with $|t-t'|<\de$
and $\nu\in W_1$,
we have
\begin{equation}
\label{eq:etaPsal}
\sup_{\eta\in\Ps}
\|\eta(\al_t(x^\nu)-\al_{t'}(x^\nu))\|<\vep,
\quad
\sup_{\eta\in\Ps}
\|(\al_t(x^\nu)-\al_{t'}(x^\nu))\eta\|<\vep.
\end{equation}

By Lemma \ref{lem:normalize},
there exist $\de'>0$ and $W_2\in\om$
such that
if $a\in\cM$ with $\|a\|\leq2$
and $\|a\xi\|+\|\xi a\|<\de'$,
then
$\sup_{\eta\in\Ps}(\|ax^\nu\eta\|+\|\eta ax^\nu\|)<\vep$,
and
$\sup_{\eta\in\Ps}(\|x^\nu a\eta\|+\|\eta x^\nu a\|)<\vep$
for all $\nu\in W_2$.

Take $\de''>0$ so that
if $t,t'\in K$ with $|t-t'|<\de''$,
then
\begin{equation}
\label{eq:ctxi}
\|(c(t,s)-c(t',s))\xi\|
+
\|\xi(c(t,s)-c(t',s))\|<\vep,
\end{equation}
\[
\|\al_{t'+s}^{-1}(c(t,s)-c(t',s))\xi\|
+
\|\xi\al_{t'+s}^{-1}(c(t,s)-c(t',s))\|<\de'.
\]

Then for $t,t'\in K$ with $|t-t'|<\min(\de,\de'')$
and $\nu\in W_1\cap W_2$,
we have
\begin{equation}
\label{eq:etaxnu}
\sup_{\eta\in\Ps}
\|\al_{t'+s}^{-1}(c(t,s)-c(t',s))x^\nu \eta\|<\vep,
\end{equation}
and
\begin{align*}
&\|(\al_t(\al_s(x^\nu))-\al_{t'}(\al_s(x^\nu)))\xi\|
\\
&=
\|(c(t,s)\al_{t+s}(x^\nu)c(t,s)^*
-c(t',s)\al_{t'+s}(x^\nu)c(t',s)^*)\xi\|
\\
&=
\|c(t,s)\al_{t+s}(x^\nu)(c(t,s)^*-c(t',s)^*)\xi\|
+
\|c(t,s)(\al_{t+s}(x^\nu)-\al_{t'+s}(x^\nu))
c(t',s)^*\xi\|
\\
&\quad+
\| (c(t,s)-c(t',s))\al_{t'+s}(x^\nu)c(t',s)^*\xi\|
\\
&\leq
\|(c(t,s)^*-c(t',s)^*)\xi\|
+
\sup_{\eta\in\Ps}
\|(\al_{t+s}(x^\nu)-\al_{t'+s}(x^\nu))\eta\|
\\
&\quad+
\|\al_{t'+s}^{-1}(c(t,s)-c(t',s))x^\nu
\al_{t'+s}^{-1}(c(t',s)^*\xi)\|
\\
&<
\vep+\vep
+
\sup_{\eta\in\Ps}
\|\al_{t'+s}^{-1}(c(t,s)-c(t',s))x^\nu \eta\|
\quad
\mbox{by }
(\ref{eq:etaPsal}),
\
(\ref{eq:ctxi})
\\
&<3\vep
\quad
\mbox{by }
(\ref{eq:etaxnu}).
\end{align*}

We can obtain a similar estimate
for $\|\xi(\al_t(\al_s(x^\nu))-\al_{t'}(\al_s(x^\nu)))\|$.
Therefore, $(\al_s(x^\nu))_\nu\in\sE_\al^\om$.

(2).
Let $(x^\nu)_\nu\in\sE_\al^\om\cap\sN_\om$.
Let $E\subs\R$ be a Borel
set with $0<\mu(E)<\infty$.
Take a compact set $K\subs E$
such that
\begin{itemize}
\item 
$\mu(E\setminus K)<\ka$;
\item
$\al,v$ are continuous on $K$;
\item
$\{K\ni t\mapsto \al_t(x^\nu)\}_\nu$
is $\om$-equicontinuous.
\end{itemize}
Then $\{K\ni t\mapsto v_t\al_t(x^\nu)v_t^*\}_\nu$
is $\om$-equicontinuous
by Lemma \ref{lem:multi-equicont}.
\end{proof}

In the following,
we generalize
the $\om$-equicontinuous part of $\cM^\omega$
introduced in \cite[Definition 2.2]{Kawamuro-RIMS}
to a Borel map.

\begin{defn}
\label{defn:equicont}
Let $\al\col\R\ra\Aut(\cM)$ be a Borel map.
We let $\cM_\al^\om$ be the quotient
C$^*$-algebra $(\sE_\al^\om\cap\sN_\om)/\sT_\om$,
and $\cM_{\om,\al}:=\cM_\al^\om\cap \cM_\om$.
We call them the $(\al,\om)$-\emph{equicontinuous parts}
of $\cM^\om$ and $\cM_\om$,
respectively.
\end{defn}

\begin{lem}
The C$^{\,*}$-subalgebras
$\cM_\al^\omega$ and $\cM_{\om,\al}$
are von Neumann subalgebras of $\cM^\omega$
and $\cM_\om$, respectively.
\end{lem}
\begin{proof}
We show the unit ball of $\cM_\al^\om$
is strongly closed
in $\cM^\om$.
Suppose a sequence $X_n\in (\cM_\al^\om)_1$
strongly converges to
$X\in (\cM^\omega)_1$
as $n\to\infty$.
Let $0<\ka<1/2$.
Let $E$ be a Borel set
with $0<\mu(E)<\infty$
and $K_0\subs\R$ a compact set
such that
$\mu(E\setminus K_0)<\ka$
and $\al\col K_0\ra\Aut(\cM)$ is continuous.
Let $\vph\in\cM_*$ be
a faithful state.
Recall the fact that
for any $\ph\in \cM_*$,
the function
$K_0\ni t\mapsto \al_t(\ph^\om)
\in (\cM^\om)_*$
is continuous
since
$\al_t(\ph^\om)=\al_t(\ph)^\om$.
Thus $\Ps:=\{\al_t^{-1}(\vph^\om)\mid t\in K_0\}$
is compact in $(\cM^\om)_*$,
and we have
$\sup_{t\in K_0}
\|X_n-X\|_{\al_t^{-1}(\vph^\om)}^\sharp\to 0$
as $n\to\infty$.

Let $\vep>0$.
Then we can find $n_0$
such that
$\sup_{t\in K_0}
\|X_{n_0}-X\|_{\al_t^{-1}(\vph^\om)}^\sharp
<\vep/3$.
Fix representing sequences
of $X_n$ and $X$,
$(x_n^\nu)_\nu$
and $(x^\nu)_\nu$ with $\|x_n^\nu\|,\|x^\nu\|\leq 1$,
respectively.
Then again by compactness of $\Ps$,
there exists $W_1\in \omega$ such that
\[
\|x_{n_0}^\nu-x^\nu\|_{\al_t^{-1}(\vph)}^\sharp
<\vep/3
\quad
\mbox{for all}\
t\in K,
\
\nu\in W_1.
\]
Since $X_{n_0}\in \cM_\al^\om$, 
there exist a compact set $K_1\subs K_0$,
$0<\delta<1$
and $W_2\in \omega$ such that
$\mu(K_0\setminus K_1)<\ka$,
and
\[
\|\al_s(x_{n_0}^\nu)-\al_t(x_{n_0}^\nu)
\|_{\vph}^\sharp
<\vep/3
\quad
\mbox{for all}\
s,t\in K_1,
\
|s-t|<\de,
\
\nu\in W_2.
\]
Thus for $s,t\in K_1$ with $|s-t|<\delta$
and $\nu\in W_1\cap W_2$,
we have
$\|\alpha_s(x^{\nu})-\al_t(x^{\nu})\|_\vph<\vep$.
This shows that $X\in \cM_\al^\om$
since $\mu(E\setminus K_1)<2\ka$.
Hence $\cM_\al^\om$ is a von Neumann algebra,
and so is $\cM_{\om,\al}=\cM_\al^\om\cap\cM_\om$. 
\end{proof}

We should note
that
$\cM\subset \cM_\al^\om$
and $\cM_{\om,\al}\subset \cM'\cap\cM_\al^\om$.

Suppose that an flow $\al$ fixes $p\in\cM^{\rm P}$.
Denote the reduced flow by $\al^p$.
It is trivial that
$\sE_{\al^p}^\om\subs \sE_{\al}^\om$.
By Lemma \ref{lem:reduced},
we obtain the following result.

\begin{cor}
\label{cor:reduced}
Let $\cM$ be a von Neumann algebra
and $\al$ a flow on $\cM$.
Suppose that
$p\in \cM^{\rm P}$ is fixed by $\al$.
Then the $(\al,\om)$-equicontinuous parts
of $(\cM_p)^\om$ and $(\cM_p)_\om$
are described as follows:
\[
(\cM_p)_{\al^p}^\om=(\cM_\al^\om)_p,
\quad
(\cM_p)_{\om,\al^p}=(\cM_{\om,\al})_p.
\]
\end{cor}

The following result
is a direct consequence of Lemma \ref{lem:omeqstab},
and this shows
that
the $(\al,\om)$-equicontinuous parts
$\cM_\al^\om$ and $\cM_{\om,\al}$
are invariant
under perturbation.

\begin{lem}
\label{lem:cocpertstab}
If $(\al,c)$ be a Borel cocycle action of $\R$
on a von Neumann algebra $\cM$.
and $(\be,d)$ is a perturbation of $(\al,c)$
by a Borel unitary path.
Then $\cM_\al^\om=\cM_{\be}^\om$
and $\cM_{\om,\al}=\cM_{\om,\be}$.
\end{lem}

\subsection{Flows on $\cM_\al^\om$ or $\cM_{\om,\al}$}

\begin{lem}
\label{lem:flowequipart}
Let $\cM$ be a von Neumann algebra.
The following statements hold:
\begin{enumerate}

\item 
If $\al$ is a flow on $\cM$,
then
so is $\al$ on $\cM_\al^\om$;

\item
If $(\al,c)$ is a Borel cocycle action of $\R$ on $\cM$,
then
$\al$ is a flow on $\cM_{\om,\al}$.
\end{enumerate}
\end{lem}
\begin{proof}
(1)
Let $\vph\in\cM_*$ be a faithful state.
Since $\vph^\om$ is faithful,
$\{a\vph^\om \mid a\in\cM_\al^\om\}$
is dense in $(\cM_\al^\om)_*$,
the predual of $\cM_\al^\om$.
Then
\begin{align*}
\|\al_t(a\vph^\om)-a\vph^\om\|_{(\cM_\al^\om)_*}
&\leq
\|\al_t(a)(\al_t(\vph^\om)-\vph^\om)\|_{(\cM_\al^\om)_*}
+
\|(\al_t(a)-a)\vph^\om\|_{(\cM_\al^\om)_*}
\\
&\leq
\|a\|\|\al_t(\vph)-\vph\|_{\cM_*}
+
\|\al_t(a)-a\|_{\vph^\om}.
\end{align*}
If $t\to0$,
the last two terms converge to 0
because
$\R\ni t\mapsto \al_t(a)$ is strongly continuous
for all $a\in\cM_\al^\om$
by Proposition \ref{prop:flow-equi}.

(2).
Note that $\al_s\al_t=\al_{s+t}$
on $Z(\cM)$.
Since the group homomorphism
$\al\col\R\ra\Aut(Z(\cM))$ is the composition
of the Borel map $\al\col\R\ra\Aut(\cM)$
and the restriction $\Aut(\cM)\ra\Aut(Z(M))$,
which is continuous,
$\al$ is a Borel homomorphism
that is in fact continuous
because $\Aut(Z(\cM))$ is Polish.

Let $\vph\in\cM_*$ be a faithful normal state
and $E\col\cM\ra Z(\cM)$ be the conditional expectation
such that $\vph\circ E=\vph$.
We put $\chi:=\vph|_{Z(M)}$.
As in (1),
$\{a\vph^\om\mid a\in \cM_{\om,\al}\}$ is
dense in $(\cM_{\om,\al})_*$.
Using $\vph^\om=\chi\circ\ta^\om$ on $\cM_\om$,
we have
\[
\|\al_t(a\vph^\om)-a\vph^\om\|_{(\cM_{\om,\al})_*}
\leq
\|a\|
\|\al_t(\chi)-\chi\|_{Z(\cM)_*}
+
\|\al_t(a)-a\|_{\vph^\om}.
\]
Since the first term in the right hand side
converges to 0 as $t\to0$,
it suffices to show
$\|\al_t(a)-a\|_{\vph^\om}\to0$
as $t\to0$.

Let $a\in\Meq$ and $(a^\nu)_\nu$ a representing sequence.
Since $\al$ is a flow on $Z(\cM)$,
the set $\{\al_r(\chi)\circ E\mid r\in [0,1]\}$
is compact in $\cM_*$.
Then for any $\vep>0$,
there exist a compact set $K\subs[0,1]$ with $\mu(K)>0$,
$\de>0$
and $W\in\om$
such that
for $s,t\in K$ with $|s-t|<\de$
and $\nu\in W$,
we have
$\|\al_s(a^\nu)-\al_t(a^\nu)\|_{\al_r(\chi)\circ E}<\vep$
for all $r\in[0,1]$.
Letting $\nu\to\om$,
we have
$\|\al_s(a)-\al_t(a)\|_{\al_r(\chi)\circ\ta^\om}
\leq\vep$.
Then
\[
\|\al_{s-t}(a)-a\|_{\vph^\om}
=
\|\al_s(a)-\al_t(a)\|_{\al_t(\chi)\circ\ta^\om}
\leq\vep
\quad
\mbox{for all }
s,t\in K,
\
|s-t|<\de
\]
because
$\al_s\al_t=\al_{s+t}$ on $\cM_{\om,\al}$.
Since the set $K-K$ contains an open neighborhood
of $0$,
there exists $\de'>0$ such that
if $|t|<\de'$,
then
$\|\al_t(a)-a\|_{\vph^\om}\leq\vep$.
Therefore, $\al$ is a flow on $\cM_{\om,\al}$.
\end{proof}

Let $\al$ be a flow on a von Neumann algebra $\cM$.
For $f\in L^1(\R)$ and $x\in\cM$,
we let
$\al_f(x):=\int_\R f(t)\al_t(x)\,dt$.
The following result provides us
with a method of creating elements
which belongs to $\cM_\al^\om$
though those may be trivial
sequences.

\begin{lem}
\label{lem:smoothing}
Let $(x^\nu)_\nu\in \ell^\infty(\cM)$
and $f\in L^1(\R)$.
If $\al$ is a flow,
then the following statements hold:
\begin{enumerate}
\item
$(\al_f(x^\nu))_\nu\in\sE_\al^\om$;

\item
If $(x^\nu)_\nu\in \sE_\al^\om\cap \sN_\om$,
then
$(\al_f(x^\nu))_\nu\in\sE_\al^\om\cap \sN_\om$;

\item
If $(x^\nu)_\nu\in \sC_\om$,
then $(\al_f(x^\nu))_\nu\in\sC_\om$.
\end{enumerate}

\end{lem}
\begin{proof}
(1).
Observe that
$\al_t(\al_f(x^\nu))-\al_f(x^\nu)=\al_{\la_t f-f}(f)$,
where $(\la_t f)(s)=f(s-t)$.
Let $C:=\sup_\nu\|x^\nu\|$.
Then
\[
\|\al_t(\al_f(x^\nu))-\al_f(x^\nu)\|
\leq
C\|\la_t f-f\|_1
\quad
\mbox{for all }
\nu\in\N.
\]
Hence $(\al_f(x^\nu))_\nu\in\sE_\al^\om$
by Lemma \ref{prop:flow-equi}.

(2).
Suppose that
$(x^\nu)_\nu\in \sE_\al^\om\cap \sN_\om$.
Let $(y^\nu)_\nu\in\sT_\om$.
Then
\[
y^\nu\al_f(x^\nu)
=
\int_{-\infty}^\infty f(t)y^\nu\al_t(x^\nu)\,dt
\quad
\mbox{for all }
\nu\in\N.
\]
By Lemma \ref{lem:uniconverge},
$y^\nu\al_t(x^\nu)\to0$
compact uniformly
in the strong topology as $\nu\to\om$.
This implies that $y^\nu\al_f(x^\nu)\to0$.
Hence $(\al_f(x^\nu))_\nu\in \sN_\om$.

(3).
Suppose that $(x^\nu)_\nu$ is $\om$-central.
For $\vep>0$,
take $T>0$
such that
$\|f-f1_{[-T,T]}\|_1<\vep$.
By Lemma \ref{lem:centralcpct},
there exists $W\in \om$
such that
if $\nu\in W$,
then
$\sup_{t\in[-T,T]}
\|x^\nu\al_{-t}(\xi)-\al_{-t}(\xi)x^\nu\|<\vep
$.
Then for any $\xi\in\cH$,
\begin{align*}
\|\al_f(x^\nu)\xi-\xi\al_f(x^\nu)\|
&\leq
\int_\R
|f(t)|
\|\al_t(x^\nu)\xi-\xi\al_t(x^\nu)\|
\,dt
\\
&\leq
\int_{-T}^T
|f(t)|
\|x^\nu\al_{-t}(\xi)-\al_{-t}(\xi)x^\nu\|
\,dt
\\
&\quad
+\int_{[-T,T]^c}
|f(t)|
\|\al_t(x^\nu)\xi-\xi\al_t(x^\nu)\|
\,dt
\\
&\leq
\|f\|_1\vep
+
2C\|\xi\|\vep.
\end{align*}
Hence $(\al_f(x^\nu))_\nu$ is $\om$-central.
\end{proof}

\begin{lem}
\label{lem:alfxnu}
Let $x=\pi_\om((x^\nu)_\nu)\in\cM_\al^\om$
and $f\in L^1(\R)$.
If $\al$ is a flow,
then $\al_f(x)=\pi_\om((\al_f(x^\nu)_\nu))$.
\end{lem}
\begin{proof}
Put $a:=\pi_\om((\al_f(x^\nu)_\nu))$
that belongs to $\cM_\al^\om$ by the previous lemma.
It suffices to show that
$\ta^\om(ay)=\ta^\om(\al_f(x)y)$
for all $y\in \cM_\al^\om$,
where $\al_f(x)$ is well-defined by Lemma \ref{lem:flowequipart}.
On the one hand,
we have
\[
\vph(\al_f(x^\nu)y^\nu)
=\int_{-\infty}^\infty f(t)
\vph(\al_t(x^\nu)y^\nu)\,dt
\quad
\mbox{for all }
\vph\in\cM_*,\
\nu\in\N.
\]
By Lemma \ref{lem:uniconverge},
$\vph(\al_t(x^\nu)y^\nu)\to\vph(\ta^\om(\al_t(x)y))$
compact uniformly as $\nu\to\om$.
Hence
\[
\vph(\ta^\om(ay))
=
\lim_{\nu\to\om}\vph(\al_f(x^\nu)y^\nu)
=\int_{-\infty}^\infty
f(t)\vph(\ta^\om(\al_t(x)y))\,dt.
\]

On the other hand,
the normality of the conditional expectation
$\ta^\om\col \cM_\al^\om\ra\cM$
and the continuity of $\al$ on $\cM_\al^\om$
implies that
$\int_{-\infty}^\infty f(t)\ta^\om(\al_t(x)y)\,dt
=\ta^\om(\al_f(x)y)$.
Hence
$\vph(\ta^\om(ay))=\vph(\ta^\om(\al_f(x)y))$
for any $\vph\in\cM_*$,
and we have
$\ta^\om(ay)=\ta^\om(\al_f(x)y)$.
\end{proof}

\subsection{Connes spectrum of $\al_\om$}
We show the fast reindexation trick
is applicable to our interesting case.
Namely,
we will construct a reindexation map
in the $(\al,\om)$-equicontinuous part $\cM_\al^\om$.
Our proof is almost in parallel
with \cite[Lemma 5.3]{Ocn-act},
but we should be careful of a construction of that
because
a reindexation map constructed in \cite[Lemma 5.3]{Ocn-act}
may not send given elements into $\cM_\al^\om$
nor commute with $\al_t$ for all $t\in\R$.

\begin{lem}[Fast reindexation trick]
\label{lem:FRT}
Let $\al$ be a flow on a von Neumann algebra $\cM$,
and
$F\subs\cM^\om$ and $N\subs\Mequ$
separable von Neumann subalgebras.
Suppose that $N$ is $\al$-invariant.
Then there exists a faithful normal
$*$-homomorphism $\Ph\col N\ra \cM_\al^\om$
with the following properties:
\begin{enumerate}

\item
$\Ph=\id$ on $N\cap \cM$;
\item
$\Ph(N\cap\Meq)\subs F'\cap\Meq$;
\item
$\ta^\om(\Ph(a)x)=\ta^\om(a)\ta^\om(x)$
for all $a\in N$, $x\in F$;
\item
$\al_t\circ\Ph=\Ph\circ\al_t$
on $N$ for all $t\in\R$.
\end{enumerate}
We call such $\Ph$ a fast reindexation map.
\end{lem}
\begin{proof}
Let us introduce the same notations as the proof
of \cite[Lemma 5.3]{Ocn-act}.
We may suppose that $\cM\subs N$.
For $n\in N$,
we take finite subsets $N_n\subs N_{n+1}$
of $N$,
$F_n\subs F_{n+1}$ of $F$,
$M_n\subs M_{n+1}$ of $\cM_*$
and $B_n\subs B_{n+1}$
of $\mathscr{B}:=\{\al_t\mid t\in\Q\}$
such that
\begin{itemize}
\item 
$\widetilde{N}:=\bigcup_n N_n$
is a unital $*$-algebra over $\Q+i\Q$,
weakly dense in $N$;

\item
$\widetilde{N}$ is
globally invariant by
$\mathscr{B}$;

\item
$\widetilde{N}\cap \cM$ is weakly dense in $\cM$;

\item
$\widetilde{N}\cap \cM_{\om,\al}$ is weakly dense
in $N\cap \cM_{\om,\al}$;

\item
$\widetilde{F}:=\bigcup_n F_n$
is weakly dense in $F$;

\item
$\bigcup_n M_n$ is norm dense in $\cM_*$;

\item
$\mathscr{B}=\bigcup_n B_n$.
\end{itemize}

For each $x\in F\cup N$,
we choose a representing sequence $(x^\nu)_\nu$
such that
for all $\nu\in\N$ and $\la\in\C$,
we have
$\|x^\nu\|\leq\|x\|$,
$(x^*)^\nu=(x^\nu)^*$,
$(\la x)^\nu=\la x^\nu$,
and
$(x^\nu)_\nu$ is constant if $x\in \cM$.

Let $\ph\in\cM_*$ be a faithful state.
For each $x\in\cM_\al^\om$
and $n\in\N$,
we find $\de_n(x)>\de_{n+1}(x)>0$
and a neighborhood $W_n(x)\supsetneq W_{n+1}(x)$
of $\om$ in $\N$ such that
for all $y\in\cM_1$ with
$\|y\|_\ph^\sharp<\de_n(x)$,
we have
$
\|x^\nu y\|_\ph^\sharp+\|yx^\nu\|_\ph^\sharp
<1/n
$
for
$\nu\in W_n(x)$.

For $n\in\N$ and $\vep>0$,
take $\ga_{n,\vep}>0$ such that
the following set belongs to $\om$:
\[
E_{n,\vep}
:=\{\nu\in\N\mid \|\al_t(x^\nu)-x^\nu\|_\ph^\sharp<\vep,
\
|t|\leq \ga_{n,\vep},
\
x\in N_n\}
\]

For $n\geq1$,
we choose $p(n)\in \N$
such that $p(n)\geq n$ and
\begin{itemize}
\item
$p(n)\in \bigcap_{x\in N_n}W_n(x)
\cap\bigcap_{m=1}^n E_{m,1/m}$;

\item
$\|x^{p(n)}y^{p(n)}-(xy)^{p(n)}\|_\ph^\sharp<1/n$
for
$x,y \in N_n$;

\item
$\|[x^{p(n)},a^n]\|_\ph^\sharp<1/n$
for
$x\in N_n\cap \cM_{\om,\al}$,
$a\in F_n$;

\item
$|\ps(a^n x^{p(n)})-\ps(a^n\ta^\om(x))|<1/n$
for
$x\in N_n$,
$a\in F_n$,
$\ps\in M_n$;

\item
$\|\be(x^{p(n)})-(\be^\om(x))^{p(n)}\|_\ph^\sharp<1/n$
for
$x\in N_n$,
$\be \in B_n$.
\end{itemize}

Letting $\Ph(x)=\pi_\om((x^{p(n)})_n)$
for $x\in \widetilde{N}$,
we obtain a faithful normal $*$-homomorphism
$\Ph\col N\ra \cM^\om$
which satisfies (1), (2) and (3),
and commutes with $\al_t$ for $t\in\Q$.
We will check that $\Ph(N)$ is contained in
the $(\al,\om)$-equicontinuous part.
Let $\vep>0$ and $x\in N_m$.
Take a large $m_\vep\in\N$ such that
$1/m_\vep<\vep$ and $m\leq m_\vep$.
Then
\[
\{n\in\N\mid n\geq m_\vep,
\|\al_t(x^{p(n)})-x^{p(n)}\|_\ph^\sharp<\vep,
|t|\leq \ga_{m_\vep,1/m_\vep}\}
=[m_\vep,\infty)\cap\N.
\]
Indeed, let $n\geq m_\vep$.
Then $p(n)\in E_{m_\vep,1/m_\vep}$.
It turns out that
$\|\al_t(x^{p(n)})-x^{p(n)}\|_\ph^\sharp<1/m_\vep<\vep$
for all $|t|\leq \ga_{m_\vep,1/m_\vep}$
since $x\in N_{m_\vep}$.
This implies that $(x^{p(n)})_n$
is $(\al,\om)$-equicontinuous,
and $\Ph(x)\in\Mequ$ for $x\in N_m$.
Since $\Ph$ is normal,
we see that $\Ph$ maps $N$ into $\Mequ$.

Then the commutativity
$\Ph\circ\al_t=\al_t\circ\Ph$ holds
for all $t\in\R$
since $\al$ is a flow
on $\cM_\al^\om$ by Lemma \ref{lem:flowequipart}.
\end{proof}

\begin{lem}\label{lem:SpGa}
Let $\al$ be a flow on a von Neumann algebra $\cM$.
Then the following statements hold:
\begin{enumerate}
\item
$\Ga(\al|_{\Meq})\subs\Ga(\al)$;

\item
If $\al$ is centrally ergodic,
then
$\Sp(\al|_{\Meq})=\Ga(\al|_{\Meq})$.
In particular,
$\Sp(\al|_{\Meq})$ is the annihilator
group of $\ker(\al|_{\Meq})$.
\end{enumerate}
\end{lem}
\begin{proof}
(1).
By Lemma \ref{lem:smoothing} and \ref{lem:alfxnu},
if $f\in L^1(\R)$ satisfies
$\al_f=0$ on $\cM$,
then $\al_f=0$ on $\cM_{\om,\al}$.
Hence
$\Sp(\al|_{\cM_{\om,\al}})\subs\Sp(\al)$.
Applying this observation to $\al^e$ with a projection
$e\in\cM^\al$,
we have
$\Sp(\al^e|_{(\cM_e)_{\om,\al^e}})\subs\Sp(\al^e)$.
By Corollary \ref{cor:reduced},
we have the natural identification
$(\Meq)_e=(\cM_e)_{\om,\al^e}$.
Thus
$\Sp(\al|_{(\Meq)_e})\subs\Sp(\al^e)$.

Let $z$ be the central support projection
of $e$ in $\Meq$.
Then $z$ is fixed by $\al$,
and the map $(\Meq)_z\ni x\mapsto xe\in(\Meq)_e$
is an isomorphism.
Obviously, this intertwines the flows coming from $\al$.
Hence,
\[
\Ga(\al|_{\Meq})
\subs
\Sp(\al|_{(\Meq)_z})
=
\Sp(\al|_{(\Meq)_e})\subs\Sp(\al^e).
\]
Since $e$ is arbitrary,
we have
$\Ga(\al|_{\Meq})\subs\Ga(\al)$.

(2).
Let $p\in\Sp(\al|_{\Meq})$,
$\vep>0$ and $T>0$
be given.
Then there exists a non-zero $x\in\Meq$
such that $\|\al_t(x)-e^{ipt}x\|
<\vep\|x\|$
for all $t\in [-T,T]$.
Let $f\in(\Meq)^\al$ be a non-zero projection,
$N:=\{\al_t(x)\mid t\in\R\}''$
and
$F=\{f\}''$.
Take a fast reindexation map
$\Ph\col N\ra F'\cap\Meq$
as in the previous lemma.
Since $\al$ is centrally ergodic,
$\ta^\om(f)\in Z(\cM)^\al=\C$.
This implies
$\|\Ph(a)f\|_2=\|a\|_2\|f\|_2$ for all $a\in N$.
Hence the $*$-homomorphism
$N\ni a\mapsto \Ph(a)f$
is faithful,
and we have $\|\Ph(a)f\|=\|a\|$.
Thus for $t\in[-T,T]$,
we obtain
\begin{align*}
\|\al_t(\Ph(x)f)-e^{ipt}\Ph(x)f\|
&=
\|\Ph(\al_t(x)-e^{ipt}x)f\|
\\
&=
\|\al_t(x)-e^{ipt}x\|
\\
&<\vep\|x\|
=
\vep\|\Ph(x)f\|.
\end{align*}
This means $p\in \Sp(\al^f|_{\Meq})$.
Therefore $p\in \Ga(\al|_{\Meq})$.
\end{proof}

In particular,
if $\al$ is a flow on a factor $\cM$
with $\Ga(\al)=\{0\}$, then $\Ga(\al|_{\Meq})=\{0\}$,
that is, $\al=\id$ on $\Meq$.
We do not know whether the converse holds or not
for injective factors.

\begin{prop}
Let $\al$ be a centrally ergodic flow
on a von Neumann algebra $\cM$.
If $\Ga(\al)=\{0\}$ and $0$ is isolated in $\Sp(\al)$,
then any element in $\cM_{\om,\al}$
is represented by a sequence in $\cM^\al$.
\end{prop}
\begin{proof}
Let $x=\pi_\om((x^\nu)_\nu)\in \cM_{\om,\al}$.
By the previous lemma,
$\al_t(x)=x$ for all $x\in\R$.
Since $0$ is isolated in $\Sp(\al)$,
there exists a non-negative $f\in L^1(\R)$
such that $\al_f$ gives a faithful normal
conditional expectation from
$\cM$ onto $\cM^\al$.
By Lemma \ref{lem:alfxnu},
we have
$x=\al_f(x)=\pi_\om((\al_f(x^\nu))_\nu)$.
\end{proof}

\subsection{Lift of Borel unitary path}
In this subsection,
we solve the problem concerning
a lift of a Borel unitary path
$U\col \R\ra\cM_\al^\om$
in Lemma \ref{lem:Borelpathlift}.
A Borel path $U\col\R\ra\Mequ$ means that
$\{U(t)\mid t\in\R\}$
generates a separable von Neumann subalgebra,
and $U$ is a Borel map into it.

\begin{lem}\label{lem:Log}
Let $\cM$ be a von Neumann algebra,
$\ph\in \cM_*$ a state
and $u\in \cM^{\rm U}$.
Then
$\|e^{\th\Log(u)}-1\|_\ph\leq \sqrt{2} \|u-1\|_\ph^{1/2}$
for $|\th|\leq1$,
where $\Log e^{ix}=ix$ for $-\pi\leq x<\pi$.
\end{lem}
\begin{proof}
Let $u=\int_{-\pi}^{\pi}e^{i\la}\,dE(\la)$
be the spectral decomposition on the torus $\R/2\pi\Z=[-\pi,\pi)$,
and $\vep:=\|u-1\|_{\ph}\leq2$.
Then we have
\[
\vep^2
=\int_{-\pi}^{\pi}|e^{i\la}-1|^2\,d \ph(E(\la))
=\int_{-\pi}^{\pi}4\sin^2(\la/2)\,d \ph(E(\la)).
\]
Thus if we set
$A_\vep:=\{\la\in[-\pi,\pi)\mid \sin(\la/2)\geq \vep^{1/2}\}$,
then we have
$\ph(E(A_\vep))\leq\vep/4$.
Using
$\Log(u)=\int_{-\pi}^{\pi}i\la\,dE(\la)$
and $e^{\th\Log(u)}=\int_{-\pi}^{\pi}e^{i\th\la}\,dE(\la)$,
we have
\begin{align*}
\|e^{\th\Log(u)}-1\|_\ph^2
&=
\int_{-\pi}^{\pi}|e^{i\th\la}-1|^2\,d\ph(E(\la))
=
\int_{-\pi}^{\pi}4\sin^2(\th\la/2)\,d\ph(E(\la))
\\
&=
\int_{A_\vep}4\sin^2(\th\la/2)\,d\ph(E(\la))
+
\int_{A_\vep^c}4\sin^2(\th\la/2)\,d\ph(E(\la))
\\
&\leq
4\ph(E(A_\vep))
+
\int_{A_\vep^c}4\sin^2(\la/2)\,d\ph(E(\la))
\\
&\leq
4\ph(E(A_\vep))
+
\int_{-\pi}^\pi4\sin^2(\la/2)\,d\ph(E(\la))
\\
&=
4\ph(E(A_\vep))
+
\vep^2.
\end{align*}
If $\vep\leq1$,
then $4\ph(E(A_\vep))\leq\vep$.
If $\vep>1$,
then $A_\vep=\emptyset$,
and $\ph(E(A_\vep))=0$.
As a result,
we obtain $\|e^{\th\Log(u)}-1\|_\ph^2\leq2\vep$
in both cases.
\end{proof}

\begin{lem}\label{lem:pathdist}
Let $t_1,\dots,t_n\in\R$ with $t_1<t_2<\dots<t_n$.
Suppose that unitaries $u_1,\dots,u_n\in \cM$ are given.
If for $\vep>0$ and a faithful state $\ph\in\cM_*$,
we have $\|u_j-u_{j+1}\|_\ph^\sharp<\vep$ for $j=1,\dots,n-1$,
then there exists a continuous unitary path
$u\col[t_1,t_n]\ra \cM$
such that $u(t_j)=u_j$ for all $j$,
and $\|u(t)-u(t_j)\|_\ph^\sharp<\sqrt{2}\vep^{1/2}$
for $t\in[t_j,t_{j+1}]$.
If moreover, we have $\|u_i-u_j\|_\ph^\sharp<\vep$
for all $i,j$,
then $\|u(s)-u(t)\|_\ph^\sharp<4\vep^{1/2}$
for all $s,t\in[t_1,t_n]$.
\end{lem}
\begin{proof}
Let $1\leq j\leq n-1$.
Set $v_j(\th):=u_j\exp(\th\Log(u_j^*u_{j+1}))$
for $\th\in[0,1]$.
The previous lemma implies the following:
\[
\|v_j(\th)-u_j\|_\ph
=
\|\exp(\th\Log(u_j^*u_{j+1}))-1\|_\ph
\leq
\sqrt{2}\|u_j-u_{j+1}\|_\ph^{1/2},
\]
and
\begin{align*}
\|v_j(\th)^*-u_j^*\|_\ph
&=
\|e^{\th\Log(u_{j+1}^*u_j)}u_j^*-u_j^*\|_\ph
=
\|u_j^*e^{\th\Log(u_ju_{j+1}^*)}-u_j^*\|_\ph
\\
&=
\|e^{\th\Log(u_ju_{j+1}^*)}-1\|_\ph
\\
&\leq
\sqrt{2}\|u_j^*-u_{j+1}^*\|_\ph^{1/2}.
\end{align*}
Hence
\[
\|v_j(\th)-u_j\|_\ph^\sharp
\leq
\sqrt{2}
(\|u_j-u_{j+1}\|_\ph^\sharp)^{1/2}
<\sqrt{2}\vep^{1/2}.
\]
Then $v_j\col[0,1]\ra\cM^{\rm U}$
is strongly continuous
and $v_j(0)=u_j$ and $v(1)=u_{j+1}$.
By connecting $v_j$'s,
we have a desired path $u(t)$.
The last statement is verified
by using the triangle inequality.
\end{proof}

\begin{lem}[Lift of Borel unitary path]
\label{lem:Borelpathlift}
Let $\al\col\R\ra\Aut(\cM)$ be a Borel map.
Let $U\col\R\ra\Mequ$
be a Borel unitary path.
Then for any Borel set
$E\subs \R$ with $0<\mu(E)<\infty$
and $\vep>0$,
there exist a compact set $K\subs E$
and a sequence $(u(t)^\nu)_\nu$ for $t\in E$
such that
\begin{itemize}
\item
$\pi_\om((u(t)^\nu)_\nu)=U(t)$
for almost every $t\in E$,
and the equality holds for all $t\in K$;

\item
$\mu(E\setminus K)<\vep$;

\item
For all $\nu\in\N$,
the map
$E\ni t\mapsto u(t)^\nu$
is Borel,
and
the map
$K\ni t\mapsto u(t)^\nu$
is strongly continuous;

\item
the family
$\{K\ni t \mapsto u(t)^\nu\in\cM\}_\nu$
is $\om$-equicontinuous.

\end{itemize}
\end{lem}
\begin{proof}
By Lemma \ref{lem:Lusin},
we have
a compact set $K\subs E$
such that
$\mu(E\setminus K)<\vep$
and $U$ is continuous on $K$.
Continuing this process,
we get a mutually disjoint
series of compact sets
$K=K_0,K_1,\dots\subs E$
such that
$\mu(E\setminus \bigcup_j K_j)=0$
and $U$ is continuous on each $K_j$.
By lifting piecewise,
we see that it suffices to show the existence
of a continuous lift
for $U\col K\ra\Mequ$.

We may and do assume that $K\subs [0,1]$
by changing the variable of $U(t)$.
Let $\vph\in\cM_*$ be a faithful state.
For each $t\in K$,
we choose a representing unitary sequence
$(\tilde{U}(t)^\nu)_\nu$ of $U(t)$.
Then for each $k\in\N$,
we can construct by induction
$N_k\in\N$ ($N_0:=1$), $F_k\in\om$ ($F_0:=\N$)
and
a finite set $A_k\subset K$ ($A_0:=\emptyset$)
with the following properties:
\begin{itemize}
\item
If $s,t\in K$
satisfies $|s-t|\leq 1/N_k$,
then
$\|U(s)-U(t)\|_{\vph^\om}^\sharp<1/2k$;

\item
$N_k>N_{k-1}$ and $2/N_k+1/(2N_{k-1})<1/N_{k-1}$;

\item
$[k,\infty)\supset F_{k-1}\supsetneq F_{k}$;

\item
$A_k:=\{a_j^k,b_j^k\}_{j=0}^{N_k-1}\cup A_{k-1}$,
where
\[
a_j^k:=\min[j/N_k,(j+1)/N_k]\cap K,
\quad
b_j^k:=\max[j/N_k,(j+1)/N_k]\cap K;
\]

\item
If $s,t\in A_k$ and $\nu\in F_k$,
then
\begin{equation}\label{eq:Utik}
\|\tilde{U}(s)^\nu-\tilde{U}(t)^\nu\|_\vph^\sharp
\leq\|U(s)-U(t)\|_{\vph^\om}^\sharp
+1/2k.
\end{equation}

\end{itemize}
Note that $\De_j^k:=[j/N_k,(j+1)/N_k]\cap K$
may be empty,
and $a_j,b_j$ are not defined in this case.
Since $|s-t|\leq 1/N_k$
for $s,t\in A_k\cap[a_j^k,b_j^k]$,
we have
\[
\|\tilde{U}(s)^\nu-\tilde{U}(t)^\nu\|_\vph^\sharp
\leq\|U(s)-U(t)\|_{\vph^\om}^\sharp
+1/2k
<1/k
\quad
\mbox{for all }\nu\in F_k.
\]
Applying Lemma \ref{lem:pathdist}
to $A_k\cap[a_j^k,b_j^k]$ for each $j$,
$\tilde{U}(t)^\nu$ and $\vep:=1/k$,
we obtain a continuous unitary path
$U(t)^{k,\nu}$ on $\bigcup_j [a_j^k,b_j^k]$
such that $U(t)^{k,\nu}=\tilde{U}(t)^\nu$
for all $t\in\bigcup_j A_k\cap[a_j^k,b_j^k]=A_k$,
and
\begin{equation}\label{eq:Ut6}
\|U(s)^{k,\nu}-U(t)^{k,\nu}\|_\vph^\sharp
\leq
4/k^{1/2}
\quad
\mbox{for all }
s,t\in[a_j^k,b_j^k],
\ \nu\in F_k.
\end{equation}
Put $u(t)^\nu:=U(t)^{k,\nu}$
for $\nu\in F_k\setminus F_{k+1}$
and $t\in K$.
We show
$\{K\ni t\mapsto u(t)^\nu\}_\nu$ is $\om$-equicontinuous.
Let $s,t\in K$ with $|s-t|<1/2N_k$
and $\nu\in F_k$.
Take $m\geq k$ with $\nu\in F_m\setminus F_{m+1}$.
Let $s_0,t_0\in A_m$ be the nearest
points from $s,t$, respectively.
Then we have
\begin{align*}
&\|u(s)^\nu-u(t)^\nu\|_\vph^\sharp
\\
&\leq
\|U(s)^{m,\nu}-U(s_0)^{m,\nu}\|_\vph^\sharp
+
\|\tilde{U}(s_0)^{\nu}-\tilde{U}(t_0)^{\nu}\|_\vph^\sharp
+
\|U(t_0)^{m,\nu}-U(t)^{m,\nu}\|_\vph^\sharp
\\
&\leq
4/m^{1/2}
+\|\tilde{U}(s_0)^{\nu}-\tilde{U}(t_0)^{\nu}\|_\vph^\sharp
+4/m^{1/2}
\quad
\mbox{by }(\ref{eq:Ut6})
\\
&\leq
8/k^{1/2}
+
\|U(s_0)-U(t_0)\|_{\vph^\om}^\sharp+1/m
\quad
\mbox{by }(\ref{eq:Utik}).
\end{align*}
Since
\[
|s_0-t_0|\leq|s_0-s|+|s-t|+|t-t_0|
\leq 1/N_m+1/(2N_k)+1/N_m\leq 1/N_k,
\]
we have
$\|U(s_0)-U(t_0)\|_{\vph^\om}^\sharp
<1/k$,
and
\[
\|u(s)^\nu-u(t)^\nu\|_\vph^\sharp
\leq 8/k^{1/2}+1/k+1/k
\leq 10/k^{1/2}.
\]
Thus
$\{K\ni t\mapsto u(t)^\nu\}_\nu$ is $\om$-equicontinuous,
and the function
$K\ni t\mapsto \pi_\om((u(t)^\nu)_\nu)\in\cM^\om$
is continuous.
Since $u(t)^\nu=\tilde{U}(t)^\nu$
for all $t\in A_k$ and $\nu\in F_k$,
$\pi_\om((u(t)^\nu)_\nu)=U(t)$
for all $t\in\bigcup_k A_k$.
It is clear that $\bigcup_k A_k$
is dense in $K$,
and we have $\pi_\om((u(t)^\nu)_\nu)=U(t)$
for all $t\in K$.
\end{proof}

We close this section
with the following three lemmas.

\begin{lem}\label{lem:alweqct}
Let $K_1,K_2\subs\R$ be compact sets.
Let $\al\col \R\ra \Aut(\cM)$ be a Borel map
and
$\{w^\nu\col K_2\ra \cM\}_\nu$ a family of
continuous maps.
Suppose that
\begin{itemize}
\item
$\al$ is continuous on $K_1$;
\item
$\{K_1\ni s\mapsto \al_s(w(t)^\nu)\}_\nu$
is $\om$-equicontinuous
for each $t\in K_2$;

\item
$\{K_2\ni t\mapsto w(t)^\nu\}_\nu$
is $\om$-equicontinuous.
\end{itemize}
Then
$\{K_1\times K_2\ni (s,t)\mapsto \al_s(w^\nu(t))\}_\nu$
is $\om$-equicontinuous.
\end{lem}
\begin{proof}
Let $\vep>0$ and $\xi\in \cH$ a cyclic and separating vector
for $\cM$.
Set $\Ps:=\{\al_s^{-1}(\xi)\mid s\in K_1\}$
that is a compact set.
By Lemma \ref{lem:eqctcpt},
there exist $\de>0$ and $W_1\in\om$
such that
for all $t,t'\in K_2$ with $|t-t'|<\de$,
$\nu\in W_1$ and $\eta\in\Ps$,
we have
\begin{equation}
\|(w(t)^\nu-w(t')^\nu)\eta\|<\vep,
\quad
\|\eta(w(t)^\nu-w(t')^\nu)\|<\vep.
\label{eq:pswnu}
\end{equation}

Take $\{t_1,\dots,t_N\}$ in $K_2$
such that each $t\in K_2$ has $t_i$
with $|t-t_i|<\de$.
By the second condition,
there exist $\de'>0$ and
$W_2\in \om$ such that
for all $s,s'\in K_1$ with $|s-s'|<\de'$,
$\nu\in W_2$
and $i=1,\dots,N$,
we have
\begin{equation}
\|(\al_s(w(t_i)^\nu)-\al_{s'}(w(t_i)^\nu))\xi\|
<\vep,
\quad
\|\xi(\al_s(w(t_i)^\nu)-\al_{s'}(w(t_i)^\nu))\|
<\vep.
\label{eq:vphalw}
\end{equation}

Now let $s,s'\in K_1$ and $t,t'\in K_2$
with $|s-s'|<\de'$ and $|t-t'|<\de$.
Take $t_i$ such that
$|t'-t_i|<\de$.
Then for $\nu\in W_1\cap W_2$,
we obtain
\begin{align*}
\|(\al_s(w(t)^\nu)-\al_{s'}(w(t')^\nu))\xi\|
&\leq
\|(\al_s(w(t)^\nu)-\al_{s}(w(t')^\nu))\xi\|
\\
&\quad
+
\|(\al_s(w(t')^\nu)-\al_{s}(w(t_i)^\nu))\xi\|
\\
&\quad
+
\|(\al_s(w(t_i)^\nu)-\al_{s'}(w(t_i)^\nu))\xi\|
\\
&\quad
+
\|(\al_{s'}(w(t_i)^\nu)-\al_{s'}(w(t')^\nu))\xi\|
\\
&
\leq
\|(w(t)^\nu-w(t')^\nu)\al_{s}^{-1}(\xi)\|
\\
&\quad
+
\|(w(t')^\nu-w(t_i)^\nu)\al_{s}^{-1}(\xi)\|
\\
&\quad
+
\vep
\quad
\mbox{by }
(\ref{eq:vphalw})
\\
&\quad
+
\|(w(t_i)^\nu-w(t')^\nu)\al_{s'}^{-1}(\xi)\|
\\
&\leq
4\vep
\quad
\mbox{by }
(\ref{eq:pswnu}).
\end{align*}
Similarly, we obtain
\[
\|\xi(\al_s(w(t)^\nu)-\al_{s'}(w(t')^\nu))\|<4\vep.
\]
Hence we are done.
\end{proof}

\begin{lem}
Let $\al\col\R\ra\Aut(\cM)$
be a Borel map
and $C\subs\R$ a compact set.
Suppose that
$\{C\ni t\mapsto x(t)^\nu\in\cM\}_\nu$
is $\om$-equicontinuous
and
$(x(t)^\nu)_\nu\in\sE_\al^\om$
for all $t\in C$.
Then
for all $\ka>0$
and Borel set $E\subs\R$ with $0<\mu(E)<\infty$,
there exists a compact set
$L\subs E$
such that
\begin{itemize}
\item
$\mu(E\setminus L)<\ka$;

\item
$\al$ is continuous on $L$;

\item
$\{L\ni s\mapsto \al_s(x(t)^\nu)\}_\nu$
is $\om$-equicontinuous
for all $t\in C$.
\end{itemize}
\end{lem}
\begin{proof}
Take an increasing sequence
of finite sets
$C_1\subs C_2\subs\cdots C$
such that
their union is dense in $C$.
Then for each $n\in\N$,
we can find a compact set $L_n\subs E$
such that
\begin{itemize}
\item
$\mu(E\setminus L_n)<\ka/2^{n+1}$;

\item
$\al$ is continuous on $L_n$;

\item
$\{L\ni s\mapsto \al_s(x(t)^\nu)\}_\nu$
is $\om$-equicontinuous
for all $t\in C_n$.
\end{itemize}
Set $L:=\bigcap_n L_n$.
Then $\mu(E\setminus L)\leq\sum_n\ka/2^{n+1}<\ka$,
and $\al$ is continuous on $L$.

We will check the third condition.
Let $\xi\in\cH$ be a cyclic and separating vector.
Let $\vep>0$
and $\Ps:=\{\al_s^{-1}(\xi)\mid s\in L\}$
that is compact.
Then there exist
$\de>0$ and $W\in\om$
such that
if $t,t'\in C$
with
$|t'-t|<\de$
and $\nu\in W$,
then
\begin{equation}
\label{eq:xnuze}
\|(x(t)^\nu-x(t')^\nu)\zeta\|
+
\|\zeta(x(t)^\nu-x(t')^\nu)\|
<\vep
\quad
\mbox{for all }
\zeta\in\Ps.
\end{equation}

Fix $t\in C$ and take $t_0\in C_n$ with $|t-t_0|<\de$.
Then by $(\al,\om)$-equicontinuity,
we have $\de'>0$ and $W'\in\om$
such that
if $s,s'\in L$ with $|s-s'|<\de'$
and $\nu\in W'$,
then
\begin{equation}
\label{eq:alsxt0}
\|(\al_s(x(t_0)^\nu)-\al_{s'}(x(t_0)^\nu))\xi\|
+
\|\xi(\al_s(x(t_0)^\nu)-\al_{s'}(x(t_0)^\nu))\|
<\vep.
\end{equation}

Then for all $s,s'\in L$ with $|s-s'|<\de'$
and $\nu\in W\cap W'$,
\begin{align*}
\|(\al_s(x(t)^\nu)-\al_{s'}(x(t)^\nu))\xi\|
&\leq
\|(\al_s(x(t)^\nu)-\al_{s}(x(t_0)^\nu))\xi\|
\\
&\quad
+
\|(\al_s(x(t_0)^\nu)-\al_{s'}(x(t_0)^\nu))\xi\|
\\
&\quad
+
\|(\al_{s'}(x(t_0)^\nu)-\al_{s'}(x(t)^\nu))\xi\|
\\
&<
\|(x(t)^\nu-x(t_0)^\nu)\al_s^{-1}(\xi)\|
\\
&\quad
+\vep
\quad
\mbox{by }
(\ref{eq:alsxt0})
\\
&\quad
+
\|(x(t_0)^\nu-x(t)^\nu)\al_{s'}^{-1}(\xi)\|
\\
&
<3\vep
\quad
\mbox{by }
(\ref{eq:xnuze}).
\end{align*}
In a similar way,
we obtain
$\|\xi(\al_s(x(t)^\nu)-\al_{s'}(x(t)^\nu))\|
<3\vep$.
Hence $\{L\ni s\mapsto \al_s(x(t)^\nu)\}_\nu$
is $\om$-equicontinuous.
\end{proof}

\begin{lem}\label{lem:Borelpatheq}
Let $(\al,c)$ be a Borel cocycle action of $\R$ on $\cM$.
Suppose that
$U\col\R\ra\Mequ$ is a Borel unitary path.
Then for any $T>0$, $\de>0$ with $0<\de<1$
and finite set $\Ph\subs\cM_*^+$,
there exist a compact set
$K\subs[-T,T]\times[-T,T]$
and
a lift
$(u(t)^\nu)_\nu$
of $U$ as in Lemma \ref{lem:Borelpathlift}
such that
\begin{itemize}
\item
$\mu(K)\geq 4T^2(1-\de)$;

\item
$
\{K\ni (t,s)\mapsto
u(t)^\nu\al_t(u(s)^\nu)c(t,s)(u(t+s)^\nu)^*
\}_\nu$
is $\om$-equicontinuous;

\item
The following limit is the uniform convergence
on $K$ for all $\vph\in\Ph$:
\[
\lim_{\nu\to\om}
\|u(t)^\nu\al_t(u(s)^\nu)c(t,s)(u(t+s)^\nu)^*-1\|_\vph^\sharp
=
\|U(t)\al_t(U(s))c(t,s)U(t+s)^*-1\|_{\vph^\om}^\sharp.
\]
\end{itemize}
\end{lem}
\begin{proof}
Let $\eta:=\de/6$, $k\in\N$.
Take a compact set
$C\subs[-2T,2T]$ for $U(t)$
as in Lemma \ref{lem:Borelpathlift},
that is,
$\mu(C)\geq 4T(1-\eta)$,
$\mu(C\cap[-T,T])\geq 2T(1-\eta)$
and
$\{C\ni t\mapsto u(t)^\nu\}_\nu$
is $\om$-equicontinuous.

By the previous lemma,
we have a compact subset $L\subs [-T,T]$
such that
$\mu(L)\geq 2T(1-\eta)$,
$L\ni t\mapsto \al_t\in\Aut(\cM)$
is continuous
and
$\{L\ni t\mapsto \al_t(u(s)^\nu)\}_\nu$
is $\om$-equicontinuous
for all $s\in C$.
Then
$\mu(C\cap L)
\geq 2T(1-2\eta)$,
and
the family
$\{L\times C\ni(t,s)\mapsto
\al_t(u(s)^\nu)\}_\nu$
is $\om$-equicontinuous
by Lemma \ref{lem:alweqct}.

Next we consider
the Borel map
$[-T,T]^2\ni(t,s)\mapsto c(t,s)\in\cM^{\rm U}$.
Take a compact subset
$M\subs [-T,T]^2$
such that
$\mu(M)\geq 4T^2(1-\eta)$,
and
$c$ is continuous on $M$
as before.

Note that
the map
$C\times C\ni (t,s)\mapsto u(t+s)^\nu$
may not be $\om$-equicontinuous.
Let $f(t,s)=t+s$ on $[-T,T]^2$
and set the compact set
$N:=f^{-1}(C)$.
Then
$\{N\ni (t,s)\mapsto u(t+s)^{\nu}\}_\nu$
is $\om$-equicontinuous,
and we have
\begin{align*}
\mu(N^c\cap[-T,T]^2)
&=
\int_{-T}^T\,dt
\int_{-T}^T\,ds
\,
1_{\{(t,s)\mid t+s\in C^c\cap[-2T,2T]\}}
(t,s)
\\
&=
\int_{-T}^T
\mu((C^c\cap[-2T,2T]-t)\cap[-T,T])
\,dt
\\
&\leq
\int_{-T}^T
\mu(C^c\cap[-2T,2T]-t)
\,dt
\\
&=2T\mu(C^c\cap[-2T,2T])
\leq
8T^2\eta.
\end{align*}

Now we set
the compact subset $K$ in $[-T,T]^2$
as follows:
\[
K:=((C\cap L)\times C)\cap M\cap N.
\]
Then
\begin{align*}
\mu(K^c\cap[-T,T]^2)
&=
\mu\left(
\left(
((C\cap L)\times C)\cap[-T,T]^2\right)^c
\cup M^c
\cup N^c
\right)
\\
&\leq
4T^2-\mu\left(
(C\cap L)
\times
(C\cap[-T,T])
\right)
+
4T^2\eta
+
8T^2\eta
\\
&=
4T^2
-
\mu(C\cap L)
\mu(C\cap[-T,T])
+
12T^2\eta
\\
&\leq
4T^2
-2T(1-2\eta)\cdot 2T(1-\eta)
+
12T^2\eta
\\
&=
4T^2\eta(6-2\eta)
<24T^2\eta
=4T^2\de.
\end{align*}

Then
$\{K\ni (t,s)\mapsto
u(t)^\nu\al_t(u(s)^{\nu})c(t,s)
(u(t+s)^\nu)^*\}_\nu$
is $\om$-equicontinuous by Lemma \ref{lem:multi-equicont},
and
we have
the uniform convergence stated
in Lemma \ref{lem:uniconverge}.
\end{proof}

\section{Rohlin flows}

\subsection{Rohlin flows}
In \cite{Kishi-CMP},
Kishimoto has introduced the notion of
the Rohlin property
for flows on C$^*$-algebras.
This property has been defined
also for finite von Neumann algebras
by Kawamuro \cite{Kawamuro-RIMS}.
Following their works,
we will introduce the Rohlin property
for a Borel cocycle action.

\begin{defn}
Let $(\al,c)$ be a Borel cocycle action of $\R$
on a separable von Neumann algebra $\cM$.
We will say that $\al$ has the \emph{Rohlin property}
if
for any $p\in\R$,
there exists a unitary $v\in\Meq$
such that
$\al_t(v)=e^{ipt}v$ for all $t\in\R$.
\end{defn}

A flow $\al$ with Rohlin property
is simply called a Rohlin flow.
We call the unitary $v$ in the above a Rohlin unitary
for $p\in\R$.
By definition,
$\al_t$ is centrally non-trivial if $t\neq0$.
Therefore,
any full factor does not admit a Rohlin flow.
Several examples are investigated in Section \ref{sec:appl}.

Lemma \ref{lem:SpGa} implies the following result.

\begin{lem}
If $\al$ is a Rohlin flow on a factor,
then $\Ga(\al)=\R$.
\end{lem}

Thus it is natural to ask
if an outer flow with full Connes spectrum
on the injective type II$_1$ factor
has the Rohlin property or not.
This problem has been open so far.
See Section \ref{sect:concl} for related problems.

We remark that there does not exist
a strongly continuous path $\R\ni p\mapsto w_p\in\cM_{\om,\al}$
such that $\al_t(w_p)=e^{ipt}w_p$
when $\cM$ is a factor.
Indeed,
$\ta_\om$ gives an $\al$-invariant inner product on $\cM_{\om,\al}$,
and $\{w_p\}_p$ is an orthonormal system.
In particular, this spans a non-separable Hilbert space.

Lemma \ref{lem:cocpertstab} implies the stability
of the Rohlin property under cocycle perturbation.

\begin{lem}
\label{lem:pertRohstab}
If a Borel cocycle action of $\R$ on a von Neumann algebra
has the Rohlin property,
then so does its any perturbation.
\end{lem}

The following result states a sequence-version of
the definition of the Rohlin property.

\begin{lem}
\label{lem:Rohlin-ptwise}
Let $\al$ be a flow on a von Neumann algebra $\cM$.
Then the following statements hold:
\begin{enumerate}
\item
$\al$ has the Rohlin property;

\item
For any $p\in\R$,
there exists a unitary central sequence
$(v^\nu)_\nu$
such that
$\al_t(v^\nu)-e^{ipt}v^\nu\to0$
compact uniformly
in the strong topology
as $\nu\to\infty$;

\item
For any $p\in\R$,
there exists a unitary central sequence
$(v^\nu)_\nu$
such that
for each $t\in\R$,
one has
$\al_t(v^\nu)-e^{ipt}v^\nu\to0$
in the strong topology
as $\nu\to\infty$.
\end{enumerate}
\end{lem}
\begin{proof}
(1)$\Rightarrow$(2).
Let $p\in\R$.
Take a unitary $v\in\cM_{\om,\al}$
with $\al_t(v)=e^{ipt}v$.
Let $(v^\nu)_\nu$ be a unitary representing sequence of $v$.

Take a compact set $K\subs\R$ with $\mu(K)>0$
such that $K=-K$,
$\al|_K$ is continuous
and $\{K\ni t\mapsto \al_t(v^\nu)\}_\nu$
is $\om$-equicontinuous.

Let $\xi\in\cH$ be a cyclic and separating vector for $\cM$.
Then
it turns out that
$\sup_{s,t\in K}\|(\al_t(v^\nu)-e^{ipt}v^\nu)\al_s(\xi)\|$
converges to 0 as $\nu\to\om$
by Lemma \ref{lem:uniconverge}
and the compactness of $\{\al_s(\xi)\mid s\in K\}$.
By taking an appropriate subsequence,
we may and do assume that
$(v^\nu)_\nu$ is central,
and
$\sup_{s,t\in K}\|(\al_t(v^\nu)-e^{ipt}v^\nu)\al_s(\xi)\|$
converges to 0 as $\nu\to\infty$.
Let $s,t\in K$.
Then
\begin{align*}
\|(\al_{s-t}(v^\nu)-e^{ip(s-t)}v^\nu)\xi\|
&=
\|\al_s(\al_{-t}(v^\nu)-e^{-ipt}v^\nu)\xi\|
+
\|e^{-ipt}(\al_s(v^\nu)-e^{ips}v^\nu)\xi\|
\\
&=
\|(\al_{-t}(v^\nu)-e^{-ipt}v^\nu)\al_{-s}(\xi)\|
+
\|(\al_s(v^\nu)-e^{ips}v^\nu)\xi\|.
\end{align*}
Hence
\[
\lim_{\nu\to\infty}
\sup_{s,t\in K}
\|(\al_{s-t}(v^\nu)-e^{ip(s-t)}v^\nu)\xi\|
=0.
\]
Let $\de>0$ with $(-\de,\de)\subs K-K$.
Then we have
$\sup_{|t|\leq\de}\|(\al_t(v^\nu)-e^{ipt}v^\nu)\xi\|$
converges to 0 as $\nu\to\infty$.
From this fact,
we can deduce that the uniform convergence
on any compact sets.

(2)$\Rightarrow$(3).
This implication is trivial.

(3)$\Rightarrow$(1).
Let $p\in\R$,
and take such a sequence $(v^\nu)_\nu$.
Let $\xi\in\cH$.
Then for each $t\in\R$,
we obtain
\begin{align*}
\|\xi(\al_t(v^\nu)-e^{ipt}v^\nu)\|
&\leq
\|[\xi,\al_t(v^\nu)]\|
+
\|(\al_t(v^\nu)-e^{ipt}v^\nu)\xi\|
+
\|[v^\nu,\xi]\|
\\
&=
\|[\al_{-t}(\xi),v^\nu]\|
+
\|(\al_t(v^\nu)-e^{ipt}v^\nu)\xi\|
+
\|[v^\nu,\xi]\|,
\end{align*}
which converges to 0 since $(v^\nu)_\nu$ is central.
Thus we have the strong$*$ convergence
$\al_t(v^\nu)-e^{ipt}v^\nu\to0$
as $\nu\to\infty$.
Let $f(t)=e^{-ipt}1_{[0,1]}(t)\in L^1(\R)$.
Then we have
\[
\|(\al_f(v^\nu)-v^\nu)\xi\|
\leq
\int_0^1 \|(\al_t(v^\nu)-e^{ipt}v^\nu)\xi\|\,dt,
\]
which converges to 0 as $\nu\to\infty$
by the dominated convergence theorem.
Likewise, we obtain
$\|\xi(\al_f(v^\nu)-v^\nu)\|\to0$ as $\nu\to\infty$.
Thus $(v^\nu)_\nu$ belongs to $\sE_\al^\om\cap \sC_\om$
by Lemma \ref{lem:Ealom} and \ref{lem:smoothing}.
Hence $v:=\pi_\om((v^\nu)_\nu)\in \cM_{\om,\al}$
satisfies $\al_t(v)=e^{ipt}v$.
\end{proof}

\subsection{Invariant approximate innerness}

We investigate a relation between the Rohlin property
and the invariant approximate innerness.

\begin{defn}
\label{defn:IAI}
Let $\al$ be a flow on a von Neumann algebra
$\cM$.
We will say that $\al$ is
\emph{invariantly approximately inner}
if for any $T\in\R$,
there exists a sequence of unitaries
$(w^\nu)_\nu$ in $\cM$
such that
\begin{itemize}
\item 
$\al_T=\lim_{\nu\to\infty}\Ad w^\nu$ in $\Aut(\cM)$;
\item
$\|(\al_t(w^\nu)-w^\nu)\xi\|+\|\xi(\al_t(w^\nu)-w^\nu)\|
\to0$
compact uniformly for $t\in\R$
as $\nu\to\infty$
for all $\xi\in\cH$.
\end{itemize}
\end{defn}

\begin{lem}
\label{lem:IAI-local}
Let $\al$ be a flow
on a von Neumann algebra $\cM$.
Then the following statements are equivalent:
\begin{enumerate}
\item
$\al$ is invariantly approximately inner;

\item
For any $T\in\R$,
there exists a sequence of unitaries
$(w^\nu)_\nu$ in $\cM$
such that
\begin{itemize}
\item 
$\al_T=\lim_{\nu\to\infty}\Ad w^\nu$ in $\Aut(\cM)$;

\item
$\|(\al_t(w^\nu)-w^\nu)\xi\|+\|\xi(\al_t(w^\nu)-w^\nu)\|
\to0$
for each $t\in\R$ and $\xi\in\cH$
as $\nu\to\infty$.
\end{itemize}

\item
For any $T\in\R$,
there exists a unitary
$w\in \cM_\al^\om$
such that
\begin{itemize}
\item
$\al_T(x)=wxw^*$ for all $x\in\cM$;
\item
$\al_t(w)=w$ for all $t\in\R$;
\end{itemize}
\end{enumerate}
\end{lem}
\begin{proof}
(1)$\Rightarrow$(2).
This implication is trivial.

(2)$\Rightarrow$(3).
Take such a sequence $(w^\nu)_\nu$.
Then as in the proof of Lemma \ref{lem:Rohlin-ptwise},
we can show that $(w^\nu)_\nu\in\sE_\al^\om$.
Since $\Ad w^\nu\to \al_T$ in $\Aut(\cM)$,
$(w^\nu)_\nu$ normalizes $\sT_\om$.
Thus we can consider a unitary
$w:=\pi_\om((w^\nu)_\nu)$
in $\cM_\al^\om$
which satisfies the required properties.

(3)$\Rightarrow$(1).
We suppose that the conditions of (2)
are fulfilled.
Let $(w^\nu)_\nu$ be a unitary
representing sequence of $w$.
Let $T>0$, $\vep>0$ and $\Ph\subs\cH$
a finite set.
By $(\al,\om)$-equicontinuity,
there exist $N\in\N$ and $W_1\in\om$
such that
if $s,t\in[-T,T]$ satisfies
$|s-t|<T/N$
and $\nu\in W_1$,
then
\[
\|(\al_s(w^\nu)-\al_t(w^\nu))\xi\|
<\vep,
\quad
\|\xi(\al_s(w^\nu)-\al_t(w^\nu))\|
<\vep
\quad
\mbox{for all }
\xi\in\Ph.
\]

Put $t_j:=jT/N$, $j=-N,\dots,N$.
Since $\al_t(w)=w$,
there exists $W_2\in\om$
such that
if $\nu\in W_2$,
then
\[
\|(\al_{t_j}(w^\nu)-w^\nu)\xi\|
<\vep,
\quad
\|\xi(\al_{t_j}(w^\nu)-w^\nu)\|
<\vep
\quad
\mbox{for all }
j=-N,\dots,N,
\
\xi\in\Ph.
\]

Let $t\in [-T,T]$ and take $t_j$
with $|t-t_j|<T/N$.
If $\nu\in W_1\cap W_2$,
then
\begin{align*}
\|(\al_{t}(w^\nu)-w^\nu)\xi\|
&\leq
\|(\al_{t}(w^\nu)-\al_{t_j}(w^\nu))\xi\|
+
\|(\al_{t_j}(w^\nu)-w^\nu)\xi\|
\\
&<2\vep.
\end{align*}
Likewise,
we obtain
$\|\xi(\al_{t}(w^\nu)-w^\nu)\|<2\vep$
for $\xi\in\Ph$, $t\in[-T,T]$ and $\nu\in W_1\cap W_2$.
Then an appropriate subsequence of $w^\nu$
satisfies the condition of Definition \ref{defn:IAI}.
\end{proof}

\begin{lem}
\label{lem:invappinn-stab}
The invariant approximate innerness
is stable under cocycle perturbation.
\end{lem}
\begin{proof}
Let $\al$ be an invariantly approximately inner
flow on a von Neumann algebra $\cM$.
Let $v$ be an $\al$-cocycle.
For $T\in\R$,
take a unitary $w\in \cM_\al^\om$
such that
$\al_T(x)=wxw^*$ for $x\in\cM$
and $w$ is fixed by $\al$.
We set $u:=v_Tw$
that belongs to $\cM_\al^\om=\cM_{\al^v}^\om$.
Then
$\al_T^v(x)=uxu^*$,
and
\[
\al_t^v(u)=v_t\al_t(v_T)\al_t(w)v_t^*
=
v_{t+T}wv_t^*
=
v_{t+T}\al_T(v_t^*)w
=u.
\]
By the previous lemma,
$\al^v$ is invariantly approximately inner.
\end{proof}

\begin{lem}
\label{lem:Rohdual}
Let $\al$ be a flow on a von Neumann algebra $\cM$
and $p\in\R$.
Suppose that there exists
a unitary central sequence
$(v^\nu)_\nu$ in $\cM$
such that
for each $t\in\R$,
$\lim_{\nu\to\infty}(\al_t(v^\nu)-e^{ipt}v^\nu)=0$
in the strong topology.
Then $\hal_p=\lim_{\nu\to\infty}\Ad \pi_\al(v^\nu)$
in $\Aut(\cM\rti_\al\R)$.
\end{lem}
\begin{proof}
Set $\cN:=\cM\rti_\al\R$
and ${\bf e}_{-p}\in C_b(\R)$
defined by ${\bf e}_{-p}(t):=e^{-ipt}$.
Then $\pi_\al(v^\nu)-v^\nu\oti {\bf e}_{-p}\to0$
in the strong topology
in $\cM\oti B(L^2(\R))$.
Indeed,
let $\xi\in \cH$ and $f\in L^2(\R)$.
Then
\[
\|(\pi_\al(v^\nu)-v^\nu\oti {\bf e}_{-p})(\xi\oti f)\|^2
=
\int_\R|f(t)|^2
\|(\al_{-t}(v^\nu)-e^{-ipt}v^\nu)\xi\|^2\,dt,
\]
which converges to 0
by the dominated convergence theorem.
Thus for $\ph\in\cM_*$
and $\ps\in B(L^2(\R))_*$,
\[
\|\pi_\al(v^\nu)(\ph\oti\ps)\pi_\al((v^\nu)^*)
-
v^\nu\ph (v^\nu)^*\oti {\bf e}_{-p}\ps {\bf e}_{-p}^*
\|
\to0.
\]
Since $(v^\nu)_\nu$ is central,
we have
\[
\|\pi_\al(v^\nu)(\ph\oti\ps)\pi_\al((v^\nu)^*)
-
\ph\oti {\bf e}_{-p}\ps {\bf e}_{-p}^*\|
\to0.
\]
This means $\Ad \pi_\al(v^\nu)\to \Ad (1\oti {\bf e}_{-p})$
in $\Aut(\cM\oti B(L^2(\R)))$.
Since $\hal_p=\Ad(1\oti {\bf e}_{-p})$ on $\cN$,
we have $\Ad \pi_\al(v^\nu)\to \hal_p$
in $\Aut(\cN)$.
\end{proof}

\begin{rem}
In the proof above,
we have used the following fact.
Let $\cN\subs\cM$ be an inclusion of von Neumann algebras.
Denote by $\Aut(\cM,\cN)$ the set of
automorphisms $\al$ on $\cM$
such that $\al(\cN)=\cN$.
It is fairly easy to see
that $\Aut(\cM,\cN)$ is a closed subgroup
of $\Aut(\cM)$ with respect to the $u$-topology.
Then the map
$\Aut(\cM,\cN)\ni\al\mapsto\al|_\cN\in\Aut(\cN)$
is continuous.

Indeed, let $\al,\be\in\Aut(\cM,\cN)$.
Take $\vph\in\cN_*$ and its normal extension
$\ps\in\cM_*$.
Then trivially,
$\|\al(\vph)-\be(\vph)\|_{\cN_*}
\leq
\|\al(\ps)-\be(\ps)\|_{\cM_*}$.
This shows the continuity.
\end{rem}

We recall the modular conjugation
of $\cM\rti_\al\R$
introduced in \cite[Lemma 2.8]{Ha-dualI}:
\[
(\tJ\xi)(s)
=
J\al_{-s}(\xi(-s))
\quad
\mbox{for }
\xi\in\cH\oti L^2(\R).
\]

\begin{lem}
\label{lem:invdual}
Let $\al$ be a flow on a von Neumann algebra $\cM$
and $p\in\R$.
Suppose that
there exists a sequence of unitaries
$(w^\nu)_\nu$ in $\cM$
such that
$\al_p=\lim_{\nu\to\infty}\Ad w^\nu$ in $\Aut(\cN)$
and
$\al_t(w^\nu)-w^\nu\to0$
as $\nu\to\infty$
in the strong$*$ topology
for each $t\in\R$.
Then the sequence $(v^\nu)_\nu$
defined by $v^\nu:=\la^\al(p)^*\pi_\al(w^\nu)$
is central in $\cM\rti_\al\R$,
and
belongs to $\sE_{\hal}^\om$.
In particular,
one has
$\hal_t(v)=e^{ipt}v$ for $t\in\R$
putting $v:=\pi_\om((v^\nu)_\nu)$.
\end{lem}
\begin{proof}
We will check that $(v^\nu)_\nu$ is central.
As in the proof of the previous lemma,
we can show that
$\pi_\al(w^\nu)-w^\nu\oti1\to0$
as $\nu\to\infty$
in the strong$*$ topology.
Then for all $\eta\in\cH$ and $f\in L^2(\R)$,
we have
\[
\limsup_{\nu\to\infty}
\|v^\nu(\eta\oti f)-(\eta\oti f)v^\nu\|
=
\limsup_{\nu\to\infty}
\|\la^\al(p)^*(w^\nu\eta\oti f)
-(\eta\oti f)\la^\al(p)^*\pi_\al(w^\nu)\|
\]
The right hand side equals 0.
Indeed,
\begin{align*}
&\|\la^\al(p)^*(w^\nu\eta\oti f)
-(\eta\oti f)\la^\al(p)^*\pi_\al(w^\nu)\|
\\
&=
\|w^\nu\eta\oti \la(p)^* f
-\tJ\pi_\al(w^\nu)^*\la^\al(p)\tJ(\eta\oti f)\|
\\
&=
\|w^\nu\eta\oti \la(p)^*f
-\al_p(\eta)w^\nu\oti \rho(p)f\|
\\
&=
\|w^\nu\eta-\al_p(\eta)w^\nu\|\|f\|,
\end{align*}
where we have used
$\tJ\la^\al(p)\tJ=U_\al(p)\oti\rho(p)$
and
$\tJ\pi_\al(x^*)\tJ=Jx^*J\oti1$
for all $p\in\R$ and $x\in\cM$.
Hence $(v^\nu)_\nu$ is central.
Since $\hal_t(v^\nu)=e^{ipt}v^\nu$ for all $\nu\in\N$,
$(v^\nu)_\nu$ belongs to $\sE_{\hal}^\om$.
Thus $v=\pi_\om((v^\nu)_\nu)\in(\cM\rti_\al\R)_{\om,\hal}$
and $\hal_t(v)=e^{ipt}v$.
\end{proof}

The following result is the von Neumann algebra version
of \cite[Theorem 1.3]{Kishi-Rohflow}.
This states that the Rohlin property and the invariantly
approximate innerness are mutually dual notions.
See \cite[Lemma 3.8]{Iz-Roh1}
for the corresponding result
in the case of finite group actions
on C$^*$-algebras.

\begin{thm}
\label{thm:dual}
Let $\al$ be a flow on a von Neumann algebra $\cM$.
Then the following statements
hold:
\begin{enumerate}
\item 
$\al$ has the Rohlin property
if and only if
$\hal$ is invariantly approximately inner;

\item
$\al$ is invariantly approximately inner
if and only if
$\hal$ has the Rohlin property.
\end{enumerate}
\end{thm}
\begin{proof}
(1).
Set $\cN:=\cM\rti_\al\R$.
Suppose that $\al$ is a Rohlin flow.
Then Lemma \ref{lem:Rohdual} shows
that $\hal$ is invariantly approximately inner
because $\hal$ fixes $\pi_\al(\cM)$.

Suppose that $\hal$ is invariantly approximately inner.
By the previous lemma,
the dual flow of $\hal$
has the Rohlin property,
and so does the flow $\al\oti\Ad\rho$
on $\cM\oti B(L^2(\R))$
by Takesaki duality \cite{Tak-dual}.
Lemma \ref{lem:tensorBH} implies that
$\al$ has the Rohlin property.

(2).
If a flow $\al$ on $\cM$ is invariantly approximately inner,
then the dual flow $\hal$ has the Rohlin property
by the previous lemma.
Conversely,
suppose that $\hal$ is a Rohlin flow.
Then by Takesaki duality
and Lemma \ref{lem:invappinn-stab},
$\be:=\al\oti\id_{B(\ell^2)}$ on
$\cN:=\cM\oti B(\ell^2)$
is
invariantly approximately inner.
Let $T\in\R$.
Then
by Lemma \ref{lem:IAI-local},
there exists a unitary
$w\in\cN_\be^\om$
such that
$wx=\be_T(x)w$
and $\be_t(w)=w$
for all $x\in\cN$ and $t\in\R$.
By the description
of $\cN^\om$
in Lemma \ref{lem:tensorBH},
we get the natural isomorphism
$\cN^\om\cong \cM^\om\oti B(\ell^2)$.
In fact, it turns out that
the isomorphism maps
$\cN_\be^\om$ onto $\cM_\al^\om\oti B(\ell^2)$.
Hence $w$ is regarded as an element
in $\cM_\al^\om\oti B(\ell^2)$,
and we have
$wx=(\al_T\oti\id)(x)w$
for $x\in\cM\oti B(\ell^2)$.
Then $w$ commutes with $1\oti y$
for any $y\in B(\ell^2)$,
and $w\in \cM_\al^\om\oti\C$.
This shows the invariantly approximate
innerness of $\al$.
\end{proof}

\begin{rem}
Let $\al$ be a flow on a von Neumann algebra $\cM$.
Then the following statements hold:
\begin{enumerate}
\item 
If $\al$ has the Rohlin property,
then
so does $\tal$;

\item
If $\al$ is invariantly approximately inner,
then
so is $\tal$.
\end{enumerate}
The first one follows from
the inclusion $\cM_\om\subs\tcM_\om$
(see the proof of \cite[Lemma 4.11]{Ma-T-endo}).
The second is directly proved.
\end{rem}

We obtain the following useful corollaries
of the previous theorem.

\begin{cor}
\label{cor:relative}
If $\al$ is a Rohlin flow on a von Neumann algebra $\cM$,
then
\[
\pi_\al(\cM)'\cap(\cM\rti_\al\R)=\pi(Z(\cM)),
\]
\[
\pi_{\tal}(\tcM)'\cap(\tcM\rti_\tal\R)
=\pi_{\tal}(Z(\tcM)).
\]
In particular,
$Z(\tcM\rti_\tal\R)=Z(\tcM)^\tal$.
\end{cor}
\begin{proof}
By the previous theorem,
$\hal$ is invariantly approximately inner.
In fact,
by Lemma \ref{lem:Rohdual},
each
$\hal_T$ is approximated by
$\Ad \pi_\al(w^\nu)$ with
$w^\nu\in\cM^{\rm U}$.
Thus $\hal$
fixes $\pi_\al(\cM)'\cap(\cM\rti_\al\R)$,
and we get the first equality.
The second equality is proved
similarly.
\end{proof}

Hence if $\cM$ is a type III$_1$ factor,
then so is $\cM\rti_\al\R$.

\begin{cor}
\label{cor:relative-appinn}
If $\al$ is an invariantly approximately inner
on a von Neumann algebra $\cM$,
then
\[
\pi_\al(\cM)'\cap(\cM\rti_\al\R)=\pi(Z(\cM\rti_\al\R)),
\]
\[
\pi_{\tal}(\tcM)'\cap(\tcM\rti_\tal\R)
=\pi_{\tal}(Z(\tcM\rti_\tal\R)).
\]
In particular,
$Z(\tcM\rti_\tal\R)^{\widehat{\tal}}=Z(\tcM)$.
\end{cor}
\begin{proof}
By Theorem \ref{thm:dual},
$\hal$ has the Rohlin property.
Then we get the result by employing the previous result
and
the following mirroring
(see \cite[Lemma 5.7]{Iz-Can} for its proof):
\[
\tJ(\pi_\al(\cM)'\cap(\cM\rti_\al\R))\tJ
=
(\cM\rti_\al\R)'\cap(\cM\oti B(L^2(\R))).
\]
\end{proof}

Therefore,
if $\al$ is a Rohlin flow on $\cM$
and invariantly approximately inner,
then
the inclusion
$\pi_\al(\cM)\subs\cM\rti_\al\R$
has the common flow of weights.
In particular,
$\cM\rti_\al\R$ is of the same type as $\cM$
when $\cM$ is a factor.
This assumption corresponds to
the central freeness and the approximate innerness
for discrete group actions on a factor.

\begin{cor}
Let $\al$ be a Rohlin flow on a von Neumann algebra $\cM$.
Suppose that $\al$ is centrally ergodic,
Then
\[
\Sp_d(\al|_{Z(\cM)})
=
\{p\in\R\mid \hal_p\in\Int(\cM\rti_\al\R)\}.
\]
\end{cor}
\begin{proof}
If $\hal_p=\Ad u$ for some unitary
$u\in \cM\rti_\al\R$,
then $u\in \pi_\al(\cM)'\cap(\cM\rti_\al\R)
=\pi_\al(Z(\cM))$.
Then putting $u=\pi_\al(v)$, $v\in Z(\cM)$,
we have
\[
\pi_\al(\al_t(v))
=\la_t^\al u(\la_t^\al)^*
=\la_t^\al \hal_p(\la_t^\al)^*u
=e^{ipt}\pi_\al(v).
\]
Hence $p\in \Sp_d(\al|_{Z(\cM)})$.

Suppose conversely
that $p\in \Sp_d(\al|_{Z(\cM)})$.
By polar decomposition,
there exists a non-zero partial isometry
$v\in Z(\cM)$ such that
$\al_t(v)=e^{ipt}v$ for $t\in\R$.
The central ergodicity implies that
$v$ is in fact a unitary.
Then $\hal_p=\Ad \pi_\al(v)$.
\end{proof}

A classification of invariantly approximately inner
flows will be treated in \S\ref{subsect:Cartan}.
A typical example of an invariantly approximately inner
flow not of infinite tensor product type
comes from a modular flow, or more generally,
an extended modular flow as introduced below.

\begin{defn}
We will say that
a flow $\be$ on a von Neumann algebra $\cN$
is \emph{extended modular}
when
$\be_t$ is an extended modular automorphism
for each $t\in\R$,
that is,
$\tbe_t\in\Int(\tcN)$.
\end{defn}

The definition above is slightly different
from that of \cite[Proposition IV.2.1]{CT}.
However,
it is essential to consider the canonical
extension in what follows,
and we adopt the definition above
(see also \cite[Proposition 5.4]{HS-pt}
and \cite[Definition 3.1]{Iz-Can}).

\begin{lem}
\label{lem:extmod-invapprox}
Let $\be$ be an extended modular flow
on a von Neumann algebra $\cN$
and $T\in\R$.
Suppose that
there exists a unitary $(v^\nu)_\nu$
in $\ell^\infty(\cN)$
such that
$\be_T=\lim_{\nu\to\infty}\Ad v^\nu$.
Then
$\be_t(v^\nu)-v^\nu$ converges to 0
compact uniformly
in the strong$*$ topology
as $\nu\to\infty$.
\end{lem}
\begin{proof}
The canonical extension
$\tbe$ is inner.
Thanks to the result due to Kallman and Moore
as mentioned in \S\ref{subsect:action-cocycle},
we can take a one-parameter unitary group
$w_t\in\tcN$
such that $\tbe_t=\Ad w_t$.

Let $\cK$ be the standard Hilbert space
of $\tcN$.
We regard $\cK$ as an $\tcN$-$\tcN$ bimodule
as usual.
Let $\xi\in\cK$,
$S>0$
and $\Ps:=\{w_t^*\xi\mid t\in [-S,S]\}$.
Then we have
$\sup_{\eta\in\Ps}\|v^\nu\eta-\tbe_T(\eta)v^\nu\|\to0$
since $\tbe_T=\lim_{\nu\to\infty}\Ad v^\nu$
in $\Aut(\tcN)$,
and $\Ps$ is compact.
Thus,
\begin{align*}
\|(\be_t(v^\nu)-v^\nu)\xi\|
&=
\|(w_t v^\nu w_t^*-v^\nu)\xi\|
\\
&=
\|w_t v^\nu w_t^*\xi-w_t \tbe_T(w_t^*\xi)v^\nu\|
+
\|w_t \tbe_T(w_t^*\xi)v^\nu-v^\nu\xi\|
\\
&\leq
\sup_{\eta\in\Ps}
\|v^\nu \eta-\tbe_T(\eta)v^\nu\|
+
\|w_t \tbe_T(w_t^*\xi)v^\nu-v^\nu\xi\|
\\
&=
\sup_{\eta\in\Ps}
\|v^\nu \eta-\tbe_T(\eta)v^\nu\|
+
\|\tbe_T(\xi)v^\nu-v^\nu\xi\|,
\end{align*}
where we have used
$\tbe_T(w_t)=w_Tw_tw_T^*=w_{T+t-T}=w_t$.
The last terms are converging to 0
as $\nu\to\infty$.
Hence
we have
$\|(\be_t(v^\nu)-v^\nu)\xi\|\to0$
uniformly on $[-S,S]$
as $\nu\to\infty$.
Similarly,
$\|\xi(\be_t(v^\nu)-v^\nu)\|\to0$
uniformly on $[-S,S]$
as $\nu\to\infty$.
\end{proof}

\begin{rem}
For a modular automorphism group,
the previous lemma is shown
without use of the canonical extension.
Indeed,
let us assume that
a faithful state $\vph\in\cN_*$
and $T\in\R$
satisfy
$\si_T^\vph=\lim_{\nu\to\infty}\Ad v^\nu$
in $\cN$ as above.
Using $\si_T^\vph(\vph)=\vph$,
we have
$\|[v^\nu,\vph]\|=\|\Ad v^\nu(\vph)-\vph\|\to0$ as $\nu\to\infty$.
Thus by \cite[Lemma 2.7]{Co-almost},
$\|\si_t^\vph(v^\nu)-v^\nu\|_\vph^\sharp\to0$
compact uniformly as $\nu\to\infty$.
\end{rem}

By Lemma \ref{lem:invdual}, Theorem \ref{thm:dual}
and Lemma \ref{lem:extmod-invapprox},
we have the following result.

\begin{prop}
\label{prop:ext-dual}
Let $\be$ be an extended modular flow
on a von Neumann algebra $\cN$.
If $\be$ is pointwise approximately inner,
the dual flow $\al:=\hbe$
has the Rohlin property.
\end{prop}

Let $\al,\be$ be as above.
We show that
the Connes-Takesaki
module flow of $\al$ is faithful.
Denote by $\cM$ the crossed product
$\cN\rti_\be\R$.
Then
$\tcM=\tcN\rti_\tbe\R$
and $\tal=\widehat{\tbe}$
by Lemma \ref{lem:can-natural}.
Since $\tbe$ is implemented by a one-parameter
unitary group as mentioned before,
we have an isomorphism
$
\{\tcM,\th,\tal\}
\cong
\{\tcN\oti L(\R),\th,\id\oti\Ad{\bf e}_{-t}\},
$
where ${\bf e}_{-t}\in C_b(\R)$ is
${\bf e}_{-t}(s):=e^{-ist}$.
By simple calculation,
we have $c_s(t)\in Z(\cN)$
satisfying
\[
\th_s(a\oti \la(t))
=(\th_s(a)\oti1)(c_s(t)^*\oti\la(t)).
\]
Since $c_s(t+t')=c_s(t)c_s(t')$,
we have a positive operator $K_s$
affiliated with $Z(\cN)$
such that $c_s(t)=K_s^{it}$.

Then by the isomorphism
$\pi\col L(\R)\to L^\infty(\R_+^*)$
with $\pi(\la(t))(h)=h^{it}$,
we have
$
\{\tcM,\th,\tal\}
\cong
\{\tcN\oti L^\infty(\R_+^*),\th,\id\oti\Ad\la(t)\}
$,
where $(\la(t)\xi)(s)=\xi(e^{-t}s)$ for $t,s>0$.
In particular,
$\mo(\al_t)$ is the translation
on $\R_+^*:=\{h\in\R\mid h>0\}$.

If we regard $K_s$
as the function $K_s\col X_\cN\to\R_+^*$,
we have the following
for all $x\in X_\cN$ and $h>0$:
\[
\th_s(1\oti f)(x,h)= f(K_s(x)h)
\]

Hence the flow space $X_\cM$
is naturally isomorphic to $X_\cN\times\R_+^*$.
Let $F^\cM$ and $F^\cN$ be the flow of weights
of $\cM$ and $\cN$, respectively.
Then we have
$F_s^\cM(x,h)=(F_s^\cN x, K_s(x)h)$.
Summarizing the discussion above,
we have the following result.

\begin{thm}
\label{thm:modular-flow}
Let $\be$ be an extended modular flow
on $\cN$
and $\al:=\hbe$ be the dual flow on
$\cM:=\cN\rti_\be\R$.
Then there exists
an $\R_+^*$-valued
$F^\cN$-cocycle $K\col X_\cN\rtimes \R\to\R_+^*$
such that
\begin{enumerate}

\item
$X_\cM=X_\cN\times\R_+^*$,
$F_s^\cM(x,h)=(F_s^\cN x,K(x,s)h)$
for all $s\in\R$, $x\in X_\cM$
and $h>0$;

\item
$\mo(\al_t)(x,h)=(x,e^{-t}h)$
for all $t\in\R$, $x\in X_\cM$
and $h>0$.
\end{enumerate}
\end{thm}

\section{Classification of Rohlin flows}

In this section,
we will prove our main theorem (Theorem \ref{thm:class2})
of this paper.

\subsection{Rohlin projection and averaging technique}
The classification of general Rohlin flows
will be reduced to that of centrally ergodic Rohlin flows
(see the proof Theorem \ref{thm:class2}).
Hence let us assume that
$(\al,c)$ is a Borel cocycle action of $\R$
on a von Neumann algebra $\cM$
with the following properties:
\begin{itemize}
\item 
Rohlin property;

\item
$Z(\cM)^\al=\C$
and
$\Sp_d(\al|_{Z(\cM)})\neq\R$.
\end{itemize}

The case that $\Sp_d(\al|_{Z(\cM)})=\R$
will be treated separately in the proof of Lemma \ref{lem:class}.
Let us put
$H_\al:=\Sp_d(\al|_{Z(\cM)})$ that is a Borel subgroup of $\R$.
The following result is probably well-known
to experts,
but we present a proof for readers' convenience.

\begin{lem}
For any $\vep>0$,
there exists $p>0$ such that
$p<\vep$ and
$\Z p\cap H_\al=\{0\}$.
\end{lem}
\begin{proof}
Let $E:=[0,1]\subs\R$,
$E_0:=\{t\in E\mid \Z t\cap H_\al=\{0\}\}$
and
\[
E_k:=\{t\in E\mid \ell t\nin H_\al,
\
\ell=1,\dots,k-1,
\
kt\in H_\al\},
\quad
k\in\N.
\]
Since
$E_0=\bigcap_{\ell=1}^\infty (1/\ell)H_\al^c\cup\{0\}$,
and
$E_k=(1/k)H_\al\cap
\bigcap_{\ell=1}^{k-1}(1/\ell)H_\al^c$,
$E_k$'s are Borel sets.
Hence
$1=\mu(E)=\sum_{k=0}^\infty\mu(E_k)$.
Suppose that $\mu(E_k)>0$ for some $k>0$.
Then there exists $\de>0$
such that
$(-\de,\de)\subs E_k-E_k$.
Since $k E_k\subs H_\al$,
we have
$(-k\de,k\de)\subs kE_k-kE_k\subs H_\al$.
This forces $H_\al$ to be $\R$,
which is a contradiction.
Thus $\mu(E_0)=1$,
and we are done.
\end{proof}

The above lemma states that
an arbitrarily large $S>0$ can be chosen
in such a way that
$(2\pi/S)\Z\cap H_\al=0$.
Let
$v\in\cM_{\om,\al}$ be a Rohlin unitary for $-2\pi/S$,
that is,
$\al_t(v)=e^{-2\pi it/S}v$.
Then $\al$ is a flow on $W^*(v)$ with period $S$.
By the equality $\ta^\om\circ\al_t=\al_t\circ\ta^\om$,
we have
$\ta^\om(v^n)\in Z(\cM)$ satisfies
$\al_t(\ta^\om(v^n))=e^{-2n\pi it/S}\ta^\om(v^n)$
for $n\in\Z$,
which yields, however,
$\ta^\om(v^n)=0$ if $n\neq0$
because $-2n\pi/S\nin H_\al$.
Hence
$\ta^\om$ is a faithful normal state
on $W^*(v)$.

Let $v=\int_0^{2\pi}e^{i\la}\,dE(\la)$
be the spectral decomposition
on $\T=[0,2\pi)$.
By easy calculation,
we have $\al_t(dE(\la))=dE(\la+2\pi t/S)$.
We set $e(\la):=E(2\pi\la/S)$ for $\la\in[0,S)$.
Then $\al_t(de(\la))=de(\la+t)$
and
$v=\int_0^S e^{2\pi i\la/S}\,de(\la)$.
Thus
$d\ta^\om(e(\cdot))$ coincides
with the Haar measure
on the torus $[0,S)=\R/S\Z$,
that is,
the normalized
Lebesgue measure.
Therefore,
for $f\in L^\infty[0,S)$,
we can define $f(v)=\int_0^S f(t)\,de(\la)$.
Then $L^\infty[0,S)\ni f\mapsto f(v)\in W^*(v)$
is an isomorphism.

Following \cite{Kawamuro-RIMS,Kishi-CMP},
we introduce a normal $*$-homomorphism
$\Th\col \cM\oti L^\infty[0,S)\ra \Mequ$
such that $\Th(a\oti f)=af(v)$.

\begin{lem}
\label{lem:theta}
Let $S>0$
with $(2\pi/S)\Z\cap H_\al=\{0\}$.
Then there exists an isomorphism
$\Th\col \cM\oti L^\infty[0,S)\ra \cM\vee W^*(v)$
such that
\begin{itemize}
\item 
$\Th(a\oti f)=af(v)$
for all $a\in\cM$ and $f\in L^\infty[0,S)$;

\item
$\al_t\circ\Th=\al_t\oti\ga_t$,
where $[0,S)$ is regarded as a circle $\R/S\Z$,
and $\ga_t$ denotes the rotation by $t$ on $[0,S)$;

\item
$\ta^\om\circ\Th=\id_{\cM}\oti\mu$,
where
$\mu$ denotes the integration
by the normalized Lebesgue measure.
\end{itemize}
\end{lem}
\begin{proof}
Let $z(\la):=e^{2\pi i\la/S}$ for $\la\in[0,S)$.
Then we have
\[
\vph^\om(a v^n)
=
\vph(a\ta^\om(v^n))
=
\de_{n,0}\vph(a)
=(\vph\oti\mu)(a\oti z^n).
\]
Since $\{a\oti z^n\mid n\in\Z\}$ and
$\{av^n\mid n\in \Z\}$
span strongly dense $*$-algebras
in $\cM\oti L^\infty[0,S)$
and $\cM\vee W^*(v)$,
respectively,
we have such $\Th$.
\end{proof}

The map $\Th$
plays a role of Shapiro's lemma,
that is,
$\Th(a)$, $a\in \cM\oti L^\infty[0,S)$,
can be regarded as
the average of $a(s)$ along with
the Rohlin tower $e(s)$.
We may write $\Th(a)$
in a formal manner as
\[
\Th(a)=\int_0^S a(s)\,d e(s).
\]
From the previous lemma,
for any
$\vph\in\cM_*^+$,
we obtain the following equality:
\begin{equation}
\label{eq:Thint}
\|\Th(w)\|_{\vph^\om}^2
=\frac{1}{S}\int_0^S \|w(s)\|_\vph^2
\,ds.
\end{equation}

\begin{lem}
\label{lem:Borelasymprep}
Let $(\al,c)$ be
a Borel cocycle action of $\R$
as before.
Let $S>0$
with $(2\pi/S)\Z\cap H_\al=\{0\}$.
Let $u\col[-T,T)\times[0,S)\ra\cM^{\rm U}$
be a Borel map,
$\Phi\subset \cM_*$ a finite set
and
$e(\lambda)\in \cM_{\om,\al}$
a Rohlin projection over $[0,S)$.
Set $w(t):=\Th(u(t,\cdot))$,
which is a Borel unitary path in $\cM_\al^\om$.
Then for any $\vep>0$
with
\[
\frac{1}{2ST}\int_{-T}^Tdt\int_{0}^S ds\,
\|[u(t,s),\varphi]\|
<\vep
\quad
\mbox{for all }
\varphi\in \Phi,
\]
there exist $W\in \omega$
and a lift $(w(t)^\nu)_\nu$
of $w(t)$
as in Lemma \ref{lem:Borelpathlift}
with respect to $E:=[-T,T)$
such that
\[
\frac{1}{2T}
\int_{-T}^T
\|[w(t)^\nu,\varphi]\|\,dt
<
3\vep
\quad
\mbox{for all }
\varphi\in \Phi,
\
\nu\in W.
\]
\end{lem}
\begin{proof}
Since $\|[u,\vph]\|=\|[u,\vph^*]\|$
for a unitary $u$,
we may and do assume that $\Ph^*=\Ph$.
Note that
$[-T,T)\ni t\mapsto u(t,\cdot)\in\cM\oti L^\infty(\T)$
is a Borel unitary path.
Hence so is $t\mapsto \Th(u(t,\cdot))\in\Mequ$.
Fix $0<\de<1$ so that for all $\vph\in\Ph$,
\begin{equation}
\label{eq:devep}
\de^2+4\de\|\vph\|
<\de^{1/2},
\
(\de+\vep)/(1-\de^{1/2})+2\de^{1/2}
\|\vph\|
<\de^{1/4}+2\vep,
\
5\de^{1/4}+2\de\|\vph\|<\vep/2.
\end{equation}
Since $u(t,s)$ is Borel,
there exists a compact set
$K\subs[-T,T)\times[0,S)$
such that
$\mu(K)\geq 2ST(1-\de)$
and $u$ is continuous on $K$.

Let $N\in\N$,
and for $i=-N,\dots,N-1$ and $j=0,\dots,N-1$,
we set
\[
I_i:=\{t\in\R\mid 
iT/N\leq t< (i+1)T/N\},
\]
\[
J_j:=\{s\in\R\mid 
jS/N\leq s< (j+1)S/N\},
\]
\[
\De_{i,j}:=I_i\times J_j.
\]
Fix a large $N$
so that
for all
$(t,s),(t',s')\in\De_{i,j}\cap K$
and $\vph\in\Ph$,
we have
\begin{equation}\label{eq:uxyiT}
\|u(t,s)-u(t',s')\|_{|\vph|}^\sharp<\de,
\quad
\|[\vph,u(t,s)-u(t',s')]\|<\de.
\end{equation}

If $\De_{i,j}\cap K\neq\emptyset$,
we fix an element $k_{i,j}\in \De_{i,j}\cap K$.
If empty,
we put $k_{i,j}:=(iT/N,jS/N)$.
We set the following unitary in $\cM$:
\[
u_0(t,s)
:=
\sum_{i,j}
u(k_{i,j})1_{\De_{i,j}}(t,s),
\quad
(t,s)\in[-T,T)\times[0,S).
\]
Then
\begin{equation}
\label{eq:utsu0}
\|u(t,s)-u_0(t,s)\|_{|\vph|}^\sharp
<\de
\quad
\mbox{for all }
(t,s)\in K,
\
\vph\in\Ph.
\end{equation}

Let $V(t):=\Th(u_0(t,\cdot))\in\cM_\al^\om$.
Then
\[
V(t)
=
\sum_{i,j}
u(k_{i,j})1_{I_i}(t)e(J_j).
\]
We estimate $\|w(t)-V(t)\|_{|\vph|^\om}^\sharp$ as follows.
Put $K_0:=\pr(K)\subs[-T,T)$,
where $\pr$ denotes the projection $(x,y)\mapsto x$. 
Then $K_0^c\times [0,S)\subs K^c$,
where $K_0^c$ and $K^c$ are the complements
in $[-T,T)$ and $[-T,T)\times[0,S)$,
respectively.
Hence $\mu(K_0^c)\cdot S\leq 2ST\de$,
and $\mu(K_0)\geq 2T(1-\de)$.
For $t\in [-T,T)$,
we set $K_t:=\{s\in[0,S)\mid (t,s)\in K\}$.
Then for $t\in[-T,T)$,
we have
\begin{align*}
\|w(t)-V(t)\|_{|\vph|^\om}^{\sharp\,2}
&=
\|\Th(u(t,\cdot))
-\Th(u_0(t,\cdot))\|_{|\vph|^\om}^{\sharp\,2}
\\
&
=\frac{1}{S}
\int_{0}^S
\|u(t,s)-u_0(t,s)\|_{|\vph|}^{\sharp\,2}
\,ds
\quad
\mbox{by }
(\ref{eq:Thint})
\\
&=
\frac{1}{S}
\int_{K_t}
\|u(t,s)-u_0(t,s)\|_{|\vph|}^{\sharp\,2}
\,ds
\\
&\quad
+
\frac{1}{S}
\int_{K_t^c}
\|u(t,s)-u_0(t,s)\|_{|\vph|}^{\sharp\,2}
\,ds
\\
&\leq
\de^2+4\mu(K_t^c)\|\vph\|/S
\quad
\mbox{by }(\ref{eq:utsu0}).
\end{align*}
Note that
\[
\int_{-T}^T\mu(K_t^c)\,dt
=\int_{-T}^T(S-\mu(K_t))\,dt
=2ST-\mu(K)
\leq 2ST\de.
\]
Then by (\ref{eq:devep}),
\[
\int_{-T}^T
\|w(t)-V(t)\|_{|\vph|^\om}^{\sharp\,2}
\,dt
\leq
(2\de^2 +8\de\|\vph\| )T
\leq
2T\de^{1/2}
\quad
\mbox{for all }
\vph\in\Ph.
\]

Let $C$ be
a  compact set in $[-T,T]$
as in Lemma \ref{lem:Borelpathlift}
with respect to $w(t)$
such that
$\mu(C)\geq 2T(1-\de)$.
By the inequality above,
we get
\begin{equation}\label{eq:KwV1}
\int_{C}
\|w(t)-V(t)\|_{|\vph|^\om}^{\sharp\, 2}
\,dt
\leq
2T\de^{1/2}
\quad
\mbox{for all }
\vph\in\Ph.
\end{equation}

Put $u_{i,j}:=u(k_{i,j})$.
For $(t,s)\in K$,
we have
\begin{align}
\sum_{i,j}
\|[\vph,u_{i,j}]\|1_{\De_{i,j}\cap K}(t,s)
&\leq
\|[\vph,u(t,s)]\|
+
\sum_{i,j}
\|[\vph,u_{i,j}-u(t,s)]\|
1_{\De_{i,j}\cap K}(t,s)
\notag\\
&\leq
\|[\vph,u(t,s)]\|
+
\de
\quad
\mbox{by }(\ref{eq:uxyiT}).
\end{align}
Integrating them by $(t,s)\in K$,
we have
\begin{align}
\sum_{i,j}\|[\vph,u_{i,j}]\|
\mu(\De_{i,j}\cap K)
&\leq
\int_K
\|[\vph,u(t,s)]\|
\,dtds
+
2ST\de
\notag\\
&\leq
\int_{-T}^T dt
\int_{0}^S ds
\,
\|[\vph,u(t,s)]\|
+
2ST\de
\notag\\
&\leq
2ST\vep
+
2ST\de
=2ST(\vep+\de).
\label{eq:intvep}
\end{align}
Note that
$\sum_{i,j}\mu(\De_{i,j}\cap K^c)=\mu(K^c)
<2ST\de$.
Set
\[
L:=\{(i,j)\mid \mu(\De_{i,j}\cap K^c)
<ST\de^{1/2}/N^2\}.
\]
Then
$|L^c|/2N^2<\de^{1/2}$ by the Chebyshev inequality,
and for $(i,j)\in L$,
\begin{align*}
\mu(\De_{i,j}\cap K)
&=
\mu(\De_{i,j})-\mu(\De_{i,j}\cap K^c)
>
ST/N^2-ST\de^{1/2}/N^2
\\
&=(1-\de^{1/2})ST/N^2.
\end{align*}
Thus by (\ref{eq:intvep}),
we have
\[
\sum_{(i,j)\in L}
\|[\vph,u_{i,j}]\|/2N^2
\leq
(\de+\vep)/(1-\de^{1/2}).
\]
By definition of $L$,
we obtain
\[
\sum_{(i,j)\in L^c}
\|[\vph,u_{i,j}]\|/2N^2
\leq
2\|\vph\||L^c|/2N^2
<2\de^{1/2}\|\vph\|.
\]
Hence by (\ref{eq:devep}),
we get the following inequality
for $\vph\in\Ph$:
\begin{equation}\label{eq:ijphu}
\sum_{i,j}
\|[\vph,u_{i,j}]\|/2N^2
\leq
(\de+\vep)/(1-\de^{1/2})+2\de^{1/2}\|\vph\|
<\de^{1/4}+2\vep.
\end{equation}

Let
$(E_j^\nu)_\nu$
be a representing sequence
of $E_j:=e(J_j)$ consisting of projections
with $\sum_j E_j^\nu=1$ for each $\nu$.
We set
$V(t)^\nu:=\sum_{i,j}1_{I_i}(t)u_{i,j}E_j^\nu$.
On $C\cap I_i$, which may be non-compact,
$\{t\mapsto w(t)^\nu\}_\nu$
and
$\{t\mapsto V(t)^\nu\}_\nu$
are $\om$-equicontinuous
since $V(t)^\nu$ is constant.
Thus by Lemma \ref{lem:uniconverge},
we have
$\lim_{\nu\to\om}
\|w(t)^\nu-V(t)^\nu\|_{|\vph|}^\sharp
=
\|w(t)-V(t)\|_{|\vph|^\om}^\sharp$
is a uniform convergence
on $C \cap I_i$.
We set
\[
\Ph'
:=
\{|[\vph,u_{i,j}]|
\mid\vph\in \Ph,\
-N\leq i\leq N-1,\ 0\leq j\leq N-1\}.
\]

Note that
$\ta^\om(E_j)=1/N$
by Lemma \ref{lem:theta}.
Using (\ref{eq:KwV1}) and the above uniform convergence,
we can find $W\in\om$
such that
\begin{equation}
\label{eq:Cwtnu}
\int_{C}
\|w(t)^\nu-V(t)^\nu\|_{|\vph|}^{\sharp\,2}
\,dt
<3T\de^{1/2}
\quad
\mbox{for all }
\vph\in\Ph,
\
\nu\in W,
\end{equation}
\[
\|[\vph,E_j^\nu]\|<\vep/6N
\quad
\mbox{for all }
\vph\in\Ph\cup\Ph',
\
0\leq j\leq N-1,
\
\nu\in W,
\]
\[
|\|\ps\|/N-\ps(E_j^\nu)|
<\vep/6N
\quad
\mbox{for all }
\ps\in\Ph',
\
0\leq j\leq N-1,
\
\nu\in W.
\]
Then for all $\vph\in\Ph$, $i,j$
and
$\nu\in W$,
\[
\|[\vph,u_{i,j}E_j^\nu]\|
\leq
\|[\vph,u_{i,j}]E_j^\nu\|
+
\|u_{i,j}[\vph,E_j^\nu]\|
\leq
\|[\vph,u_{i,j}]E_j^\nu\|
+\vep/6N.
\]
Putting $\ps:=[\vph,u_{i,j}]$
and $\ps=u|\ps|$, the polar decomposition,
we obtain the following estimate
for $x\in\cM$ with $\|x\|\leq1$ and $\nu\in W$:
\begin{align*}
|([\vph,u_{i,j}]E_j^\nu)(x)|
&=
||\ps|(E_j^\nu xu)|
\leq
|[|\ps|,E_j^\nu](E_j^\nu xu)|
+
||\ps|(E_j^\nu xu E_j^\nu)|
\\
&\leq
\|[\ps,E_j^\nu]\|\|x\|
+
|\ps|(E_j^\nu)^{1/2}
\|x\|
|\ps|(E_j^\nu)^{1/2}
\\
&<
\vep/6N+|\ps|(E_j^\nu)
<
\vep/3N+\|\ps\|/N.
\end{align*}
Then
for all $\vph\in\Ph$, $i,j$
and
$\nu\in W$,
we obtain
\[
\|[\vph,u_{i,j}E_j^\nu]\|
<\vep/2N+\|[\vph,u_{i,j}]\|/N.
\]
Hence for $\vph\in\Ph$ and $\nu\in W$,
\begin{align}
\int_{-T}^T
\|[\vph,V(t)^\nu]\|
\,dt
&\leq
\sum_{i,j}
\int_{-T}^T
\|[\vph,u_{i,j}E_j^\nu]\|
1_{I_i}(t)
\,dt
\notag\\
&<
\sum_{i,j}
(\vep /2N+
\|[\vph,u_{i,j}]\|/N)
\cdot(T/N)
\notag\\
&\leq
\vep T
+
2(\de^{1/4}+2\vep)T
\quad
\mbox{by }(\ref{eq:ijphu})
\notag\\
&\leq
2T(\de^{1/4}+5\vep/2)
\label{eq:phVnu7}.
\end{align}

We estimate
$\int_{C}
\|[\vph,w(t)^\nu-V(t)^\nu]\|\,dt$
as follows:
\begin{align}
&\int_{C}
\|[\vph,w(t)^\nu-V(t)^\nu]\|\,dt
\notag\\
&\leq
\int_{C}
\|(w(t)^\nu)^*-(V(t)^\nu)^*\|_{|\vph|}\,dt
+
\int_{C}
\|w(t)^\nu-V(t)^\nu\|_{|\vph^*|}\,dt
\notag\\
&\leq
\mu(C)^{1/2}
\left(
\int_{C}
\|(w(t)^\nu)^*-(V(t)^\nu)^*\|_{|\vph|}^2\,dt
\right)^{1/2}
\notag\\
&\quad
+
\mu(C)^{1/2}
\left(
\int_{C}
\|w(t)^\nu-V(t)^\nu\|_{|\vph^*|}^2\,dt
\right)^{1/2}
\notag\\
&\leq
\mu(C)^{1/2}
\cdot(2\cdot3T\de^{1/2})^{1/2}
+
\mu(C)^{1/2}
\cdot(2\cdot3T\de^{1/2})^{1/2}
\quad
\mbox{by }
(\ref{eq:Cwtnu})
\notag\\
&\leq
7T\de^{1/4}.
\label{eq:vphwV}
\end{align}

Then for $\nu\in W$,
\begin{align*}
\int_{C}
\|[\vph,w(t)^\nu]\|
\,dt
&\leq
\int_{C}
\|[\vph,w(t)^\nu-V(t)^\nu]\|
\,dt
+
\int_{C}
\|[\vph,V(t)^\nu]\|
\,dt
\\
&\leq
2T(5\de^{1/4}+5\vep/2)
\quad
\mbox{by }
(\ref{eq:phVnu7}),
\
(\ref{eq:vphwV}),
\end{align*}
and
\begin{align*}
\int_{-T}^T
\|[\vph,w(t)^\nu]\|
\,dt
&\leq
\int_{C}
\|[\vph,w(t)^\nu]\|
\,dt
+
\int_{C^c}
\|[\vph,w(t)^\nu]\|
\,dt
\\
&\leq
2T(5\de^{1/4}+5\vep/2)
+
2\|\vph\|\mu(C^c)
\\
&\leq
2T(5\de^{1/4}+5\vep/2)
+
2\|\vph\|\cdot 2T\de 
\\
&=2T
(5\de^{1/4}+2\de\|\vph\|+5\vep/2)
<6T\vep
\quad
\mbox{by }
(\ref{eq:devep}).
\end{align*}
\end{proof}

\subsection{2-cohomology vanishing}
Let $(\al,c)$ be a Borel cocycle action
of $\R$
on a von Neumann algebra $\cM$
as in the previous subsection,
that is,
it has the Rohlin property and
the ergodicity on $Z(\cM)$
such that
$\al$ on $Z(\cM)$
is not conjugate to the translation
on $L^\infty(\R)$.
We will show that the 2-cocycle $c$
can be perturbed to be close to $1$.
Let $0<\de<1$, $T>0$ and
a finite set $\Ph\subs\cM_*^+$.
Take
$S>T$
such that $(2\pi/S)\Z\cap H_\al=\{0\}$
and
\begin{equation}
\label{eq:4T2T}
4T^{1/2}/S^{1/2}
<\de/24T^2.
\end{equation}

Let $e(\la)$ be a Rohlin projection
over $[0,S)$.
We put $U(t):=\Th(\tilde{c}(t,\cdot-t)^*)$,
where $\tilde{c}$ denotes the periodization
of $c$
with respect to the second variable,
that is,
$\tilde{c}(x,y):=c(x,y)$
for $y\in[0,S)$,
and $\tilde{c}(x,y+S)=\tilde{c}(x,y)$
for $y\in\R$.

\begin{lem}\label{lem:Borelperturb}
In the above setting,
there exist $W\in\om$
and a lift $(u(t)^\nu)_\nu$
of $U(t)$
as in
Lemma \ref{lem:Borelpatheq}
such that for all $\vph\in\Ph$,
\[
\int_{-T}^T dt
\int_{-T}^T ds
\,
\|u(t)^\nu\al_t(u(s)^\nu)c(t,s)(u(t+s)^\nu)^*-1\|_\vph^\sharp
<\de
\quad\mbox{for all }
\nu\in W.
\]
If $\vep>0$ satisfies
\[
\frac{1}{2ST}
\int_{-T}^T
dt
\int_{0}^S
ds
\,
\|[c(t,s),\vph]\|
<\vep
\quad
\mbox{for all }
\vph\in\Ph,
\]
then one can take $W$ so that
\[
\frac{1}{2T}
\int_{-T}^T\|[u(t)^\nu,\vph]\|
\,dt<3\vep
\quad
\mbox{for all }
\vph\in\Ph,
\
\nu\in W.
\]
\end{lem}
\begin{proof}
By Lemma \ref{lem:theta},
we have
$\al_t(U(s))=\Th(\al_t(\tilde{c}(s,\cdot-t-s)^*))$.
Let $-T\leq s,t\leq T$.
When $t+s\geq0$,
then
\[
\tilde{c}(x,\la-t-s)
=c(x,\la-t-s)1_{[t+s,S)}(\la)
+c(x,\la-t-s+S)1_{[0,t+s)}(\la).
\]
Thus we have
\begin{align*}
&U(t)\al_t(U(s))c(t,s)U(t+s)^*
\\
&=
\Th(
\tilde{c}(t,\cdot-t)^*
\al_t(\tilde{c}(s,\cdot-t-s)^*)
c(t,s)\tilde{c}(t+s,\cdot-t-s)
)
\\
&=
\Th(
c(t,\cdot-t)^*
\al_t(c(s,\cdot-t-s)^*)
c(t,s)c(t+s,\cdot-t-s)
1_{[t+s,S)}(\cdot))
\\
&\quad
+
\Th(
\tilde{c}(t,\cdot-t)^*
\al_t(\tilde{c}(s,\cdot-t-s)^*)
c(t,s)\tilde{c}(t+s,\cdot-t-s)
1_{[0,t+s)}(\cdot))
\\
&=
\Th(
1_{[t+s,S)}(\cdot))
\\
&\quad
+
\Th(
\tilde{c}(t,\cdot-t)^*
\al_t(\tilde{c}(s,\cdot-t-s)^*)
c(t,s)\tilde{c}(t+s,\cdot-t-s)
1_{[0,t+s)}(\cdot)).
\end{align*}
Then
for
$\vph\in\Ph_+$
and
$t,s$ with $t+s\geq0$:
\begin{align*}
&
\|
U(t)\al_t(U(s))c(t,s)U(t+s)^*-1
\|_\vph^\sharp
\\
&\leq
\|1_{[t+s,S)}(\cdot)-1\|_{\vph\oti\mu}^\sharp
+
\|1_{[0,t+s)}(\cdot)\|_{\vph\oti\mu}^\sharp
\\
&=
2\|1_{[0,t+s)}(\cdot)\|_{\vph}^\sharp
\\
&=
2\|\vph\|^{1/2}(t+s)^{1/2}/S^{1/2}
\\
&\leq
2\sqrt{2}\|\vph\|^{1/2}T^{1/2}/S^{1/2}
<
\de/24T^2
\quad
\mbox{by }
(\ref{eq:4T2T}).
\end{align*}
The same inequality also holds when $t+s\leq0$.
Hence for all $t,s\in[-T,T]$
and $\vph\in\Ph$,
\begin{equation}
\label{Uc1}
\|
U(t)\al_t(U(s))c(t,s)U(t+s)^*-1
\|_\vph^\sharp
<\de/24T^2.
\end{equation}
Then by Lemma \ref{lem:Borelpatheq},
there exist a compact subset
$K\subs[-T,T]^2$
and a lift $(u(t)^\nu)_\nu$
of $U(t)$
such that
$\mu(K)\geq 4T^2(1-\de/24T^2)$,
and
for all $\vph\in\Ph$,
we have the following
uniform convergence
on $K$ as $\nu\to\om$:
\begin{align*}
\|u(t)^\nu\al_t(u(s)^{\nu})c(t,s)(u(t+s)^\nu)^*
-1\|_\vph^\sharp
\to
\|U(t)\al_t(U(s))c(t,s)U(t+s)^*
-1\|_{\vph^\om}^\sharp.
\end{align*}
By (\ref{Uc1}),
there exists $W\in\om$
such that if $(t,s)\in K$,
$\nu\in W$ and $\vph\in\Ph$,
then
\[
\|u(t)^\nu\al_t(u(s)^{\nu})c(t,s)(u(t+s)^\nu)^*
-1\|_\vph^\sharp
<\de/24T^2.
\]
If $\nu\in W$ and $\vph\in\Ph$,
then we obtain
\begin{align*}
&
\int_{-T}^T
dt
\int_{-T}^T
ds
\,
\|u(t)^\nu\al_t(u(s)^{\nu})c(t,s)(u(t+s)^\nu)^*
-1\|_\vph^\sharp
\\
&=
\int_{K}dtds\,
\|u(t)^\nu\al_t(u(s)^{\nu})c(t,s)(u(t+s)^\nu)^*
-1\|_\vph^\sharp
\\
&\quad
+
\int_{K^c}dtds\,
\|u(t)^\nu\al_t(u(s)^{\nu})c(t,s)(u(t+s)^\nu)^*
-1\|_\vph^\sharp
\\
&\leq
\mu(K)\de/24T^2
+
2\|\vph\|^{1/2}\mu(K^c)
\\
&\leq
4T^2(1+
2\|\vph\|^{1/2})\de/24T^2
<\de/2
\quad
\mbox{by }
(\ref{eq:4T2T}).
\end{align*}

Next if we have
$(1/2ST)
\int_{-T}^T
dt
\int_0^S
ds
\,
\|[\vph,c(t,s)]\|<\vep$
for all $\vph\in\Ph$,
then
\[
\frac{1}{2ST}
\int_{-T}^T
dt
\int_0^S
ds
\,
\|[\vph,\tilde{c}(t,s-t)]\|
=
\frac{1}{2ST}
\int_{-T}^T
dt
\int_0^S
ds
\,
\|[\vph,c(t,s)]\|
<\vep,
\]
and
we can apply 
Lemma \ref{lem:Borelasymprep}
to $U(t)=\Th(\tilde{c}(t,\cdot-t))$.
Then
we have the following
for $\nu$ close to $\om$:
\[
\frac{1}{2T}
\int_{-T}^T
\|[\vph,u(t)^\nu]\|_\vph\,dt
<3\vep
\quad
\mbox{for all }
\vph\in\Ph.
\]
\end{proof}

Let us take a 
decreasing sequence
$\{\vep_n\}_n$,
increasing sequences
$\{T_n\}_n$
and
$\{S_n\}_n$
such that
$0<\vep_n<1/n$,
$T_n,S_n>n$,
$(2\pi/S_n)\Z\cap H_\al=\{0\}$,
and
\begin{equation}
\label{eq:SnTn}
T_n+S_n<T_{n+1},
\quad
\sum_{k=n+1}^\infty
\sqrt{44T_k\vep_k}<\vep_{n},
\quad
4T_{n+1}^{1/2}/S_{n+1}^{1/2}<\vep_{n+1}/24T_{n+1}^2.
\end{equation}
The last inequality
satisfies (\ref{eq:4T2T}) for $\de=\vep_{n+1}$.
For a finite set $\Ph\subs\cM_*$,
we define
\[
d(\Ph):=\max
\left(\{1\}\cup\{\|\vph\|\mid \vph\in\Ph\}\right).
\]

\begin{thm}[2-cohomology vanishing]
\label{thm:Borel2van}
Let $(\al,c)$ be a Borel cocycle action
of $\R$ on a von Neumann algebra $\cM$.
Suppose that
$(\al,c)$ has the Rohlin property,
and
$\al$ is an ergodic flow on $Z(\cM)$
that is not conjugate to the translation
on $L^\infty(\R)$.
Then the following statements hold:
\begin{enumerate}
\item
The 2-cocycle
$c$ is a coboundary,
that is,
there exists a Borel unitary path $v$ in $\cM$
such that
\[
v(t)\al_t(v(s))c(t,s)v(t+s)^*=1
\quad
\mbox{for almost every }(t,s)\in\R^2;
\]

\item
If for some $n\geq 2$
and a finite set $\Ph\subs(\cM_*)_+$,
one has
\[
\int_{-T_{n+1}}^{T_{n+1}}
dt
\int_{-T_{n+1}}^{T_{n+1}}
ds
\,
\|c(t,s)-1\|_\varphi^{\sharp}
\leq
\vep_{n+1}
\quad
\mbox{for all }
\varphi\in \Phi,
\]
then
one can choose
$v(t)$ in (1)
such that
\[
\int_{-T_n}^{T_n}
\|v(t)-1\|_\vph^{\sharp}
\,dt
<\vep_{n-1}
d(\Ph)^{1/2}
\quad
\mbox{for all }
\varphi\in \Phi;
\]

\item
If for some $n\geq 2$
and a finite set $\Ph\subs\cM_*$,
one has
\[
\int_{-T_{n}}^{T_{n}}
dt
\int_{0}^{T_{n+1}}
ds\,
\|[c(t,s),\vph]\|
<\vep
\quad
\mbox{for all }
\varphi\in \Phi,
\]
then one can take $v$ in (1)
satisfying
\[
\int_{-T_{n}}^{T_{n}}\|[v(t),\vph]\|
\,dt
\leq
(3\vep_{n-1}+3\vep)
d(\Ph)
\quad
\mbox{for all }
\varphi\in \Phi.
\]
\end{enumerate}
\end{thm}
\begin{proof}
We may assume that $\Ph$ is contained in
the unit ball of $\cM_*$.

(1), (2).
First we assume that $\cM$ is finite.
Let $\ta\in\cM_*$ be a faithful tracial state.
Let $I_n:=[-T_n,T_n]$
and $J_n:=I_n\times I_n$.

Employing Lemma \ref{lem:Borelperturb},
we have a Borel unitary path $v_n(t)$
such that
with $\al_t^n:=\Ad v_n(t)\circ\al_t$
and $c_n(t,s):=v_n(t)\al_t(v_n(s))c(t,s)v_n(t+s)^*$,
we get
\begin{equation}
\label{eq:Jncn}
\int_{J_{n+1}}\|c_n(t,s)-1\|_2
\,dtds
<
\vep_{n+1},
\end{equation}
where $\|\cdot\|_2=\|\cdot\|_\ta=\|\cdot\|_\ta^\sharp$.
It suffices to prove (1) and (2)
for $\al^n$.
Then $(\al^n,c^n)$ has the Rohlin property by Lemma \ref{lem:pertRohstab}.
Again by Lemma \ref{lem:Borelperturb},
there exists a Borel path $u(t)$
such that
\[
\int_{J_{n+2}}
\|u(t)\al_t^n(u(s))c_n(t,s)u(t+s)^*-1\|_2
\,dtds
<\vep_{n+2}.
\]
By (\ref{eq:Jncn}),
we have
\begin{equation}\label{eq:Jn1}
\int_{J_{n+1}}
\|u(t)\al_t^n(u(s))u(t+s)^*-1\|_2
\,dtds
<2\vep_{n+1}.
\end{equation}

Let $e(\la)\in\Meq$ be a Rohlin projection
over $[0,S_n)$ with respect to $\al$.
Then
$e(\la)$ is also a Rohlin projection
for $\al^n$ by Lemma \ref{lem:pertRohstab}.
Set $W:=\Th(\tilde{u}(\cdot))\in\Mequ$,
where $\tilde{u}$ is the periodization
of $u$ with period $S_n$.
By Lemma \ref{lem:theta},
$\al_t^n(W)=\Th(\al_t^n(\tilde{u}(\cdot-t)))$.
For $0\leq t\leq T_n$,
we have
\[
\tilde{u}(\la-t)
=u(\la-t)1_{[t,S_n)}(\la)
+u(\la-t+S_n)1_{[0,t)}(\la).
\]
Then
\begin{align*}
W^*u(t)\al_t^n(W)
&=
\Th(\tilde{u}(\cdot)
u(t)\al_t^n(\tilde{u}(\cdot-t)))
\\
&=
\Th(u(\cdot)^*
u(t)\al_t^n(u(\cdot-t))
1_{[t,S_n)}(\cdot))
\\
&\quad+
\Th(u(\cdot)^*
u(t)\al_t^n(u(\cdot-t+S_n))1_{[0,t)}(\cdot)).
\end{align*}
Thus
\begin{align*}
&\int_{0}^{T_n}
\|W^*u(t)\al_t^n(W)-1\|_2^2
\,dt
\\
&\leq
2\int_{0}^{T_n}
\|
\Th((u(\cdot)^*u(t)\al_t^n(u(\cdot-t)-1)
1_{[t,S_n)}(\cdot))\|_2^2
\,dt
\\
&\quad
+
2\int_{0}^{T_n}
\|\Th(u(\cdot)^*
u(t)\al_t^n(u(\cdot-t+S_n)-1)1_{[0,t)}(\cdot))
\|_2^2
\,dt
\\
&\leq
\frac{2}{S_n}
\int_{0}^{T_n}
dt
\int_{t}^{S_n}
ds\,
\|u(s)^*u(t)\al_t^n(u(s-t))-1\|_2^2
\\
&\quad
+
\frac{2}{S_n}
\int_{0}^{T_n}
4t
\,dt
\\
&\leq
\frac{2}{S_n}
\int_{J_{n+1}}
\|u(t+s)^*u(t)\al_t^n(u(s))-1\|_2^2
\,
dtds
\\
&\quad
+
4T_n^2/S_n
\\
&\leq
8\vep_{n+1}/S_n+4T_n^2/S_n<9\vep_n
\quad\mbox{by }
(\ref{eq:SnTn}),
\
(\ref{eq:Jn1}).
\end{align*}
Similarly, we have the same inequality as the above
for the integration over $[-T_n,0]$.
Hence
\[
\int_{-T_n}^{T_n}
\|W^*u(t)\al_t^n(W)-1\|_2^2
\,dt
<18\vep_n.
\]

Let $(w^\nu)_\nu$ be a representing
unitary sequence of $W\in\Mequ$.
Take a compact set $K\subs[-T_n,T_n]$
with
$\mu(K)>2T_n(1-\vep_n/2T_n)$,
such that
$\al$, $v_n$
and $u$
are continuous on $K$,
and moreover,
the family
$\{K\ni t\mapsto \al_t^n(w^\nu)\}_\nu$
is $\om$-equicontinuous.
Then we have the following estimate
by Lemma \ref{lem:uniconverge}:
\[
\lim_{\nu\to\om}
\int_{K}
\|(w^\nu)^*u(t)\al_t^n(w^\nu)-1\|_2^2
\,dt
=
\int_{K}
\|W^*u(t)\al_t^n(W)-1\|_2^2
\,dt
<18\vep_n.
\]
Hence
\begin{align*}
&\lim_{\nu\to\om}
\int_{-T_n}^{T_n}
\|(w^\nu)^*u(t)\al_t^n(w^\nu)-1\|_2^2
\,dt
\\
&<18\vep_n
+
\lim_{\nu\to\om}
\int_{K^c\cap[-T_n,T_n]}
\|(w^\nu)^*u(t)\al_t^n(w^\nu)-1\|_2^2
\,dt
\\
&\leq
18\vep_n
+
4\mu(K^c\cap[-T_n,T_n])
\\
&<
18\vep_n
+
4\vep_n=22\vep_n.
\end{align*}

Thus for some $\nu\in\N$,
we obtain
\[
\int_{-T_n}^{T_n}
\|(w^\nu)^*u(t)\al_t^n(w^\nu)-1\|_2^2
\,dt
<22\vep_n.
\]
We set
$v_{n+1}(t):=(w^\nu)^*u(t)\al_t^n(w^\nu)$
for $t\in\R$.
Then
\begin{align*}
\int_{-T_n}^{T_n}\|v_{n+1}(t)-1\|_2\,dt
&\leq
(2T_n)^{1/2}
\left(\int_{-T_n}^{T_n}\|v_{n+1}(t)-1\|_2^2\,dt\right)^{1/2}
\\
&
<(44T_n\vep_n)^{1/2},
\end{align*}
and
\[
\int_{J_{n+2}}
\|v_{n+1}(t)\al_t^n(v_{n+1}(s))
c^n(t,s)v_{n+1}(t+s)^*-1\|_2
\,dtds
<\vep_{n+2}.
\]
Let
$(\al^{n+1},c_{n+1})$ be the perturbation
of $(\al^n,c_n)$ by $v_{n+1}$.

Repeating the above process,
we obtain a family of Borel cocycle actions
$(\al^k,c_k)$
and Borel unitary paths $v_k$,
$k=n,n+1,\dots$
such that
$(\al^{k+1},c_{k+1})$
is the perturbation
of $(\al^k,c_k)$ by $v_{k+1}$,
and
for $k\geq n+1$,
\[
\int_{I_{k-1}}
\|v_k(t)-1\|_2
\,dt
<
(44T_{k-1}\vep_{k-1})^{1/2},
\quad
\int_{J_{k+1}}
\|c_k(t,s)-1\|_2
\,dtds
<\vep_{k+1}.
\]
Then a subsequence of
$v_k(t)v_{k-1}(t)\cdots v_{n+1}(t)$
strongly converges
to a Borel path $v(t)$
almost everywhere on $\R$,
and
we have
$v(t)\al_t^n(v(s))c(t,s)v(t+s)^*=1$
almost everywhere on $\R^2$.
On the norm $\|v(t)-1\|_2$,
we have
\begin{align*}
\int_{-T_n}^{T_n}
\|v(t)-1\|_2\,dt
&=\limsup_{k\to\infty}
\int_{-T_n}^{T_n}
\|v_k(t)v_{k-1}(t)\cdots v_{n+1}(t)-1\|_2
\,dt\\
&\leq
\sum_{k=n}^\infty\sqrt{44 T_k\vep_k}
<\vep_{n-1}
\quad
\mbox{by }
(\ref{eq:SnTn}).
\end{align*}

Next we consider the case
that $\cM$ is properly infinite.
By Lemma \ref{lem:Borelperturb},
we perturb $(\al,c)$ to $(\al^n,c^n)$ so that
\[
\int_{J_{n+1}}\|c^n(t,s)-1\|_\vph^\sharp\,dtds
<\vep_{n+1}
\quad
\vph\in\Ph.
\]
By Lemma \ref{lem:propinf2coho},
we can take
a Borel map $u(t)$ as
\[
u(t)\alpha_t^n(u(s))c^n(t,s)u(t+s)^*=1
\quad
\mbox{for all }
(t,s)\in\R^2.
\]
Thus
we have
\[
\int_{J_{n+1}}
\|u(t+s)^*u(t)\alpha_t^n(u(s))-1\|_\vph^{\sharp}
\,dtds
<\varepsilon_{n},
\quad
|t|,|s|\leq T_{n+1},
\
\vph\in\Ph.
\]

A computation as given in the finite case
shows that for $W:=\Th(\tilde{u})$,
we have
\[
\int_{-T_n}^{T_n}
\|W^*u(t)\alpha^n_t(W)-1\|_{\vph}^{\sharp\,2}
\,dt
<18\vep_n
\quad
\vph\in\Ph.
\]
Then we can prove (1) and (2) in a similar
way to the above.

(3).
We may assume that $\Ph^*=\Ph$.
By Lemma \ref{lem:Borelperturb},
we find a Borel unitary path $w$
such that
\[
\int_{-T_{n+1}}^{T_{n+1}} dt
\int_{-T_{n+1}}^{T_{n+1}} ds\,
\|w(t)\al_t(w(s))c(t,s)w(t+s)^*-1\|_{|\vph|}^\sharp
<\vep_{n+1},
\]
and
$\int_{-T_n}^{T_n}
\|[\vph,w(t)]\|
\,dt<3\vep$
for all $\vph\in \Ph$.
Let $(\al',c')$ be the perturbation
of $(\al,c)$ by $w$.
Then $(\al',c')$ is a Borel cocycle action
with Rohlin property.
By (2),
there exists a Borel unitary path
$v$ such that
$v(t)\al_t'(v(s))c'(t,s)v(t+s)^*=1$
almost everywhere on $\R^2$,
and
$\int_{-T_n}^{T_n}
\|v(t)-1\|_{|\vph|}^\sharp\,dt
<\vep_{n-1}$
for all $\vph\in\Ph$.
Thus
\begin{align*}
\int_{-T_n}^{T_n}
\|[\vph,v(t)]\|\,dt
&=
\int_{-T_n}^{T_n}
\|[\vph,v(t)-1]\|\,dt
\\
&\leq
\int_{-T_n}^{T_n}
\|v(t)^*-1\|_{|\vph|}\,dt
+
\int_{-T_n}^{T_n}
\|v(t)-1\|_{|\vph^*|}\,dt
\\
&<
2\sqrt{2}\vep_{n-1}.
\end{align*}
Then
we obtain
\begin{align*}
\int_{-T_n}^{T_n}
\|[\vph,v(t)w(t)]\|\,dt
&\leq
\int_{-T_n}^{T_n}
\|[\vph,v(t)]\|\,dt
+
\int_{-T_n}^{T_n}
\|[\vph,w(t)]\|\,dt
\\
&<
2\sqrt{2}\vep_{n-1}+3\vep.
\end{align*}
\end{proof}

\subsection{Approximation by cocycle perturbation}
Let $\al,\be$ be flows on a von Neumann algebra $\cM$
with $\al_t\be_t^{-1}\in\oInt(\cM)$.
Then
we can approximate $\be_t$ by a perturbation
of $\al_t$ for finite $t$'s.
We would like to connect these unitaries
by a continuous path,
but we have not solved this problem.
We can do for an ITPFI factor
with a lacunary product state
(see Proposition \ref{prop:appro}).
Instead, we connect those by a Borel unitary path.

\begin{lem}
\label{lem:adBorel}
For any $T>0$, $\vep>0$
and a compact set $\Ph\subs\cM_*$,
there exists a Borel unitary path
$\{u(t)\}_{|t|\leq T}$
such that
\[
\|\Ad u(t)\circ\al_t(\vph)-\be_t(\vph)\|
<\vep
\quad
\mbox{for all }
\vph\in\Ph,
\
t\in[-T,T].
\]
\end{lem}
\begin{proof}
Since $\Ph$ is compact,
we can take a finite set
$\Ph_0\subs\Ph$
such that
each $\vph\in\Ph$ has $\vph_0\in\Ph_0$
such that
$\|\vph-\vph_0\|<\vep/4$.
Choose a large $N\in\N$
so that
\[
\|\al_t(\vph)-\vph\|<\vep/6,
\quad
\|\be_t(\vph)-\vph\|<\vep/6
\quad
\mbox{for all }
\vph\in\Ph_0,
\
|t|\leq T/N.
\]
For each $t_j:=jT/N$,
$j=-N,\dots,N$,
we can take a unitary $u_j$
such that
\[
\|\Ad u_j\circ\al_{t_j}(\vph)-\be_{t_j}(\vph)\|
<\vep/6
\quad
\mbox{for all }
\vph\in\Ph_0,
\
j=-N,\dots,N.
\]
We put
a unitary
$u(t):=
\sum_{j=-N}^{N-1}u_j 1_{[t_j,t_{j+1})}(t)+u_N 1_{\{t_N\}}(t)$.
Then for $t\in[t_j,t_{j+1})$, $j=-N,\dots,N-1$
and $\vph\in\Ph_0$,
we have
\begin{align*}
\|\Ad u(t)\circ\al_t(\vph)-\be_t(\vph)\|
&=
\|\Ad u_j\circ\al_t(\vph)-\be_t(\vph)\|
\\
&\leq
\|\Ad u_j(\al_t(\vph)-\al_{t_j}(\vph))\|
+
\|\Ad u_j\circ\al_{t_j}(\vph)-\be_{t_j}(\vph)\|
\\
&\quad
+
\|\be_{t_j}(\vph)-\be_t(\vph)\|
\\
&<\vep/6+\vep/6+\vep/6=\vep/2.
\end{align*}
For $\vph\in\Ph$,
take $\vph_0\in\Ph_0$ such that
$\|\vph-\vph_0\|<\vep/4$.
Then
\begin{align*}
\|\Ad u(t)\circ\al_t(\vph)-\be_t(\vph)\|
&\leq
\|\Ad u(t)
\left(
\al_t(\vph)
-\al_t(\vph_0)
\right)\|
\\
&\quad
+
\|\Ad u(t)\circ\al_t(\vph_0)-\be_t(\vph_0)\|
+\|\be_t(\vph_0)-\be_t(\vph)\|
\\
&\leq
\|\vph-\vph_0\|+\vep/2+\|\vph-\vph_0\|
<\vep.
\end{align*}
\end{proof}

\begin{lem}\label{lem:Borel2-cocycle1}
For any $T>0$, $\vep>0$
and a finite set $\Phi\subset \cM_*$,
there exists a Borel unitary path
$\{u(t)\}_{|t|\leq 2T}$
such that 
\[
\|\Ad u(t)\circ \alpha_t(\varphi)-\beta_t(\varphi)\|
<\varepsilon
\quad
\mbox{for all }
\varphi\in \Phi,
\
t\in[-2T,2T]
\]
\[
\|[u(t)\alpha_t(u(s))u(t+s)^*, \vph]\|
<\varepsilon
\quad
\mbox{for all }
\varphi\in \Phi,
\
t,s\in [-T,T].
\]
\end{lem}
\begin{proof}
Let
$\Ps:=
\{\be_t(\vph)\mid \vph\in\Ph, |t|\leq 2T\}$,
which is a compact set by continuity of $\be$.
By the previous lemma,
we can take a Borel unitary path
$u(t)$
such that
\[
\|\Ad u(s)\circ \alpha_s(\ps)-\beta_s(\ps)\|
<\vep/3
\quad
\mbox{for all }
\ps\in\Ps,
s\in[-2T,2T].
\]
Then for $t,s\in[-T,T]$ and $\ps\in\Ps$,
we have
\begin{align*}
&\|\Ad u(t)\al_t(u(s))\circ \alpha_{t+s}(\ps)-
\beta_{t+s}(\ps)\|
\\
&\leq
\|\Ad u(t)\circ\al_t(\Ad u(s)\circ
\alpha_{s}(\ps))
-
\Ad u(t)\circ\al_t(\beta_{s}(\ps))\|
\\
&\quad
+
\|\Ad u(t)\circ\al_t(\beta_{s}(\ps))
-\be_t(\be_s(\ps))\|
\\
&=
\|\Ad u(s)\circ \alpha_s(\ps)-\beta_s(\ps)\|
+
\|\Ad u(t)\circ\al_t(\beta_{s}(\ps))
-\be_t(\be_s(\ps))\|
\\
&<
\vep/3+\vep/3=2\vep/3.
\end{align*}
Together with
$\|\al_{t+s}(\ps)-\Ad u(t+s)^*\be_{t+s}(\ps)\|<\vep/3$,
we have
\[
\|\Ad u(t)\al_t(u(s))u(t+s)^*
\circ\be_{t+s}(\ps)
-
\be_{t+s}(\ps)\|<\vep
\quad
\mbox{for all }
\ps\in\Ps,
\
t,s\in[-T,T].
\]
Since $\be_{-t-s}(\vph)\in\Ps$
for $\vph\in\Ph$,
we are done.
\end{proof}

\begin{lem}\label{lem:Borelapproby1cocycle}
Suppose that
$\alpha$ is a centrally ergodic
Rohlin flow on $\cM$
such that $\Sp_d(\al|_{Z(\cM)})\neq\R$.
Then for any $T>0$, $\vep>0$
and finite set $\Phi\subset \cM_*$,
there exists an $\al$-cocycle $u$
for $\alpha$ such that
\[
\int_{-T}^T
\left\|\Ad u(t)\circ \alpha_t(\varphi)
-\beta_t(\varphi) \right\|
\,dt
<\vep
\quad\mbox{for all }
\varphi\in \Phi.
\]
\end{lem}
\begin{proof}
We may and do assume that $\Ph$
is contained in the unit ball of $\cM_*$.
Let $\eta>0$ be such that
$8\eta^{1/4}<1$
and
$2\eta+8\eta^{1/4}<\vep$.
Take a large $N\in\N$
so that
\begin{equation}
\label{eq:alvphibe}
\|\alpha_t(\varphi)-\varphi\|<\eta,
\quad
\|\beta_t(\varphi)-\varphi\|<\eta/T
\quad
\mbox{for all }
\varphi\in \Phi,
\
|t|\leq T/N.
\end{equation}
Set
\begin{equation}
\label{eq:ANT}
A(N,T)
:=
\{t_j\mid j=-N,\dots,N\},
\quad
t_j:=jT/N,
\end{equation}
and
\[
\Ps:=\{\beta_t(\varphi)
\mid \varphi\in \Phi,
\
t\in A(N,T)\}.
\]
Recall (\ref{eq:SnTn})
and
fix $n\in\N$ with $T<T_n$ and $3\vep_{n-1}<\eta/(4N+2)$.
By Lemma \ref{lem:Borel2-cocycle1},
there exists a Borel unitary path
$v(t)$
such that 
\[
\|\Ad v(t)
\circ\alpha_t(\ps)
-\beta_t(\ps)\|
\,dt
<\eta/T_{n+1}
\quad
\mbox{for all }
\ps\in \Ps,
\
|t|,|s|
\leq
2T_{n+1},
\]
\[
\|[v(t)\alpha_t(v(s))v(s+t)^*, \ps]\|
<\eta/(8N+4)T_nT_{n+1}
\quad
\mbox{for all }
\ps\in \Ps,
\
|t|,|s|
\leq T_{n+1}.
\]

Set
$\ga_t:=\Ad v(t)\circ\alpha_t$,
and $c(t,s):=v(t)\alpha_t(v(s))v(t+s)^*$.
Then $(\ga,c)$ is a Borel cocycle action.
By Theorem \ref{thm:Borel2van},
there exists a Borel unitary path $w$
such that
$w(t)\ga_t(w(s))c(t,s)w(t+s)^*=1$
almost everywhere on $\R^2$,
and
\[
\int_{-T_n}^{T_n}
\|[\beta_{t_j}(\vph),w(t)]\|\,dt
<2\eta/(2N+1)
\quad\mbox{for all }
j=-N,\dots,N,
\
\vph\in\Ph.
\]
It is known that
there exists a Borel $\al$-cocycle
$u(t)$ that coincides with $w(t)v(t)$
almost everywhere on $\R$
(see \cite[Remark III.1.9]{CT}).
Moreover,
any Borel $\al$-cocycle is automatically strongly continuous
(see the remark after the proof).
We set
\[
B_{\vph,j}
:=\{t\in[-T_n,T_n]
\mid
\|[\be_{t_j}(\vph),w(t)]\|
<\eta^{1/2}/T_n
\},
\quad
j=-N,\dots,N,
\
\vph\in\Ph.
\]
By the Chebyshev inequality,
we have
$\mu(B_{\vph,j}^c)<\eta^{1/2}/(2N+1)$,
where $B_{\vph,j}^c=[-T_n,T_n]\setminus B_{\vph,j}$.
We let
$B_\vph:=\bigcap_{j=-N}^N B_{\vph,j}$.
Then
$
\mu(B_\vph^c)\leq\sum_j\mu(B_{\vph,j}^c)
<
\eta^{1/2}
$.
Thus
for all $\vph\in\Ph$,
\begin{align}
\int_{t_j}^{t_{j+1}}
\|[\be_{t_j}(\vph),w(t)]\|\,dt
&=
\int_{B_\vph\cap[t_j,t_{j+1}]}
\|[\be_{t_j}(\vph),w(t)]\|\,dt
\notag\\
&\quad
+
\int_{B_\vph^c\cap[t_j,t_{j+1}]}
\|[\be_{t_j}(\vph),w(t)]\|\,dt
\notag\\
&<
\eta^{1/2}T/NT_n
+
2\mu(B_\vph^c\cap[t_j,t_{j+1}]).
\label{eq:TjBvph}
\end{align}

Therefore,
for $\vph\in\Ph$,
we have
\begin{align*}
\int_{-T}^{T}
\|[\be_t(\vph),w(t)]\|\,dt
&=
\sum_{j=-N}^{N-1}
\int_{t_j}^{t_{j+1}}
\|[\be_t(\vph),w(t)]\|\,dt
\\
&\leq
\sum_{j=-N}^{N-1}
\int_{t_j}^{t_{j+1}}
\|[\be_t(\vph)-\be_{t_j}(\vph),w(t)]\|\,dt
\\
&\quad
+
\sum_{j=-N}^{N-1}
\int_{t_j}^{t_{j+1}}
\|[\be_{t_j}(\vph),w(t)]\|\,dt
\\
&<
\sum_{j=-N}^{N-1}
\int_{t_j}^{t_{j+1}}
2\|\be_t(\vph)-\be_{t_j}(\vph)\|
\,dt
\\
&\quad
+
\sum_{j=-N}^{N-1}
\left(
\eta^{1/2}T/NT_n
+
2\mu(B_\vph^c\cap[t_j,t_{j+1}])
\right)
\quad
\mbox{by }
(\ref{eq:TjBvph})
\\
&<
2N\times 2\eta/N
+
2\eta^{1/2}
+2\mu(B_\vph^c)
\quad
\mbox{by }
(\ref{eq:alvphibe})
\\
&\leq
4\eta+2\eta^{1/2}
+2\eta^{1/2}
<
8\eta^{1/2}.
\end{align*}

By the inequality above,
we obtain
\begin{align*}
\int_{-T}^{T}
\|\Ad u(t)\circ\al_t(\vph)
-\be_t(\vph)\|
\,dt
&=
\int_{-T}^{T}
\|\Ad w(t)v(t)\circ\al_t(\vph)
-\be_t(\vph)\|
\,dt
\\
&\leq
\int_{-T}^{T}
\|\Ad w(t)(\Ad v(t)\circ\al_t(\vph)-\be_t(\vph))\|
\,dt
\\
&\quad+
\int_{-T}^{T}
\|\Ad w(t)\circ\be_t(\vph)-\be_t(\vph)\|
\,dt.
\\
&<
2\eta+
\int_{-T}^{T}
\|[w(t),\be_t(\vph)]\|
\,dt
\\
&<2\eta+8\eta^{1/2}
<\vep.
\end{align*}
\end{proof}

\begin{rem}
Let $\al$ be a flow
on a separable von Neumann algebra
$\cM$,
and
$v$ a Borel $\al$-cocycle in $\cM$.
Then $v$ is strongly continuous.
Indeed,
in the crossed product $\cM\rtimes_\al\R$,
we have the Borel
one-parameter unitary group
$v(t)\la^\al(t)$.
Since $(\cM\rti_\al\R)^{\rm U}$ is Polish,
$v(t)\la^\al(t)$ is continuous,
and so is $v(t)$.
\end{rem}

\begin{lem}\label{lem:uvphal}
Let $\alpha$ be a flow on a von Neumann algebra
$\cM$
and $u$ an $\al$-cocycle.
Then for all $t\in\R$ and $\vph\in\cM_*$,
one has
\[
\|[u(t),\vph]\|
=
\|\al_{-t}(\vph)-\Ad u(-t)\circ\al_{-t}(\vph)\|.
\]
\end{lem}
\begin{proof}
Since $\Ad u(t)\circ \alpha_t\circ \Ad u(-t)\circ
\alpha_{-t}=\id$,
we have
\begin{align*}
\|[u(t),\vph]\|
&=
\|\Ad u(t)(\varphi)-\varphi\|
\\
&=
\|\Ad u(t)\circ \alpha_t\circ\alpha_{-t}(\varphi)
-\Ad u(t)\circ 
\alpha_t\circ \Ad u(-t)\circ\alpha_{-t}(\varphi)\|\\
&=
\|\alpha_{-t}(\varphi)-\Ad u(-t)\circ \alpha_{-t}(\varphi)\|. 
\end{align*}
\end{proof}

\subsection{Approximate vanishing of 1-cohomology}

\begin{thm}
\label{thm:Borelnear1vanish}
Let $\alpha$ be a Rohlin flow on a von Neumann algebra $\cM$.
Suppose that
$\al$ is centrally ergodic
and $\Sp_d(\al|_{Z(\cM)})\neq\R$.
Let $\vep,\de,T>0$
and $\Ph\subs\cM_*$ a finite set.
Take $S>0$
such that
$T\|\vph\|/S<\vep^2/4$
for all $\vph\in\Ph$
and $(2\pi/S)\Z\cap \Sp_d(\al|_{Z(\cM)})=\{0\}$.
Then for
any $\al$-cocycle $u$
with
\[
\frac{1}{S}
\int_{0}^S
\|[u(t), \varphi]\|
\,dt
<\delta
\quad\mbox{for all }
\varphi \in \Phi,
\]
there exists a unitary $w\in\cM$ such that 
\[
\|[w,\varphi]\|<3\delta
\quad\mbox{for all } \varphi\in \Phi,
\]
\[
\|\varphi\cdot (u(t)\alpha_t(w)w^*-1)\|<\vep,
\ 
\|(u(t)\alpha_t(w)w^*-1)\cdot \varphi\|<\vep
\quad\mbox{for all }
|t|\leq T,
\varphi\in \Phi.
\]
\end{thm}
\begin{proof}
We may and do assume that $\Ph^*=\Ph$.
Let $e(\la)\in \cM_{\om,\al}$ be
a Rohlin projection over $[0,S)$. 
Put $W:=\Th(\tilde{u})\in U(\cM_\al^\om)$,
where $\tilde{u}$ is the periodization of $u$
with period $S$.
Then it is trivial that
\[
\tilde{u}(\la-t)=u(\la-t)1_{[t,S)}(\la)
+u(\la-t+S)1_{[0,t)}(\la)
\quad\mbox{for all }
\la\in[0,S).
\]
For $0\leq t\leq T$, we have
\begin{align*}
u(t)\alpha_t(W)W^*
&=
\Th(u(t)\al_t(\tilde{u}(\cdot-t))
\tilde{u}(\cdot)^*)
\\
&=
\Th(u(t)\al_t(u(\cdot-t))u(\cdot)^*1_{[t,S)}(\cdot))
+
\Th(u(t)\al_t(u(\cdot-t+S))u(\cdot)^*1_{[0,t)}(\cdot))
\\
&=
\Th(1_{[t,S)}(\cdot))
+
\Th(u(t)\al_t(u(\cdot-t+S))u(\cdot)^*1_{[0,t)}(\cdot)).
\end{align*}
Hence for all $\vph\in\Ph$,
we have
\begin{align*}
\|u(t)\alpha_t(W)W^*-1\|_{|\vph|^\om}^\sharp
&\leq
\|\Th(1_{[t,S)}(\cdot)-1)\|_{|\vph|^\om}^\sharp
\\
&\quad
+
\|\Th(u(t)\al_t(u(\cdot-t+S))u(\cdot)^*1_{[0,t)}(\cdot))
\|_{|\vph|^\om}^\sharp
\\
&=
2\|1_{[0,t)}(\cdot)\|_{|\vph|\oti\mu}^\sharp
\\
&\leq
2t^{1/2}\|\vph\|^{1/2}/S^{1/2}.
\end{align*}
The same estimate as the above
is valid for $-T\leq t\leq 0$.
Thus
if $|t|\leq T$,
we have
\[
\|u(t)\alpha_t(W)W^*-1\|_{|\vph|^\om}^\sharp
\leq
2t^{1/2}\|\vph\|^{1/2}/S^{1/2}
<\vep
\quad\mbox{for all }
\vph\in\Ph.
\]
Take a representing sequence
$(w^\nu)_\nu$
of $W\in \cM_\al^\om$. 
By Lemma \ref{lem:uniconverge},
we have the following uniform convergence
on $[-T,T]$:
\[
\lim_{n\rightarrow \omega}
\|u(t)\alpha_t(w^\nu)(w^\nu)^*-1\|_{|\vph|}^\sharp
=
\|u(t)\alpha_t(W)W^*-1\|_{|\vph|^\om}^\sharp.
\]
Hence if $\nu\in\N$ is close to $\om$,
then
for all $t\in[-T,T]$,
\[
\|\vph\cdot(u(t)\alpha_t(w^\nu)(w^\nu)^*-1)\|
<\vep,
\quad
\|(u(t)\alpha_t(w^\nu)(w^\nu)^*-1)\cdot\vph\|
<\vep.
\]
Applying Lemma \ref{lem:Borelasymprep}
to $u(t,s):=u(s)$,
we have
$\|[w^\nu, \vph]\|<3\de$
for $\nu$ being close to $\om$.
\end{proof}

\subsection{Proof of the main theorem}

We will prove our main theorem
for centrally ergodic Rohlin flows
by using the Bratteli-Elliott-Evans-Kishimoto intertwining argument \cite{EK}.

\begin{lem}
\label{lem:class}
Let $\alpha$ and $\beta$ be Rohlin flows
on a von Neumann algebra $\cM$.
Suppose that
$\al$ is centrally ergodic.
Then
$\alpha$ and $\beta$ are strongly cocycle conjugate
if and only if
$\al_t\be_{-t}\in\oInt(\cM)$ for all $t\in\R$.
\end{lem}
\begin{proof}
The ``only if'' part is trivial.
We will prove the ``if'' part.
Assume
$\al_t\be_{-t}\in\oInt(\cM)$ for all $t\in\R$.

\vspace{5pt}
\noindent
{\bf Case 1.}
$\Sp_d(\al|_{Z(\cM)})=\R$.
\vspace{5pt}

In this case,
the covariant system $\{L^\infty(\R),\Ad\la\}$
embeds into $\{Z(\cM),\al\}$.
Since $\al$ is centrally ergodic,
this embedding is surjective.
(See Remark \ref{rem:duality-cocycle}.)
Note that $\al=\be$ on $Z(\cM)$
since $\al_t\be_{-t}\in\oInt(\cM)$.
By duality theorem,
we obtain the following decompositions:
\begin{equation}
\label{eq:Mtensordecomp}
\cM
=
\cM^\al\vee Z(\cM)
\cong\cM^\al\oti Z(\cM).
\end{equation}
Note that $Z(\cM^\al)\subs Z(\cM)$,
and $Z(\cM^\al)=\C=Z(\cM^\be)$,
that is,
the fixed point algebra
$\cM^\al$ is a factor.

The ergodic flow $\al=\be$ on $Z(\cM)$
is identified with the translation
on $L^\infty(\R)$, which produces
the groupoid $\sG:=\R\ltimes\R$.
Applying \cite[Corollary XIII.3.29]{TaIII}
to $\sG$
and $\al,\be\col \sG\ra G$
with $G:=\oInt(\cM^\al)$
and $H:=\Int(\cM^\al)$,
we can find $\th\in \Aut(\cM)$
and a Borel unitary path $u\col\R\ra\cM$
such that
$\th=\int_\R \th_x\,dx$ with $\th_x\in\oInt(\cM^\al)$
and
\[
\Ad u(t)\circ\al_t=\th\circ\be_t\circ\th^{-1}.
\]
We should note that
the statement of \cite[Corollary XIII.3.29]{TaIII}
is concerned with a properly ergodic flow,
but the proof is also applicable to $\R\ltimes\R$
by setting a base space $\Z$ with natural transformation
and a ceiling function $r=1$.

By Theorem \ref{thm:appdisint},
$\th$ is approximately inner.
If we use \cite[Proposition XIII.3.34]{TaIII},
then it turns out that we may arrange $u(t)$
to an $\al$-cocycle.
In our case, it is directly checked as follows.
Since $\be$ is a flow,
the cocycle action $(\Ad u(t)\circ\al_t,c(t,s))$
must be a flow, where we have put
$c(t,s)=u(t)\al_t(u(s))u(t+s)^*$ which belongs to $Z(\cM)$.
By the conjugacy
$\{\cM,\al\}\cong\{\cM^\al\oti L^\infty(\R),\id\oti\Ad\la\}$,
$c$ is regarded as an $L^\infty(\R)$-valued
2-cocycle with respect to the translation.
Then $c$ is a coboundary by \cite[Proposition A.2]{CT}.
Hence we may assume that $v$ is an $\al$-cocycle,
and we are done.

\vspace{5pt}
\noindent
{\bf Case 2.}
$\Sp_d(\al|_{Z(\cM)})\neq\R$.
\vspace{5pt}

Take $\vep_n,T_n>0$ and $S_n>0$
which satisfy (\ref{eq:SnTn}).
Let us denote $H_\al:=\Sp_d(\al|_{Z(\cM)})$
as before.
We should note that the choice of
$S_n$ depends on $\al$,
that is,
$(2\pi/S_n)\Z\cap H_\al=\{0\}$.
In what follows,
we introduce a sequence of
flows $\ga^{(m)}$.
They are cocycle perturbations
of $\al$ and $\be$,
and $H_\al=H_\be=H_{\ga^{(n)}}$.
Hence $(2\pi/S_n)\Z\cap H_{\ga^{(m)}}=\{0\}$,
and we can apply the preceding results
to $\ga^{(m)}$.

Now for $N\in\N$ and $T>0$,
let
$A(N,T)$ be as defined in (\ref{eq:ANT}).
Let $\{\varphi_i\}_{i=1}^\infty$
be
a dense countable set of
the unit ball of $\cM_*$,
and set $\Phi_n=\{\varphi_i\}_{i=0}^n$
with a faithful state $\vph_0\in\cM_*$.
Set $\hat{\Ph}_0:=\Ph_0$,
$\hat{\Ph}_1:=\Ph_1$,
$\gamma^{(-1)}_t:=\beta_t$
and $\gamma^{(0)}_t:=\alpha_t$.

By Lemma \ref{lem:Borelapproby1cocycle},
there exists a $\gamma^{(-1)}$-cocycle
$u^1(t)$ such that
\[
\int_{-T_2}^{T_2}
\|\Ad u^1(t)\circ \gamma^{(-1)}_t(\varphi)-\gamma_t^{(0)}(\varphi)\|
\,dt
<\vep_1
\quad\mbox{for all }
\varphi\in \hat{\Phi}_1.
\]

Set $\gamma^{(1)}_t:=\Ad u^1(t)\circ \gamma_t^{(-1)}$,
$w_1:=1$,
$\hat{v}^{-1}(t):=1$,
$v^1(t):=u^1(t)=:\hat{v}^1(t)$
and
$\th_{-1}=\th_0=\th_1:=\id$.
Choose $M_1\in\N$ such that
$\|(\hat{v}^1(t)-\hat{v}^1(s))\varphi\|<\vep_1$ and
$\|\varphi\cdot(\hat{v}^1(t)-\hat{v}^1(s))\|<\vep_1$ 
for $t,s\in [-T_1,T_1]$ with
$|t-s|\leq T_1/M_1$ and $\varphi\in \hat{\Phi}_{0}$.

We will inductively construct a flow
$\gamma^{(n)}$,
an automorphism $\theta_n\in \Int(\cM)$,
a $\ga^{(n-2)}$-cocycle  $u^n(t)$, 
unitary paths $v^n(t)$, $\hat{v}^n(t)$,
a unitary $w_n\in \cM^{\rm U}$,
a natural number $M_n\in \mathbb{N}$,
and a finite set $\hat{\Phi}_n\subset \cM$
satisfying the following conditions:
\begin{enumerate}
\renewcommand{\labelenumi}{$(n.\arabic{enumi})$}
\setlength{\itemsep}{-1pt}
\item
$\hat{\Phi}_n
=
\Ph_n
\cup \theta_{n-1}(\Phi_n)
\cup
\left
\{\varphi_0\hat{v}^{k}(t),
\hat{v}^{k}(t)\varphi_0
\mid
t\in A(M_k,T_k),
1\leq k\leq n-1\right\}$;

\item
$\gamma_t^{(n)}=\Ad u^n(t)\circ \gamma^{(n-2)}_t$;

\item
$v^n(t)=u^n(t)\gamma^{(n-2)}_t(w_n)w_n^*$,
$\hat{v}^n(t)=v^n(t)w_n\hat{v}^{n-2}(t)w_n^*$;
\item
$\theta_{n}=\Ad w_n\circ\theta_{n-2}$;

\item
$
\int_{-T_{n+1}}^{T_{n+1}}
\|\gamma_t^{(n)}(\varphi)-\gamma^{(n-1)}_t(\varphi)\|\,dt
<\vep_n$
for
$\varphi\in 
\hat{\Phi}_n$;

\item
$\|(v^n(t)-1)\varphi\|< \vep_n$, 
$\|\varphi\cdot (v^n(t)-1)\|< \vep_n$, $|t|\leq T_n$, $\varphi\in \hat{\Phi}_{n-1}$;

\item
$\|[w_n,\varphi]\|
<3\vep_{n-1}$, $\varphi\in \hat{\Phi}_{n-1}$;

\item
$\|(\hat{v}^n(t)-\hat{v}^n(s))\varphi\|<\vep_n$ and
$\|\varphi\cdot(\hat{v}^n(t)-\hat{v}^n(s))\|<\vep_n$ 
for $t,s\in [-T_n,T_n]$ with
$|t-s|\leq T_n/M_n$ and $\varphi\in \hat{\Phi}_{n-1}$.
\end{enumerate}

Suppose we have constructed them up to the $n$-th step.
Define $\hat{\Phi}_{n+1}$ as the condition $(n+1.1)$.
Employing Lemma \ref{lem:Borelapproby1cocycle},
we take a $\ga^{(n-1)}$-cocycle
$u^{n+1}(t)$
such that 
\[
\int_{-T_{n+2}}^{T_{n+2}}
\|\Ad u^{n+1}(t)\circ \gamma^{(n-1)}_t(\varphi)-\gamma^{(n)}_t(\varphi)\|
\,dt
<\vep_{n+1}
\quad\mbox{for }
\vph\in \hat{\Phi}_n\cup \hat{\Phi}_{n+1}.
\]
Combining this with $(n.5)$, we have
\[
\int_{-T_{n+1}}^{T_{n+1}}
\|\Ad u^{n+1}(t)\circ\gamma^{(n-1)}_t(\varphi)-\gamma^{(n-1)}_t(\varphi)\|
\,dt
<
2\vep_{n}
\quad\mbox{for }
\vph\in \hat{\Phi}_{n}.
\]
By Lemma \ref{lem:uvphal}, 
\[
\int_{-T_{n+1}}^{T_{n+1}}
\|[u^{n+1}(t), \varphi]\|<2\vep_{n}
\quad\mbox{for }
\varphi\in \hat{\Phi}_{n}.
\]
Using $T_n/S_n<\vep_n/4$,
$S_n< T_{n+1}$
and
Theorem \ref{thm:Borelnear1vanish},
we get a unitary $w_{n+1}\in\cM$
such that
\[
\|[w_{n+1},\varphi]\|
<6\vep_n/S_n<3\vep_n
\quad\mbox{for }
\varphi\in \hat{\Phi}_n,
\]
\[
\|\varphi\cdot
(u^{n+1}(t)\gamma^{(n-1)}_t(w_{n+1})w_{n+1}^*-1)\|<\vep_n
\quad
\mbox{if }
|t|\leq T_n,
\
\vph\in\hat{\Ph}_n,
\]
and
\[
\|(u^{n+1}(t)\gamma^{(n-1)}_t(w_{n+1})w_{n+1}^*-1)
\vph\|_2<\vep_n
\quad
\mbox{if }
|t|\leq T_n,
\
\vph\in\hat{\Ph}_n.
\]
Set
$\gamma^{(n+1)}_t=\Ad
u^{n+1}(t)\circ\gamma^{(n-1)}_t$,
$\theta_{n+1}=\Ad w_{n+1}\circ \theta_{n-1}$
and
\[ 
v_{t}^{n+1}=u^{n+1}(t)\gamma^{(n-1)}_t(w_{n+1})w_{n+1}^*,
\quad
\hat{v}^{n+1}(t)
=
v^{n+1}(t)w_{n+1}\hat{v}^{n-1}(t)w_{n+1}^*.
\]
Then the conditions
from $(n+1,2)$ to $(n+1.7)$ are satisfied.
Choose $M_{n+1}\in \mathbb{N}$ as in $(n+1.8)$
and the induction procedure is done.

We first show the convergence
of
$\{\theta_{2n}\}_n$ and $\{\theta_{2n+1}\}_n$.
By $(n+2.7)$,
\[
\|\theta_{n+2}(\varphi)-\theta_n(\varphi)\|=\|[w_{n+2},\theta_n(\varphi)]\|_2<4\vep_{n+1}
\]
and
\[
\|\theta_{n+2}^{-1}(\varphi)-\theta_n^{-1}(\varphi)\|=\|[w_{n+2},\varphi]\|_2< 4\vep_{n+1}
\quad
\mbox{for all }
\varphi\in \Phi_{n+1}.
\]
Thus the limits $\lim_{n\rightarrow
\infty}\theta_{2n}=\sigma_{0}$ and 
$\lim_{n\rightarrow
\infty}\theta_{2n+1}=\sigma_{1}$ exist.

We next show $\hat{v}^{2n}(t)$ 
converges compact uniformly.
If $|t|\leq T_{n+2}$,
then
\begin{align}
&\|\varphi_0\cdot (\hat{v}^{n+2}(t)-\hat{v}^n(t))\|
\notag\\
&=
\|\varphi_0\cdot({v}^{n+2}(t)w_{n+2}\hat{v}^n(t)w_{n+2}^*-\hat{v}^n(t))\|
\notag
\\
&\leq
\|\varphi_0\cdot ({v}^{n+2}(t)-1)\| 
+
\|\varphi_0\cdot (w_{n+2}\hat{v}^n(t)w_{n+2}^*-\hat{v}^n(t))\|
\notag\\
&<
\vep_{n+2}
+
\|\varphi_0\cdot (w_{n+2}\hat{v}^n(t)w_{n+2}^*-\hat{v}^n(t))\|
\quad
\mbox{by }
(n+2.6).
\label{eq:ph0w}
\end{align}
For $|t|\leq T_{n}$, take $t_0\in 
A(M_n,T_n)$ so that
$0\leq |t-t_0|<T_{n}/M_{n}$.
Note $\varphi_0\hat{v}^{n}(t_0)\in \hat{\Phi}_{n+1}$. 
By $(n.8)$,
\[
\|\varphi_0\cdot (\hat{v}^n(t)-\hat{v}^n(t_0))\|
<\vep_n.
\]
Hence the second term of (\ref{eq:ph0w})
can be estimated as follows;
\begin{align*}
&\|\varphi_0\cdot
(w_{n+2}\hat{v}^n(t)w_{n+2}^*-\hat{v}^n(t))\|
\\
&\leq
\|[w_{n+2},\vph_0]\|
+ \|w_{n+2}\varphi_0\hat{v}^n(t)
-\varphi_0\hat{v}^n(t)w_{n+2}\|
\\
&\leq
3\vep_{n+1}
+
2\vep_{n}
+
\|w_{n+2}\varphi_0\hat{v}^n(t_0)
-\varphi_0\hat{v}^n(t_0)w_{n+2}\|
\quad
\mbox{by }
(n+2.7),
\
(n.8)\\
&<
3\vep_{n+1}
+
2\vep_{n}
+
3\vep_{n+1}
<8\vep_n
\quad
\mbox{by }
(n+2.7).
\end{align*}
Thus we get
\[
\|\varphi_0\cdot (\hat{v}^{n+2}(t)
-\hat{v}^n(t))\|<9\vep_n
\quad\mbox{if }
|t|\leq T_n.
\]
We estimate
$\|(\hat{v}^{n+2}(t) -\hat{v}^n(t))\varphi_0\|$
as follows:
\begin{align*}
&\left
\|(\hat{v}^{n+2}(t) -\hat{v}^n(t))\varphi_0\right\|
\\
&=
\|({v}^{n+2}(t)w_{n+2}\hat{v}^n(t)w_{n+2}^* -\hat{v}^n(t))
\varphi_0\|\\ 
&\leq
\|({v}^{n+2}(t)-1)w_{n+2}\hat{v}^n(t)w_{n+2}^*
\varphi_0\|
+
\|(w_{n+2}\hat{v}^n(t)w_{n+2}^*
-\hat{v}^n(t))\varphi_0\|
\end{align*}
In the same way as above,
we can show
$\|(w_{n+2}\hat{v}^n(t)w_{n+2}^*
-\hat{v}^n(t))\varphi_0\|<8\vep_n$.
The first term is estimated as follows.
We take $t_0$ as above.
Since $\hat{v}^n(t_0)\varphi_0\in \hat{\Phi}_{n+1}$,
we have
\begin{align*}
&\|({v}^{n+2}(t)-1)w_{n+2}\hat{v}^n(t)w_{n+2}^*\varphi_0\|
\\
&\leq
2\|[w_{n+2}^*,\vph_0]\|
+
\|({v}^{n+2}(t)-1)w_{n+2}\hat{v}^n(t)\varphi_0\|
\\
&\leq
6\vep_{n+1}
+
\|({v}^{n+2}(t)-1)
w_{n+2}
(\hat{v}^n(t)-\hat{v}^n(t_0))\varphi_0\|
\\
&\quad
+
\|({v}^{n+2}(t)-1)[w_{n+2},\hat{v}^n(t_0)\varphi_0]\|
+\|({v}^{n+2}(t)-1)\hat{v}^n(t_0)\varphi_0 w_{n+2}\|
\\
&\leq
6\vep_{n+1}
+2\vep_{n}
+3\vep_{n+1}
+
\vep_{n+2}
\\
&<12\vep_{n+1}.
\end{align*}
Hence we have
\[
\|(\hat{v}^{n+2}(t)
-\hat{v}^n(t))\cdot \varphi_0\|
<20\vep_n
\quad\mbox{if }
|t|\leq T_n.
\]
Thus $\hat{v}^{2n}(t)$
converges compact uniformly
in the strong* topology.
In the same way, so does $\hat{v}^{2n+1}(t)$. 
Put
$\bar{v}^0(t):=\lim_{n\rightarrow \infty}\hat{v}^{2n}(t)$
and
$\bar{v}^1(t):=\lim_{n\rightarrow \infty}\hat{v}^{2n+1}(t)$,
which are cocycles
of $\sigma_0\circ \alpha_t\circ\sigma_0^{-1}$
and
$\sigma_1\circ \beta_t\circ \sigma_1^{-1}$,
respectively.
By $(n.5)$, we have
\[
\Ad \bar{v}^0(t)\circ \sigma_0\circ \alpha_t\circ
\sigma_0^{-1}=
\Ad \bar{v}^1(t)\circ \sigma_1\circ \beta_t\circ \sigma_1^{-1}
\quad\mbox{for all }
t\in\R.
\]
Therefore, $\alpha$ and $\beta$ are strongly cocycle conjugate.
\end{proof}

\begin{rem}
\label{rem:duality-cocycle}
Let $\cP$ be a von Neumann algebra
and $\al$ a flow on $\cP$.
Suppose that
the covariant system
$\{L^\infty(\R),\Ad\la\}$ is embedded into
$\{Z(\cP),\al\}$.
Then any $\al$-cocycle is a coboundary as checked below.

Thanks to \cite[Theorem 1]{Land},
it turns out that $\al$ is a dual flow.
Indeed,
${\bm e}_s\in L^\infty(\R)$
satisfies
$\Ad\la_t({\bm e}_s)=e^{-ist}{\bm e}_s$,
where ${\bm e}_s(x)=e^{isx}$.
Let $\pi\col L^\infty(\R)\ra Z(\cP)$
be an equivariant normal $*$-homomorphism.
Put $w(s):={\bm e}_s$.
Since $w(s)$ belongs to $Z(\cP)$,
we obtain
\[
\cP=\cP^\al\vee\{w(\R)\}''
\cong
\cP^\al\rti_{\Ad w}\R
\cong
\cP^\al\oti L^\infty(\R).
\]
Then $\al$ is conjugate to $\id\oti\Ad\la$.
Thus $\al$ is stable
(see the proof of \cite[Theorem XII.1.11]{TaII}).
\end{rem}

Now we will prove the main theorem for general Rohlin flows.

\begin{thm}
\label{thm:class2}
Let $\alpha$ and $\beta$ be Rohlin flows
on a von Neumann algebra $\cM$.
Then
$\alpha$ and $\beta$ are strongly cocycle conjugate
if and only if
$\al_t\be_{-t}\in\oInt(\cM)$ for all $t\in\R$.
\end{thm}
\begin{proof}
We only prove the ``if'' part.
The assumption implies that
$\al=\be$ on $Z(\cM)$.
Let $(X,\mu)$ be a measure theoretic
spectrum of $Z(\cM)^\al=Z(\cM)^\be$.
Then we obtain the following disintegrations:
\[
\cM=\int_X^\oplus \cM_x\,d\mu(x),
\quad
\al_t^x=\int_X^\oplus\al_t^x\,d\mu(x),
\quad
\be_t^x=\int_X^\oplus\be_t^x\,d\mu(x).
\]
Note that $\al^x$ and $\be^x$ are
centrally ergodic flows on $\cM_x$
for almost every $x$.

\begin{clm}
For almost every $x\in X$,
$\al_t^x\be_{-t}^x\in\oInt(\cM_x)$
for all $t\in\R$.
\end{clm}
\begin{proof}[Proof of Claim 1.]
Employing Theorem \ref{thm:appdisint},
we deduce that
for each $t\in\R$,
$\al_t^x\be_{-t}^x\in\oInt(\cM_x)$
for almost every $x\in X$.
Thus by usual measure theoretic discussion,
it turns out that
for almost every $x\in X$,
$\al_t^x\be_{-t}^x\in\oInt(\cM_x)$
for all $t\in\Q$.
Since $\R\ni t\mapsto\al_t^x\be_{-t}^x\in\Aut(\cM_x)$
is continuous,
we see that
for almost every $x\in X$,
$\al_t^x\be_{-t}^x\in\oInt(\cM_x)$
for all $t\in\R$.
\end{proof}

\begin{clm}
For almost every $x\in X$,
$\al^x$ and $\be^x$ are Rohlin flows.
\end{clm}
\begin{proof}[Proof of Claim 2.]
Let $p\in\R$.
Employing Lemma \ref{lem:Rohlin-ptwise},
we take a central sequence $(v^\nu)_\nu$
in $\cM$ such that
$v^\nu\in\cM^{\rm U}$
and
$\al_t(v^\nu)-e^{ipt}v^\nu\to0$
compact uniformly
in the strong$*$ topology
as $\nu\to\infty$.
By Lemma \ref{lem:centralseq-disint},
a subsequence of
$(v_x^\nu)_\nu$ is central
for almost every $x\in X$.
Hence we may and do assume that
$(v_x^\nu)_\nu$ is central
for almost every $x\in X$.

Let $\{\vph^k\}_k$ be a norm dense sequence in $\cM_*$.
Then
\[
\|(\al_t(v^\nu)-e^{ipt}v^\nu)\vph^k\|
=
\int_X \|(\al_t^x(v_x^\nu)-e^{ipt}v_x^\nu)\vph_x^k\|\,d\mu(x),
\]
\[
\|\vph^k\cdot(\al_t(v^\nu)-e^{ipt}v^\nu)\|
=
\int_X \|\vph^k\cdot(\al_t^x(v_x^\nu)-e^{ipt}v_x^\nu)\|\,d\mu(x),
\]
As the discussion in the proof of Lemma \ref{lem:centralseq-disint},
we may and do assume that for each $t\in\R$,
$\al_t^x(v_x^\nu)-e^{ipt}v_x^\nu$ converges to $0$
in the strong$*$ topology
as $\nu\to\infty$ for almost every $x\in X$.

Let $f(t)=e^{-ipt}1_{[0,1]}(t)$.
Then for all $\nu\in\N$ and $x\in X$,
we have
\[
\|(\al_f(v_x^\nu)-v_x^\nu)\vph_x^k\|
\leq
\int_0^1\|(\al_t^x(v_x^\nu)-e^{ipt}v_x^\nu)\vph_x^k\|\,dt.
\]
Hence
\begin{align*}
\int_X
\|(\al_f(v_x^\nu)-v_x^\nu)\vph_x^k\|
\,d\mu(x)
&\leq
\int_0^1dt
\int_Xd\mu(x)\,
\|(\al_t^x(v_x^\nu)-e^{ipt}v_x^\nu)\vph_x^k\|
\\
&=
\int_0^1
\|(\al_t(v^\nu)-e^{ipt}v^\nu)\vph^k\|
\,dt,
\end{align*}
which converges to 0 as $\nu\to\infty$.
Similarly we have
$\int_X
\|\vph_x^k\cdot(\al_f(v_x^\nu)-v_x^\nu)\|
\to0$.
Again by taking a subsequence if necessary,
we may and do assume that
$(\al_f(v_x^\nu)-v_x^\nu)_\nu$ is a trivial sequence
for almost every $x\in X$.

By Lemma \ref{lem:smoothing},
$(v_x^\nu)_\nu$ is a Rohlin unitary for $p\in\R$
for almost every $x\in X$.
\end{proof}

Combining the above claims,
Lemma \ref{lem:class} and Theorem \ref{thm:cocdisint},
we see that
$\al$ and $\be$ are strongly cocycle conjugate.
\end{proof}

Thanks to \cite[Theorem 1 (i)]{KawST}
and Theorem \ref{thm:genKST},
we know $\oInt(\cM)=\ker(\mo)$
when $\cM$ is injective.
Hence we obtain the following corollary.

\begin{cor}
\label{cor:class-inj}
Let $\cM$ be an injective von Neumann algebra
and $\al,\be$ Rohlin flows on $\cM$.
Then $\al$ and $\be$ are strongly cocycle conjugate
if and only if
$\mo(\al_t)=\mo(\be_t)$ for all $t\in\R$.
\end{cor}

\begin{cor}
\label{cor:II1III1}
If $\cM$ is an injective factor of type II$_1$ or III$_1$,
then any Rohlin flow
is cocycle conjugate to $\id_{\cM}\oti \al^0$,
where $\al^0$ is a (unique) Rohlin flow on the injective factor
of type II$_1$.
\end{cor}
\begin{proof}
Since $\Aut(\cM)=\oInt(\cM)$,
any Rohlin flows are cocycle conjugate.
\end{proof}

\begin{cor}
\label{cor:Rohlin-II1}
Let $\al$ be a flow on
the type II$_1$ injective factor $\cM$.
Then the following statements are equivalent:
\begin{enumerate}
\item
$\al$ has the Rohlin property;

\item
$\al$ is invariantly approximately inner
and $\Ga(\al)=\R$.
\end{enumerate}
In this case,
$\al$ is stably self-dual,
that is,
$\hal\sim \al\oti\id_{B(\ell^2)}$.
\end{cor}
\begin{proof}
(1)$\Rightarrow$(2).
The $\al$ is cocycle conjugate
to a product type action as we will see
in Example \ref{exam:ITP} in the next section.
Thus $\al$ is invariantly approximately inner
by Lemma \ref{lem:invappinn-stab}.

(2)$\Rightarrow$(1).
Suppose that the condition (2) holds.
Then the dual flow $\hal$ has the Rohlin property
by Theorem \ref{thm:dual}.
Since $\hal$ preserves the trace
on the type II$_\infty$ factor
$\cN:=\cM\rti_\al\R$,
$\mo(\hal)$ is trivial.
Hence $\hal$
is cocycle conjugate to 
$\sigma\oti\id$ on $\cM\oti B(\ell^2)$,
where $\sigma$ is the a product type action
given in Example \ref{exam:ITP}.
Thus
the bidual flow
$\widehat{\hal}$
is conjugate
to $\hat{\sigma}\oti\id$.

By Takesaki duality,
$\widehat{\hal}
\sim
\al\oti\id$
on $\cM\oti B(\ell^2)$.
Hence $\al\oti\id\sim\hat{\sigma}\oti\id$
on
$\cM\oti B(\ell^2)$.
Since $\sigma$ is invariantly approximately inner, 
$\hat{\sigma}$ has the  Rohlin property by Theorem \ref{thm:dual}.
Thus $\al$ has the Rohlin property
by Lemma \ref{lem:tensorBH}. 
Since both $\alpha\otimes \id$ and $\hal$ are trace preserving 
Rohlin flows, they are cocycle conjugate.
\end{proof}

Since the tensor product flow of an arbitrary flow
and a Rohlin flow has the Rohlin property,
we see the following result.

\begin{cor}
Let $\al$ and $\be$ be flows on an injective factor $\cM$
such that $\mo(\al_t)=\mo(\be_t)$ for all $t\in\R$.
Let $\si$ be a Rohlin flow on the injective factor $\cN$
of type II$_1$.
Then $\al\oti\si$ is strongly cocycle conjugate to $\be\oti\si$.
\end{cor}

\section{Applications}
\label{sec:appl}

In this section,
we discuss several applications of
our classification result.
Among them,
we present a new proof
of Kawahigashi's work on
flows on the injective type II$_1$ factor
\cite{Kw-cent,Kw-irrat,Kw-cartan}
and also that
of
classification of injective type III
factors.

\subsection{Classification of invariantly approximately inner flows}
\label{subsect:Cartan}

Let $\al,\be$ be flows on an injective factor $\cM$.
Let $\cN_1:=\cM\rti_\al\R$
and $\cN_2:=\cM\rti_\be\R$.
We denote
by $\{Z(\tcN_1),\th^1,\mo(\hal)\}$
and
$\{Z(\tcN_2),\th^2,\mo(\hbe)\}$
the triples
of the flow spaces, the flows of weights
and the Connes-Takesaki modules
for $\hal$ and $\hbe$,
respectively.

\begin{thm}
\label{thm:class-invapprox}
Let $\al,\be$ be flows
on an injective factor $\cM$.
Suppose that they are invariantly approximately inner.
Then the following statements hold:
\begin{enumerate}
\item
Two flows $\al$ and $\be$ are stably conjugate
if and only if
the types (I, II and III)
of $\cM\rti_\al\R$ and $\cM\rti_\be\R$ are same,
and
the $\R^2$-actions
$\th_s^1\circ\mo(\hal_t)$
and
$\th_s^2\circ\mo(\hbe_t)$
on the corresponding flow spaces
are conjugate;

\item
In (1),
if one of the following conditions holds,
then $\al$ and $\be$ are cocycle conjugate:
\begin{itemize}
\item
$\cM$ is infinite;

\item
$\cM\rti_\al\R$ and $\cM\rti_\be\R$ are not
of type I.
\end{itemize}
\end{enumerate}
\end{thm}
\begin{proof}
(1).
The ``only if'' part is clear.
We will show the ``if'' part.
Recall that the dual flows
$\hal$ and $\hbe$ have the Rohlin property
by Theorem \ref{thm:dual}.
Since $\cM$ is a factor,
the dual flows are centrally ergodic.
In particular, $\cN_1$ and $\cN_2$
are von Neumann algebras of type I, II$_1$, II$_\infty$
or III.
To show the stable conjugacy,
we may and do assume that $\cN_1$ and $\cN_2$
are properly infinite by considering $\al\oti\id_{B(\ell^2)}$
and $\be\oti\id_{B(\ell^2)}$.
Then their core von Neumann algebras are isomorphic.
Let $\Th\col Z(\widetilde{\cN_1})\ra Z(\widetilde{\cN_2})$
be an isomorphism which conjugates
$\th_s^1\circ\mo(\hal_t)$
and
$\th_s^2\circ\mo(\hbe_t)$.
Then by Lemma \ref{lem:P1P2}, which will be proved later,
there exist an isomorphism $\rho\col\cN_1\ra \cN_2$
and a self-adjoint operator $h$ affiliated with $Z(\cN_2)$
such that
$\Th=\th_h^2\circ\mo(\rho)$.
(See (\ref{eq:thhx}) for the definition of $\theta_h$.)

We set $\ga_t:=\rho\circ\hal_t\circ\rho^{-1}$.
Then
\begin{align*}
\mo(\ga_t)
&=\mo(\rho)\circ\mo(\hal_t)\circ\mo(\rho)^{-1}
\\
&=
\th_{-h}^2\circ\Th\circ
\mo(\hal_t)\circ\Th^{-1}\circ\th_h^2
\\
&=
\th_{-h}^2\circ
\Th\circ\mo(\hal_t)\circ\Th^{-1}
\circ\th_h^2
\\
&=
\th_{-h}^2\circ
\mo(\hbe_t)
\circ\th_h^2
\\
&=
\mo(\hbe_t).
\end{align*}

By Corollary \ref{cor:class-inj},
it turns out that $\ga\sim\hbe$.
Hence $\hal\sim\hbe$,
and $\al\oti\id_{B(\ell^2)}
\sim\be\oti\id_{B(\ell^2)}$
by Takesaki duality.

(2).
If $\cM$ is infinite,
then $\al\sim\be$ as usual.
Hence we suppose that $\cM$ is of type II$_1$.
Then $\cN_1$ and $\cN_2$ are of type II by our assumption.
Since $\alpha$ and $\beta$ are stably conjugate,
there exist an
$\alpha\otimes \id_{B(\ell^2)}$-cocycle $w$
and $\theta\in\Aut(\cM\otimes B(\ell^2))$ such that
\[
\Ad w_t \circ(\alpha_t\otimes \id_{B(\ell^2)})
=\theta\circ (\beta_t\otimes \id_{B(\ell^2)})\circ \theta^{-1}.
\]

It turns out from (1) that
$\alpha$ and $\alpha\otimes\id_{\cM}$ are also stably conjugate.
Thus
there exist an $\alpha\otimes \id_{B(\ell^2)}$-cocycle $v$
and an isomorphism
$\theta_\alpha\col \cM\otimes \cM\otimes B(\ell^2)
\longrightarrow  \cM\otimes B(\ell^2)$
such that
\[
\Ad v_t \circ (\alpha_t\otimes \id_{B(\ell^2)})
=\theta_\alpha\circ
(\alpha_t\otimes\id_{\cM}
\otimes \id_{B(\ell^2)})\circ \theta_\alpha^{-1}.
\]
\[
\Ad \theta^{-1}_\alpha(v_t^*)
\circ
(\alpha_t\otimes \id_\cM\otimes \id_{B(\ell^2)})
=\theta_\alpha^{-1}\circ
(\alpha_t\otimes \id_{B(\ell^2)})\circ \theta_\alpha.
\]

We set
\[
\theta_1
:=\theta_\alpha\circ
(\id_{\cM}\otimes\th^{-1} )\circ\theta_\alpha^{-1}\circ\theta,
\
v_t':=\th_1(\th^{-1}(v_t^*)),
\
w_t':=\th_1(\th^{-1}(w_t)).
\]
The $\th_1$ is an automorphism on $\cM\otimes B(\ell^2)$
which satisfies $\mo(\theta_1)=1$.
This is verified by computing its module
with respect to a tracial weight.
Then
\begin{align*}
\lefteqn{ \theta_1\circ (\beta_t\otimes\id_{B(\ell^2)})\circ \theta_1^{-1}}
\\
&=
\Ad w_t'\circ
\theta_\alpha\circ(\id_{\cM}\otimes\theta^{-1} )\circ\theta_\alpha^{-1}
\circ
(\alpha_t\otimes\id_{B(\ell^2)})\circ
\theta_\alpha\circ(\id_{\cM}\otimes\theta )
\circ\theta_\alpha^{-1}
\\ 
&=
\Ad (w_t'v_t')
\circ
\theta_\alpha\circ(\id_{\cM}\otimes\theta^{-1} )
\circ
(\alpha_t\otimes\id_{\cM}\otimes \id_{B(\ell^2)})\circ
(\id_{\cM}\otimes\theta )\circ\theta_\alpha^{-1}
\\ 
&=
\Ad (w_t'v_t')\circ
\theta_\alpha
\circ
(\alpha_t\otimes\id_{\cM}\otimes \id_{B(\ell^2)})\circ
\theta_\alpha^{-1}
\\ 
&=
\Ad(w_t'v_t'v_t)
\circ
(\alpha_t\otimes \id_{B(\ell^2)}).
\end{align*}
Hence we may and do assume that $w$ and $\th$ satisfy
\[
\Ad w_t \circ (\alpha_t\otimes \id_{B(\ell^2)})
=\theta\circ
(\beta_t\otimes \id_{B(\ell^2)})\circ \theta^{-1},
\quad
\mo(\th)=1.
\]
By the latter condition,
we can take a unitary $u\in\cM\oti B(\ell^2)$
and $\si\in\Aut(\cM)$
such that $\Ad u\circ\theta=\sigma\otimes \id_{B(\ell^2)}$.
By setting $u_t:=uw_t(\al_t\otimes \id_{B(\ell^2)})(u^*)$, we have 
\[
 \Ad u_t\circ (\alpha_t\otimes \id_{B(\ell^2)})
=
\left(\sigma\circ \beta_t\circ
 \sigma^{-1}\right)\otimes \id_{B(\ell^2)}.
\]
From the above relation, $u_t\in \cM\oti\C$ is clear,
and we obtain
\[
\Ad u_t\circ \alpha_t
=\sigma\circ \beta_t\circ\sigma^{-1}.
\]
Therefore $\alpha$ and $\beta$ are cocycle conjugate.
\end{proof}

\begin{lem}
\label{lem:Shapiro}
Let $\th$ be an automorphism on a von Neumann algebra
$\cN$
and
$\al$ an automorphism on the crossed product
$\cN\rti_\th\Z$.
Suppose that $\th$ is ergodic
and faithful on $Z(\cN)$,
and $\al=\id$ on $\cN$.
Then
there exists a sequence of unitaries
$(v_n)_n$ in $Z(\cN)$
such that
$\al=\lim_{n\to\infty}\Ad v_n$
in $\Aut(\cM)$.
\end{lem}
\begin{proof}
Put $\cA:=Z(\cN)$
and $U:=\la^\th(1)$,
the implementing unitary of
$\cM:=\cN\rti_\th\Z$.
Since $Z(\cM)^{\hat{\th}}=\cA^\th=\C$,
$\hat{\th}$ is a centrally ergodic action
of $\T$.
By \cite[Corollary VI.1.3]{NT},
we have $\cA'\cap\cM=\cN$.
In particular, $Z(\cM)=\cA^\th=\C$.

Then $c(m):=\al(U^m)U^{*m}$ belongs to $\cN'\cap\cM=\cA$,
which is a $\th$-cocycle.
We will show that $c$ is approximated by
a coboundary.

Let $\vph\in \cN_*$ be a faithful state
and $\hvph$ the dual state on $\cM$.
Note that $\hvph\circ\al=\hvph$
since $\al$ fixes $\cN$.
Let $n\in\N$ and $\vep_n:=1/2n^2(2n+1)$.
Take $\de_n>0$ so that
if $x\in \cA_1$ satisfies
$|x|_\vph<\de_n$,
then $|\th^k(x)|_\vph<\vep_n$ with $|k|\leq n$.
Next take $N_n\in\N$ such that
$12/N_n<\de_n$.

By \cite[Lemma 10]{KawST},
there exists a partition of unity
$\{f\}\cup\{e_i\}_{i=0}^{N_n}$ in $\cA$
such that
$\th(e_i)=e_{i+1}$,
$i=0,\dots,N_n-1$,
$|f|_\vph<1/N_n$, $|e_0|_\vph<1/N_n$
and $|e_{N_n}|_\vph<2/N_n$.
Using the following inequalities:
\[
\th(e_{N_n})\leq e_0+f\leq\th(e_{N_n})+\th(f)+f,
\quad
\th(f)=f+e_0-\th(e_{N_n}),
\]
we get
\[
|\th(e_{N_n})|_\vph
<2/N_n,
\quad
|\th(f)|_\vph
<4/N_n
\quad
|\th^{-1}(e_0)|_\vph<3/N_n,
\quad
|\th^{-1}(f)|_\vph<6/N_n.
\]

Set
$v_n:=f+\sum_{i=0}^{N_n}c(i)e_i\in\cA$.
Then $v_n$ is a unitary,
and
\begin{align*}
c(1)\th(v_n)-v_n
&=
c(1)\th(f)
+
\sum_{i=0}^{N_n}
c(1)\th(c(i))\th(e_i)-v_n
\\
&=
c(1)\th(f)
+
\sum_{i=0}^{N_n}
c(1+i)\th(e_i)-v_n
\\
&=
c(1)\th(f)
+
\sum_{i=1}^{N_n}
c(i)e_i
+
c(1+{N_n})\th(e_{N_n})
-v_n
\\
&=
c(1)\th(f)-f
-
c(0)e_0
+
c(1+N_n)\th(e_{N_n}).
\end{align*}
Hence we have
\begin{align*}
|c(1)\th(v_n)-v_n|_\vph
&\leq
|\th(f)|_\vph+|f|_\vph+|e_0|_\vph+|\th(e_{N_n})|_\vph
\\
&<
4/N_n+1/N_n+1/N_n+2/N_n
\\
&=8/N_n<\de_n,
\end{align*}
and by $c(-1)=\th^{-1}(c(1)^*)$,
\begin{align*}
|c(-1)\th^{-1}(v_n)-v_n|_\vph
&=
|c(1)\th(v_n)-v_n|_{\th(\vph)}
\\
&\leq
|\th(f)|_{\th(\vph)}+|f|_{\th(\vph)}
+|e_0|_{\th(\vph)}+|\th(e_{N_n})|_{\th(\vph)}
\\
&=
|f|_\vph+|\th^{-1}(f)|_\vph
+
|\th^{-1}(e_0)|_\vph+|e_{N_n}|_\vph
\\
&<
1/N_n+6/N_n+3/N_n+2/N_n
\\
&=
12/N_n<\de_n.
\end{align*}

Therefore, if $|k|\leq n$,
then
\begin{equation}
\label{eq:thkcv1}
|\th^k(c(1)\th(v_n)-v_n)|_\vph
<
\vep_n,
\end{equation}
\begin{equation}
\label{eq:thkcv2}
|\th^{k}(c(-1)\th^{-1}(v_n)-v_n)|_\vph
<
\vep_n.
\end{equation}

We will prove the following inequality by induction:
\begin{equation}
\label{eq:ckthk}
|c(k)\th^k(v_n)-v_n|_\vph
<
|k|\vep_n
\quad
\mbox{if }|k|\leq n.
\end{equation}

We first consider when $k>0$.
Suppose that we have proved the inequality above
for $k-1$.
Using the cocycle identity
$c(k)=c(k-1)\th^{k-1}(c(1))$,
we obtain the following:
\begin{align*}
|c(k)\th^k(v_n)-v_n|_\vph
&=
|c(k-1)\th^{k-1}(c(1)\th(v_n))-v_n|_\vph
\\
&\leq
|\th^{k-1}(c(1)\th(v_n)-v_n)|_\vph
+
|c(k-1)\th^{k-1}(v_n)-v_n|_\vph
\\
&<
\vep_n
+
|c(k-1)\th^{k-1}(v_n)-v_n|_\vph
\quad
\mbox{by }
(\ref{eq:thkcv1})
\\
&<
k\vep_n.
\end{align*}

Next we consider when $k<0$.
Suppose that we have proved the inequality above
for $k+1$.
Using the cocycle identity
$c(k)=c(k+1)\th^{k+1}(c(-1))$,
we obtain the following:
\begin{align*}
|c(k)\th^k(v_n)-v_n|_\vph
&=
|c(k+1)\th^{k+1}(c(-1)\th^{-1}(v_n))-v_n|_\vph
\\
&\leq
|\th^{k+1}(c(-1)\th^{-1}(v_n)-v_n)|_\vph
+
|c(k+1)\th^{k+1}(v_n)-v_n|_\vph
\\
&<
\vep_n
+
|c(k+1)\th^{k+1}(v_n)-v_n|_\vph
\quad
\mbox{by }
(\ref{eq:thkcv2})
\\
&<
-k\vep_n.
\end{align*}
Hence (\ref{eq:ckthk}) holds.

We will prove that
$\al=\lim_n\Ad v_n$
in $\Aut(\cM)$.
Since $\hvph\circ\al=\hvph$ and $v_n\in \cM_{\hvph}$,
it suffices to show that
$\al(x)=\lim_n\Ad v_n(x)$
in the strong topology
for $x=\sum_{|k|\leq \ell}x_k U^k$ with $x_k\in \cN$.
When $n\geq\ell$,
we obtain
\begin{align*}
\|\al(x)-v_n xv_n^*\|_\hvph^2
&=
\Big{\|}
\sum_{|k|\leq \ell}x_k(c(k)-v_n\th^k(v_n^*))U^k
\Big{\|}_\hvph^2
\\
&<
\sum_{|k|\leq \ell}
\|x_k (c(k)-v_n\th^k(v_n^*))\|_\vph^2
\\
&\leq
\sum_{|k|\leq \ell}
\|x\|^2\cdot
2|c(k)-v_n\th^k(v_n^*)|_\vph
\\
&<
\|x\|^2\cdot 2(2\ell+1)\ell\vep_n
\quad
\mbox{by }
(\ref{eq:ckthk})
\\
&\leq
\|x\|^2/n.
\end{align*}
Thus $\al(x)=\lim_n v_nxv_n^*$,
and we are done.
\end{proof}

\begin{lem}
\label{lem:fixCartan}
Let $\al$ be a flow on an injective factor $\cM$.
Suppose that $\al$ fixes a Cartan subalgebra $\cA$
of $\cM$.
Then $\al$ is invariantly approximately inner.
\end{lem}
\begin{proof}
By \cite[Theorem 10]{CFW},
we may and do assume
that there exists a $\Z$-action
$\th$ on $\cA$ such that
$\cM=\cA\rti_\th\Z$.
The factoriality of $\cM$ implies
the ergodicity and the faithfulness
of $\th$.
Thus by the previous lemma,
$\al$ is invariantly approximately inner.
\end{proof}

In particular,
Kawahigashi's example \cite[Theorem 1.4]{Kw-Cotriv}
is invariantly approximately inner.

If a flow $\al$ fixes a Cartan subalgebra
of an injective type II$_1$ factor
and $\Ga(\al)=\R$,
then $\al$ has the Rohlin property
by Corollary \ref{cor:Rohlin-II1}
and the previous lemma.
This means the uniqueness of $\al$.
As a result,
we have proved the following main result
of \cite{Kw-cartan}.
See Example \ref{exam:ITP} for a product type
flow with the Rohlin property.

\begin{thm}[Kawahigashi]
\label{thm:Kaw-Cartan}
Let $\al$ be a flow on the injective
type II$_1$ factor $\cM$.
If $\al$ fixes a Cartan subalgebra of $\cM$
and $\Ga(\al)=\R$,
then $\al$ is cocycle conjugate
to a product type flow,
and absorbs any product type actions.
Thus such an action $\al$ is unique
up to cocycle conjugacy.
\end{thm}

In \cite[Proposition 6.5, Theorem 6.6]{AY},
Aoi and Yamanouchi have generalized
the above Kawahigashi's result
to the case of
an action of a locally compact group
on injective factors by groupoid method.
We will prove their result for flows
making use of a classification of Rohlin flows.

\begin{cor}[Aoi-Yamanouchi]
Let $\al,\be$ be flows on an injective injective factor $\cM$.
Suppose that they fix a Cartan subalgebra
of $\cM$,
and both $\cM\rti_\al\R$ and $\cM\rti_\be\R$ are not
of type I.
Then $\al$ and $\be$ are cocycle conjugate
if and only if
the $\R^2$-actions
$\th_s^1\circ\mo(\hal_t)$
and
$\th_s^2\circ\mo(\hbe_t)$
on the corresponding flow spaces are conjugate.
\end{cor}
\begin{proof}
By Lemma \ref{lem:fixCartan},
$\al$ and $\be$ are invariantly approximately inner.
When $\cN_1$ is of type III,
so is $\cN_2$
since their flow spaces are isomorphic.
Hence when $\cN_1$ is of type II,
so is $\cN_2$.
By Theorem \ref{thm:class-invapprox},
$\al\sim\be$.
\end{proof}

\begin{rem}
\label{rem:typeI-Gamma}
If $\cN:=\cM\rti_\al\R$ is of type I,
then $\Ga(\al)=\{0\}$.
This fact is verified as follows.
Take a type I subfactor $\cP$ in $\cN$
such that $\cP'\cap \cN=Z(\cN)$.
Hence we have the tensor product
decomposition $\cN=\cP\oti Z(\cN)$.
Let $\ga_t$ be the restriction of $\hal_t$
on $Z(\cN)$.
Then $\ga$ is an ergodic flow.
Since $(\id\oti\ga_{-t})\circ\hal_t=\id$
on $Z(\cN)$,
$(\id\oti\ga_{-t})\circ\hal_t$ is inner.
Take a measurable unitary map
$\R\ni t\mapsto U_t\in\cN$
such that $\hal_t=\Ad U_t\circ(\id\oti\ga_t)$.
Then $c(t,s):=(\id\oti\ga_t)(U_s^*)U_t^*U_{t+s}$
is a 2-cocycle of $\ga$
which belongs to $\C\oti Z(\cN)$.
Thanks to \cite[Proposition A.2]{CT},
$c$ is a coboundary.
Hence we may and do assume that
$U$ is an $(\id\oti\ga)$-cocycle.
Then
$\cP\oti (Z(\cN)\rti_\ga\R)\cong\cN\rti_\hal\R
\cong \cM\oti B(L^2(\R))$
that is a factor.
In particular,
$\ga$ is faithful,
that is,
$\{0\}=\ker\ga=\Ga(\al)$.
\end{rem}

By the remark above,
we have interest in a classification
of flows with trivial Connes spectrum.

\begin{lem}
\label{lem:Connes-spectrum0}
Let $\cM$ be a factor and $\al$ a flow on $\cM$.
If $\Ga(\al)=\{0\}$,
then $\al$ is invariantly approximately inner.
\end{lem}
\begin{proof}
When $\al$ is inner,
$\al$ is implemented by one-parameter unitary
group as usual.
Thus the statement is trivial.

Suppose that $\al$ is not inner.
Let $\cN:=\cM\oti B(\ell^2)$
and $\bar{\al}:=\al\oti\id_{B(\ell^2)}$.
By the same discussion as \cite[\S 5.3]{Co-une},
there exist a perturbation $\si$ of $\bar{\al}$
and $\mu>0$
such that
$0\in \Sp(\si)$ is isolated
and
$\cN=\cN^\si\rti_\ga\Z$ whose
implementing unitary $\la^\ga(1)$
satisfies
$\Sp_\si(\la^\ga(1))\subs[\mu,\infty)$.

\begin{clam}
The $Z(\cN^\si)$ is non-atomic,
and $\ga$ is a faithful ergodic action on $Z(\cN^\si)$.
\end{clam}
\begin{proof}[Proof of Claim]
The factoriality of $\cN$ implies
the central ergodicity of $\ga$.
Note that $Z(\cN^\si)\neq\C$.
If so,
then
$\Ga(\si)=\Sp(\si)\supset \Sp_\si(\la^\ga(1))\neq\{0\}$.
This is a contradiction.
Assume that $Z(\cN^\si)$ were atomic.
Take a minimal projection $e\in Z(\cN^\si)$.
By \cite[Lemme 2.3.3]{Co-une},
$\{0\}=\Ga(\si)=\Ga(\si^e)$.
However, $(\cN_e)^{\si^e}=(\cN^\si)_e$ is a factor,
and $\Ga(\si^e)=\Sp(\si^e)$.
Hence $\si^e$ is trivial.
Since $\cN$ is a factor,
$\si$ must be inner by \cite[Lemme 1.5.2]{Co-une},
and this is a contradiction.
Therefore,
$Z(\cN^\si)$ is non-atomic.

Suppose that $n>0$ is the period of
$\ga$ on $Z(\cN^\si)$.
The ergodicity of $\ga$ implies
that $Z(\cN^\si)\cong\ell^\infty(\Z/n\Z)$,
which is atomic.
Hence $\ga$ must be faithful on $\cN^\si$.
\end{proof}

By Lemma \ref{lem:Shapiro},
we can deduce the invariantly
approximate innerness of $\si$
and that of $\al$.
\end{proof}

\begin{rem}
We can also prove the previous lemma in the following way.
The condition $\Ga(\al)=\{0\}$
means that $\hal$ is faithful on $Z(\cM\rti_\al\R)$.
If we use Theorem \ref{thm:propergRoh},
then it turns out that
$\hal$ has the Rohlin property.
Thus $\al$ is invariantly approximately inner
by Theorem \ref{thm:dual}.
\end{rem}

By Theorem \ref{thm:class-invapprox}
and the previous lemma,
we obtain the following result due to
Kawahigashi for type II$_1$ case
\cite[Theorem 1.4]{Kw-cent}
and
Hui for type III case
\cite[Theorem 1.3]{Hui}.

\begin{cor}
\label{cor:Kaw-spectrum0}
Let $\al,\be$ be flows on the injective factor $\cM$.
Suppose that $\Ga(\al)=\{0\}=\Ga(\be)$.
Then the following statements hold:
\begin{enumerate}
\item
The $\al$ is stably conjugate to $\be$
if and only if
the types (I, II and III)
of $\cM\rti_\al\R$ and $\cM\rti_\be\R$
are same,
and
the $\R^2$-actions
$\th_s^1\circ\mo(\hal_t)$
and
$\th_s^2\circ\mo(\hbe_t)$
are conjugate as before;

\item
In (1),
if one of the following conditions hold,
then $\al$ and $\be$ are cocycle conjugate:
\begin{itemize}
\item
$\cM$ is infinite;

\item
$\cM\rti_\al\R$ and $\cM\rti_\be\R$ are not of type I.
\end{itemize}
\end{enumerate}
\end{cor}

\begin{exam}[Kawahigashi]
The difference between the cocycle conjugacy and the stable conjugacy
occurs only when $\cM$ is of type II$_1$.
We let $\al^{(k)}$ be the flow on
the injective factor of type II$_1$
defined by
\[
\al_t^{(k)}
=\bigotimes_{n=1}^\infty
\Ad
\begin{pmatrix}
1&0\\
0&e^{2\pi i\cdot 3^{n+k}t}
\end{pmatrix}
.
\]
They are mutually stably conjugate,
but not cocycle conjugate.
See \cite[Theorem 2.9]{Kw-Cotriv}.
In particular,
their crossed products are of type I.
\end{exam}

\subsection{Examples of Rohlin flows on the injective factor of type II$_1$}
We recall the notion of minimality
(cf. \cite{PW}).

\begin{defn}
Let $\al$ be an action of a locally compact
group $G$ on a factor $\cM$.
We will say that $\al$ on $\cM$
is \emph{minimal}
if $\al$ is faithful and $(\cM^\al)'\cap\cM=\C$.
\end{defn}

\begin{thm}\label{thm:almost}
Let $\alpha$ be an almost periodic and minimal
flow on the injective type II$_1$ factor
$\cM$.
Then $\alpha$ has the Rohlin property. 
\end{thm}
\begin{proof}
Set $\cM(p):=\{x\in \cM\mid
\alpha_t(x)=e^{ipt}x\}$, $p\in \R$, and 
$H:=\mathrm{Sp}_d(\alpha)=\{p\in \mathbb{R}\mid \cM(p)\ne \{0\}\}$.
As shown in \cite[Lemma 2.4, Proposition 7.3]{Thom-prime},
the eigenspace $\cM(p)$ contains a unitary,
and
$H$ is a dense
countable subgroup of $\mathbb{R}$.
Take a unitary $v(p)\in \cM(p)$,
and set
$\gamma_p:=\Ad v(p)|_{\cM^\alpha}$,
$c(p,q):=v(p)v(q)v(p+q)^*$.
Then $(\ga,c)$ is a free cocycle action
on $\cM^\al$,
where we regard $H$ as a discrete group.
Since
$(\cM^{\alpha})'\cap \cM=\mathbb{C}$,
$\cM^\alpha$ is an injective subfactor of type II$_1$.
By the 2-cohomology vanishing
theorem \cite[Theorem 7.6]{Ocn-act},
we may assume
$v(p)$ is a unitary representation of $H$, and
$\gamma$ is a free action of $H$ on $\cM^\alpha$.
By \cite[Lemma 9.2]{Ocn-act},
there exists
$\{u(p)_n\}_{n=1}^\infty \subset U(\cM^\alpha)$
for each $p\in H$
such that 
for all $p,q\in H$,
\[
\gamma_p
=\lim_{n\rightarrow \infty}\Ad u(p)_n,
\quad
\lim_{n\rightarrow \infty}
\left\|\gamma_p(u(q)_n)-u(q)_n\right\|_2
=0.
\]
In fact, we may assume
$\lim_{n\rightarrow \infty}
\|u(p)_nu(q)_n-u(p+q)_n\|_2=0$ holds,
but this is unnecessary in this proof.

Fix $p\in H$, and set $w_n:=u(p)_n^*v(p)$.
Then trivially
$\alpha_t(w_n)=e^{ipt}w_n$,
and $(w_n)_n$ is a central sequence
as verified below.
For $x\in \cM^\alpha$ and $q\in H$,
we have
\begin{align*}
\|w_nxv(q)-xv(q)w_n\|_2
&=
\|u(p)_n^*v(p)xv(q)-xv(q)u(p)_n^*v(p)\|_2
\\
&=
\|u(p)_n^*\gamma_p(x)v(p+q)-x\gamma_q(u(p)_n^*)v(p+q)\|_2
\\
&=
\|\gamma_p(x)-u(p)_nx\gamma_q(u(p)_n^*)\|_2
\\
&\leq
\|\gamma_p(x)-u(p)_nxu(p)_n^*\|_2
+
\left\|x\left(\gamma_q(u(p)_n^*)-u(p)_n^*\right)\right\|_2,
\end{align*}
and the last terms converge to $0$
by the choice of $\{u(p)_n\}$.
Since the linear span of $\{\cM^\al v(q)\}_{q\in H}$
is a strongly dense
$*$-subalgebra of $\cM$, 
$(w_n)_n$ is central.
Thus
$\pi_\om((w_n)_n)\in \cM_{\om,\al}$
is a Rohlin unitary for $p\in H$.

We next show the existence of a Rohlin unitary
for an arbitrary $p\in \mathbb{R}$.
Take a strongly dense countable set
$\{x_j\}_{j=1}^\infty\subset \cM_1$.

For any $n>0$, take $q_n\in H$ such that
$|e^{ipt}-e^{iq_n t}|<1/n$, $|t|\leq n$. 
Let $w_n\in\cM$ be a unitary such that 
\[
\alpha_t(w_n)=e^{iq_n t}w_n,
\quad
\|[w_n,x_j]\|_2<1/n,
\quad 1\leq j\leq n.
\]
Then for $|t|\leq n$,
\[
\|\alpha_t(w_n)-e^{ipt}w_n\|_2
<
1/n.
\]
Thus $\pi_\om((w_n)_n)$ is a Rohlin unitary for $p$.
\end{proof}

\begin{exam}
\label{exam:ITP}
\upshape 
Let $\mu,\nu\in \R\setminus\{0\}$
be such that $\mu/\nu\nin\Q$. 
Let $\si$ be the flow on the injective
type II$_1$ factor $\cM$
as defined below
\[
\sigma_t:=\bigotimes_{n=1}^\infty 
\Ad
\begin{pmatrix}
1 & 0 &0
\\
0&e^{i\mu t} &0 
\\
0&0 &e^{i\nu t}
\\
\end{pmatrix}.
\]
It is straightforward to check that
$\si$ satisfies the condition of
Theorem \ref{thm:Kaw-Cartan}
or
Theorem \ref{thm:almost}.
Indeed, $\Sp_d(\si)=\mu\Z+\nu\Z$.
Hence $\si$ has the Rohlin property.
We can also show this fact from
Corollary \ref{cor:faithful-product}
without using the classification
result of discrete amenable group actions
due to Ocneanu.

By uniqueness of
a Rohlin flow on the injective
type II$_1$ factor,
any Rohlin flow is cocycle conjugate to $\si$.
We will study product type flows as above
in more general setting in \S\ref{subsect:itp}.
\end{exam}

\begin{exam}\label{exam:irat}
\upshape
Let $\theta\in \mathbb{R}\setminus\mathbb{Q}$
and
$A_\theta$
be the irrational
rotation C$^*$-algebra generated by $u$ and $v$
satisfying $uv=e^{2\pi i\th}vu$.
For $\mu,\nu\in\mathbb{R}\setminus\{0\}$,
we introduce the ergodic flow $\alpha$
on
$A_\theta$
defined by $\alpha_t(u)=e^{i\mu t}u$
and
$\alpha_t(v)=e^{i\nu t}v$.
If
$\mu/\nu\not\in \mathbb{Q}$
and 
$\mu/\nu\not \in GL(2,\Q)\th$,
where
each $g\in GL(2,\Q)$ acts on $\R$
as the linear fractional transformation,
then $\alpha$ is a Rohlin flow
in the C$^*$-sense by \cite[Proposition 2.5]{Kishi-CMP}.

We lift $\alpha$ to the weak
closure of $A_\theta$ with respect to the unique trace.
Then we obtain a Rohlin flow
on the injective type II$_1$ factor.
Hence Theorem \ref{thm:class2}
implies that
the flow $\si$
introduced in Example \ref{exam:ITP}
and $\al$ are mutually cocycle conjugate.
This fact is originally proved by
Kawahigashi \cite[Theorem 16]{Kw-irrat}.

This example can be generalized
to a higher-dimensional noncommutative torus
\cite[Proposition 2.5]{Kishi-CMP}.
\end{exam}

\begin{rem}
Let $A$ be a C$^*$-algebra and $\pi\col A\ra B(H)$
be a representation.
A central sequence $(x^\nu)_\nu$
in $A$
needs not to be central in $\pi(A)''$ in the von Neumann algebra sense.
Indeed,
let $\vph\in A^*$ be a state and
$\pi_\vph\col A\ra B(H_\vph)$ be the GNS representation.
Suppose that the normal extension of $\vph$ on $\cM:=\pi_\vph(A)''$
is faithful.
Then $\{\vph \pi_\vph(a)\mid a\in A\}$ is a dense subspace of $\cM_*$.

Let $(x^\nu)_\nu$ be a central sequence in $A$.
Then
we have
\[
\|[\vph \pi_\vph(a),\pi_\vph(x^\nu)]\|_{\cM_*}
\leq
\|a\|\|[\vph,\pi_\vph(x^\nu)]\|_{\cM_*}
+
\|[a,x^\nu]\|.
\]
By the Kaplansky density theorem,
we have
$\|[\vph,\pi_\vph(x^\nu)]\|_{\cM_*}
=
\|[\vph,x^\nu]\|_{A^*}$.
Therefore,
$(\pi_\vph(x^\nu))_\nu$ is central in $\cM$
if and only if
$\lim_{\nu\to\infty}\|[\vph,x^\nu]\|_{A^*}=0$.
\end{rem}

\subsection{The classification of injective factors of type III}
\label{subsect:III-flow}

We recall the following fundamental result.
A sketch of a proof is given
in order that
we can understand the outline.

\begin{thm}[Connes, Haagerup, Kawahigashi-Sutherland-Takesaki]
\label{thm:modular-appinn}
Let $\cM$ be an injective von Neumann algebra
and $\vph$ a faithful normal state on $\cM$.
Then
$\si_t^\vph\in\oInt(\cM)$
for all $t\in\R$.
\end{thm}
\begin{proof}
By Theorem \ref{thm:appdisint} and the disintegration
of $\si^\vph$ as (\ref{eq:moddisint}),
we may and do assume that $\cM$ is a factor.
The semifinite case is trivial.

If $\cM$ is of type III$_\la$ with $0<\la<1$,
then there exists an automorphism
$\th$ on the injective type II$_1$ factor $\cN$
such that
$\cM=\cN\rti_\th\Z$ and $\ta\circ\th=\la\ta$.
Then by \cite[Lemma 5, Theorem 1.2.5]{Co-outer},
$\th$ has the Rohlin property
as a $\Z$-action,
and the dual action $\hat{\th}_t=\si_t^\vph$
is invariantly approximately inner
as a torus action.
This fact is proved in a similar way
to Theorem \ref{thm:dual}.

By \cite[Proposition 3.9]{Co-almost},
a modular automorphism
is approximately inner
for any factor of type III$_0$.
We will prove this fact
by using Rohlin flows
in Corollary \ref{cor:III0-modular}.

If $\cM$ is of type III$_1$,
then the asymptotic bicentralizer of
any faithful normal state
is trivial \cite[Corollary 2.4]{Ha-III1}.
As is proved in \cite[Theorem IV.1]{Co-III1},
the semidiscreteness
implies
the approximate innerness
of
a modular automorphism.
\end{proof}

Assuming this theorem,
we present the classification of injective
type III von Neumann algebras
from a viewpoint of Rohlin property.

\begin{thm}[Connes, Haagerup, Krieger]
\label{thm:inj-class}
Let $\cM_1$ and $\cM_2$
be injective von Neumann algebras of type III.
Then
they are isomorphic
if and only if
their flows of weights are isomorphic.
\end{thm}
\begin{proof}
The ``only if'' part is clear,
and we will show the ``if'' part.
Let $\cM_1$ and $\cM_2$ as above.
Let $\vph_1$ and $\vph_2$ be
faithful normal states
on $\cM_1$ and $\cM_2$,
respectively.
By our assumption
and the uniqueness of the injective type II$_\infty$ factor,
we may regard $\widetilde{\cM_1}=\widetilde{\cM_2}$,
and
the dual flows $\th^1:=\widehat{\si^{\vph_1}}$
and $\th^2:=\widehat{\si^{\vph_2}}$
are equal on $Z(\widetilde{\cM_1})$.
By standard prescription (see the proof of Lemma \ref{lem:P1P2}),
we may and do assume that
$\th^1$ and $\th^2$ are scaling
the same trace $\ta$ as
$\ta\circ\th_t^1=e^{-t}\ta=\ta\circ\th_t^2$.
Thus $\mo(\th_t^1)=\mo(\th_t^2)$.

By Proposition \ref{prop:ext-dual}
and the previous theorem,
it turns out that
$\th^1$ and $\th^2$
have the Rohlin property.
Thus by Corollary \ref{cor:class-inj},
they are cocycle conjugate.
In particular,
$\widetilde{\cM_1}\rti_{\th^1}\R
\cong\widetilde{\cM_2}\rti_{\th^2}\R$.
Hence $\cM_1\cong\cM_2$
by Takesaki duality.
\end{proof}

Let us denote by
$\cR_0$, $\cR_{0,1}$ and $\cR_\infty$
the injective factors of
type II$_1$, II$_\infty$ and III$_1$.
Let us focus on
the injective
type III$_1$ factor $\cR_\infty$.
The core of $\cR_\infty$ is isomorphic to $\cR_{0,1}$.
The dual flow $\th$ scales
the trace $\ta$
as $\ta\circ\th_t=e^{-t}\ta$.
As we have seen,
$\th$ has the Rohlin property.
We summarize this fact as follows.

\begin{thm}\label{thm:tracescaling}
A trace scaling flow on $\cR_{0,1}$
has the Rohlin property.
\end{thm}
\begin{proof}
Let $\th$ be a trace scaling flow on $\cR_{0,1}$
and $\ta$ a faithful normal tracial weight on $\cR_{0,1}$.
Then there exists a non-zero $p\in\R$ such that
$\ta\circ\th_t=e^{-pt}\ta$ for $t\in\R$.
Using $\th'_t:=\th_{t/p}$ if $p\neq1$,
we may and do assume $\ta\circ\th_t=e^{-t}\ta$.
We have known that $\th$ has the Rohlin property.
\end{proof}

Let us refine
Theorem \ref{thm:inj-class}
for later use (see Theorem \ref{thm:itp-class}).
Let $\cM$ be a von Neumann algebra.
For a self-adjoint operator $h$ affiliated with
$Z(\cM)$,
we define $\th_h\in\Aut(\tcM)$ by
\begin{equation}
\label{eq:thhx}
\th_h(x)=x,
\quad
\th_h(\la^\vph(t))=e^{-ith}\la^\vph(t)
\quad
\mbox{for }
x\in\cM,
\
\vph\in W(\cM),
\
t\in\R.
\end{equation}

\begin{lem}
\label{lem:P1P2}
Let $\cP_1$ and $\cP_2$ be injective von Neumann algebras
which are one of type I, II$_1$, II$_\infty$ or III.
Let $\{Z(\widetilde{\cP_1}),\th^1\}$
and $\{Z(\widetilde{\cP_2}),\th^2\}$
be the flow of weights of
$\cP_1$ and $\cP_2$,
respectively.
Suppose that
$\widetilde{\cP_1}\cong\widetilde{\cP_2}$
as von Neumann algebras
and
there exists an isomorphism
$\Th\col Z(\widetilde{\cP_1})\ra Z(\widetilde{\cP_2})$
with $\Th\circ\th_t^1=\th_t^2\circ\Th$
on $Z(\widetilde{\cP_1})$
for all $t\in\R$.
Then
there exist
an isomorphism
$\rho\col\cP_1\ra\cP_2$
and a self-adjoint operator $h$ that is affiliated with
$Z(\cP_2)$
such that
$\Th=\th_h^2\circ\mo(\rho)=\mo(\rho)\circ\th_{\rho^{-1}(h)}^1$.
\end{lem}
\begin{proof}
By the assumption of the injectivity,
there exist a semifinite injective factor $\cQ$,
an abelian von Neumann algebra $\cA$
and isomorphisms
$\pi_j\col \widetilde{\cP_j}\ra \cQ\oti\cA$
with $j=1,2$.
We put
$\ps_t^j:=\pi_j\circ\th_t^j\circ\pi_j^{-1}$.
By Proposition \ref{prop:ext-dual}
and Theorem \ref{thm:modular-appinn},
they are Rohlin flows.

Then
$\Th_0:=\pi_2\circ\Th\circ\pi_1^{-1}|_\cA$
is an automorphism on $\cA$
such that
$\Th_0\circ\ps_t^1=\ps_t^2\circ\Th_0$
on $\cA$.
By replacing with $\pi_1$ with
$(\id_\cQ\oti\Th_0)\circ\pi_1$,
we may and do assume that
$\ps_t^1=\ps_t^2$ on $\cA$
and
$\pi_2^{-1}\circ\pi_1|_{Z(\widetilde{\cP_1})}=\Th$.

\begin{clam}
There exists $\ga\in\Aut(\cQ\oti \cA)$
such that
$\ga|_\cA=\id_\cA$
and $\ga\circ\ps_t^1\circ\ga^{-1}=\ps_t^2$
modulo $\oInt(\cQ\oti\cA)$.
\end{clam}
\begin{proof}[Proof of Claim.]
When $\cP_1$ is of type I,
so is $\cQ$.
Since $\ps_t^1\ps_{-t}^2|_\cA=\id_\cA$
and $\Aut(\cQ)=\Int(\cQ)$,
we can deduce
$\ps_t^1\ps_{-t}^2\in \Int(\cQ\oti\cA)$
by Theorem \ref{thm:intdisint}.

When $\cP_1$ is of type II$_1$,
so is $\cQ$.
Since $\Aut(\cQ)=\oInt(\cQ)$,
$\ps_t^1\ps_{-t}^2\in \oInt(\cQ\oti\cA)$
by Theorem \ref{thm:appdisint}.

When $\cP_2$ is of type II$_\infty$
or III,
$\cQ$ is the injective factor of type II$_\infty$.
Let $\ta$ be a faithful normal semifinite tracial
weight on $\cQ$.
Take $\vph_j\in W(\cA)$
with $(\ta\oti\vph_j)\circ\ps_t^j=e^{-t}(\ta\oti\vph_j)$
for $j=1,2$.

Realize $\cA$ as $L^\infty(X,\mu_2)$
where $\vph_2(a)=\int_X a(x)\,d\mu_2(x)$.
Let $h(x)$ be a positive measurable function
such that
$\vph_1(a)=\int_X h(x)a(x)d\mu_2(x)$.
Take
$\ga\in\Aut(\cQ\oti \cA)$
of the form
$\ga=\int_X^\oplus \ga_x\,d\mu_2(x)$
such that
$\ta\circ\ga_x=h(x)\ta$.
Then
$(\ta\oti\vph_2)\circ\ga=\ta\oti\vph_1$,
and
\[
(\ta\oti\vph_2)\circ\ga\circ\ps_t^1\circ\ga^{-1}
=
e^{-t}(\ta\oti\vph_2).
\]
Since
$\ga\circ\ps_t^1\circ\ga^{-1}=\ps_t^2$
on $\cA$,
we obtain
$\mo(\ga\circ\ps_t^1\circ\ga^{-1})=\mo(\ps_t^2)$.
By Theorem \ref{thm:genKST},
$\ga\circ\ps_t^1\circ\ga^{-1}\circ\ps_{-t}^2\in\oInt(\cQ\oti\cA)$.
\end{proof}

Replacing $\pi_1$ with
$\ga\circ\pi_1$,
we may assume that
$\ps_t^1=\ps_t^2$ modulo
$\oInt(\cQ\oti\cA)$.
Then by Theorem \ref{cor:class-inj},
there exist
$\ga'\in\oInt(\cQ\oti\cA)$
and
a $\ps^2$-cocycle $w$
such that
$\ga'\circ\ps_t^1\circ\ga'^{-1}=\Ad w_t\circ\ps_t^2$.
Then
$\pi_2^{-1}\circ\ga'\circ\pi_1=\Th$
on $Z(\widetilde{\cP_1})$.
Thus
we can extend $\Th$
to the isomorphism
from $\widetilde{\cP_1}$ onto
$\widetilde{\cP_2}$
by putting
$\Th:=\pi_2^{-1}\circ\ga'\circ\pi_1$
on $\widetilde{\cP_1}$.

Put $v_t:=\pi_2^{-1}(w_t)$.
Then $v$ is a $\th^2$-cocycle
and
we obtain
$\Th\circ\th_t^1\circ\Th^{-1}
=\Ad v_t\circ\th_t^2$.
By stability of $\th^2$ \cite[Theorem III.5.1]{CT},
there exists a unitary
$w\in\widetilde{\cP_2}$
such that $v_t=w^*\th_t^2(w)$.
Set $\Th':=\Ad w\circ\Th$.
Then
$\Th'=\Th$ on $Z(\widetilde{\cP_1})$
and
$\Th'\circ\th_t^1\circ(\Th')^{-1}
=\th_t^2$ for $t\in\R$.
Thus we may and do assume
that $\Th\circ\th_t^1\circ\Th^{-1}=\th_t^2$.
This implies that $\Th(\cP_1)=\cP_2$.
Put $\rho:=\Th|_{\cP_1}$.

Let $\vph\in W(\cP_1)$.
Then
$u_t:=
\Th(\la^\vph(t))
\la^{\rho(\vph)}(t)^*$
belongs to $\widetilde{\cP_2}^{\th^2}=\cP_2$,
and
\begin{align*}
u_t\si_t^{\rho(\vph)}(u_s)
&=
\Th(\la^\vph(t))
\la^{\rho(\vph)}(t)^*
\cdot
\la^{\rho(\vph)}(t)
\Th(\la^\vph(s))
\la^{\rho(\vph)}(s)^*
\la^{\rho(\vph)}(t)^*
\\
&=
u_{t+s}.
\end{align*}
Hence there exists
$\ps\in W(\cP_2)$
such that $u_t=[D\ps:D\rho(\vph)]_t$.
For $x\in\cP_2$ and $t\in\R$,
we have
\begin{align*}
\si_t^\ps(x)
&=
u_t\si_t^{\rho(\vph)}(x)u_t^*
\\
&=
\Th(\la^\vph(t))x\Th(\la^\vph(t)^*)
=
\rho(\si_t^\vph(\rho^{-1}(x)))
\\
&=
\si_t^{\rho(\vph)}(x).
\end{align*}
Thus $u_t\in Z(\cP_2)$.
Let $h$ be a positive operator affiliated
with $Z(\cP_2)$
such that
$u_t=e^{ith}$.
The $u$ does not depend on $\vph$.
Indeed,
for another
$\chi\in \cP_1$,
we have
\begin{align*}
\Th(\la^\chi(t))
&=
\Th([D\chi:D\vph]_t\la^\vph(t))
\\
&=
[D\rho(\chi):D\rho(\vph)]_t\cdot
u_t
\la^{\rho(\vph)}(t)
\\
&=
u_t
\la^{\rho(\chi)}(t).
\end{align*}
This shows $\Th=\th_h^2\circ\trho$.
\end{proof}

\subsection{Non-fullness of type III$_0$ factors}

In \cite{Co-almost},
it is shown that
a modular automorphism group
on any type III$_0$ factor $\cM$
is approximately inner.
This implies the non-fullness
of an arbitrary type III$_0$ factors.
If we apply Lemma \ref{lem:Connes-spectrum0}
to the discrete decomposition of $\cM$,
then the invariant approximate innerness
of $\si^\vph$
is immediately obtained.
We will present another approach to that
by showing that
any non-periodic ergodic flow
on a commutative von Neumann algebra
has the Rohlin property.

Let $(X,\mu, \mathcal{F}_t)$
be a non-singular properly ergodic flow
with $\mu(X)=1$.
We can assume that 
$(X,\mu, \mathcal{F}_t)$
is given by a flow built under a ceiling
function with a base space $(Y,\nu,S)$
and
a positive function $r(y)$ on $Y$.
We represent $x\in X$ as $(\pi(x),h(x))\in Y\times\R$
with $0\leq h(x)<r(\pi(x))$.
We can assume that $0<R_1\leq r(y)<R_2 $
for some $R_i>0$.
(In fact,
we can assume that $r(y)$
takes only two values.
See \cite{Kren,Rudo}.)
Let $\mathbb{Z}\ltimes Y$
be the groupoid whose multiplication rule
is given by
$(n,S^my)(m, y):=(n+m,y)$.
For $(n,y)\in\mathbb{Z}\ltimes Y$,
we set
\[T(n,y)
:=
\begin{cases}
\sum_{k=0}^{n-1}r(S^ky)
&\mbox{if }n\geq 1,
\\
0&\mbox{if }n=0,
\\
-T(-n, S^ny)&
\mbox{if }n\leq -1.
\end{cases}
\]
Then $T\col\mathbb{Z}\ltimes Y\ra\mathbb{R}$
is a homomorphism.

Next for $(t,x)\in\R\ltimes X$,
we define $N(t,x)=m$
if $T(m,\pi(x))\leq t+h(x)< T(m+1,\pi(x))$,
where
we note that
$\lim_{n\rightarrow \infty}T(n,x)=+\infty$
since $r(y)\geq R_1>0$.
It turns out that
$N\col\R\ltimes X\ra\Z$
is a homomorphism.
The maps $T$ and $N$ are related to
each other as follows:
\[
S^{N(t,x)}\pi(x)=\pi(\mathcal{F}_t x),
\]
\begin{equation}
\label{eq:tTN}
t=T(N(t,x),\pi(x))+h(\cF_t x)-h(x)
\quad
\mbox{for all } (t,x)\in\R\ltimes X.
\end{equation}

Define a homomorphism
$p\col\R\ltimes X\ra\Z\ltimes Y$
by
$p(t,x)=(N(t,x),\pi(x))$.
This induces the group homomorphism
$p^*\col Z^1(\mathbb{Z}\ltimes Y)
\ra
Z^1(\mathbb{R}\ltimes X)$
by $p^*(c):=c\circ p$,
where each $Z^1(\cdot)$ denotes the set
of $\T$-valued cocycles.
Let us denote by $B^1(\cdot)$
the set of coboundaries.
Then $p^*(B^1(\Z\ltimes Y))\subs B^1(\R\ltimes X)$.
Hence the following group homomorphism is well-defined:
\begin{equation}
\label{eq:pH1}
p^*\col H^1(\Z\ltimes X)\ra H^1(\R\ltimes X).
\end{equation}
In fact,
$p^*$ is an isomorphism
though we do not use this fact in what follows.
See
\cite[Proposition A.2]{CT} or 
\cite[Theorem 3.1]{Su-Tak-Pac}
for its proof.

For $c_1,c_2\in Z^1(\mathbb{R}\ltimes X)$
and $d_1,d_2\in Z^1(\mathbb{Z}\ltimes Y)$,
we set the following metrics:
\[
\rho_X(c_1,c_2)
:=
\max_{0\leq t<R_1}
\int_X |c(t,x)-c'(t,x)|\,d\mu(x),
\]
\[
\rho_Y(d_1,d_2)
:=
\int_Y|d_1(1,y)-d_2(1,y)|\,d\nu(y).
\]

\begin{lem}\label{lem:ZtoR}
For $d_1,d_2\in Z^1(\mathbb{Z}\ltimes Y)$,
we have
\[
\rho_X(p^*(d_1),p^*(d_2))
\leq
R_2
\rho_Y(d_1,d_2).
\]
In particular,
$p^*$ is continuous.
\end{lem}
\begin{proof}
Note that $N(t,x)\in \{0,1\}$ for $0\leq t<R_1$.
Indeed,
since $0\leq h(x)<r(\pi(x))$,
we obtain
\[
t\leq h(x)+t< r(\pi(x))+t<r(\pi(x))+r(S\pi(x)).
\]
If $h(x)+t<r(\pi(x))$, then $N(t,x)=0$.
If $r(\pi(x))\leq h(x)+t$, then $N(t,x)=1$.

Fix $t$ with $0\leq t<R_1$.
Let
\[
X_1:=\{x\in X\mid 0\leq h(x)<r(x)-t\},
\quad
X_2:=\{x\in X\mid r(x)-t\leq h(x)<r(x)\}.
\]
Then for $x_1\in X_1$
and $x_2\in X_2$,
we have $N(t,x_1)=0$ and $N(t,x_2)=1$.
Hence
$p^*(d_1)(t,x_1)=1=p^*(d_2)(t,x_1)$.
Then
\begin{align*}
\int_X|p^*(d_1)(t,x)-p^*(d_2)(t,x)|
\,d\mu(x)
&=
\int_{X_2}
|p^*(d_1)(t,x)-p^*(d_2)(t,x)|
\,d\mu(x)
\\
&=
\int_{X_2}
|d_1(1,\pi(x))-d_2(1,\pi(x))|
\,d\mu(x)
\\
&\leq
R_2
\int_{Y}
|d_1(1,y)-d_2(1,y)|
\,d\nu(y).
\end{align*}
\end{proof}

\begin{thm}
\label{thm:propergRoh}
Let $\mathcal{A}$
be an abelian von Neumann algebra
and $\alpha$ a non-periodic ergodic flow on $\cA$.
Then $\alpha$ has the Rohlin property.
\end{thm}
\begin{proof}
If $\Sp_d(\al)=\R$,
then $\al$ is conjugate to the translation
on $L^\infty(\R)$.
Thus the function $t\mapsto e^{ipt}$ does the job.

We consider when $\Sp_d(\al)\neq\R$.
Let $(X,\mu,\cF_t)$
be a point realization of $\al$ with $\mu(X)=1$.
Note that $\cF$ is properly ergodic
since $\al$ is non-periodic.
Then we represent $(X,\mu,\cF_t)$
as a flow built under a ceiling function $r$
with a base space $(Y,\nu,S)$ as before.
We let $\th(f)(y):=f(S^{-1}y)$ for $f\in L^\infty(Y)$.

Let $p\in\R$.
For $(t,x)\in\R\ltimes X$
and $(n,y)\in\Z\ltimes Y$,
we set $c(t,x):=e^{ipt}$
and $d(n,y):=e^{ip T(n,y)}$
which are
cocycles of $\R\ltimes X$
and $\Z\ltimes Y$, respectively.
Then for $(t,x)\in\R\ltimes X$,
we have
\begin{align*}
p^*(d)(t,x)
&=
d(N(t,x),\pi(x))
=
e^{ipT(N(t,x),\pi(x))}
\\
&=
e^{ip(t-h(\cF_t x)+h(x))}
\quad
\mbox{by }
(\ref{eq:tTN})
\\
&=
c(t,x)e^{-ip h(\cF_t x)}e^{ip h(x)}.
\end{align*}
This means $p^*(d)=c$ in $H^1(\R\ltimes X)$.

We let $u_n(y):=d(-n,y)$ for $n\in\Z$ and $y\in Y$.
Then $u\col \Z\ra L^\infty(Y)^{\rm U}$
is a $\th$-cocycle.
By the proof of Lemma \ref{lem:Shapiro},
for any $\vep>0$,
there exists $v\in L^\infty(Y)$
such that $|u_k-v^*\th^k(v)|_\nu<\vep$
for $k$ with $|k|\leq 1$.

Let $v(y)$ be a bounded measurable function
representing $v\in L^\infty(Y)$.
We set a coboundary
$\partial v$ on $\Z\ltimes Y$ defined by
$\partial v(n,y):=v(S^n y)v(y)^*$.
Then
\[
\rho_Y(d,\partial v)
=
|u_{-1}-v^*\th^{-1}(v)|_\nu
<\vep.
\]
By the previous lemma,
we obtain
$\rho_X(p^*(d),p^*(\partial v))\leq R_2\vep$.

Therefore,
we have proved that $c$ is approximated by a coboundary.
More precisely,
for any $n\in\N$,
there exists a measurable function
$w_n\col X\ra\T$ such that
\[
\rho_X(c,\partial w_n)
=
\max_{0\leq t<R_1}
\int_X
|e^{ipt}-w_n(x)^*\al_t(w_n)(x)|\,d\mu(x)
<1/n.
\]
Regarding $w_n$ as a unitary in $L^\infty(X,\mu)$,
we obtain
\[
|\al_t(w_n)-e^{ipt}w_n|_\mu<1/n
\quad
\mbox{if }
0\leq t<R_1.
\]

Then $w:=\pi_\om((w_n)_n)\in\cA_\om$
satisfies $\al_t(w)=e^{ipt}w$ for all $t\in\R$.
We will check that
$(w_n)_n$ is $(\al,\om)$-equicontinuous.
Let $f(t):=R_1^{-1}e^{-ipt}1_{[0,R_1]}(t)$
and $w_n':=\al_f(w_n)$.
Then
\[
|\al_f(w_n)-w_n|_\mu
\leq
R_1^{-1}
\int_0^{R_1}
|\al_t(w_n)-e^{ipt}w_n|_\mu\,dt
<1/nR_1.
\]
Thus $\al_f(w_n)-w_n\to0$ in the strong$*$ topology
as $n\to\infty$.
By Lemma \ref{lem:smoothing}, $w\in\cA_{\al,\om}$.
\end{proof}

\begin{rem}
By Theorem \ref{thm:Borelnear1vanish},
any $\al$-cocycle is approximated by a coboundary.
To see this fact,
we can avoid
the Shapiro type argument for the flow $\al$
when we use the fact that
$p^*$ defined in (\ref{eq:pH1})
is isomorphism.
The surjectivity of $p^*$
implies that any $\al$-cocycle $c$
may be assumed to be of the form
$p^*(d)$ with $d$ a $\th$-cocycle.
As we have seen $d$
is approximated by a coboundary,
and so is $c$.
\end{rem}

\begin{cor}
Let $\mathcal{N}$ be a von Neumann algebra
and $\alpha$ a flow on $\mathcal{N}$.
If $\alpha$ is non-periodic and ergodic
on $Z(\mathcal{N})$,
then $\alpha$ has the Rohlin property.
\end{cor}

Since the dual flow of a Rohlin flow
is invariantly approximately inner
by Theorem \ref{thm:dual},
we get the following
result due to Connes \cite{Co-almost}.

\begin{cor}[Connes]
\label{cor:III0-modular}
Let $\cM$ be a type III$_0$ factor.
Then any modular
automorphism group of $\cM$ is approximately inner,
and hence $\cM$ is not a full factor.
\end{cor}

\subsection{Product type flows}
\label{subsect:itp}
We will generalize Example \ref{exam:ITP}
to a factor of type III.
Let ${\bm\la}=(\la_1,\dots,\la_m)$ with all $\la_j>0$
and
${\bm\mu}=(\mu_1,\dots,\mu_m)\in\R^m$.
We consider the following flow:
\[
\cM_{\bm\la}:=
\bigotimes_{k=1}^\infty
(M_{m+1}(\C),\phi_{\bm \la}),
\quad
\phi_{\bm \la}
:=
\frac{1}{1+\la_1+\cdots+\la_{m}}
\Tr
\cdot
\begin{pmatrix}
1&0&\cdots&0\\
0&\la_1&\cdots&0\\
\vdots&\vdots&\ddots&\vdots\\
0&0&\cdots&\la_{m}
\end{pmatrix}
,
\]
\[
\al_t^{{\bm \la},{\bm \mu}}
:=
\bigotimes_{k=1}^\infty
\Ad
\begin{pmatrix}
1&0&\cdots&0\\
0&e^{i\mu_1 t}&\cdots&0\\
\vdots&\vdots&\ddots&\vdots\\
0&0&\cdots&e^{i\mu_{m} t}
\end{pmatrix}
,
\quad
t\in\R.
\]
Let $\vph_{\bm\la}=\bigotimes_{k=1}^\infty\ph_{\bm\la}$.
In what follows,
we simply write $\al=\al^{{\bm\la},{\bm\mu}}$,
$\cM=\cM_{\bm\la}$
and $\vph=\vph_{\bm\la}$.
It is trivial
that $\al$ is invariantly approximately inner,
and the dual flow has the Rohlin property.
This fact will enable us to classify $\al$
in terms of ${\bm\la}$ and ${\bm\mu}$.
Note that $((\cM_\vph)^{\al})'\cap\cM=\C$
because the infinite symmetric group
$\mathfrak{S}_\infty$ is represented
into $(\cM_\vph)^{\al}$ canonically.

The flow $\al$ fixes the diagonal Cartan subalgebra
of $\cM$,
and we have already classified such flows
in Theorem \ref{thm:class-invapprox}
for injective infinite factors.
However, we can easily compute the flow of weights
of $\cM\rti_\al\R$
because of the minimality of $\al$,
and we will classify $\al$
without use of the results obtained
in \S\ref{subsect:Cartan}.

By Lemma \ref{lem:can-natural},
we have the following identification:
\[
\tcN=\tcM\rti_\tal\R=(\cM\rti_{\si^\vph}\R)\rti_\tal\R.
\]
Note that $\tal_t(\la^\vph(s))=\la^\vph(s)$.
Since
$(\cM_\vph)^\al\subs\cM$ has the trivial relative commutant,
we have
$\tcM'\cap \tcN
\subs \C\oti \{\la^\vph(\R)\}''\otimes\{\la^\tal(\R)\}''$.
From the picture of the product type flows,
we can observe that
$\cM$ is strongly densely spanned
by eigenvectors for both $\si^\vph$ and $\al$.
Take a non-zero element $x_j\in \cM$
such that
$\si_s^\vph(\al_t(x_j))
=e^{i(\log\la_j s+\mu_j t)}x_j$.
Then we have
\[
\pi_\tal(\pi_{\si^\vph}(x_j))
=x_j
\oti
{\bf e}_{-\log\la_j}\oti{\bf e}_{-\mu_j},
\]
where ${\bf e}_{s}(t):=e^{ist}$ for $s,t\in\R$.
Hence we obtain the natural identification
\begin{equation}
\label{eq:ZcN}
Z(\tcN)=\tcM'\cap \tcN=(L(\R)\oti L(\R))
\cap\{{\bf e}_{-\log\la_j}\oti{\bf e}_{-\mu_j}
\mid j=1,\dots,m\}',
\end{equation}
where the flow of weights $\th_s$
and the dual flow $\widehat{\tal}_t$ are given by
the restrictions of
$\Ad {\bf e}_{-s}\oti1$ and $1\oti \Ad{\bf e}_{-t}$.
By the Fourier transform,
we have the isomorphism
$L(\R)\oti L(\R)\to L^\infty(\R)\oti L^\infty(\R)$
such that
$\la^\vph(s)\oti\la^{\tal}(t)\mapsto {\bf e}_s\oti {\bf e}_t$.
Then we have
\[
\tcM\cap(\tcM\rti_\tal\R)=Z(\tcN)\cong (L^\infty(\R)\oti L^\infty(\R))
\cap
\{\la(\log\la_j)\oti\la(\mu_j)\}'.
\]
The $\th_t$ and $\widehat{\tal}_s$ are transformed to
$\Ad\la(t)\oti1$ and $1\oti \Ad\la(s)$,
respectively.
Note that $Z(\tcM)=(\tcM'\cap(\tcM\rti_\tal\R))^{\widehat{\tal}}$,
and
\[
Z(\tcM)\cong L^\infty(\R)\cap\{\la(\log\la_j)\}'.
\]

For $s_1,\dots,s_k\in\R^2$,
we denote by $\langle s_1,\dots,s_k\rangle$
the closed subgroup generated by
them.
We put
\[
G_{{\bm \la},{\bm \mu}}
:=\langle
(\log\la_j,\mu_j),
\mid
j=1,\dots,m
\rangle
.
\]
Put $\pr_1(x,y):=x$ and $\pr_2(x,y)=y$.
Then we have
\[
\ovl{\pr}_1(G_{{\bm \la},{\bm \mu}})
=
\Ga(\si^{\vph_{\bm\la}}),
\quad
\ovl{\pr}_2(G_{{\bm \la},{\bm \mu}})
=
\Ga(\al^{{\bm\la},{\bm\mu}}),
\]
where
$\ovl{\pr}_j(G_{{\bm \la},{\bm \mu}})$
denotes
the closure of
$\pr_j(G_{{\bm \la},{\bm \mu}})$.

\begin{thm}
\label{thm:flowwt-itp}
Let $\cM_{\bm\la}$, $\al:=\al^{{\bm\la},{\bm \mu}}$
and $\cN:=\cM_{\bm\la}\rti_\al\R$ as before.
Let $(X_\cN,F^\cN)$ be the flow of weights
of $\cN$.
Then the following holds:
\begin{enumerate}
\item
One has the identification
$L^\infty(X_\cN)=L^\infty(\R^2)^{G_{{\bm \la},{\bm \mu}}}$
that is the fixed point algebra
of the translation action
of $G_{{\bm \la},{\bm \mu}}$ on $\R^2$;

\item
The flow of weights $F^\cN$ is given by
\[
f(F_{t}^\cN (r,s))=f(r+t,s)
\quad
\mbox{for }
f\in L^\infty(X_\cN),
\
r,s,t\in\R;
\]

\item
The Connes-Takesaki module of $\hal_t$ is given by
\[
(\mo(\hal_t)f)(r,s)=f(r,s-t)
\quad
\mbox{for }
f\in L^\infty(X_\cN),
\
r,s,t\in\R.
\]
\end{enumerate}
\end{thm}

Thus we obtain the following characterization
of the factoriality and the type of $\cN$.
Remark \ref{rem:typeI-Gamma} states that
if $\cN$ were of type I, then $\Ga(\al)=\{0\}$.
Hence $\mu_j=0$ for all $j$,
and $\cN\cong \cM\oti L^\infty(\R)$
that is not of type I,
and this is a contradiction.
Thus
$\cN$ must be a von Neumann algebra
of type II$_1$, II$_\infty$ or III.
When $\cN$ is a factor,
its type is either II$_\infty$ or III.

\begin{cor}
\label{cor:product-flowwt}
The following statements hold:
\begin{enumerate}
\item
$\cM_{\bm\la}\rti_{\al^{{\bm \la},{\bm\mu}}}\R$
is a factor
if and only if
$\ovl{\pr}_2(G_{{\bm \la},{\bm\mu}})=\R$;

\item
$\cM_{\bm\la}\rti_{\al^{{\bm \la},{\bm\mu}}}\R$ is a factor of
\begin{itemize}
\item
type II$_\infty$
if and only if
$G_{{\bm \la},{\bm\mu}}\cong\R$
and
$G_{{\bm \la},{\bm\mu}}\neq \R\times\{0\}$;

\item
type III$_1$
if and only if
$G_{{\bm \la},{\bm\mu}}=\R^2$;

\item
type III$_0$
if and only if
$G_{{\bm \la},{\bm\mu}}\cong\Z^2$
and
$\ovl{\pr}_2(G_{{\bm \la},{\bm\mu}})=\R$;

\item
type III$_\rho$,
$0<\rho<1$
if and only if
$G_{{\bm \la},{\bm\mu}}\cong\R\oplus\Z$
and
$G_{{\bm \la},{\bm\mu}}\neq\R\times\Z$.
\end{itemize}
\end{enumerate}
\end{cor}
\begin{proof}
We let
$\cN:=\cM_{\bm\la}\rti_{\al^{{\bm \la},{\bm\mu}}}\R$.

(1).
This is because $\ovl{\pr}_2(G_{{\bm\la},{\bm\mu}})=\Ga(\al)$,
or $Z(\cN)=Z(\tcN)^\th$.

(2).
Recall that
any closed subgroup in $\R^2$
is isomorphic to one of the following:
\[
\R^2,\ \R\times\Z,\ \R,
\
\Z^2,\ \Z,\ 0.
\]
The factoriality of $\cN$
excludes
the cases $G_{{\bm\la},{\bm\mu}}\cong\Z,0$.

We know that
$\cN$ is semifinite
if and only if
the flow of weights is the translation on $\R$,
that is,
$G_{{\bm\la},{\bm\mu}}\cong\R$.

Since $\cN$ is a factor of type III$_1$
if and only if
$L^\infty(X_\cN)=\C$,
we have $G_{{\bm\la},{\bm\mu}}=\R^2$.

Let us consider the case
that $G_{{\bm\la},{\bm\mu}}\cong\Z^2$.
Then we can take a parallelogram
as a fundamental domain of
the action of $G_{{\bm\la},{\bm\mu}}$
on $\R^2$.
Then the flow of weights of $\cN$
must be non-periodic
because of the ergodicity.
Thus $\cN$ is of type III$_0$
if $\cN$ is a factor.

When $G_{{\bm\la},{\bm\mu}}\cong\R\oplus\Z$,
a fundamental domain is a segment.
By the factoriality,
we have
$\ovl{\pr}_2(G_{{\bm\la},{\bm\mu}})=\R$,
and $G_{{\bm\la},{\bm\mu}}\neq\R\times\Z$.
Then it is easy to see that
the flow of weights of $\cN$
is periodic,
that is,
$\cN$ is of type III$_\rho$
with $0<\rho<1$.
\end{proof}

\begin{exam}
\label{ex:faithful-module}
We consider the case of $m=2$.
Then $G_{{\bm\la},{\bm\mu}}$ is isomorphic to
one of $\R$, $\Z$ and $\Z^2$.
Hence if $\cN$ is a factor,
then $\cN$ must be of type II$_\infty$ or III$_0$.
The former comes from the modular flow
$\si_{at}^{\vph_{\bm\la}}$ for some $a\in\R$.
Let us consider the latter.
Namely,
$\mu_1,\mu_2\neq0$ and $\mu_1/\mu_2\nin\Q$,
and the vectors
$v_1:=(\log\la_1,\mu_1)$ and $v_2:=(\log\la_2,\mu_2)$
are linearly independent.
With this assumption,
we obtain $(\la_1,\la_2)\neq(1,1)$,
that is, $\cM_{\bm\la}$ is of type III.
Then
$G_{{\bm\la},{\bm\mu}}=\Z v_1+\Z v_2$
and
$\ovl{\pr}_2(G_{{\bm \la},{\bm\mu}})=\R$.
Denote by $\be$
the dual flow of
$\al^{{\bm \la},{\bm\mu}}$.

We let $S_t^1(r,s):=(r+t,s)$
and $S_t^2(r,s):=(r,s+t)$
for $r,s,t\in\R$.
Then $L^\infty(\R)^{G_{{\bm\la},{\bm\mu}}}$
is nothing but
the fixed point algebra
of $L^\infty(\R)$
with respect to the transformations
$S_{\log\la_j}^1S_{\mu_j}^2$,
$j=1,2$.
Recall that
the flow of weights of $\cN$
and
the Connes-Takesaki module
of $\be$
are given by $S_t^1$
and $S_t^2$ on $L^\infty(\R)^{G_{{\bm\la},{\bm\mu}}}$,
respectively.

We pull back the flows
$S^1$ and $S^2$
through the linear transformation
$T\col\R^2\ra\R^2$,
where
\[
T:=
\begin{pmatrix}
\log\la_1&\log\la_2\\
\mu_1&\mu_2
\end{pmatrix}.
\]

Then we have
\[
T^{-1}\circ S_t^1\circ T
\begin{pmatrix}
x\\
y
\end{pmatrix}
=
\begin{pmatrix}
x\\
y
\end{pmatrix}
+
t
T^{-1}
\begin{pmatrix}
1\\
0
\end{pmatrix}
,
\
T^{-1}\circ S_t^2\circ T
\begin{pmatrix}
x\\
y
\end{pmatrix}
=
\begin{pmatrix}
x\\
y
\end{pmatrix}
+
t
T^{-1}
\begin{pmatrix}
0\\
1
\end{pmatrix}
.
\]
Note that
\[
T^{-1}\circ S_{\log\la_1}^1 S_{\mu_1}^2\circ T
\begin{pmatrix}
x\\
y
\end{pmatrix}
=
\begin{pmatrix}
x\\
y
\end{pmatrix}
+
\begin{pmatrix}
1\\
0
\end{pmatrix}
,
\
T^{-1}\circ S_{\log\la_2}^1S_{\mu_2}^2\circ T
\begin{pmatrix}
x\\
y
\end{pmatrix}
=
\begin{pmatrix}
x\\
y
\end{pmatrix}
+
\begin{pmatrix}
0\\
1
\end{pmatrix}
.
\]

Thus the triple
$(X_\cN,\cF_t^\cN,\mo(\be_t))$
is conjugate
to the Kronecker flow on
$[0,1]^2$.
More precisely,
with this identification,
we have
\[
\cF_t^\cN
\begin{pmatrix}
x\\
y
\end{pmatrix}
=
\begin{pmatrix}
x\\
y
\end{pmatrix}
+
\frac{t}{\De_{{\bm\la},{\bm\mu}}}
\begin{pmatrix}
\mu_2\\
-\mu_1
\end{pmatrix}
,\quad
\mo(\be_t)
\begin{pmatrix}
x\\
y
\end{pmatrix}
=
\begin{pmatrix}
x\\
y
\end{pmatrix}
+
\frac{t}{\De_{{\bm\la},{\bm\mu}}}
\begin{pmatrix}
-\log\la_2\\
\log\la_1
\end{pmatrix}
,
\]
where
$\De_{{\bm\la},{\bm\mu}}
:=\det(T)=\mu_2\log\la_1-\mu_1\log\la_2$.
This means
that the vectors ${\bm\la}$ and ${\bm\mu}$
determine
the directions of
$\mo(\be_t)$ and $\cF_t^{\cN}$, respectively.
In particular,
$\mo(\be_t)$ is non-periodic
if and only if
$\log\la_1\Z+\log\la_2\Z$ is dense in $\R$,
that is,
$\cM_{\bm\la}$ is of type III$_1$.

We will prove in Lemma \ref{lem:modular-part}
that the modular part of $\al$
is equal to
$\pr_2(G_{{\bm\la},{\bm\mu}}^\perp)$,
where we denote by $G_{{\bm\la},{\bm\mu}}^\perp$
the annihilator group of $G_{{\bm\la},{\bm\mu}}$.
Hence $(x,y)\in G_{{\bm\la},{\bm\mu}}^\perp$
if and only if
$(x,y)T\in (2\pi\Z,2\pi\Z)$.
This implies that
$\pr_2(G_{{\bm\la},{\bm\mu}}^\perp)
=(2\pi/\Delta_{{\bm\la},{\bm\mu}})(\log\la_1\Z+\log\la_2\Z)$.
This is not equal to $\R$.
Hence $\al$ is not extended modular.
\end{exam}

Now we will prove that
the closed subgroup $G_{{\bm \la},{\bm\mu}}$
is a complete invariant of
the cocycle conjugacy class
of $\al^{{\bm\la},{\bm\mu}}$.

\begin{thm}
\label{thm:itp-class}
Let $m,n\in\N$.
Let ${\bm\la}\in\R^m$ and ${\bm\rho}\in\R^n$
be such that
$\la_j,\rho_k>0$ for all $j,k$.
Let ${\bm\mu}\in\R^m$ and ${\bm\nu}\in\R^n$.
Consider the flows
$\al^{{\bm\la},{\bm\mu}}$
and $\al^{{\bm\rho},{\bm\nu}}$
on
$\cM_{\bm\la}$ and
$\cM_{\bm\rho}$,
respectively.
Then
they are cocycle conjugate
if and only if
$G_{{\bm \la},{\bm\mu}}=G_{{\bm \rho},{\bm\nu}}$.
\end{thm}
\begin{proof}
Let us simply write
$\al^{\bm\mu}:=\al^{{\bm\la},{\bm\mu}}$
and
$\al^{\bm\nu}:=\al^{{\bm\rho},{\bm\nu}}$.
Let
$\be^{\bm\mu}:=\widehat{\al^{\bm\mu}}$
and
$\be^{\bm\nu}:=\widehat{\al^{\bm\nu}}$.
We put $\cP_1:=\cM_{\bm\la}\rti_{\al^{\bm\mu}}\R$
and $\cP_2:=\cM_{\bm\rho}\rti_{\al^{\bm\nu}}\R$.
Recall that we have the isomorphism
$Z(\tcP_1)
\cong
L^\infty(\R^2)^{G_{{\bm \la},{\bm\mu}}}$
and
$Z(\tcP_2)
\cong
L^\infty(\R^2)^{G_{{\bm \rho},{\bm\nu}}}$.

Suppose that
$\al^{\bm\mu}$
and $\al^{\bm\nu}$
are cocycle conjugate flows,
that is,
there exist
an $\al^{\bm\nu}$-cocycle
$v$
and
an isomorphism
$\ph\col\cM_{\bm\la}\ra\cM_{\bm\rho}$
such
that
$\Ad v(t)\circ\al_t^{\bm\nu}
=
\ph\circ\al_t^{\bm\mu}\circ\ph^{-1}$.
Then $\ph$ extends to
the isomorphism $\ph\col\cP_1\ra\cP_2$
such that
$\ph(\pi_{\al^{\bm\mu}}(x))
=\pi_{\al^{\bm\nu}}(\ph(x))$
and
$\ph(\la^{\al^{\bm\mu}}(t))
=\pi_{\al^{\bm\nu}}(v(t))\la^{\al^{\bm\nu}}(t)$
for $x\in\cM_{\bm\la}$ and $t\in\R$.
Then
$\be^{\bm\nu}_t
=
\ph\circ\be^{\bm\mu}_t\circ\ph^{-1}$.

Let
$\tph\col\tcP_1\ra\tcP_2$
be the canonical extension of $\ph$.
Let $\th^1$ and $\th^2$ be the dual flows on $\tcP_1$
and $\tcP_2$, respectively.
Then
$\th_t^2=\tph\circ\th_t^1\circ\tph^{-1}$
and
$\widetilde{\be_t^{\bm\nu}}
=\tph\circ\widetilde{\be_t^{\bm\mu}}\circ\tph^{-1}$.
Hence
we have an isomorphism
$L^\infty(\R^2)^{G_{{\bm\la},{\bm\mu}}}
\cong
L^\infty(\R^2)^{G_{{\bm\rho},{\bm\nu}}}$
preserving the translation of
$\R^2$.
Thus $G_{{\bm\la},{\bm\mu}}=G_{{\bm\rho},{\bm\nu}}$.

Next suppose
that $G_{{\bm\la},{\bm\mu}}=G_{{\bm\rho},{\bm\nu}}$.
Then
we have an isomorphism
$\si\col
Z(\tcP_1)\ra
Z(\tcP_2)$
preserving their flows of weights
and the Connes-Takesaki modules
of $\be^{\bm\mu}$ and $\be^{\bm\nu}$.
Applying $\pr_j$
to $G_{{\bm\la},{\bm\mu}}=G_{{\bm\rho},{\bm\nu}}$,
we obtain
$\Ga(\si^{\vph_{\bm\la}})=\Ga(\si^{\vph_{\bm\rho}})$
and
$\Ga(\al^{\bm\mu})=\Ga(\al^{\bm\nu})$.
In particular,
$\cM_{\bm\la}$ is semifinite
if and only if
$\la_j=1$ for all $j$.
In this case,
$\rho_j=1$ for all $j$,
and $\cM_{\bm\la}\cong\cM_{\bm\rho}$
is of type II$_1$.
When $\cM_{\bm\la}$ is of type III,
then $\cM_{\bm\rho}$ is the same type as $\cM_{\bm\la}$.
Since they are not of type III$_0$,
$\cM_{\bm\la}\cong \cM_{\bm\rho}$.

\vspace{5pt}
\noindent{\bf Case 1.}
$\Ga(\al^{\bm\mu})=\R$.
\vspace{5pt}

In this case,
$\Ga(\al^{\nu})=\ovl{\pr}_2(G_{{\bm\rho},{\bm\nu}})=\R$.
Thus
$\cP_1$ and $\cP_2$ are factors.
By Corollary \ref{cor:product-flowwt} (2),
$\tcP_1$ and $\tcP_2$ must be a von Neumann algebra
of type II$_\infty$.
Hence $\tcP_1\cong\tcP_2$.

By Lemma \ref{lem:P1P2},
there exist $s\in\R$
and an isomorphism
$\ph\col\cP_1\ra\cP_2$
such that
$\si=\mo(\ph)\circ\th_s^1$.

Consider the flow
$\ga_t:=\ph\circ\be_t^{\bm\mu}\circ\ph^{-1}$
on $\cP_2$.
Then they satisfy
\begin{align*}
\mo(\ga_t)
&=\mo(\ph)\circ\mo(\be_t^{\bm\mu})\circ\mo(\ph)^{-1}
\\
&=\si\circ \th_{-s}^1
\circ\mo(\be_t^{\bm\mu})\circ
\th_s^1\circ\si^{-1}
\\
&=
\si\circ\mo(\be_t^{\bm\mu})\circ\si^{-1}
\\
&=
\mo(\be_t^{\bm\nu}).
\end{align*}
Since they have the Rohlin property,
$\ga\sim \be^{\bm\nu}$
by Corollary \ref{cor:class-inj}.
Hence by Takesaki duality,
$\al^{\bm\mu}\oti\id_{B(L^2(\R))}
\sim\al^{\bm\nu}\oti\id_{B(L^2(\R))}$.

\vspace{5pt}
\noindent
{\bf Case 2.}
$\Ga(\al^{\bm\mu})\neq\R$.
\vspace{5pt}

When $\Ga(\al^{\bm\mu})=\{0\}$,
$\mu_j=0=\nu_j$ for all $j$.
Thus there is nothing to prove
because
$\al_t^{\bm\mu}=\id_{\cM_{\bm\la}}$
and $\al_t^{\bm\nu}=\id_{\cM_{\bm\rho}}$.

We consider the case that
$\Ga(\al^{\bm\mu})=p\Z$ for some $p>0$.
Then $\al^{\bm\mu}$ and $\al^{\bm\nu}$
have the period $T:=2\pi/p$.
When we regard them as the actions
of the torus $\R/T\Z$,
we denote them by $\ga^1$ and $\ga^2$,
respectively.
Note that they are minimal.

Let $\cQ_1:=\cM_{\bm\la}\rti_{\ga^1}\R/T\Z$
and $\cQ_2:=\cM_{\bm\rho}\rti_{\ga^2}\R/T\Z$.
Let $\de^1$ and $\de^2$
be the dual $p\Z$-actions
of $\ga^1$ and $\ga^2$, respectively.
By a similar computation to (\ref{eq:ZcN}),
we obtain the following
natural isomorphism:
\[
Z(\tcQ_1)\cong L^\infty(\R\times p\Z)^{G_{{\bm\la},{\bm\mu}}},
\quad
Z(\tcQ_2)\cong L^\infty(\R\times p\Z)^{G_{{\bm\rho},{\bm\nu}}},
\]
where the flows of weights act on the first
coordinate,
and the Connes-Takesaki modules of
$\de^1$ and $\de^2$
do on the second.
Define an isomorphism
$\si\col Z(\tcQ_1)\ra Z(\tcQ_2)$
by this expression.

Since $\tcQ_1$ and $\tcQ_2$
must be of type II$_\infty$,
there exist an isomorphism
$\ph\col\cQ_1\ra\cQ_2$
and
$s\in\R$
such that
$\si=\mo(\ph)\circ\th_s^1$
by Lemma \ref{lem:P1P2},
where $\th^1$ denotes the flow of weights of $\cQ_1$.
Putting $\de':=\ph\circ\de^1\circ\ph^{-1}$,
we obtain $\mo(\de')=\mo(\de^2)$.

We compute the modular invariants of
$\de^1$ and $\de^2$.
Suppose that
$\widetilde{\de}_{pn}^1=\Ad u$ for some $n\in\Z$
and $u\in\tcQ_1^{\rm U}$.
Since
$\widetilde{\de}^1$ fixes
$\widetilde{\cM_{\bm\la}}$
and
$\widetilde{\cM_{\bm\la}}'\cap \tcQ_1=Z(\tcQ_1)$,
we must have
$u\in Z(\tcQ_1)$,
that is,
$\widetilde{\de}_{pn}^1=\id_{\tcQ_1}$.
Thus $n=0$.
Likewise,
it turns out that
the modular part of $\de^2$ is trivial.
Hence
$\de^1$ and $\de^2$ are centrally free.

Therefore, $\de^1\sim\de^2$ as $p\Z$-actions
by \cite[Theorem 6.1]{Kat-S-T}
or \cite[Theorem 20]{KawST}
(see \cite[Theorem 3.1]{Ma-class} for a simple proof).
Thus
$\ga^1\oti\id_{B(\ell^2)}\sim \ga^2\oti\id_{B(\ell^2)}$
by Takesaki duality.
This implies that
$\al^{\bm\mu}\oti\id_{B(\ell^2)}
\sim\al^{\bm\nu}\oti\id_{B(\ell^2)}$
as $\R$-actions.

\vspace{9pt}
Hence in either case,
$\al^{\bm\mu}$ is stably conjugate to
$\al^{\bm\nu}$.
If $\cM$ is infinite,
this implies
$\al^{\bm\mu}\sim\al^{\bm\nu}$
as shown in Remark \ref{rem:alid}.
When $\cM_{\bm\la}\cong \cM_{\bm\rho}$ is finite,
$\{0\}=\Ga(\si^{\vph_{\bm\la}})=\Sp(\si^{\vph_{\bm\la}})$.
Thus $\la_j=1$ for all $j$.
Likewise, $\rho_k=1$ for all $k$.
Then
the condition
$G_{{\bm\la},{\bm\mu}}=G_{{\bm\rho},{\bm\nu}}$
implies
that
$H:=
\langle \mu_j\mid j\rangle
=\langle \nu_j\mid j\rangle$.

When $H=T\Z$ for some $T>0$,
$\al^{\bm\mu}$ and $\al^{\bm\nu}$
have the period $2\pi/T$.
Then they are regarded
as minimal actions of the torus
$\R/(2\pi/T)\Z$
on the injective type II$_1$ factor.
Hence
$\al^{\bm\mu}$ is conjugate to $\al^{\bm\nu}$
by the uniqueness of a minimal action
of the torus on the injective type II$_1$
factor.

When $H=\R$,
the both actions have the Rohlin property
by Theorem \ref{thm:almost} or \ref{thm:itp-Rohlin}.
Therefore,
they are cocycle conjugate
by Corollary \ref{cor:II1III1}.
\end{proof}

Next
we describe the modular part $\La(\al^{{\bm\la},{\bm\mu}})$
of $\al:=\al^{{\bm\la},{\bm\mu}}$
on $\cM:=\cM_{\bm\la}$.
By Lemma \ref{lem:cnt-borel},
$\La(\al^{{\bm\la},{\bm\mu}})$
is a Borel subgroup
of $\R$ which consists of elements $t\in\R$
such that $\tal_t$ are inner.

\begin{lem}
One has
$\La(\al)=\Sp_d(\widehat{\tal}|_{Z(\tcM\rti_\tal\R)})$.
\end{lem}
\begin{proof}
Let $p\in \La(\al)$.
Take a unitary $u_p\in \tcM$ such that
$\tal_p=\Ad u_p$.
Then $v_p:=\pi_\tal(u_p^*)\la^\tal(p)$ belongs to
$\tcM'\cap(\tcM\rti_\tal\R)=Z(\tcM\rti_\tal\R)$,
and $\widehat{\tal}_t(v_p)=e^{-ipt}v_p$.
Thus $-p$, and hence, $p$
belong to $\Sp_d(\widehat{\tal}|_{Z(\tcM\rti_\tal\R)})$.

Suppose that
$-p\in \Sp_d(\widehat{\tal}|_{Z(\tcM\rti_\tal\R)})$.
Take
a non-zero $z\in Z(\tcM\rti_\tal\R)$
such that
$\widehat{\tal}_t(z)=e^{-ipt}z$.
We put $a:=z^*\la^\tal(p)$.
Then $a$ is fixed by $\widehat{\tal}$,
and $a\in \pi_\tal(\tcM)$.
Let $b$ be such that $a=\pi_\tal(b)$.
For $x\in \tcM$,
we have $bx=\tal_p(x)b$.
Then by ergodicity of $\th$ on $Z(\tcM)$,
we can take a unitary $u\in\tcM$
such that $ux=\tal_p(x)u$ for all $x\in\tcM$
(see the proof of \cite[Proposition 3.4]{Iz-Can}).
Hence $p\in\La(\al)$.
\end{proof}

Let $p\in\La(\al)$
and take $u_p,v_p$ as given in the proof above.
Since $Z(\tcM\rti_\tal\R)$ is included
in $\C\oti L(\R)\oti L(\R)$,
there exists $w_p\in L(\R)=\{\la^\vph(\R)\}''$
such that
$v_p=1\oti w_p\oti \la^\tal(p)=\pi_{\tal}(1\oti w_p)\la^\tal(p)$.
Thus $u_p=1\oti w_p^*$.
In particular,
$\tal_t(u_p)=u_p$.

Next we compute the modular invariant of $u_p$.
By (\ref{eq:ZcN}),
$v_p$ commutes with
$1\oti {\bf e}(\log\la_j)\oti {\bf e}(\mu_j)$
for all $j$.
Hence
\[
\th_{\log\la_j}(w_p)=e^{ip\mu_j}w_p.
\]
Then we obtain
the character
$\langle \log\la_j\mid j\rangle=\Ga(\si^\vph)
\ni t\mapsto \th_t(w_p)w_p^*\in\T$.
Thus there exists $q$
such that
$e^{iq\log\la_j}=e^{ip\mu_j}$
for all $j$.

\begin{lem}
An element
$p\in\R$ belongs to $\La(\al)$
if and only if
there exists $q\in\R$
such that
$e^{iq\log\la_j}=e^{ip\mu_j}$
for all $j=1,\dots,m$.
\end{lem}
\begin{proof}
The ``only if'' part has been shown.
We show the ``if'' part.
Suppose that
$q$ satisfies
$e^{iq\log\la_j}=e^{ip\mu_j}$
for all $j=1,\dots,m$.
Put $u:=\pi_\tal(\la^\vph(q)^*)\la^\tal(p)$.
Then $u$ is fixed by the action
of $G_{{\bm\la},{\bm\mu}}$,
and it turns out that $u\in Z(\tcM\rti_\tal\R)$.
Thus we are done from the previous lemma.
\end{proof}

We let $G_{{\bm\la},{\bm\mu}}^\perp$ be the annihilator
group of $G_{{\bm\la},{\bm\mu}}$
with respect to the pairing of $\R^2$,
that is,
$(a,b)\in G_{{\bm\la},{\bm\mu}}^\perp$
if and only if
$e^{i(ax+by)}=1$
for all $(x,y)\in G_{{\bm\la},{\bm\mu}}$.
Then the following lemma is an immediate
consequence of the previous result.

\begin{lem}
\label{lem:modular-part}
The modular part $\La(\al)$ coincides
with $\pr_2(G_{{\bm\la},{\bm\mu}}^\perp)$.
\end{lem}

Now we will characterize when
$\al^{\bm\mu}$ has the Rohlin property.

\begin{thm}
\label{thm:itp-Rohlin}
Let $\cM:=\cM_{\bm\la}$
and $\al:=\al^{{\bm\la},{\bm\mu}}$
as before.
Then the following statements are equivalent:
\begin{enumerate}
\item 
The flow $\al$ has the Rohlin property;

\item
$\tcM'\cap(\tcM\rti_\tal\R)=Z(\tcM)$;

\item
$\Ga(\al)=\R$ and $\al_t\nin\Cnt(\cM)$
for all $t\neq0$;

\item
$\{0\}\times\R\subs G_{{\bm \la},{\bm\mu}}$.
\end{enumerate}
In this case,
$G_{{\bm \la},{\bm\mu}}=\Ga(\si^\vph)\times\R$.
Moreover,
$\cM\rti_\al\R$ is an injective infinite factor
of the same type as $\cM\oti B(L^2(\R))$.
\end{thm}
\begin{proof}
(1)$\Rightarrow$(2).
This is proved in Corollary \ref{cor:relative}.

(2)$\Rightarrow$(3).
The relative commutant property
$\tcM'\cap(\tcM\rti_\tal\R)=Z(\tcM)$
implies that
$\tal$ is an outer flow,
and $\al_t\nin\Cnt(\cM)$ for $t\neq0$.
By assumption,
we have $Z(\cM\rti_\al\R)=Z(\tcM\rti_\tal\R)^\th=\C$.
Thus $\Ga(\al)=\R$.

(3)$\Rightarrow$(4).
Since $\Ga(\al)=\R$,
$\cM\rti_\al\R$ is a factor.
By the previous lemma,
we obtain $\pr_2(G_{{\bm\la},{\bm\mu}}^\perp)=\{0\}$.
Thus $G_{{\bm\la},{\bm\mu}}^\perp$
is of the form $H\times\{0\}$
for a unique closed subgroup $H\subs\R$.
Then $G_{{\bm\la},{\bm\mu}}=H^\perp\times\R$,
and (4) follows.

(4)$\Rightarrow$(1).
The assumption of (4)
implies
that
$G_{{\bm \la},{\bm\mu}}=\Ga(\si^\vph)\times \R$.
This implies
the factoriality of $\cM\rti_\al\R$
and
$\mo(\hal_t)=\id$
for all $t\in\R$.
Hence
$\hal$ is pointwise approximately inner.
Since
$G_{{\bm \la},{\bm\mu}}=\Ga(\si^\vph)\times \R$,
we have
$Z(\tcN)=L^\infty(\R)^{\Ga(\si^\vph)}$.
Thus $\cN$ is an injective infinite factor
of the same type as $\cM\oti B(L^2(\R))$.

As remarked before,
$\hal$ has the Rohlin property.
Hence $\hal\sim\al^{0}\oti\id_{\cN}$
by Theorem \ref{thm:class2},
where $\al^0$ is a Rohlin flow
on the injective type II$_1$ factor.
Then the Takesaki duality
implies that $\al$ has the Rohlin property
by Theorem \ref{thm:dual}
and Corollary \ref{cor:Rohlin-II1}.
\end{proof}

\begin{rem}
When $m=2$,
$G_{{\bm \la},{\bm\mu}}$
is generated by two vectors
$(\log\la_j,\mu_j)$.
Hence $G_{{\bm \la},{\bm\mu}}$
is isomorphic to one of
$\Z^2$, $\Z$ and $\R$.
Then
$G_{{\bm \la},{\bm\mu}}=\Ga(\si^\vph)\times\R$
if and only if
$\Ga(\si^\vph)=\{0\}$.
This means that
$\vph$ is tracial
and $\langle\mu_1,\mu_2\rangle=\R$,
that is,
$\al^{{\bm\la},{\bm\mu}}$
is the flow given in Example \ref{exam:ITP}.
\end{rem}

The following lemma is also used in the next subsection.

\begin{lem}
\label{lem:torus-minimal}
Let $\cM$ be an injective factor
and $\ga$ a minimal action of $\T^m$.
Let $E\col\cM\ra\cM^\ga$
be the faithful normal conditional expectation.
Suppose that $\cM^\ga$ has the faithful normal
tracial state $\ta$.
Let $\vph:=\ta\circ E$.
Then the following statements hold:
\begin{enumerate}
\item 
There exists ${\bm\la}=(\la_1,\dots,\la_m)$
such that
$\la_j>0$
and
$\si_t^\vph
=\ga_{(e^{i\log\la_1 t},\dots,e^{i\log\la_m t})}$;

\item
$\ga_{(z_1,\dots,z_m)}\sim
\ga_{z_1}^{\la_1}\oti\cdots\oti\ga_{z_m}^{\la_m}$,
where
the $\T$-action
$\ga^{\la_j}$ is defined as
\[
\cM_{\la_j}
:=
\bigotimes_{n=1}^\infty
(M_2(\C),\ph_j)'',
\quad
\ph_{\la_j}
:=
\frac{1}{1+\la_j}
\Tr
\cdot
\begin{pmatrix}
1&0\\
0&\la_j
\end{pmatrix}
,
\]
\[
\ga_z^{\la_j}
=
\bigotimes_{n=1}^\infty
\Ad
\begin{pmatrix}
1&0\\
0&z
\end{pmatrix}
.
\]
\end{enumerate}
\end{lem}
\begin{proof}
(1).
See \cite[Proposition 5.2 (5)]{Iz-Can}.

(2).
Let $\cN:=\cM\rti_\ga\T^m$.
Then $\si_t^{\hvph}$ is implemented by
$\la^\ga((e^{i\log\la_1 t},\dots,e^{i\log\la_m t}))$.
Thus $\cN$ is the injective factor of
type II$_\infty$.
Let $h$ be a positive operator affiliated with
$\cN$ such that
\[
h^{it}
=\la^\ga((e^{i\log\la_1 t},\dots,e^{i\log\la_{m}t})).
\]
Then $\ta:=\hvph_{h^{-1}}$
is a faithful normal tracial weight on $\cN$.
We compute the module of $\hga$ as follows.
Let $(k_1,\dots,k_m)\in\Z^m$.
By definition of the dual action,
we get
\[
\hga_{(k_1,\dots,k_m)}(h^{it})
=
e^{-i(k_1\log\la_1+\cdots+k_m\log\la_m)t}h^{it},
\]
and
\[
\hga_{(k_1,\dots,k_m)}(h)
=
e^{-(k_1\log\la_1+\cdots+k_m\log\la_m)}h.
\]
Hence
\[
\ta\circ\hga_{(k_1,\dots,k_m)}
=
\hvph_{\hga_{(-k_1,\dots,-k_m)}(h)^{-1}}
=
e^{-(k_1\log\la_1+\cdots+k_m\log\la_m)}
\ta
=
\la_1^{-k_1}\cdots\la_m^{-k_m}
\ta.
\]

Thanks to \cite[Theorem 2.9]{Ocn-act},
we have the cocycle conjugacy
as the $\Z^m$-actions:
\[
\hga_{(k_1,\dots,k_m)}\sim \th_1^{k_1}\oti\cdots\oti\th_m^{k_m},
\quad
(k_1,\dots,k_m)\in\Z^m,
\]
where $\th_j$ is an aperiodic automorphism on the injective factor
$\cR_{0,1}$
such that $\ta_j\circ\th_j=\la_j^{-1}\ta_j$
for a faithful semifinite normal trace $\ta_j$ on $\cR_{0,1}$.
By Takesaki duality,
$\widehat{\hga}\sim \ga\oti\id_{B(\ell^2)}$.
Thus we have the following conjugacy:
\[
\ga_{(z_1,\dots,z_m)}\oti\id_{B(\ell^2)}
\approx
\widehat{\th_1}_{z_1}
\oti\cdots
\oti\widehat{\th_m}_{z_m},
\quad
(z_1,\dots,z_m)\in\T^m.
\]

First we consider $j$
such that $\la_j\neq1$.
Let $\widehat{\ta_j}$ be the dual weight
of $\ta_j$
on $\cR_{0,1}\rti_{\th_j}\Z$.
Then we have
$(\widehat{\th_j})_{\la_j^{it}}
=
\si_t^{\widehat{\ta_j}}$,
which is cocycle conjugate
to
$\si_t^{\ph_{\la_j}}$
because
$\cR_{0,1}\rti_{\th_j}\Z$
is isomorphic to $\cM_{\la_j}$.
Thus
$(\widehat{\th_j})_{e^{it}}
\sim
\si_{t/\log\la_j}^{\ph_{\la_j}}
=\ga_{e^{it}}^{\la_j}$.

Next we suppose that $\la_j=1$.
In this case,
$\th_j$ comes from an aperiodic
automorphism on $\cR_0$
that is unique up to cocycle conjugacy.
Thus $\widehat{\th_j}\sim \be\oti\id_{B(\ell^2))}$,
where $\be$ is a minimal action of
$\T$ on $\cR_0$.
By uniqueness of $\be$,
$\be_{z}$ is conjugate to $\ga_z^{\la_j}$.
Hence
$(\widehat{\th_j})_{z}
\sim \ga_z^{\la_j}\oti\id_{B(\ell^2)}$.
Therefore,
\begin{equation}
\label{eq:al-decomp}
\ga_{(z_1,\dots,z_m)}\oti\id_{B(\ell^2)}
\sim
\ga_{z_1}^{\la_1}\oti\cdots\oti\ga_{z_m}^{\la_m}
\oti\id_{B(\ell^2)}.
\end{equation}

We will remove $\id_{B(\ell^2)}$ as follows.
Recall that
$\Ga(\si^{\vph})
=\Sp(\si^{\vph})
=\langle \log\la_j\mid j\rangle$.
Hence $\cM$ is infinite,
then $\la_j\neq1$ for some $j$.
By Remark \ref{rem:alid} (to locally compact abelian groups),
$\ga\sim\ga\oti\id_{B(\ell^2)}$
and
$\ga^{\la_j}\sim
\ga^{\la_j}\oti\id_{B(\ell^2)}$.
Thus we are done.

Let us consider the case
that $\cM$ is finite,
that is,
$\la_j=1$ for all $j$.
Then
$\ga_{(z_1,\dots,z_m)}$
and $\ga_{z_1}^{\la_1}\oti\cdots\oti\ga_{z_m}^{\la_m}$
are minimal actions of $\T^m$
on $\cR_0$,
and they are conjugate.
\end{proof}

Let ${\bm\la},{\bm\mu}\in\R^m$ as before.
Regarding
$\la_j,\mu_j\in\R^1$,
we obtain the the flow $\al^{\la_j,\mu_j}$
as follows:
\[
\cM_{\la_j}
:=
\bigotimes_{n=1}^\infty
(M_2(\C),\ph_j)'',
\quad
\ph_{\la_j}
:=
\frac{1}{1+\la_j}
\Tr
\cdot
\begin{pmatrix}
1&0\\
0&\la_j
\end{pmatrix}
,
\]
\[
\al_t^{\la_j,\mu_j}
=
\bigotimes_{n=1}^\infty
\Ad
\begin{pmatrix}
1&0\\
0&e^{i\mu_j t}
\end{pmatrix}
.
\]

Let us consider
the gauge action $\ga$
of $\T^{m}$ on $\cM_{\bm\la}$:
\begin{equation}
\label{eq:gauge}
\ga_z
:=
\bigotimes_{k=1}^\infty
\Ad
\begin{pmatrix}
1&0&\cdots&0\\
0&z_1&\cdots&0\\
\vdots&\vdots&\ddots&\vdots\\
0&0&0&z_{m}
\end{pmatrix}
,
\quad
(z_1,\dots,z_{m})\in \T^{m}.
\end{equation}
Then
\[
\al_t^{{\bm\la},{\bm\mu}}
=\ga_{(e^{i\mu_1 t},\dots,e^{i\mu_{m}t})},
\quad
\si_t^{\vph_{\bm\la}}
=
\ga_{(e^{i\log\la_1 t},\dots,e^{i\log\la_{m}t})}.
\]
Employing Lemma \ref{lem:torus-minimal},
we obtain the following result.

\begin{thm}
\label{thm:allamu-prod}
Let ${\bm\la},{\bm\mu}$ be as before.
Then
$\al^{{\bm\la},{\bm\mu}}
\sim \al^{\la_1,\mu_1}\oti\cdots \oti\al^{\la_m,\mu_m}$.
\end{thm}

This implies the following result.

\begin{cor}
Let ${\bm\la},{\bm\mu}\in\R^m$
and ${\bm\rho},{\bm\nu}\in\R^n$
with $\la_j>0$ and $\rho_k>0$
for all $j=1,\dots,m$ and $k=1,\dots,n$.
Then we have
$\al^{{\bm\la}\oplus{\bm\rho},{\bm\mu}\oplus{\bm\nu}}
\sim
\al^{{\bm\la},{\bm\mu}}
\oti\al^{{\bm\rho},{\bm\nu}}$.
\end{cor}

\begin{rem}
If we put $n=1$, $\rho=1$ and $\nu=0$
in the above corollary,
we have
$\al^{{\bm\la}\oplus{\bm\rho},{\bm\mu}\oplus{\bm\nu}}
\sim
\al^{{\bm\la},{\bm\mu}}\oti\id_{\cR_0}$.
Since
$G_{{\bm\la}\oplus{\bm\rho},{\bm\mu}\oplus{\bm\nu}}
=
G_{{\bm\la},{\bm\mu}}$,
$\al^{{\bm\la}\oplus{\bm\rho},{\bm\mu}\oplus{\bm\nu}}
\sim
\al^{{\bm\la},{\bm\mu}}$
by Theorem \ref{thm:itp-class}.
Thus we have
$\al^{{\bm\la},{\bm\mu}}
\sim
\al^{{\bm\la},{\bm\mu}}\oti\id_{\cR_0}$
\end{rem}

As an application of
Theorem \ref{thm:itp-Rohlin} and \ref{thm:allamu-prod},
we will give an example of Rohlin flows on
the injective type III$_1$ factor.
Let $\cP$ be the injective type III$_1$ factor
and $\vph_1,\vph_2\in W(\cP)$.
Let $\cM:=\cP\oti\cP$.
We study when
the flow $\si_{\mu t}^{\vph_1}\oti\si_{\nu t}^{\vph_2}$
on $\cM$
has the Rohlin property
for given $\mu$ and $\nu$.

Thanks to the Connes cocycle and the uniqueness
of an injective type III$_1$ factor,
we may and do assume that
the both $\vph_1$ and $\vph_2$ are equal to
the following product state $\chi=\ph_\la\oti\ph_\rho$,
where $\la,\rho$ satisfy
$0<\la,\rho<1$,
$\log\la/\log\rho\nin\Q$
and
\[
\ph_\la:=
\bigotimes_{n=1}^\infty
\frac{1}{1+\la}
\Tr\cdot
\begin{pmatrix}
1&0\\
0&\la
\end{pmatrix}
,
\quad
\ph_\rho:=
\bigotimes_{n=1}^\infty
\frac{1}{1+\rho}
\Tr\cdot
\begin{pmatrix}
1&0\\
0&\rho
\end{pmatrix}
.
\]
Then trivially we have
\[
\si_{\mu t}^{\vph_1}\oti\si_{\nu t}^{\vph_2}
\sim
\si_{\mu t}^{\ph_\la}\oti\si_{\nu t}^{\ph_\la}
\oti
\si_{\mu t}^{\ph_\rho}\oti\si_{\nu t}^{\ph_\rho}.
\]
Thus
letting
\[
{\bm\la}=(\la,\la,\rho,\rho),
\quad
{\bm\mu}=(\mu\log\la,\nu\log\la,\mu\log\rho,\nu\log\rho),
\]
we have
$\si_{\mu t}^{\vph_1}\oti\si_{\nu t}^{\vph_2}
\sim\al_t^{{\bm\la},{\bm\mu}}$.
Then $G_{{\bm\la},{\bm\mu}}$
is the closure of
\[
(\Z\log\la+\Z\log\rho)(1,\mu)+(\Z\log\la+\Z\log\rho)(1,\nu).
\]
Since $\log\la/\log\rho$ is irrational,
$G_{{\bm\la},{\bm\mu}}$ is the closure of
$\R(1,\mu)+\R(1,\nu)$.
Thus if $\mu\neq\nu$,
then $G_{{\bm\la},{\bm\mu}}=\R^2$,
which is also equivalent to
say
$\{0\}\times\R\subs G_{{\bm\la},{\bm\mu}}$.

\begin{prop}
\label{prop:modular-tensor}
The $\si_{\mu t}^{\vph_1}\oti\si_{\nu t}^{\vph_2}$
has the Rohlin property
if and only if
$\mu\neq\nu$.
In this case,
$\si_{\mu t}^{\vph_1}\oti\si_{\nu t}^{\vph_2}\sim
\al_t^0\oti\id_{\cR_\infty}$,
where $\al^0$ is a (unique) Rohlin flow on
$\cR_0$.
\end{prop}

Therefore,
for any $\vph\in W(\cR_\infty)$,
$\si_{\mu t}^\vph\oti\id_{\cR_\infty}$
is a Rohlin flow unless $\mu=0$
though this sounds a little strange
since the modular
flow is centrally trivial.

\subsection{Quasi-free flows on Cuntz algebras}
We recall basic facts on a Cuntz algebra
and a quasi-free flow.
Our standard references
are \cite{Cu,E-On,Iz-WatI}.

Let $2\leq n<\infty$ and $\cO_n$
the Cuntz algebra generated by isometries
$s_1,\dots,s_n$ satisfying
$\sum_j s_js_j^*=1$.
Then
$\cO_n$ admits the canonical action of
the unitary group $U(n)$,
that is,
each unitary $u=(u_{ij})_{i,j}\in U(n)$
maps the generator $s_k$
to $\sum_{j}u_{jk}s_j$.
We embed $\T^n$ into $U(n)$ diagonally.
Denote by $\ga$ the action
of $\T^n$ on $\cO_n$,
that is,
\[
\ga_{(z_1,\dots,z_n)}(s_j)=z_j s_j,
\quad
j=1,\dots,n.
\]
We regard $\T$ as a closed subgroup
of $\T^n$ via the map
$z\mapsto (z,\dots,z)$.
Denote by $\cO_{U(n)}$ and $\cF^{\,n}$
the fixed point algebras
by $U(n)$ and $\T$,
respectively.
Let us denote by $A_n$ the fixed point algebra
$(\cF^{\,n})^\ga$.
Then it is trivial that
\[
\cO_{U(n)}\subs A_n=\cO_n^\ga\subs \cF^{\,n}\subs\cO_n.
\]
It is known that
$\cF^{\,n}$ is canonically isomorphic to
$\bigotimes_\N M_n(\C)$.
Let us
denote by $F$
the conditional expectation
from $\cO_n$ onto
$\cF^{\,n}$
given by averaging the action of $\T$.
For ${\bm \mu}:=(\mu_1,\dots,\mu_n)\in\R^n$,
we introduce the quasi-free flow $\al^{\bm \mu}$
as follows:
\[
\al_t^{\bm \mu}(s_j)=e^{i\mu_j t}s_j,
\quad
j=1,\dots,n.
\]
Then we have
$\cO_{U(n)}\subs \cO_n^\ga\subs\cO_n^{\al^{\bm \mu}}$.
Put $\btr:=(1,\dots,1)$.
Then $\al^\btr$ is nothing but
the action of $\T$ as stated above.
Thus $F(x)=\int_0^{2\pi}\al_t^\btr(x)\,dt$
for $x\in\cO_n$.
Note that
the restriction of $\ga_{(1,z_1,\dots,z_{n-1})}$
on $\cF^{\,n}$ for $z_j\in\T$
is of the form
defined in (\ref{eq:gauge}).

By \cite[Proposition 2.2]{E-On}
(and also \cite[Theorem 2]{Ol-Ped}
in the case of $\mu_j=1$),
$\al^{\bm \mu}$ has a KMS state
if and only if
$\mu_i\mu_j>0$ for all $i,j$.
In fact,
a KMS state $\vph^{\bm \mu}$
and an inverse temperature $\be\in\R$
are unique,
and given by
\begin{equation}
\label{eq:inversetemp}
\vph^{\bm \mu}=\ps^{\bm \mu}\circ F,
\quad
\sum_{j=1}^n e^{-\be\mu_j}=1,
\end{equation}
where
\[
\ps^{\bm \mu}
:=
\bigotimes_{k=1}^\infty
\Tr
\cdot
\begin{pmatrix}
e^{-\be\mu_1}&0&\cdots&0\\
0&e^{-\be\mu_2}&\cdots&0\\
\vdots&\vdots&\ddots&\vdots\\
0&0&\cdots&e^{-\be\mu_n}
\end{pmatrix}
.
\]

Let $\pi_\mu\col\cO_n\ra B(H_\mu)$
be the GNS representation
with respect to the KMS state
$\vph^{\bm \mu}$.
We simply write $\cM$ and $\cN$
for $\pi_\mu(\cO_n)''$
and $\pi_\mu(A_n)''$,
respectively.

The modular automorphism of $\vph^{\bm \mu}$
is given by
$\si_t^{\vph^{\bm \mu}}=\al_{-\be t}^{\bm \mu}$.
Hence
we have
$\pi_{\bm\mu}(\cO_{U(n)})''\subs \cN
\subs \cM_{\vph^{\bm\mu}}$.
Thanks to
\cite[Proposition 4.5]{Iz-WatI},
we have
$\pi_{\bm\mu}(\cO_{U(n)})'\cap\cM=\C$.
In particular,
$\cM$ is a type III factor,
which is of type III$_\la$
($0<\la<1$)
if $\mu_i/\mu_j\in\Q$ for all $i,j$,
and 
of type III$_1$ otherwise
\cite[Theorem 4.7]{Iz-WatI}.

Since the $\T^n$-action
$\ga$ preserves $\vph^{\bm\mu}$,
it extends to $\cM$.
Thus so does $\al^{\bm\om}$
for any ${\bm \om}\in\R^n$.
Note $\cN=\cM^\ga$
and the following formulae:
\[
\al_t^{\bm\om}
=
\ga_{(e^{i\om_1 t},\dots,e^{i\om_n t})},
\quad
\si_t^{\vph^{\bm \mu}}
=
\ga_{(e^{-i\be\mu_1 t},\dots,e^{-i\be\mu_n t})}.
\]

Since $\pi_{\bm\mu}(\cO_{U(n)})'\cap\cM=\C$,
$\ga$ is a minimal action of $\T^n$
on an injective factor $\cM$.
Applying Lemma \ref{lem:torus-minimal}
to $\vph:=\vph_{\bm\mu}$,
we have
\[
\ga_{(z_1,\dots,z_n)}
\sim \ga_{z_1}^{\la_1}\oti\cdots\oti\ga_{z_n}^{\la_n},
\]
where $\la_j:=e^{-\be\mu_j}<1$.
Recall $\al^{{\bm\la},{\bm\om}}$
defined in the previous subsection.
Then by Theorem \ref{thm:allamu-prod},
we obtain
\[
\al_t^{\bm\om}
=
\ga_{(e^{i\om_1 t},\dots,e^{i\om_n t})}
\sim
\ga_{e^{i\om_1 t}}^{\la_1}\oti\cdots\oti\ga_{e^{i\om_n t}}^{\la_n}
=
\al_t^{\la_1,\om_1}\oti\cdots\oti\al_t^{\la_n,\om_n}
\sim
\al_t^{{\bm\la},{\bm\om}}.
\]
Lemma \ref{lem:invappinn-stab} implies
that $\al^{\bm\om}$ is invariantly approximately inner.

We also get the following results
by Corollary \ref{cor:product-flowwt},
Theorem \ref{thm:itp-class}
and Theorem \ref{thm:itp-Rohlin}
putting
\[
H_{{\bm\mu},{\bm\om}}
:=
\langle(-\be\mu_j,\om_j)\mid j=1,\dots,n\rangle.
\]
Note that $\be$ depends on ${\bm \mu}$
as (\ref{eq:inversetemp}).

\begin{thm}
\label{thm:QF-class}
Let ${\bm\mu}\in\R^m$ and ${\bm\nu}\in\R^n$
with $\mu_i\mu_j>0$
for all $i,j$
and
$\nu_k\nu_{\ell}>0$
for all $k,\ell$.
Let
${\bm\om}\in\R^m$ and ${\bm\eta}\in\R^n$.
Then
the flows
$\al^{\bm\om}$ on $\pi_{\bm\mu}(\cO_m)''$
and
$\al^{\bm\eta}$ on $\pi_{\bm\nu}(\cO_n)''$
are cocycle conjugate
if and only if
$H_{{\bm\mu},{\bm\om}}
=
H_{{\bm\nu},{\bm\eta}}$.
\end{thm}

\begin{thm}
\label{thm:QF-Rohlin}
Let $\al^{\bm\om}$ be
the quasi-free flow
on $\cM:=\pi_{\bm\mu}(\cO_n)''$
as before.
Then the following conditions are equivalent:
\begin{enumerate}
\item
$\al^{\bm\om}$ has the Rohlin property;

\item
$\tcM'\cap(\tcM\rti_{\widetilde{\al^{\bm\om}}}\R)=Z(\tcM)$;

\item
$\Ga(\al^{\bm\om})=\R$
and
$\al_t^{\bm\om}\nin\Cnt(\cM)$;

\item
$\{0\}\times\R\subs H_{{\bm\mu},{\bm\om}}$.
\end{enumerate}
Moreover,
$\cM\rti_{\al^{\bm\om}}\R$ is an injective type III factor
of the same type as $\cM$.
\end{thm}

When $n=2$,
$H_{{\bm\mu},{\bm\om}}$
never fulfills the fourth condition above.
Thus any quasi-free flow
on $\pi_{\bm\mu}(\cO_2)''$
does not have the Rohlin property.

\begin{exam}
Put $n=3$
and
$\mu_1=\mu_2=\mu_3=1$.
Then the inverse temperature
$\be$ equals $\log 3$.
Put $\om_1=1$, $\om_2=2$ and $\om_3=\sqrt{2}$.
The group
$G_{{\bm\mu},{\bm\om}}$
contains
$(0,1)=(-\be\mu_2,\om_2)-(-\be\mu_1,\om_1)$
and
$(0,\sqrt{2})=(-\be\mu_3,\om_3)-(-\be\mu_1,\om_1)$.
Thus
$\al^{\bm\om}$ on $\pi_{\bm\mu}(\cO_3)''$
has the Rohlin property.
\end{exam}

So far,
we have obtained the classification
of $\al^{\bm\om}$
by using product type flows.
Let us prove the invariant approximate innerness
of $\al^{\bm\om}$ in another way
which is motivated
by Kishimoto's results \cite{Kishi-O2,Kishi-One}.
In those works,
he has shown,
in the C$^*$-algebra level,
that
$\al^{\bm \om}$
has the Rohlin property
as a flow on $\cO_n$
if and only if
the semigroup generated
by $\om_j$, $j=1,\dots,n$,
is dense in $\R$.
In his proof,
a sequence of endomorphisms
$\ph_k\col\cO_n\ra\cO_n$
plays a crucial role.
Those are based on the 1-cocycle property
of the shift automorphism on $\bigotimes_\Z M_n(\C)$,
which is proved by deeply understanding
its gauge invariant C$^*$-subalgebra.

However,
it turns out not so involved
to prove that the 1-cocycle property
in the von Neumann algebra level.
For this,
we should understand
how the shift endomorphism $\si$
on
$\cF^{\,n}=\bigotimes_\N M_n(\C)$
acts
on $\cN$,
and furthermore,
the ultraproduct von Neumann algebras
$\cN^\om$ and $\cN_\om$
(one should not confuse the vector ${\bm\om}$
with a free ultrafilter $\om$).

By definition of $\si$,
we have $\ps^{\bm\mu}\circ\si=\ps^{\bm\mu}$
and $\si(\ga_z(x))=\ga_z(\si(x))$
for $z\in \T^n$
and $x\in\cF^{n}$.
Thus
$\si$ extends to $\cR:=\pi_{\bm\mu}(\cF^{\,n})''$,
and $\si(\cN)\subs\cN\subs\cR$.

Let us denote by $\cR'\cap\cN^\om$
the subalgebra of $\cN^\om$
which consists of
$\pi_\om((x^\nu)_\nu)\in\cN^\om$
such that
$\|yx^\nu-x^\nu y\|_{\ps^{\bm\mu}}^\sharp\to0$
as $\nu\to\infty$
for all $y\in\cR$.
Since there exists a faithful normal
conditional expectation from $\cR$
onto $\cN$ by averaging $\ga$,
we obtain
$\cN^\om\subs\cR^\om$,
and
$\cR'\cap\cN^\om$ is a von Neumann subalgebra
of $\cN^\om$.
Moreover,
$\ps^{\bm\mu}$ is a trace
on $\cN$,
and $\cR'\cap\cN^\om\subs\cN_\om$.

\begin{lem}
The $\cR'\cap\cN^\om$ is a type II$_1$ factor.
\end{lem}
\begin{proof}
Let $x$ be a non-zero element in $\cR'\cap\cN^\om$
with $\ta_\om(x)=0$.
We will show that $x$ is not central.
We can take a representing sequence
$(x^\nu)_\nu$ of $x$
and
$W_m\in\om$ for $m\in\N$
such that
$[m,\infty)\supset W_m\supsetneq W_{m+1}$,
$x^\nu\in \bigotimes_{k=m}^\infty M_n(\C)''$
for $\nu\in W_m$
and
$\ps^{\bm\mu}(x^\nu)=0$.

Since $\ps^{\bm\mu}$ is the tracial state
on the type II$_1$ factor $\cN$,
there exists a sequence of unitaries
$(u^\nu)_\nu$ in $\cN$
such that
$u^\nu\in \bigotimes_{k=m}^\infty M_n(\C)''$
and
$\|[u^\nu,x^\nu]\|_2\geq\|x^\nu\|_2/2$.
Putting $u:=\pi_\om((u^\nu)_\nu)$,
we have $u\in\cR'\cap\cN^\om$ and $ux\neq xu$.
\end{proof}

\begin{lem}
The $\si$ is an aperiodic automorphism on $\cR'\cap\cN^\om$.
\end{lem}
\begin{proof}
Let $\ph(x):=n^{-1}\sum_j s_j^* xs_j$ for $x\in\cM$.
Then $\ph\circ\si=\id_{\cM}$.
Since $\ph(\cN)\subs\cN$,
$\ph$ extends to $\cN^\om$
(see \cite[Lemma 3.2]{Ma-T}).
Using $\cR=\{s_i s_j^*\}_{i,j}''\vee \si(\cR)$,
we obtain $\si(\cR'\cap\cN^\om)\subs \cR'\cap\cN^\om$.
Since $a\ph(x)b=\ph(\si(a)x\si(b))$
for all $x,a,b\in \cM$,
we have $\ph(\cR'\cap\cN^\om)\subs\cR'\cap\cN^\om$.
Take any $x=\pi_\om((x^\nu)_\nu)$
and $y=\pi_\om((y^\nu)_\nu)$
in $\cR'\cap\cN^\om$.
Then in the strong$*$ topology,
we have
\[
\ph(x^\nu)\ph(x^\nu)
=
\frac{1}{n^2}
\sum_{i,j}
s_i^* x^\nu s_is_j^*y^\nu s_j
\to
\ph(x^\nu y^\nu),
\quad
\nu\to\om.
\]
Hence $\ph(xy)=\ph(x)\ph(y)$,
that is,
$\ph$ is a faithful endomorphism on $\cR'\cap\cN^\om$.
Since $\ph\circ\si=\id_{\cR'\cap\cN^\om}$,
$\ph$ is a surjection, and an automorphism.
Thus so is $\si=\ph^{-1}$.

Suppose that for some non-zero
element
$a\in\cR'\cap\cN^\om$
and $k\in\N$,
we have $ax=\si^k(x)a$
for all $x\in\cR'\cap\cN^\om$.
Let us denote by $\{e_{ij}^\nu\}_{i,j}$
a system of $n\times n$-matrix units
in the $\nu$-th matrix algebra
$M_n(\C)$ in $\cF^{\,n}$.
We let $x:=\pi_\om((e_{11}^\nu)_\nu)$.
Then $x\in\cR'\cap\cN^\om$
and
\begin{align}
\|\si^k(x)-x\|_2^2
&=\lim_{\nu\to\om}
\left(
\ps^{\bm\mu}(e_{11}^{k+\nu}+e_{11}^\nu)
-
2\ps^{\bm\mu}(e_{11}^{k+\nu}e_{11}^\nu)
\right)
\notag
\\
&=
2e^{-\be\mu_1}(1-e^{-\be\mu_1})>0.
\label{eq:sikx}
\end{align}

Since $\cR'\cap\cN^\om\subs\cR_\om$,
we have a fast reindexation map
$\Ph\col \{\si^\ell(x)\}_{\ell\in\Z}''\ra \{a\}'\cap\cR_\om$
such that
$\Ph(\si(y))=\si(\Ph(y))$
and
$\ta^\om(\Ph(y)aa^*)=\ta^\om(y)\ta^\om(aa^*)$
for all $y\in \{\si^\ell(x)\}_{\ell\in\Z}''$.
By constructing method of $\Ph$,
it turns out that $\Ph(x)\in \cR'\cap \cN^\om$.
On the one hand,
we have
\begin{align*}
\ta^\om
\left(|\Ph(x)-\si^k(\Ph(x))|^2 aa^*
\right)
&=
\ta^\om
\left((\Ph(x^*)-\si^k(\Ph(x^*)))
\cdot
(\Ph(x)-\si^k(\Ph(x)))aa^*
\right)
\\
&=
\ta^\om
\left((\Ph(x^*)-\si^k(\Ph(x^*)))
\cdot
(a\Ph(x)-\si^k(\Ph(x))a)a^*
\right)
\\
&=0.
\end{align*}
On the other hand,
\begin{align*}
\ta^\om
\left(|\Ph(x)-\si^k(\Ph(x))|^2 aa^*\right)
&=
\ta^\om\left(\Ph(|x-\si^k(x)|^2) aa^*\right)
\\
&=
\|x-\si^k(x)\|_2^2\ta^\om(aa^*)
\\
&>0
\quad
\mbox{by }
(\ref{eq:sikx}).
\end{align*}
This is a contradiction.
\end{proof}

Thanks to \cite[Theorem 1.2.5]{Co-outer},
we can prove the following 1-cohomology
almost vanishing
by Shapiro's lemma
(see Lemma \ref{lem:Shapiro}).

\begin{lem}
For any unitary $v\in \cN$ and $\vep>0$,
there exists a unitary $\cN$
such that
$\|v-w\si(w^*)\|_2<\vep$,
where $\|\cdot\|_2=\|\cdot\|_{\ps^{\bm\mu}}$.
\end{lem}

This is the von Neumann algebra version
of \cite[Theorem 1.1 (1)]{Kishi-One}.
We apply this not to the permutation unitary
$\sum_{i,j}s_is_js_i^*s_j^*$
but to
$u_t:=\sum_{j}e^{i\om_j t}s_js_j^*\in\cN$.

\begin{lem}
The flow $\al^{\bm\om}$ is invariantly approximately inner
on $\cM$.
\end{lem}
\begin{proof}
Fix $T\in\R$.
By the previous lemma,
we obtain a sequence of unitaries
$(w_k)_k$ in $\cN$
such that
$\|u_T-w_k\si(w_k)^*\|_2\to0$
as $k\to\infty$.
Then we have
\[
\al_T^{\bm\om}(s_j)=u_Ts_j=\lim_{k\to\infty}w_ks_jw_k^*
\]
in the strong$*$ topology.
Thus for all $x\in\cO_n$,
we have $\al_T^{\bm\om}(x)=\lim_{k\to\infty}w_kxw_k^*$.
Let $a\in\cO_n$.
Since $w_k\in \cN\subs\cM_{\vph^{\bm\mu}}$,
$w_k(\vph^{\bm\mu} a)w_k^*=\vph^{\bm\mu}w_kaw_k^*$,
which converges to
$\vph^{\bm\mu}\al_T^{\bm\om}(a)=\al_T^{\bm\om}(\vph^{\bm\mu}a)$
in the norm topology of $\cM_*$.
Hence $\lim_{k\to\infty}\Ad w_k=\al_T^{\bm\om}$
in $\Aut(\cM)$.
Then the statement is clear
because $w_k\in\cN\subs \cM^{\al^{\bm\om}}$.
\end{proof}

With this fact,
we can proceed to compute the flow of weights
of $\cM\rti_{\al^{\bm\om}}\R$
as shown in the previous subsection,
and obtain Theorem \ref{thm:QF-class} and \ref{thm:QF-Rohlin}.

\begin{rem}
The pointwise approximate innerness of $\al^{\bm\om}$
is easily verified.
Indeed,
if $\cM$ is of type III$_1$,
then it is trivial because $\Aut(\cM)=\oInt(\cM)$.
When $\cM$ is of type III$_\la$ with $0<\la<1$,
then $\vph^{\bm\mu}$ is the periodic state
which is invariant under $\al^{\bm\om}$.
Thus we have $\mo(\al_t^{\bm\om})=\id$
for all $t\in\R$.
\end{rem}

\begin{rem}
If we apply the 1-cohomology
almost vanishing to
the permutation
$\sum_{i,j}s_is_js_i^*s_j^*$,
then by the same argument as
that of \cite{Kishi-One},
we can show that
there exist endomorphisms
$\{\rho_k\}_{k\in\N}$ on $\cM$
with the following conditions:
\begin{enumerate}
\item 
There exists a unitary
$u_k\in \cN$
such that
$\rho_k(s_j)=u_k s_j$
for all $j$;

\item
$\rho_k\circ\ga_z=\ga_z\circ\rho_k$
for all $z\in\T^n$;

\item
$\vph^{\bm \mu}\circ\rho_k=\vph^{\bm \mu}$;

\item
$\lim_{k\to\infty}[\rho_k(x),y]=0$
in the strong$*$ topology
for all $x,y\in\cM$.
\end{enumerate}

On the last condition,
$(\rho_k(x))_k$ is central in $\cM$
if and only if
$x\in \cM_\vph$.
Indeed,
since $\rho_k(\cM)\subs\cM$ and
$\vph^{\bm\mu}\circ\rho_k=\vph^{\bm\mu}$,
we have
$\|[\rho_k(x),\vph^{\bm\mu}]\|
\geq\|[x,\vph^{\bm\mu}]\|$
for all $k$.
\end{rem}

\section{A characterization of Rohlin property}
The main purpose of this section
is to show
Theorem \ref{thm:faithful-Rohlin}
which states
that
a flow $\al$ on a factor $\cM$
has the Rohlin property
if and only if
$\al$ faithfully acts on $\cM_{\om,\al}$.

We begin with the following lemma.
\begin{lem}
\label{lem:pmuval}
Let $p\in\Sp(\al|_{\cM_{\om,\al}})$,
and $\mu_1,\dots,\mu_n$ be
finite Borel measures on $\R$.
Then for any $\vep>0$,
there exists a non-zero partial isometry $v\in \cM_{\om,\al}$
such that
\[
\int_\R\|\al_t(v)-e^{ipt}v\|_2\,d\mu_j(t)
<\vep\|v\|_2\sum_{j=1}^m \mu_j(\R).
\]
\end{lem}
\begin{proof}
Fix a small positive number $\eta$
such that $\eta<\min(\vep^2,1/400)$.
Let $\mu:=\sum_{j=1}^n\mu_j$ and $\nu:=\mu(\R)^{-1}\mu$.
Note that $\nu$ is a regular Borel measure with $\nu(\R)=1$.
Take $R>0$ so that $\nu(\R\setminus [-R,R])\leq\eta/4$.

Since $\al$ preserves the tracial state $\ta_\om$ on $\cM_{\om,\al}$,
there exists a non-zero element $x\in\cM_{\om,\al}$
such that
\[
\|\al_t(x)-e^{ipt}x\|_2 \leq\eta\|x\|_2/2
\quad\mbox{for all }
t\in[-R,R].
\]
Then
\begin{align*}
\int_\R\|\al_t(x)-e^{ipt}x\|_2\,d\nu(t)
&=
\int_{-R}^R\|\al_t(x)-e^{ipt}x\|_2\,d\nu(t)
\\
&\quad+
\int_{\R\setminus[-R,R]}\|\al_t(x)-e^{ipt}x\|_2\,d\nu(t)
\\
&\leq
\eta\|x\|_2/2\cdot \nu([-R,R])
+
2\|x\|_2\cdot \nu(\R\setminus[-R,R])
\\
&\leq
\eta\|x\|_2.
\end{align*}
Since $\cM_{\om,\al}$ is finite,
we can take a unitary $w\in \cM_{\om,\al}$
with $x=w|x|$.
Put $y_t:=e^{-ipt}w^*\al_t(x)$.
Then we have
\[
\|y_t^*-|x|\|_2
=
\|y_t-|x|\|_2
=
\|\al_t(x)-e^{ipt}x\|_2.
\]
By the Powers-St\o rmer inequality,
we have
\[
\|\al_t(|x|)-|x|\|_2^2
\leq2\|y_t-|x|\|_2\|x\|_2.
\]
Thus
\begin{align*}
\int_\R
\|\al_t(|x|)-|x|\|_2^2
\,d\nu(t)
&\leq
\int_\R
2\|y_t-|x|\|_2\|x\|_2
\,d\nu(t)
\\
&\leq
2\eta\|x\|_2^2.
\end{align*}
Next we have
\begin{align*}
\int_\R
\|y_t-|y_t|\|_2^2
\,d\nu(t)
&\leq
\int_\R
2\|y_t-|x|\|_2^2
\,d\nu(t)
+
\int_\R
2\||x|-\al_t(|x|)\|_2^2
\,d\nu(t)
\\
&\leq
4\|x\|_2\cdot\eta\|x\|_2
+
2\cdot2\eta\|x\|_2^2
\\
&=
8\eta\|x\|_2^2.
\end{align*}
Using Fubini's theorem
and \cite[Lemma 1.2.5, 1.2.6]{Co-inj},
we have
\[
\int_{\R_+^*}
\left(
\int_\R\|u_{\sqrt{a}}(y_t)-u_{\sqrt{a}}(|y_t|)\|_2^2\,d\nu(t)
\right)
\,da
=
\int_\R\|y_t-|y_t|\|_2^2\,d\nu(t)
\leq
8\eta\|x\|_2^2,
\]
and
\begin{align*}
\int_{\R_+^*}
\left(
\int_\R
\|u_{\sqrt{a}}(|y_t|)-u_{\sqrt{a}}(|x|)\|_2^2\,d\nu(t)
\right)
\,da
&\leq
\int_\R
\|y_t-|x|\|_2\|y_t+|x|\|_2
\,d\nu(t)
\\
&\leq
2\eta\|x\|_2^2.
\end{align*}
Thus we obtain
\[
\int_{\R_+^*}
\left(
\int_\R\|u_{\sqrt{a}}(y_t)-u_{\sqrt{a}}(|x|)\|_2^2\,d\nu(t)
\right)
\,da
\leq
2(8\eta+2\eta)\|x\|_2^2
=20\eta\|x\|_2^2.
\]

Now let $G(a)=\|u_{\sqrt{a}}(|x|)\|_2^2$.
Then $\int_{\R_+^*}G(a)\,da=\|x\|_2^2$.
We let
\[
A:=\left\{
b>0\,\middle|\,
\int_\R
\|u_{\sqrt{b}}(y_t)-u_{\sqrt{b}}(|x|)\|_2^2
\,d\nu(t)
>\eta^{1/2}\|u_{\sqrt{b}}(|x|)\|_2^2
\right\}.
\]
Then
\[
\int_{A}G(a)\,da
<\eta^{-1/2}\cdot 20\eta\|x\|_2^2
=20\eta^{1/2}\|x\|_2^2.
\]
Since the measure $G(a)\,da$ is normal,
we can take an open set $U\subs \R_+$
such that
$0\in U$, $A\subs U$ and
$\int_{U\cap\R_+^*}G(a)\,da
<
20\eta^{1/2}\|x\|_2^2$.
Take the smallest $b>0$ such that $b\in U^c$ satisfies
$(0,b)\subs U$.
Then
\[
\int_{(0,b)}G(a)\,da
\leq
\int_U G(a)\,da
<
20\eta^{1/2}\|x\|_2^2<\|x\|_2^2.
\]
Hence $b<\infty$.
Then we have
\[
\int_\R
\|u_{\sqrt{b}}(y_t)-u_{\sqrt{b}}(|x|)\|_2^2
\,d\nu(t)
\leq\eta^{1/2}\|u_{\sqrt{b}}(|x|)\|_2^2,
\]
and
\begin{align*}
\|x-u_{\sqrt{b}}(x)|x|\|_2^2
&=
\ta(|x|^2(1-E_b(|x|^2)))
\leq
\int_0^b \ta(E_s(|x|^2))\,ds
\\
&=
\int_0^b G(s)\,ds
<
20\eta^{1/2}\|x\|_2^2
\\
&<\|x\|_2^2.
\end{align*}
Hence $v:=u_{\sqrt{b}}(x)$
is a non-zero partial isometry.
Trivially, $u_{\sqrt{b}}(y_t)=\al_t(v)$,
and we are done.
\end{proof}

We consider the standard Hilbert space $H\oti L^2(\R)$
of the crossed product $\cM\rti_\al\R$.
Recall that for $x\in\cM$,
the right action of $\pi_\al(x)$
on this Hilbert space is nothing but $Jx^*J\oti1$.
Hence
\[
(\zeta\oti f)\pi_\al(x)=\zeta x\oti f
\quad\mbox{for all }
\zeta\in H,\ f\in L^2(\R).
\]
Note that the one-parameter unitary group
associated with $\hal$ is $1\oti {\bm e}_{-\cdot}$,
that is,
\[
\hal_p(\xi)=(1\oti {\bm e}_{-p})\xi
\quad\mbox{for all }
p\in\R,\ \xi\in H\oti L^2(\R).
\]

\begin{lem}
\label{lem:SLIAI}
Let $p\in \Sp(\al|_{\cM_{\om,\al}})$
and $f\in\cM^{\rm P}$.
For any $\vep>0$
and $\xi_1,\dots,\xi_n\in \pi_\al(f)(H\oti L^2(\R))\pi_\al(f)$,
there exists a non-zero $x\in \cM_f$
such that
\[
\|\pi_\al(x)\xi_j
-\hal_p(\xi_j)\pi_\al(x)\|^2
\leq\vep\sum_{j=1}^n\|\pi_\al(x)\xi_j\|^2
\quad\mbox{for all }
j=1,\dots,n.
\]
\end{lem}
\begin{proof}
Note that for $\xi\in H\oti L^2(\R)$,
the functionals $\cM\ni x\mapsto \langle \pi_\al(x)\xi,\xi\rangle$
and $\cM\ni x\mapsto \langle (x\oti1)\xi,\xi\rangle$
are normal.
Hence for $x=\pi_\om((x^\nu)_\nu)\in\cM_\om$,
we obtain
\begin{equation}
\label{eq:pialx}
\lim_{\nu\to\om}\|\pi_\al(x^\nu)\xi\|=\|x\|_2\|\xi\|
=\lim_{\nu\to\om}\|(x^\nu\oti1)\xi\|.
\end{equation}

Let $p,\vep$ and $\xi_j$ be given as in the statement above.
Take $\de>0$ and $R>0$ with $8\de^2+42\de(1-\de)^{-2}<\vep$
and $\|\xi_j-\eta_j\|<\de\|\xi_j\|$,
where we have put $\eta_j:=(1\oti 1_{[-R,R]})\xi_j$.
By Lemma \ref{lem:pmuval},
there exists a non-zero partial isometry $v\in \cM_{\om,\al}$
such that
\[
\int_{-R}^R\|\al_t(v)-e^{ipt}v\|_2\|\eta_j(-t)\|^2\,dt
<\de\|v\|_2\sum_{j=1}^n\|\eta_j\|^2.
\]
From the inequality $\|\al_t(v)-e^{ipt}v\|_2\leq 2\|v\|_2$,
we obtain
\begin{equation}
\label{eq:alveta}
\int_{-R}^R\|\al_t(v)-e^{ipt}v\|_2^2\|\eta_j(-t)\|^2\,dt
<2\de\|v\|_2^2\sum_{j=1}^n\|\eta_j\|^2.
\end{equation}

Take a representing sequence $(v^\nu)_\nu$ of $v$
such that each $v^\nu$ is a non-zero partial isometry.
To proceed a proof, we need the following claim.

\begin{clam}
For each $j=1,\dots,n$, the following holds:
\begin{equation}
\lim_{\nu\to\om}
\|\pi_\al(v^\nu)\eta_j-(v^\nu\oti {\bm e}_{-p})\eta_j\|^2
<
21\de\|v\|_2^2\sum_{j=1}^n\|\eta_j\|^2.
\label{eq:6deltav}
\end{equation}
\end{clam}
\begin{proof}[Proof of Claim]
Let $\ka>0$ so that $8\ka^2<\de$.
We fix $j$.
Take $\zeta_1,\dots,\zeta_m\in H$ and continuous functions
$f_1,\dots,f_m$ such that
$\supp f_k\subs[-R,R]$ for all $k$,
and
$\eta_j^o:=\sum_{k=1}^m \zeta_k\oti f_k$
satisfies
$\|\eta_j-\eta_j^o\|<\ka\|\eta_j\|$.
Then
\begin{align*}
\|\pi_\al(v^\nu)\eta_j^o-(v^\nu\oti {\bm e}_{-p})\eta_j^o\|^2
&=
\int_{-R}^R
\|(\al_{-t}(v^\nu)-e^{-ipt}v^\nu)\eta_j^o(t)\|^2\,dt
\\
&=
\int_{-R}^R
\|\sum_{k=1}^m f_k(t)
(\al_{-t}(v^\nu)-e^{-ipt}v^\nu)\zeta_k\|^2\,dt.
\end{align*}

Note that $(v^\nu)_\nu$ is $(\al,\om)$-equicontinuous.
Hence the following functions
\[
\|\sum_{k=1}^m f_k(t)
(\al_{-t}(v^\nu)-e^{-ipt}v^\nu)\zeta_k\|^2
=
\sum_{k,\ell=1}^m
f_k(t)\overline{f_\ell(t)}
\langle|\al_{-t}(v^\nu)-e^{-ipt}v^\nu|^2\zeta_k,\zeta_\ell\rangle
\]
converge to
\[
\sum_{k,\ell=1}^m
f_k(t)\overline{f_\ell(t)}
\|\al_{-t}(v)-e^{-ipt}v\|_2^2
\langle\zeta_k,\zeta_\ell\rangle
=
\|\al_{-t}(v)-e^{-ipt}v\|_2^2
\|\eta_j^o(t)\|^2.
\]
uniformly on $[-R,R]$ as $\nu\to\om$.
Thus we have
\begin{align}
\lim_{\nu\to\om}
\|\pi_\al(v^\nu)\eta_j^o-(v^\nu\oti {\bm e}_{-p})\eta_j^o\|^2
&=
\int_{-R}^R
\|\al_{-t}(v)-e^{-ipt}v\|_2^2\|\eta_j^o(t)\|^2\,dt
\notag
\\
&\leq
\int_{-R}^R
2\|\al_{-t}(v)-e^{-ipt}v\|_2^2\|\eta_j^o(t)-\eta_j(t)\|^2\,dt
\notag
\\
&\quad+
\int_{-R}^R
2\|\al_{-t}(v)-e^{-ipt}v\|_2^2\|\eta_j(t)\|^2\,dt
\notag
\\
&\leq
8\|v\|_2^2\|\eta_j^o-\eta_j\|^2
+
4\de\|v\|_2^2
\sum_{j=1}^n\|\eta_j\|^2
\quad
\mbox{by }(\ref{eq:alveta})
\notag
\\
&\leq
(8\ka^2+4\de)\|v\|_2^2\sum_{j=1}^n\|\eta_j\|^2
\notag
\\
&<5\de\|v\|_2^2\sum_{j=1}^n\|\eta_j\|^2.
\label{eq:5deltav}
\end{align}
From (\ref{eq:pialx}), (\ref{eq:5deltav})
and the following inequality:
\begin{align*}
\|\pi_\al(v^\nu)\eta_j-(v^\nu\oti {\bm e}_{-p})\eta_j\|^2
&\leq
4\|\pi_\al(v^\nu)(\eta_j-\eta_j^o)\|^2
+
4\|(v^\nu\oti {\bm e}_{-p})(\eta_j-\eta_j^o)\|^2
\\
&\quad+
4\|\pi_\al(v^\nu)\eta_j^o-(v^\nu\oti {\bm e}_{-p})\eta_j^o\|^2,
\end{align*}
we obtain
\begin{align*}
\lim_{\nu\to\om}
\|\pi_\al(v^\nu)\eta_j-(v^\nu\oti {\bm e}_{-p})\eta_j\|^2
&\leq
8\|v\|_2^2\|\eta_j-\eta_j^o\|^2
+
5\de\|v\|_2^2\sum_{j=1}^n\|\eta_j\|^2
\\
&\leq
(8\ka^2+20\de)\|v\|_2^2\sum_{j=1}^n\|\eta_j\|^2
\\
&<
21\de\|v\|_2^2\sum_{j=1}^n\|\eta_j\|^2.
\end{align*}
Hence Claim follows.
\end{proof}

In the following inequality:
\begin{align*}
\|\pi_\al(v^\nu)\xi_j-\hal_p(\xi_j)\pi_\al(v^\nu)\|
&\leq
\|\pi_\al(v^\nu)(\xi_j-\eta_j)\|
+\|\hal_p(\eta_j-\xi_j)\pi_\al(v^\nu)\|
\\
&\quad+
\|\pi_\al(v^\nu)\eta_j-\hal_p(\eta_j)\pi_\al(v^\nu)\|,
\end{align*}
we let $\nu\to\om$ and then
\begin{align}
\lim_{\nu\to\om}\|\pi_\al(v^\nu)\xi_j-\hal_p(\xi_j)\pi_\al(v^\nu)\|
&\leq
2\|v\|_2\|\xi_j-\eta_j\|
+
\lim_{\nu\to\om}\|\pi_\al(v^\nu)\eta_j-\hal_p(\eta_j)\pi_\al(v^\nu)\|
\notag\\
&<
2\de\|v\|_2\|\xi_j\|
+
\lim_{\nu\to\om}\|\pi_\al(v^\nu)\eta_j
-\hal_p(\eta_j)\pi_\al(v^\nu)\|.
\label{eq:pialveta}
\end{align}
On the last term, we have
$\hal_p(\eta_j)\pi_\al(v^\nu)=(1\oti {\bm e}_{-p})\eta_j(v^\nu\oti1)$,
where $\eta_j(v^\nu\oti1)$ means that $(J (v^\nu)^*J\oti1)\eta_j$.
Thus
\begin{align*}
\|\pi_\al(v^\nu)\eta_j-\hal_p(\eta_j)\pi_\al(v^\nu)\|
&\leq
\|\pi_\al(v^\nu)\eta_j-(v^\nu\oti {\bm e}_{-p})\eta_j\|
+
\|[v^\nu\oti1,\eta_j]\|.
\end{align*}
Since $(v^\nu)_\nu$ is $\om$-central,
we have
\begin{align}
\lim_{\nu\to\om}
\|\pi_\al(v^\nu)\eta_j-\hal_p(\eta_j)\pi_\al(v^\nu)\|^2
&\leq
21\de\|v\|_2^2\sum_{j=1}^n\|\eta_j\|^2
\quad
\mbox{by }(\ref{eq:6deltav})
\notag
\\
&<
21\de(1-\de)^{-2}\|v\|_2^2\sum_{j=1}^n\|\xi_j\|^2.
\label{eq:pialdev}
\end{align}

By (\ref{eq:pialveta}) and (\ref{eq:pialdev}),
we get
\begin{align*}
\lim_{\nu\to\om}\|\pi_\al(v^\nu)\xi_j-\hal_p(\xi_j)\pi_\al(v^\nu)\|^2
&\leq
(8\de^2+42\de(1-\de)^{-2})
\|v\|_2^2\sum_{j=1}^n\|\xi_j\|^2
\\
&<
\vep\|v\|_2^2\sum_{j=1}^n\|\xi_j\|^2
\\
&=
\lim_{\nu\to\om}
\vep\sum_{j=1}^n\|\pi_\al(v^\nu)\xi_j\|^2.
\end{align*}
Since $[v^\nu,f]\to0$ as $\nu\to\om$ in the strong$*$ topology
and $\xi_j=\pi_\al(f)\xi_j\pi_\al(f)$,
we have
\[
\lim_{\nu\to\om}\|\pi_\al(f v^\nu f)\xi_j-\hal_p(\xi_j)\pi_\al(f v^\nu f)\|^2
<
\lim_{\nu\to\om}
\vep\sum_{j=1}^n\|\pi_\al(f v^\nu f)\xi_j\|^2.
\]
Note that $fvf=fv\neq0$ in $\cM^\om$.
Hence $x:=fv^\nu f$ does the job
for a sufficiently large $\nu$.
\end{proof}

\begin{lem}\label{lem:uax}
Let $\cN$ be a von Neumann algebra
and $L^2(\cN)$ its standard Hilbert space.
Then for $x,y\in \cN$
and $\xi,\eta\in L^2(\cN)$,
one has
\begin{align*}
&\int_0^\infty
\|u_{\sqrt{a}}(x)\xi-\eta u_{\sqrt{a}}(y)\|^2\,da
\\
&\leq
4(\|x\xi-\eta y\|+\|x^* \eta-\xi y^*\|)
(\|x\xi\|+\|x^*\eta\|+\|\xi y^*\|+\|\eta y\|).
\end{align*}
\end{lem}
\begin{proof}
Let
\[
H:=
\begin{pmatrix}
0&x^*\\
x & 0
\end{pmatrix}
,
\quad
K:=
\begin{pmatrix}
0&y^*\\
y & 0
\end{pmatrix}
,\quad
\zeta
:=
\begin{pmatrix}
\xi&0\\
0&\eta
\end{pmatrix}
.
\]
Then we have
\[
u_{\sqrt{a}}(H)
=
\begin{pmatrix}
0& u_{\sqrt{a}}(x)^*\\
u_{\sqrt{a}}(x)&0
\end{pmatrix}
,
\quad
u_{\sqrt{a}}(K)
=
\begin{pmatrix}
0& u_{\sqrt{a}}(y)^*\\
u_{\sqrt{a}}(y)&0
\end{pmatrix}
,
\]
and
\[
H\zeta-\zeta K
=
\begin{pmatrix}
0& x^* \eta-\xi y^*\\
x\xi-\eta y&0
\end{pmatrix}
.
\]

\begin{clam}
\begin{align*}
\int_0^\infty
\|u_{\sqrt{a}}(H)\zeta-\zeta u_{\sqrt{a}}(K)\|^2\,da
&\leq
4\|H\zeta-\zeta K\|(\|H\zeta\|+\|\zeta K\|).
\end{align*}
\end{clam}
\begin{proof}[Proof of Claim]
Applying the proof of \cite[Proposition IX.1.22]{TaII} to
\[
H':=
\begin{pmatrix}
H&0\\
0&K
\end{pmatrix}
,\quad
\zeta'
:=
\begin{pmatrix}
0&\zeta\\
0&0
\end{pmatrix}
,
\]
we obtain
\[
\int_0^\infty
\|[u_{\sqrt{a}}(H'),\zeta']\|^2\,da
\leq
4\|[H',\zeta']\|(\|H'\zeta'\|^2+\|\zeta'H'\|^2)^{1/2}.
\]
This proves the claim.
\end{proof}

By the claim above,
we have
\begin{align*}
&\int_0^\infty
\|u_{\sqrt{a}}(x)\xi-\eta u_{\sqrt{a}}(y)\|^2
+
\|u_{\sqrt{a}}(x)^*\eta-\xi u_{\sqrt{a}}(y)^*\|^2
\,da
\\
&=
\int_0^\infty
\|u_{\sqrt{a}}(H)\zeta-\zeta u_{\sqrt{a}}(K)\|^2
\,da
\\
&\leq
4\|H\zeta-\zeta K\|(\|H\zeta\|+\|\zeta K\|)
\\
&\leq
4(\|x\xi-\eta y\|+\|x^* \eta-\xi y^*\|)
(\|x\xi\|+\|x^*\eta\|+\|\xi y^*\|+\|\eta y\|)
.
\end{align*}
\end{proof}

\begin{lem}
Let $p\in \Sp(\al|_{\cM_{\om,\al}})$
and $f\in\cM^{\rm P}$.
For any $\vep>0$
and $\xi_1,\dots,\xi_n\in \pi_\al(f)(H\oti L^2(\R))\pi_\al(f)$
with $\xi_j=\tJ\xi_j$,
there exists a non-zero partial isometry $v\in \cM$
such that
$v=fvf$ and
\[
\|\pi_\al(v)\xi_j
-\hal_p(\xi_j)\pi_\al(v)\|^2
\leq\vep\sum_{j=1}^n\|\pi_\al(v)\xi_j\|^2
\quad\mbox{for all }
j=1,\dots,n.
\]
\end{lem}
\begin{proof}
Let $\de>0$ with $32\de^{1/2}+16\de\leq\vep$.
By Lemma \ref{lem:SLIAI},
we have a non-zero $x\in \cM_f$ such that
\[
\|\pi_\al(x)\xi_j
-\hal_p(\xi_j)\pi_\al(x)\|^2
\leq\de\sum_{j=1}^n\|\pi_\al(x)\xi_j\|^2
\quad\mbox{for all }
j=1,\dots,n.
\]
Employing the previous lemma,
we obtain
\begin{align*}
&\int_0^\infty
\|\pi_\al(u_{\sqrt{a}}(x))\xi_j
-\hal_p(\xi_j)\pi_\al( u_{\sqrt{a}}(x))\|^2\,da
\\
&\leq
4(\|\pi_\al(x)\xi_j-\hal_p(\xi_j) \pi_\al(x)\|
+\|\pi_\al(x^*)\hal_p(\xi_j)-\xi_j \pi_\al(x^*)\|)
\\
&\quad\cdot
(\|\pi_\al(x)\xi_j\|+\|\pi_\al(x^*)\hal_p(\xi_j)\|
+\|\xi_j \pi_\al(x^*)\|+\|\hal_p(\xi_j)\pi_\al(x)\|)
\\
&\leq
8\|\pi_\al(x)\xi_j-\hal_p(\xi_j) \pi_\al(x)\|
\cdot
2(\|\pi_\al(x)\xi_j\|+\|\hal_p(\xi_j)\pi_\al(x)\|)
\\
&\leq
16
\|\pi_\al(x)\xi_j-\hal_p(\xi_j) \pi_\al(x)\|
\\
&\quad\cdot
(2\|\pi_\al(x)\xi_j\|+\|\pi_\al(x)\xi_j-\hal_p(\xi_j) \pi_\al(x)\|)
\\
&\leq
32\de^{1/2}\sum_{j=1}^n\|\pi_\al(x)\xi_j\|^2
+
16\de\sum_{j=1}^n\|\pi_\al(x)\xi_j\|^2
\\
&\leq
\vep \|\pi_\al(x)\xi_j\|^2
=
\vep\int_0^\infty\|\pi_\al(u_{\sqrt{a}}(x))\xi_j\|^2\,da.
\end{align*}
Thus for some $a>0$,
$v:=u_{\sqrt{a}}(x)$ does the job.
\end{proof}

We can show the following lemma
in the same way as in \cite[Lemma III.3]{Co-III1}.

\begin{lem}
Let $p\in \Sp(\al|_{\cM_{\om,\al}})$
and $f\in\cM^{\rm P}$.
For any $\vep>0$ and
$\xi_1,\dots,\xi_n\in \pi_\al(f)(H\oti L^2(\R))\pi_\al(f)$
with $\tJ\xi_j=\xi_j$,
there exist a non-zero projection $E\in \cM_f$
and a non-zero $x\in \cM_f$ such that
\begin{enumerate}
\item
$\|x\|\leq1$, $x=ExE$;

\item
$\sum_{j=1}^n \|\pi_\al(x)\xi_j\|^2
\geq
\frac{1}{12800}
\sum_{j=1}^n\|\pi_\al(E)\xi_j\|^2$;

\item
$\|[\pi_\al(E),\xi_k]\|^2\leq\vep^2\sum_{j=1}^n\|\pi_\al(E)\xi_j\|^2$,
\quad
$1\leq k\leq n$;

\item
$\|\pi_\al(x)\xi_k-\hal_p(\xi_k)\pi_\al(x)\|^2
\leq \vep^2\sum_{j=1}^n\|\pi_\al(x)\xi_j\|^2$,
\quad
$1\leq k\leq n$.
\end{enumerate}
\end{lem}

Let $\cP_\cN^\natural$ be the natural cone of
$\cN:=\cM\rti_\al\R$ in the standard Hilbert space
$H\oti L^2(\R)$.

\begin{lem}
Let  $p\in\Sp(\al|_{\cM_{\om,\al}})$
and $\xi_0$ be
a cyclic and separating unit vector for $\cM\rti_\al\R$.
Then for any $\vep>0$ and $\xi_1,\dots,\xi_n\in \cP_\cN^\natural$,
there exists a non-zero $x\in \cM$
such that
\begin{itemize}
\item
$\|x\|\leq1$;

\item
$\|\pi_\al(x)\xi_0\|^2
\geq
1/102400
$;

\item
$\|\pi_\al(x)\xi_k-\hal_p(\xi_k)\pi_\al(x)\|
\leq
\vep\sum_{j=1}^n\|\xi_j\|^2$
for all $k=1,\dots,n$.
\end{itemize}
\end{lem}
\begin{proof}
We prove this lemma in a similar way to that of 
\cite[Lemma III.4]{Co-III1}.
Put $d=12800^{-1}$.
Let $\sJ$ be the set of all
$(n+2)$-tuples $(x,E,\beta_1,\dots,\beta_n)$
in $\cM\times\cM^{\rm P}\times (H\oti L^2(\R))^n$
such that
\begin{enumerate}
\item
$\|x\|\leq1$, $x=ExE$;

\item
$\pi_\al(E)\be_j=\be_j$,
$\eta_j:=\xi_j-\be_j-\tJ\be_j\in \cP_\cN^\natural$
and $[\pi_\al(E),\eta_j]=0$,
\quad
$1\leq j\leq n$;

\item
$\|\be_k\|^2\leq \vep^2\sum_{j=1}^n\|E\xi_j\|^2$;

\item
$\sum_{j=1}^n \|\pi_\al(x)\eta_j\|^2
\geq
d/2
\sum_{j=1}^n\|\pi_\al(E)\eta_j\|^2$;

\item
$\|\pi_\al(x)\eta_k-\hal_p(\eta_k)\pi_\al(x)\|^2
\leq \vep^2\sum_{j=1}^n\|\pi_\al(x)\xi_j\|^2$,
\quad
$1\leq k\leq n$.
\end{enumerate}

We will equip $\sJ$ with a partial order as
$(x,E,\beta_1,\dots,\beta_n)\leq(x',E',\beta_1',\dots,\beta_n')$
if
\begin{itemize}
\item
$E\leq E'$;

\item
$x=x'E$;

\item
$\pi_\al(E)\be_j'=\be_j$,
\quad
$1\leq j\leq n$;

\item
$\|\be_k'-\be_k\|^2\leq\vep^2\sum_{j=1}^n\|\pi_\al(E'-E)\xi_j\|^2$,
\quad
$1\leq k\leq n$.
\end{itemize}

Take a maximal element $(x,E,\beta_1,\dots,\beta_n)$.
It suffices to show that $E=1$.
Indeed, suppose that we have proved $E=1$.
Let $\vep>0$ and $\xi_1,\dots,\xi_n\in\cP_\cN^\natural$
with $\|\xi_j\|=1$.
Put $\zeta_j=\xi_0$ for $j=1,\dots,n^2$
and $\zeta_{n^2+k}=\xi_k$ for $k=1,\dots,n$.
Then for $\vep$ and $\{\zeta_j\}_{j=1}^{n^2+n}$,
we have non-zero $x\in\cM_1$ such that
\[
\sum_{j=1}^{n^2+n} \|\pi_\al(x)\zeta_j\|^2
\geq
d/4
\sum_{j=1}^{n^2+n}\|\zeta_j\|^2
=
(n^2+n)d/4,
\]
\[
\|\pi_\al(x)\zeta_k-\hal_p(\zeta_k)\pi_\al(x)\|^2
\leq \vep^2\sum_{j=1}^n\|\pi_\al(x)\zeta_j\|^2.
\]
Readers are referred to
the proof of \cite[Lemma B.1]{Masuda-CH} for a detail.
Then we have
\[
\|\pi_\al(x)\xi_0\|^2
\geq
\left((n^2+n)d/4-n\right)/n^2
=
d/4-(1-d/4)/n.
\]
If we take a sufficiently large $n$,
then we have $\|\pi_\al(x)\xi_0\|^2\geq d/8$.

Suppose on the contrary that $f:=1-E$ is not equal to 0.
Then $\pi_\al(f)\beta_j=0=(J\beta_j)\pi_\al(f)$.
Hence $\pi_\al(f)\eta_j \pi_\al(f)=\pi_\al(f)\xi_j \pi_\al(f)$.
Take a non-zero projection $F\in \cM_f$ and a non-zero $y\in \cM_F$
which satisfy the conditions in the previous lemma
for $\pi_\al(f)\xi_1 \pi_\al(f),\dots,\pi_\al(f)\xi_n \pi_\al(f)$.

We set
\[
x':=x+y,
\quad
E':=E+F,
\quad
\be_j':=\be_j+\pi_\al(F)\eta_j\pi_\al(f-F).
\]
We will check that $(x',E',\beta_1',\dots,\beta_n')$ belongs to $\sJ$.
The condition (1) is trivial.
On the condition (2),
we put $\eta_j':=\xi_j-\be_j'-J\be_j'$.
Then
\begin{align*}
\eta_j'
&=\eta_j-\pi_\al(F)\eta_j\pi_\al(f-F)-\pi_\al(f-F)\eta_j\pi_\al(F)
\\
&=
\pi_\al(E)\eta_j \pi_\al(E)+\pi_\al(F)\eta_j \pi_\al(F)
+\pi_\al(f-F)\eta_j\pi_\al(f-F)\\
&=
\pi_\al(F)\eta_j \pi_\al(F)+\pi_\al(1-F)\eta_j\pi_\al(1-F)
\in\cP_\cN^\natural,
\end{align*}
and
$[\pi_\al(E'),\eta_j']=0$.

Next, the condition (3) is verified as follows:
\begin{align*}
\|\be_k'\|^2
&=
\|\be_k\|^2+\|\pi_\al(F)\eta_j\pi_\al(f-F)\|^2
\\
&\leq
\vep^2\sum_{j=1}^n\|\pi_\al(E)\xi_j\|^2
+
\|\pi_\al(F)[\pi_\al(f)\eta_j\pi_\al(f),\pi_\al(F)]\|^2
\\
&\leq
\vep^2\sum_{j=1}^n\|\pi_\al(E)\xi_j\|^2
+
\vep^2\sum_{j=1}^n\|\pi_\al(F)\pi_\al(f)\xi_j\pi_\al(f)\|^2
\\
&\leq
\vep^2\sum_{j=1}^n\|\pi_\al(E')\xi_j\|^2.
\end{align*}

We will verify the condition (4).
Since $x=xE$ and $[\pi_\al(E),\eta_j]=0$, we have
\begin{align}
\pi_\al(x')\eta_j'
&=
\pi_\al(x+y)(\pi_\al(F)\eta_j \pi_\al(F)+\pi_\al(1-F)\eta_j\pi_\al(1-F))
\notag\\
&=\pi_\al(x)\eta_j\pi_\al(1-F)
+
\pi_\al(y)\eta_j\pi_\al(F)
\notag\\
&=
\pi_\al(x)\eta_j
+
\pi_\al(y)\eta_j\pi_\al(F),
\label{eq:xeta}
\end{align}
and
\begin{align}
\hal_p(\eta_j')\pi_\al(x')
&=
\pi_\al(1-F)\hal_p(\eta_j)\pi_\al(x)+\pi_\al(F)\hal_p(\eta_j)\pi_\al(y)
\notag\\
&=
\hal_p(\eta_j)\pi_\al(x)+\pi_\al(F)\hal_p(\eta_j)\pi_\al(y).
\label{eq:halpetax}
\end{align}

Since
\begin{align*}
\|\pi_\al(y)\eta_j\pi_\al(f)\|^2
&=
\|\pi_\al(y)\eta_j\pi_\al(F)\|^2
+
\|\pi_\al(y)\eta_j\pi_\al(f-F)\|^2
\\
&=
\|\pi_\al(y)\eta_j\pi_\al(F)\|^2
+
\|\pi_\al(y)[\pi_\al(f)\eta_j\pi_\al(f),\pi_\al(F)\|^2
\\
&\leq
\|\pi_\al(y)\eta_j\pi_\al(F)\|^2
+
\vep^2\sum_{j=1}^n
 \|\pi_\al(F)\pi_\al(f)\eta_j\pi_\al(f)\|^2,
\end{align*}
we have 
\[
\|\pi_\al(y)\eta_j\pi_\al(F)\|^2
\geq
\|\pi_\al(y)\eta_j\pi_\al(f)\|^2
-
\vep^2\sum_{j=1}^n \|\pi_\al(F)\pi_\al(f)\eta_j\pi_\al(f)\|^2.
\]

Then it follows that 
\begin{align*}
\|\pi_\al(x')\eta_j'\|^2
&=
\|\pi_\al(x)\eta_j\|^2
+
\|\pi_\al(y)\eta_j\pi_\al(F)\|^2
\\
&\geq
d/2\sum_{j=1}^n\|\pi_\al(E)\eta_j\|^2
+
\|\pi_\al(y)\eta_j\pi_\al(f)\|^2
\\
&\quad-
\vep^2\sum_{j=1}^n \|\pi_\al(F)\pi_\al(f)\eta_j\pi_\al(f)\|^2
\\
&\geq
d/2\sum_{j=1}^n\|\pi_\al(E)\eta_j\|^2
+
(d-\vep^2)\sum_{j=1}^n\|\pi_\al(F)\pi_\al(f)\eta_j\pi_\al(f)\|^2
\\
&\geq
d/2\sum_{j=1}^n\|\pi_\al(E)\pi_\al(1-F)\eta_j\pi_\al(1-F)\|^2
\\
&\quad+
d/2\sum_{j=1}^n\|\pi_\al(F)\eta_j\pi_\al(F)\|^2
\\
&=
d/2\sum_{j=1}^n\|\pi_\al(E)\pi_\al(1-F)\eta_j\pi_\al(1-F)
+\pi_\al(F)\eta_j\pi_\al(F)\|^2
\\
&=
d/2\sum_{j=1}^n
\|\pi_\al(E+F)(\pi_\al(1-F)\eta_j\pi_\al(1-F)
+\pi_\al(F)\eta_j\pi_\al(F))\|^2
\\
&=
d/2\sum_{j=1}^n
\|\pi_\al(E')\eta_j'\|^2,
\end{align*}
which shows (4).

Using $\pi_\al(F)\be_j=0=(\tJ\be_j)\pi_\al(F)$
in (\ref{eq:xeta}) and (\ref{eq:halpetax}),
we can verify (5) as follows:
\begin{align*}
&\|\pi_\al(x')\eta_k'-\hal_p(\eta_k')\pi_\al(x')\|^2
\\
&=
\|\hal_p(\eta_k)\pi_\al(x)-\hal_p(\eta_k)\pi_\al(x)\|^2
\\
&\quad+
\|\pi_\al(y)\eta_k\pi_\al(F)-\pi_\al(F)\hal_p(\eta_k)\pi_\al(y)\|^2
\\
&\leq
\|\hal_p(\eta_k)\pi_\al(x)-\hal_p(\eta_k)\pi_\al(x)\|^2
\\
&\quad+
\|\pi_\al(y)\cdot\pi_\al(f)\xi_k\pi_\al(f)-\pi_\al(f)\hal_p(\xi_k)\pi_\al(f)
\cdot\pi_\al(y)\|^2
\\
&\leq
\vep^2\sum_{j=1}^n\|\pi_\al(x)\xi_j\|^2
+
\vep^2\sum_{j=1}^n\|\pi_\al(y)\xi_j\|^2
\\
&=
\vep^2\sum_{j=1}^n\|\pi_\al(x')\xi_j\|^2.
\end{align*}

Therefore, $(x',E',\beta_1',\dots,\beta_n')$ belongs to $\sJ$,
and is strictly larger than $(x,E,\beta_1,\dots,\beta_n)$.
This is a contradiction.
Hence $E=1$.
\end{proof}

Thus we have proved
there exists a sequence $(x_n)_n$ in $\cM_1$
which does not converge to 0 in the strong$*$
topology, and
\[
\lim_{n\to\infty}\|\pi_\al(x_n)\xi-\hal_p(\xi)\pi_\al(x_n)\|=0
\quad
\mbox{for all }
\xi\in H\oti L^2(\R).
\]
Note that $(x_n)_n$ is a central sequence in $\cM$.
Indeed, let $\xi\in H\oti L^2(\R)$
and set $\vph(a):=\langle \pi_\al(a)\xi,\xi\rangle$
for $a\in \cM$.
Then for $y\in\cM_1$,
\begin{align*}
[x_n,\vph](y)
&=
\langle \pi_\al(yx_n)\xi,\xi\rangle
-
\langle \pi_\al(x_ny)\xi,\xi\rangle
\\
&=
\langle \pi_\al(y)
\left(\pi_\al(x_n)\xi-\hal_p(\xi)\pi_\al(x_n)\right),\xi\rangle
\\
&\quad+
\langle \pi_\al(y)\hal_p(\xi)\pi_\al(x_n),\xi\rangle
-
\langle \pi_\al(x_ny)\xi,\xi\rangle
\\
&=
\langle \pi_\al(y)
\left(\pi_\al(x_n)\xi-\hal_p(\xi)\pi_\al(x_n)\right),\xi\rangle
\\
&\quad+
\langle \pi_\al(y)\hal_p(\xi),\xi\pi_\al(x_n^*)\rangle
-
\langle \pi_\al(x_ny)\xi,\xi\rangle
\\
&=
\langle \pi_\al(y)
\left(\pi_\al(x_n)\xi-\hal_p(\xi)\pi_\al(x_n)\right),\xi\rangle
\\
&\quad+
\langle \pi_\al(y)\hal_p(\xi),
\left(\xi\pi_\al(x_n^*)-\pi_\al(x_n^*)\hal_p(\xi)\right)\rangle
\\
&\quad
+
\langle \pi_\al(y)\hal_p(\xi),
\pi_\al(x_n^*)\hal_p(\xi)\rangle
-
\langle \pi_\al(x_ny)\xi,\xi\rangle
\\
&=
\langle \pi_\al(y)
\left(\pi_\al(x_n)\xi-\hal_p(\xi)\pi_\al(x_n)\right),\xi\rangle
\\
&\quad+
\langle \pi_\al(y)\hal_p(\xi),
\left(\xi\pi_\al(x_n^*)-\pi_\al(x_n^*)\hal_p(\xi)\right)\rangle
\end{align*}
holds, 
since
\begin{align*}
\langle \pi_\al(y)\hal_p(\xi),
\pi_\al(x_n^*)\hal_p(\xi)\rangle
&=
\langle \pi_\al(x_ny)\hal_p(\xi),\hal_p(\xi)\rangle
=
\langle \hal_p(\pi_\al(x_ny)\xi),\hal_p(\xi)\rangle
\\
&=
\langle \pi_\al(x_ny)\xi,\xi\rangle.
\end{align*}
Thus
\[
\|[x_n,\vph]\|_{\cM_*}
\leq
\|\pi_\al(x_n)\xi-\hal_p(\xi)\pi_\al(x_n)\|
+
\|\xi\pi_\al(x_n^*)-\pi_\al(x_n^*)\hal_p(\xi)\|,
\]
and the right hand side
converges to 0 as $n\rightarrow \infty$.

\begin{thm}
Suppose that $p\in\Sp(\al|_{\cM_{\om,\al}})$.
Then there exists a unitary central sequence $(u_n)_n$
in $\cM$ such that
$\hal_p=\lim_{n\to\infty}\Ad\pi_\al(u_n)$.
\end{thm}
\begin{proof}
Set $\cQ:=M_2(\C)\oti\cM$.
Let
$A_\om$ be the set of all
$(X_n)_n$ in $\ell^\infty(\cQ)$
such that $[(\id\oti\pi_\al)(X_n),\vph\oplus \hal_p(\vph)]\to0$
as $\nu\to\om$ for all $\vph\in \cN_*$.
Then $\cP_\om:=A_\om/\sT_\om(\cQ)$ is a von Neumann algebra.
By the remark above,
$\cP_\om\subs M_2(\C)\oti\cM_\om$ naturally.

We will show that $e:=e_{11}\oti1$ is equivalent to $f:=e_{22}\oti1$
in $\cP_\om$.
Let $z$ be a central element of $\cP_\om$.
Clearly, $z=e_{11}\oti a+e_{22}\oti b$
for some $a,b\in \cM_\om$.
Take $(x_n)_n$ as above.
Then $(x_n)_n$ defines a non-zero element $x$ of $\cM_\om$
and $e_{21}\oti x\in e\cP_\om f$.
Then $[z,e_{21}\oti x]=0$ implies $xa=bx$.
Since $x$ is central,
we may assume that $xa=ax$, and moreover
$\|(a-b)x\|_2=\|a-b\|_2\|x\|_2$
by the fast reindexation trick.
Thus we have $a=b$.
Hence $Z(\cP_\om)\subs \C\oti Z(\cM_\om)$.
This implies that $e\sim f$.
\end{proof}

By Lemma \ref{lem:invdual},
we obtain the following result.

\begin{thm}
\label{thm:faithful-Rohlin}
Let $\al$ be a flow on a factor $\cM$.
Then $p\in\Sp(\al|_{\Meq})$
if and only if
there exists a unitary $v\in\Meq$
such that $\al_t(v)=e^{ipt}v$ for all $t\in\R$.
In particular,
the following statements are equivalent:
\begin{enumerate}
\item
$\al$ is faithful on $\Meq$;

\item
$\Sp(\al|_{\Meq})=\Gamma(\al|_{\Meq})=\R$;

\item
$\al$ is a Rohlin flow on $\cM$.
\end{enumerate}
\end{thm}

This result immediately implies the following.

\begin{cor}\label{cor:faithful-product}
Let $\al$ be a faithful flow on a finite factor.
Then the product type flow $\bigotimes_{n=1}^\infty\al_t$
has the Rohlin property.
\end{cor}

\section{Concluding remarks and Problems}
\label{sect:concl}
In this paper,
we have studied Rohlin flows on von Neumann algebras.
We can generalize the Rohlin property
for actions of a locally compact abelian group
by using the dual group.
Hence it is natural attempt
to extend our work more general.
In studying this,
one needs to think of the map $\Th$
that is introduced in Lemma \ref{lem:theta}.

\begin{prob}
Classify actions with the Rohlin property
of a locally compact abelian group
on von Neumann algebras.
\end{prob}

Our main theorem is applicable
to non-McDuff factors.
When its central sequence algebra
is non-trivial and commutative,
there is a chance that they have a Rohlin flow.

\begin{prob}
Characterize a factor admitting a Rohlin flow.
\end{prob}

We should notice that
the classification of ``outer actions'' of $\R$
on injective factors
has not yet been completely finished.
Indeed,
we do not know a characterization
of the Rohlin property
without using a central sequence algebra.
In the light of classification results
for amenable discrete or compact group actions
obtained so far,
we will pose the following plausible conjecture.

\begin{conj}
\label{conj:Rohlin-outer}
Let $\al$ be a flow on an injective factor $\cM$.
Then the following conditions are equivalent:
\begin{enumerate}
\item
$\al$ has the Rohlin property;

\item
$\pi_\tal(\widetilde{\cM})'\cap (\tcM\rti_\tal\R)
=\pi_\tal(Z(\widetilde{\cM}))$.
\end{enumerate}
\end{conj}

We have seen
the implication (1)$\Rightarrow$(2)
holds
for a general von Neumann algebra
in Corollary \ref{cor:relative}.

\begin{exam}
Let
$\al$ be a flow on an injective type III$_0$ factor $\cM$
such that $\al$ is ergodically and faithfully acting on the
space of the flow of weights.
Let $\cN=\cM\rti_\al\R$.
Then $\tcN=\tcM\rti_\tal\R$ canonically.
Since the action $\tal$ on $Z(\tcM)$
is faithful,
we have $\pi_\tal(Z(\tcM))'\cap\tcN=\pi_\tal(\tcM)$
by \cite[Corollary VI.1.3]{NT}.
Hence $\al$ satisfies (2)
in Conjecture \ref{conj:Rohlin-outer}.
The dual flow
of an extended modular flow has such property
(see Theorem \ref{thm:modular-flow}).
\end{exam}

Thus it must be interesting to consider
the weak version of
Conjecture \ref{conj:Rohlin-outer}.

\begin{prob}
Suppose that $\al$ is a flow on an injective
type III$_0$ factor
such that $\mo(\al)$ is faithful and ergodic.
Then does $\al$ have the Rohlin property?
\end{prob}

Readers are referred to \cite{Iz-Can}
for a faithful action of a compact group on a flow space.
The following problem is related with Lemma \ref{lem:SpGa}.

\begin{prob}
Assume that $\cM$ is an injective factor
and $\al$ is a pointwise centrally non-trivial flow.
Then $\Ga(\al|_{\Meq})=\Ga(\al)$?
\end{prob}

Since we always have
$\Ga(\al|_{\Meq})\subs\Ga(\al)$
by Lemma \ref{lem:SpGa},
what we must do is to study the case
when $\Ga(\al)\neq\{0\}$.
If this problem is affirmatively solved,
the following condition,
which looks much weaker than (2)
in the above conjecture,
could imply the Rohlin property of $\al$
by Theorem \ref{thm:faithful-Rohlin}:
{\it
\begin{enumerate}
\addtocounter{enumi}{+2}
\item
$\al$ is pointwise centrally non-trivial
and
$\Ga(\al)=\R$.
\end{enumerate}
}

Indeed
in Theorem \ref{thm:itp-Rohlin}
and \ref{thm:QF-Rohlin},
we have shown that the three conditions mentioned above
are equivalent for product type flows
and quasi-free flows induced from Cuntz algebras.
We should note that
if $\al$ is an action of a compact abelian group,
then
the above condition (3) indeed implies the Rohlin property.

\begin{prob}
Let $\al$ be a flow on the injective type II$_1$ factor $\cR_0$.
Suppose that $\al\oti\id_{\cR_0}$ has the Rohlin property.
Then does $\al$ also have?
\end{prob}

This problem is a direct consequence of the above conjecture.
Proposition \ref{prop:modular-tensor} shows
that the assumption on the type is necessary.

In Theorem \ref{thm:tracescaling},
we have proved
that
a trace scaling flow on the injective
type II$_\infty$ factor
has the Rohlin property
by using Connes-Haagerup theory.
Another proof of that theorem
will yield the uniqueness of
the injective type III$_1$ factor.
Hence we have interest in the following
problem.

\begin{prob}
Prove Theorem \ref{thm:tracescaling}
without using the fact that
$\si_t^\vph\in\oInt(\cM)$ for
an injective type III$_1$ factor.
\end{prob}

\section{Appendix}
\subsection{Basic measure theoretic results}
We recall the following elementary result
of measure theory.

\begin{lem}
\label{lem:Lusin}
Let $X$ be a Polish space
and $f\col \R^n\ra X$
be a Borel map.
Let $E\subs\R^n$ be a Borel set
with $0<\mu(E)<\infty$,
where $\mu$ denotes the Lebesgue measure on $\R^n$.
Then the following statements hold:
\begin{enumerate}
\item
For any $\vep>0$,
there exists a compact set $K\subs E$
such that
$\mu(E\setminus K)<\vep$
and $f$ is continuous on $K$;

\item
Let $d$ be a complete metric
on the Polish space $X$.
For any $\vep_1,\vep_2>0$,
there exist disjoint Borel sets
$A_0,A_1,\dots,A_N\subs E$
and $x_1,\dots,x_N\in X$
such that
$\sum_{j=0}^N A_j=E$,
$\mu(A_0)<\vep_1$,
$d(f(t),x_j)<\vep_2$
for all $t\in A_j$,
$j=1,\dots,N$.
\end{enumerate}
\end{lem}
\begin{proof}
(1).
This is shown by using Lusin's theorem.
See \cite[Theorem 17.12]{Kech}

(2).
Let $K$ be as above with $\vep=\vep_1$.
Since $f|_K$ is uniformly continuous,
for any $\vep_2>0$,
there exists a Borel partition
$\{A_j\}_{j=1}^N$ of $K$
such that any $s,t\in A_j$ satisfy $d(f(s),f(t))<\vep_2$.
Put $A_0:=E\setminus K$,
and take an element $x_j\in f(A_j)$ for each $j$.
Then these $x_j$ have the required property.
\end{proof}

Employing the previous lemma,
we have the following result.

\begin{lem}\label{lem:simple}
Let $\cM$ be a separable finite von Neumann algebra
with a faithful normal tracial state $\ta$.
Let us write $\|x\|_2:=\ta(x^*x)^{1/2}$ as usual.
Let $w\col [0,1]\ra \cM^{\rm U}$ be a Borel map.
Then for any $\vep_1,\vep_2>0$,
there exist disjoint Borel sets
$A_0,A_1,\dots,A_N$ of $[0,1]$
and $u_1,\dots,u_N\in \cM^{\rm U}$
such that
$\mu(A_0)<\vep_1$,
$\|w_t-u_j\|_2<\vep_2$
for all $t\in A_j$,
$j=1,\dots,N$.
\end{lem}

\subsection{Disintegration of automorphisms}

Let $(X,\sB)$ be a standard Borel space
and $\mu$ a $\si$-finite Borel measure.
Recall the following basic result.

\begin{lem}
\label{lem:meas-Borel}
Let $f\col X\ra \R$ be a $\mu$-measurable function.
Then there exists a Borel $\mu$-null set $N\subs X$
such that the restriction
$f\col X\setminus N\ra \R$ is Borel.
\end{lem}
\begin{proof}
Take a sequence $\{q_n\}_n$
whose union equals $\Q$.
Then $A_n:=\{x\mid f(x)>q_n\}$ is $\mu$-measurable.
Since $A_n$ is $\mu$-measurable,
we can take Borel sets
$B_n,B_n'\subs X$
such that
$B_n\subs A_n\subs B_n'$
and
$\mu(B_n'\setminus B_n)=0$.
We put $N_n:=B_n'\setminus B_n$.
Then we have
$A_n\setminus N_n
=B_n$
that is Borel.
We let $N:=\bigcup_{k=1}^\infty N_n$.
Then
\[
\{x\in X\mid x\nin N,\ f(x)>q_n\}
=A_n\cap N_n^c\cap
\bigcap_{k\neq n}N_k^c,
\]
and this set is Borel.
Since $\Q$ is dense in $\R$,
$f\col X\setminus N\ra \R$ is Borel.
\end{proof}

Let $\{H_x\}_{x\in X}$ and
$\{\cM_x\}_{x\in X}$ be measurable fields
of separable Hilbert spaces
and separable von Neumann algebras, respectively,
such that $\cM_x\subs B(H_x)$.
Let $\{\al_x\}_{x\in X}$ be a measurable
field of automorphisms
with $\al_x\in\Aut(\cM_x)$.
We set
\[
\cM:=\int_X^\oplus\cM_x\,d\mu(x),
\quad
H:=\int_X^\oplus H_x\,d\mu(x),
\quad
\al:=\int_X^\oplus\al_x\,d\mu(x).
\]

\begin{thm}
\label{thm:appdisint}
Let $\al$ and $\al^x$ be as above.
Then
$\al\in\oInt(\cM)$ if and only if
$\al_x\in\oInt(\cM_x)$ for almost every $x\in X$;
\end{thm}
\begin{proof}
We may and do assume that
$H_x=H_0$ with $H_0\cong \ell^2$ for all $x\in X$,
and $\mu(X)<\infty$.
Let $\vN(H_0)$ be the set
of von Neumann algebras on $H_0$.
We equip $\vN(H_0)$ with the Effros Borel structure
as usual \cite{Ef}.
Then we have a $\mu$-measurable map
$X\ni x\mapsto \cM_x\in\vN(H_0)$.
Using the previous lemma and choice functions
of $\vN(H_0)$,
we may and do assume that the maps
$x\mapsto \cM_x$,
$x\mapsto \cM_x'$
and
$x\mapsto \al_x$ are Borel.

Suppose that $\al\in\oInt(\cM)$.
Then there exists a sequence of unitaries
$\{v^\nu\}_\nu$ in $\cM$
such that $\al=\lim_\nu\Ad v^\nu$
in the $u$-topology.
Then we obtain
\[
\|\al(\vph)-v^\nu\vph v^{\nu *}\|
=
\int_X\|\al_x(\vph_x)-v_x^\nu\vph_x v_x^{\nu *}\|_{(\cM_x)_*}
\,d\mu(x)
\quad
\mbox{for all }\vph\in\cM_*.
\]
Thus
there exists a subsequence
$\{v^{\nu_k}\}_k$
such that
for all $\vph\in\cM_*$,
we have
$\|\al_x(\vph_x)-v_x^{\nu_k}\vph_x v_x^{\nu_k *}\|\to0$
as $k\to\infty$
for almost all $x$.
Hence $\al_x\in\oInt(\cM_x)$ for almost all $x$
since $\{\vph_x\mid\vph\in\cM_*\}$
is dense in $(\cM_x)_*$
for almost every $x$.

Suppose conversely that
$\al_x\in\oInt(\cM_x)$ for almost every $x$.
For an integrable Borel map
$X\ni x\mapsto \vph_x\in B(H_0)_*$,
we set
\[
F_\vph\col X\times B(H_0)^{\rm U}
\ni (x,v)\mapsto \|\al_x(\vph_x|_{M_x})-v\vph_x v^*|_{M_x}\|_{(\cM_x)_*}\in\R,
\]
where $B(H_0)^{\rm U}$ denotes the unitary group
that is Polish with respect to the strong* topology.
The function $F_\vph$ is Borel.
Indeed, let $\{a_x\}_x$ be a Borel operator field
such that $a_x\in\cM_x$.
Then
\begin{equation}
\label{eq:alxph}
(\al_x(\vph_x|_{M_x})-v\vph_x v^*)(a_x)
=
\vph(\al_x^{-1}(a_x))-v\vph_x v^*(a_x).
\end{equation}
It is trivial that $x\mapsto \vph_x(\al_x^{-1}(a_x))$ is Borel.
Since the maps
$B(H_0)^{\rm U}\times B(H_0)_*\ni(v,\vph)
\mapsto v\vph v^*\in B(H_0)_*$
and
the coupling $B(H_0)_*\times B(H_0)\to \C$
are both continuous,
the second term in (\ref{eq:alxph})
is Borel.
Hence $F_\vph$ is Borel.

Let $b^j\col \vN(H_0)\ra B(H_0)_1$ be a Borel choice function
such that $\{b^j(\cN)\}_{j=1}^\infty$
is strongly dense in $\cN_1$
for all $\cN\in\vN(H_0)$.
We let $b_x^j:=b^j(\cM_x')$
that is a Borel function from $X$ into $B(H_0)_1$.
Let $\{\vph^k\}_{k=1}^\infty$ be a norm dense set
in $L_{B(H_0)_*}^1(X,\mu)$.
We set the function
\[
G_n\col X\times B(H_0)^{\rm U}\ni (x,v)
\mapsto
\sup_{1\leq j,j'\leq n}\|[v,b_x^j]\vph_x^{j'}\|\in\R.
\]
Then $G_n$ is Borel
since
the left multiplication
$B(H_0)\times B(H_0)_*\to B(H_0)_*$
is continuous.
For $m,n\in\N$,
we set the following Borel subset:
\[
Z_m:=
\bigcap_{k=1}^m
F_{\vph^k}^{-1}([0,1/m])\cap
\bigcap_{n=1}^\infty
G_{n}^{-1}(\{0\}).
\]
Note that
\[
Z_m=
\left\{(x,v)\in X\times B(H_0)^{\rm U}
\,\middle|\,
\sup_{1\leq k\leq m}
\|\al_x(\vph_x^k|_{M_x})-v\vph_x^k v^*|_{M_x}\|_{(\cM_x)_*}
\leq1/m,
\
v\in M_x
\right\}.
\]

Let $\pr_1\col X\times B(H_0)^{\rm U}\to X$ be the projection.
By approximate innerness of $\al_x$,
it turns out that $\pr_1|_{Z_m}\col Z_m\ra X$ is surjective.
Thanks to the measurable cross section theorem
(see \cite[Theorem 3.2.4]{Arv-inv} or \cite[Theorem A.16]{TaI}),
we have a $\mu$-measurable map
$f\col X\ra Z_m$
such that $\pr_1\circ f=\id_X$.

Let $\pr_2\col X\times B(H_0)^{\rm U}\ra B(H_0)^{\rm U}$
be the projection.
Since $\pr_2$ is Borel,
$\pr_2\circ f\col X\ra B(H_0)^{\rm U}$ is $\mu$-measurable.
We set $v_x:=\pr_2(f(x))\in \cM_x$
and
\[
v:=\int_X^\oplus v_x\,d\mu(x)
\in\cM.
\]
Then for all
$k=1,\dots,m$,
\[
\|\al(\vph^k)-v\vph^k v^*\|_{\cM_*}
=
\int_X
\|\al_x(\vph_x^k|_{\cM_x})-v_x\vph_x^k v_x^*|_{\cM_x}\|_{(\cM_x)_*}
\,d\mu(x)
\leq\mu(X)/m.
\]
This means that $\al\in\oInt(\cM)$.
\end{proof}

In the proof above,
we have implicitly proved
the following result
which has been proved by Lance \cite[Theorem 3.4]{La}.

\begin{thm}
\label{thm:intdisint}
Let $\al=\int_X^\oplus \al^x\,d\mu(x)$ be an automorphism
on $\cM$ as before.
Then
$\al\in\Int(\cM)$
if and only if
$\al_x\in\Int(\cM_x)$
for almost every $x$.
\end{thm}

Next we study centrally trivial automorphisms.

\begin{lem}
Let us fix a faithful normal state $\ps$ on $\cM$.
An automorphism $\al$ on $\cM$
is centrally trivial
if and only if
for any $\vep>0$,
there exist $\de>0$ and a finite set
$F\subs\cM_*$
such that
if $a\in\cM_1$ satisfies
$\|[a,\vph]\|<\de$ for all $\vph\in F$,
then
$\|\al(a)-a\|_\ps^\sharp<\vep$.
\end{lem}

\begin{lem}
\label{lem:cnt-borel}
The subgroup $\Cnt(\cM)$ is Borel in $\Aut(\cM)$.
\end{lem}
\begin{proof}
Note that if $\al\in\Cnt(\cM)$ and
$(a^\nu)_\nu$ is central,
then $\|\al(a^\nu)-a^\nu\|_\ps^\sharp\to0$
as $\nu\to\infty$.
Let $\{F_m\}_m$ be an increasing sequence
of finite subsets in $\cM_*$
such that their union is norm dense in $\cM_*$.
Let $\{a_j\}_j$ be a strongly dense sequence
in $\cM_1$.
For $m\in\N$,
we let $J_m:=\{j\in\N\mid \|[a_j,\vph]\|<1/m,\ \vph\in F_m\}$.

Then the previous lemma implies
that
$\al\in\Aut(\cM)$ is centrally trivial
if and only if
\[
\inf_{m\in\N}\sup_{j\in J_m}\|\al(a_j)-a_j\|_\ps^\sharp=0.
\]
Since $\|\al(a_j)-a_j\|_\ps^\sharp$
is continuous with respect to $\al\in\Aut(\cM)$,
the function
$\al\mapsto \inf_m\sup_{j\in J_m}\|\al(a_j)-a_j\|_\ps^\sharp$
is Borel.
In particular,
$\Cnt(\cM)$ is a Borel subset in $\Aut(\cM)$.
\end{proof}

\begin{lem}
\label{lem:centralseq-disint}
If a sequence $a^\nu:=\int_X^\oplus a_x^\nu\,d\mu(t)$
is central in $\cM$,
then a subsequence
$(a_x^{\nu_m})_m$ is central in $\cM_x$
for almost every $x\in X$.
\end{lem}
\begin{proof}
Let $\{\vph^k\}_{k\in\N}$ be a dense sequence
of $\cM_*$.
Then $(a^\nu)_\nu$ is central
if and only if
\[
\|[a^\nu,\vph^k]\|
=
\int_X
\|[a_x^\nu,\vph_x^k]\|\,d\mu(x)
\to0
\quad
\mbox{as }
\nu\to\infty
\]
for all $k$.
This means
$\|[a_x^\nu,\vph_x^k]\|\to0$
in $L^1(X,\mu)$.
Hence we are done.
\end{proof}

\begin{thm}
\label{thm:centdisint}
Let $\al=\int_X^\oplus \al^x\,d\mu(x)$ be an automorphism
on $\cM$ as before.
Then
$\al\in\Cnt(\cM)$ if and only if
$\al_x\in\Cnt(\cM_x)$ for almost every $x\in X$.
\end{thm}
\begin{proof}
We may and do assume that
all relevant maps such as $x\mapsto \cM_x$ and
$x\mapsto \al_x$ are Borel as before.
By replacing $\mu$ if necessary,
$\mu$ is assumed to be finite.
Let $\vph_0\in B(H_0)_*$ be a faithful state
and $\ps:=\vph_0\otimes\mu$.

Suppose that
$\al_x\in\Cnt(\cM_x)$ for almost every $x\in X$.
If $\al$ were not centrally trivial,
there exist $\vep_0>0$
and a central sequence $(a^\nu)_\nu$
in $\cM$ such that
$\inf_\nu\|\al(a^\nu)-a^\nu\|_\ps^\sharp\geq\vep_0$.
This implies for all $\nu$,
\[
\int_X
\|\al_x(a_x^\nu)-a_x^\nu\|_{\vph_0}^{\sharp\,2}\,d\mu(t)
\geq\vep_0^2.
\]
By the previous lemma,
a subsequence $(a_x^{\nu_m})_m$ is central
for almost every $x\in X$.
Hence $\|\al_x(a_x^{\nu_m})-a_x^{\nu_m}\|_{\vph_0}^\sharp
\to0$
as $m\to\infty$.
Then by the dominated convergence theorem,
the left hand side above converges to 0,
and this is a contradiction.

Suppose conversely
$\al\in\Cnt(\cM)$.
We let $a^j\col X\ra B(H_0)_1$ be a choice function
such that $\{a_x^j\}_j$ is strongly dense in $(\cM_x)_1$.
By discarding $\mu$-null sets,
we may take a norm dense sequence $\{\vph^k\}_k$
in $L_{B(H_0)_*}(X,\mu)$
such that
$\{\vph_x^k|_{\cM_x}\}_k$ is norm dense in $(\cM_x)_*$
for all $x\in X$.
We set
\[
A_{j,m}:=\{x\in X\mid
\|[a_x^j,\vph_x^k]\|_{(\cM_x)_*}<1/m,\ k=1,\dots,m\},
\]
which is
a Borel subset of $X$
since the map
$B(H_0)_*\times \vN(H_0)\ni (\vph,\cM)\mapsto \|\vph|_\cM\|$
is Borel.
Then
\[
\{x\in X\mid \al_x\in\Cnt(\cM_x)\}
=
\left
\{x\in X
\,\middle|\, 
\inf_{m\in\N}\sup_{j\in\N} 1_{A_{j,m}}(x)
\|\al_x(a_x^j)-a_x^j\|_{\vph_0}^\sharp=0
\right\}.
\]
Since
$x\mapsto \|\al_x(a_x^j)-a_x^j\|_{\vph_0}^\sharp$
is Borel,
the set $\{x\in X\mid \al_x\in\Cnt(\cM_x)\}$
is Borel.

Suppose that
$N:=\{x\in X\mid \al_x\nin\Cnt(\cM_x)\}$
satisfied
$\mu(N)>0$.
Since the positive function
$g\col N\ni x\mapsto
\inf_m\sup_j 1_{A_{j,m}}(x)
\|\al_x(a_x^j)-a_x^j\|_{\vph_0}$
is Borel,
there exists $\vep_1>0$
such that
$N_1:=\{x\in N\mid g(x)>\vep_1\}$ satisfies $\mu(N_1)>0$.
Thus for all $m\in\N$,
we obtain $j_m\in\N$ such that
\[
\mu\Big{(}
\bigcup_{j=1}^{j_m}
\{x\in A_{j,m}\mid
\|\al_x(a_x^j)-a_x^j\|_{\vph_0}^\sharp\geq\vep_1
\}\Big{)}
\geq\mu(N_1)/2.
\]

Let
$B_{j,m}:=
\{x\in A_{j,m}\mid
\|\al_x(a_x^j)-a_x^j\|_{\vph_0}^\sharp\geq\vep_1
\}$
and
$X_m:=B_{1,m}\cup\cdots\cup B_{j_m,m}$.
We set $c_x^m$ as follows:
\[
c_x^m
:=
\begin{cases}
a_x^j&
\mbox{if }
x\in B_{j,m}\cap B_{1,m}^c
\cap\cdots\cap B_{j-1,m}^c,
\\
0&
\mbox{if }
x\in X\setminus X_m.
\end{cases}
\]

Then $\|[c_x^m,\vph_x^k]\|_{(\cM_x)_*}\leq1/m$
for all $x\in X$
and $k=1,\dots,m$,
and $\|\al_x(c_x^m)-c_x^m\|_{\vph_0}^\sharp\geq\vep_1$
for all $x\in X_m$.
Put
$c^m:=\int_X^\oplus c_x^m\,d\mu(x)$.
Then we obtain
\[
\|[c^m,\vph^k]\|_{\cM_*}
=\int_X\|[c_x^m,\vph_x^k]\|_{(\cM_x)_*}\,d\mu(x)\leq\mu(X)/m
\quad\mbox{for all }
k=1,\dots,m,
\]
and
\[
\|\al(c^m)-c^m\|_\ps^{\sharp\,2}
=
\int_{X_m}
\|\al_x(c_x^m)-c_x^m\|_{\vph_0}^{\sharp\, 2}\,d\mu(t)
\geq
\vep_1^2\mu(X_m)
\geq
\vep_1^2\mu(N_1)/2.
\]
Then $(c^m)_m$ is central,
but $\liminf_m\|\al(c^m)-c^m\|_\ps^{\sharp\,2}\geq \vep_1^2\mu(N_1)/2$.
This is a contradiction.
\end{proof}

Let $\cM$ be a von Neumann algebra
and $\vph$ a faithful normal state on $\cM$.
Then we obtain the central decompositions
of $\cM$ and $\vph$ as follows:
\[
\cM=\int_X^\oplus \cM_x\,d\mu(x),
\quad
\vph=\int_X^\oplus \vph_x\,d\mu(x),
\]
where $Z(\cM)$ is identified with $L^\infty(X,\mu)$.

We may and do assume that
all $\cM_x$ are von Neumann subalgebra
acting on a common Hilbert space $H_0$ as before.
Thus $\cM$ acts on
$H:=L^2(X,\mu)\oti H_0$.

Let $\cN_x:=\cM_x\rti_{\si^{\vph_x}}\R$
and $K_x:=H_0\oti L^2(\R)$.
We will show that
$\{\cN_x,K_x\}_x$ is a measurable field of
von Neumann algebras
with respect to
$\int_X^\oplus K_x\,d\mu(x)
=L^2(X)\oti H_0\oti L^2(\R)$.

Let $a^j\col X\ra B(H_0)_1$
be a $\mu$-measurable choice function
such that
$\{a_x^j\}_{j=1}^\infty$
is strongly dense in $(\cM_x)_1$ for almost every $x$.
Then $x\mapsto \pi_{\si^{\vph_x}}(a_x^j)$
is $\mu$-measurable.
Indeed,
let
$\xi,\eta\in \int_X^\oplus K_x\,d\mu(x)$.
Then
\[
\langle\pi_{\si^{\vph_x}}(a_x^j)\xi_x,
\eta_x\rangle
=
\int_\R
\langle \si_{-t}^{\vph_x}(a_x^j)\xi_x(t),
\eta_x(t)\rangle
\,d\mu(t).
\]
Since
$(x,t)\mapsto
\langle \si_{-t}^{\vph_x}(a_x^j)\xi_x(t),
\eta_x(t)\rangle$
is $\mu$-measurable,
$x\mapsto
\langle\pi_{\si^{\vph_x}}(a_x^j)\xi_x,
\eta_x\rangle$
is $\mu$-measurable
by Fubini's theorem.

Note that
$\{\pi_{\si^{\vph_x}}(a_x^j)\}_j$
is strongly dense in $\pi_{\si^{\vph_x}}(\cM_x)_1$.
Each $\cN_x$ contains the left regular
representation $\la^{\vph_x}(t)=1\oti\la(t)$.
Thus
$\{\cN_x,K_x\}$ is a $\mu$-measurable field.
Let $\cN$ be the disintegration of $\cN_x$,
that is,
\[
\cN:=\int_X^\oplus\cN_x\,d\mu(x),
\]
which acts on the Hilbert space
$L^2(X,\mu)\oti H_0\oti L^2(\R)$.
Using
\begin{equation}
\label{eq:moddisint}
\si_t^\vph=\int_X^\oplus \si_t^{\vph_x}\,d\mu(t),
\end{equation}
we obtain
\[
\int_X^\oplus \pi_{\si^{\vph_x}}(a_x^j)\,d\mu(x)
=
\pi_{\si^\vph}(a^j).
\]
It is trivial that
$\int_X^\oplus \la^{\vph_x}(t)\,d\mu(x)=\la^\vph(t)$.
Thus
$\cN$ is nothing but $\cM\rti_{\si^\vph}\R$.
Summarizing this discussion,
we obtain the following result.

\begin{lem}
In the above setting,
one has the following natural identification:
\[
\cM\rti_{\si^\vph}\R
=
\int_X^\oplus
\cM_x\rti_{\si^{\vph_x}}\R\,d\mu(x).
\]
\end{lem}

Let $\al\in\Aut(\cM)$
which fixes any element of $Z(\cM)$.
Then $\al$ is described as
\[
\al=\int_X^\oplus\al_x\,d\mu(x).
\]
By the previous lemma,
we have
\[
\tal=\int_X^\oplus
\widetilde{\al_x}\,d\mu(x),
\quad
\mo(\al)=\int_X^\oplus
\mo(\al_x)\,d\mu(x).
\]

Combining Theorem \ref{thm:appdisint},
\ref{thm:intdisint},
\ref{thm:centdisint}
and Kawahigashi-Sutherland-Takesaki's
result \cite[Theorem 1]{KawST},
we obtain the following.

\begin{thm}
\label{thm:genKST}
Let $\cM$ be an injective von Neumann algebra
with separable predual.
Then the following statements hold:
\begin{enumerate}
\item
$\oInt(\cM)=\ker(\mo)$;

\item
$\Cnt(\cM)=\{\al\in\Aut(\cM)\mid \tal\in\Int(\tcM)\}$.
\end{enumerate}
\end{thm}

We need the following lemma.

\begin{lem}
\label{lem:appinner-weak}
Let $\cM$ be a von Neumann algebra
and $\th\in\Aut(\cM)$.
Then $\th\in\oInt(\cM)$
if and only if
for any $\vep>0$ and a finite $\Ph\subs \cM_*^+$,
there exists $a\in\cM_1$
such that
\[
\|\th(\vph)-a\vph a^*\|<\vep,
\quad
\|a^*a-1\|_\vph+\|aa^*-1\|_\vph<\vep.
\]
\end{lem}
\begin{proof}
By assumption,
we obtain a sequence $(a_n)_n$
in $\cM_1$ such that for all positive $\vph\in\cM_*$,
\[
\th(\vph)=\lim_{n\to\infty}a_n\vph a_n^*,
\quad
\lim_{n\to\infty}
\|a_n^*a_n-1\|_\vph+\|a_na_n^*-1\|_\vph=0.
\]
This also implies that
$\lim_n a_n^* \vph a_n=\th^{-1}(\vph)$.

It turns out that $(a_n)_n$ belongs to $\sN_\om(\cM)$.
Indeed, let $(x_n)_n\in \sT_\om(\cM)$
with $\sup_n\|x_n\|\leq1$.
Then we have
\begin{align*}
\|x_n a_n \vph\|
&\leq
\|x_n a_n \vph\cdot (1-a_n^*a_n)\|
+
\|x_n a_n \vph a_n^*a_n\|
\\
&\leq
\|\vph\cdot (1-a_n^*a_n)\|
+
\|x_n (a_n \vph a_n^*-\th(\vph))a_n\|
+
\|x_n \th(\vph)a_n\|
\\
&\leq
\|1-a_n^*a_n\|_\vph
+
\|a_n \vph a_n^*-\th(\vph)\|
+
\|x_n\th(\vph)\|,
\end{align*}
and
\begin{align*}
\|\vph a_n x_n\|
&\leq
\|(1-a_n a_n^*)\vph a_n x_n\|
+
\|a_na_n^*\vph a_n x_n\|
\\
&\leq
\|1-a_na_n^*\|_\vph
+
\|a_n(a_n^*\vph a_n-\th{-1}(\vph))x_n\|
+
\|a_n\th^{-1}(\vph)x_n\|
\\
&\leq
\|1-a_na_n^*\|_\vph
+
\|a_n^*\vph a_n-\th{-1}(\vph)\|
+
\|a_n\th^{-1}(\vph)x_n\|.
\end{align*}
Thus $\|x_n a_n\vph\|\to0$ and $\|\vph a_n x_n\|\to0$
as $n\to\infty$.
Hence $(a_n)_n$ normalizes $\sT_\om(\cM)$.

We let $u:=\pi_\om((a_n)_n)\in\cM^\om$
that is a unitary.
Take a unitary representing sequence $(u_n)_n$
of $u$.
This satisfies $u_n-a_n\to0$ in the strong$*$ topology
as $n\to\om$,
and we are done.
\end{proof}

Let $\cM$ be a von Neumann algebra and
$\cA$ a von Neumann subalgebra of $Z(\cM)$.
Let $(X,\mu)$ be the measure theoretic spectrum
of $\cA$.
Let $\cH$ be the standard Hilbert space of $\cM$.
Then we have the following disintegrations
putting $\cA=L^\infty(X,\mu)$ as usual:
\[
\cM=\int_X^\oplus\cM_x\,d\mu(x),
\quad
\cH=\int_X^\oplus \cH_x\,d\mu(x).
\]
We may assume that $\dim\cH_x$ is constant
in what follows.
Let $\cK$ be a Hilbert space with
$\dim\cK=\dim\cH_x$ for all $x$.
Then $\{\cH_x\}_x$ is regarded
as a constant field $\{\cK\}_x$,
and we obtain the natural identification
\[
\cH=
\int_X^\oplus \cK\,d\mu(x)
=L^2(X,\mu)\oti \cK.
\]
Note that any automorphism $\th$ on $\cM_x$
is implemented by a unitary on $\cK$.

Now let $\al$ and $\be$ be
actions of a locally compact group $G$
on $\cM$ which are fixing $\cA$.
Then they are written as follows:
\[
\al_t=\int_X^\oplus \al_t^x\,d\mu(x),
\quad
\be_t=\int_X^\oplus \be_t^x\,d\mu(x)
\quad
\mbox{for all }
t\in G.
\]

\begin{thm}
\label{thm:cocdisint}
Let $\al$, $\be$ be as above.
Then the following statements hold:
\begin{enumerate}
\item 
They are cocycle conjugate
if and only if
$\al^x$ and $\be^x$ are for almost every $x\in X$;

\item 
They are strongly cocycle conjugate
if and only if
$\al^x$ and $\be^x$ are for almost every $x\in X$.
\end{enumerate}
\end{thm}
\begin{proof}
(1).
It is useful to consider those actions
in terms of a Kac algebra \cite{ES}.
Namely,
$\al, \be$ are regarded as
the faithful normal $*$-homomorphisms
$\al,\be\col \cM\ra\cM\oti L^\infty(G)$
by putting $(\al(a)\xi)(t)=\al_t(a)\xi(t)$,
$(\be(a)\xi)(t)=\be_t(a)\xi(t)$
for all $\xi\in \cH\oti L^2(G)$
and $t\in G$.
Then we obtain
\[
(\al\oti\id)\circ\al=(\id\oti\de)\circ\al,
\quad
(\be\oti\id)\circ\be=(\id\oti\de)\circ\be,
\]
where the coproduct
$\de\col L^\infty(G)\ra \lG\oti\lG$ is defined by
$\de(f)(r,s):=f(rs)$ for $f\in\lG$ and $r,s\in G$.

Take unitary representations
$U,V\col G\ra B(\cH)$ such that
$\al_t=\Ad U_t$ and $\be_t=\Ad V_t$
on $\cM$.
Regarding $U,V\in B(\cH)\oti\lG$,
we have
\[
(\id\oti\de)(U)
=
U_{12}U_{13},
\quad
(\id\oti\de)(V)
=
V_{12}V_{13},
\]
and
\[
\al(a)=U(a\oti1)U^*,
\quad
\be(a)=V(a\oti1)V^*
\quad
\mbox{for all }
a\in\cM.
\]

Since $\cA$ is fixed by $\al$ and $\be$,
$U$ and $V$ are diagonalizable,
that is,
\[
U=\int_X^\oplus U^x\,d\mu(x),
\quad
V=\int_X^\oplus V^x\,d\mu(x),
\]
where
$U^x,V^x\in B(\cK)\oti \lG$
and we have used the following identification:
\[
\cH\oti L^2(G)
=\int_X^\oplus \cK\oti L^2(G)\,d\mu(x).
\]

Then $U^x$ and $V^x$ implement
$\al^x$ and $\be^x$, respectively,
for almost every $x$.
Note that a unitary
$v\in B(\cK)\oti\lG$
is an $\al^x$-cocycle
if and only if
$v\in \cM_x\oti\lG$
and
$vU^x$ is a unitary representation,
that is,
it satisfies
$(\id\oti\de)(vU^x)=(vU^x)_{12}(vU^x)_{13}$.

By Lemma \ref{lem:meas-Borel},
we may and do assume that
all the relevant measurable maps in what follows
are in fact Borel.
Take Borel maps
$a^j\col X\ra (\cM_x)_1$ and $b^k\col X\ra(\cM_x')_1$
for $j,k\in\N$.

Let $\{\xi_i\}_{i\in\N}$ be a dense
sequence of $K\oti L^2(G)$.
Let $Y_m$ be the subset of
$X\times (B(\cK)\oti\lG)^{\rm U}\times B(\cK)^{\rm U}$
which consists of elements $(x,v,w)$
such that for $i,j,k=1,\dots,m$,
\[
\|[v,b_x^k\oti1]\xi_i\|<1/m,
\quad
(\id\oti\de)(vU^x)=(vU^x)_{12}(vU^x)_{13},
\]
\[
\|[wa_x^j w^*,b_x^k]\xi_i\|+
\|[w^*a_x^j w,b_x^k]\xi_i\|<1/m,
\]
\[
\|\left(
v\al^x(a_x^j)v^*
-(w\oti1)U^x(w^* a_x^j w\oti1)(U^x)^*(w^*\oti1)
\right)\xi_i\|<1/n.
\]
We can show that $Y_m$ is Borel as before.
Thus $Y:=\bigcap_m Y_m$ is Borel.
Then
$(x,v,w)\in Y$
if and only if
$v\in\cM_x\oti\lG$,
$w\cM_x w^*=\cM_x$
and
\[
(\id\oti\de)(v)=(v\oti1)(\al^x\oti\id)(v),
\quad
\Ad v\circ\al^x=(\th\oti\id)\circ\be^x\circ\th^{-1},
\]
where we have put $\th:=\Ad w|_{\cM_x}$.
By our assumption,
$Y$ is non-empty, and
we get the Borel projection $\pr_X\col Y\ra X$.

Then there exists
a measurable cross section
$s\col X\ra Y$ with $\pr_X\circ s=\id_X$.
We let $s(x)=(x,v^x,w^x)$.
Hence $\{v^x\}_x$ and $\{w^x\}_x$ are measurable,
and we set
\[
v:=\int_X^\oplus v^x\,d\mu(x)\in \cM\oti\lG,
\quad
w:=\int_X^\oplus w^x\,d\mu(x)
\in \int_X^\oplus B(\cK)\,d\mu(x).
\]
Put $\th:=\Ad w|_\cM$ and we obtain
\[
(\id\oti\de)(v)=(v\oti1)(\al\oti\id)(v),
\quad
\Ad v\circ\al=(\th\oti\id)\circ\be\circ\th^{-1}.
\]
Thus we are done.

(2).
We first show that
the following set is Borel:
\[
Z:=\{(x,w)\in X\times B(\cK)^{\rm U}
\mid
\Ad w|_{\cM_x}\in\oInt(\cM_x)
\}.
\]
Take a norm dense sequence $\{\vph^k\}_k$
in $L_{B(\cK)}^1(X,\mu)$ such that
$\{\vph_x^k|_{\cM_x}\}_k$ is norm dense
in $\cM_x$ for almost every $x$.
For $m\in\N$,
we define $Z_m\subs X\times B(\cK)^{\rm U}$ which consists of
elements $(x,w)$ such that
there exists $\ell\in\N$ satisfying
the following conditions for all $i,j,k=1,\dots,m$:
\[
\|[wa_x^jw^*,b_x^k]\xi_i\|+\|[w^*a_x^jw,b_x^k]\xi_i\|<1/m,
\]
\[
\|w\vph_x^i w^*|_{\cM_x}
-a_x^{\ell}\vph_x^i (a_x^{\ell})^*|_{\cM_x}\|_{(\cM_x)_*}
<1/m,
\]
\[
\|((a_x^\ell)^*a_x^\ell-1)\xi_i\|+\|(a_x^\ell(a_x^\ell)^*-1)\xi_i\|<1/m.
\]
Then $Z_m$ is Borel.

We will show that $Z=\bigcap_m Z_m$.
Let $(x,w)\in \bigcap_m Z_m$.
Then $\th=\Ad w|_{\cM_x}\in \Aut(\cM_x)$,
and
for any $\vep>0$ and
a finite set $\Ph\subs (\cM_x)_*^+$,
there exists an element $a\in(\cM_x)_1$
such that
\[
\|\th(\vph)-a\vph a^*\|<\vep,
\]
\[
\|a^*a-1\|_\vph+\|aa^*-1\|_\vph<\vep
\quad
\mbox{for all }\vph\in\Ph.
\]
This implies that $\th$ is approximately inner
by Lemma \ref{lem:appinner-weak}.
Hence $(x,w)\in Z$.

Suppose conversely that
$(x,w)\in Z$.
Put $\th:=\Ad w|_{\cM_x}\in\oInt(\cM_x)$.
Then for any $\vep>0$ and
a finite $\Ph\subs(\cM_x)_*^+$,
there exists a unitary $u\in \cM_x$
such that
$\|\th(\vph)-u\vph u^*\|<\vep$
for $\vph\in\Ph$.
We can take a subsequence
$\{a_x^{k_n}\}_n$ which converging to
$u$ in the strong$*$ topology.
Thus $(x,w)\in\bigcap_m Z_m$.

Therefore $Z=\bigcap_m Z_m$,
which is Borel.
We modify $Y$ defined above as follows:
\[
Y':=Y\cap \{(x,v,w)\mid (x,w)\in Z\}.
\]
By our assumption,
$Y'$ is non-empty, and we obtain the projection
$\pr_X\col Y'\ra X$.
Then we are done in a similar way to the above.
\end{proof}

As an application,
we obtain the following result due to
Kallman and Moore.
See \cite[Theorem 0.1]{Kal}
or \cite[Theorem 5]{MooreIV}.

\begin{cor}[Kallman, Moore]
\label{cor:ptwiseinner}
Any inner flow on a separable von Neumann algebra
is implemented by a one-parameter unitary group.
\end{cor}
\begin{proof}
Let $\al$ be such a flow on a von Neumann algebra $\cM$.
Since $\al$ fixes $Z(\cM)$,
we obtain the central decomposition
$\al_t=\int_X^\oplus \al_t^x\,d\mu(x)$.
Then $\al^x$ is an inner flow on a factor $\cM_x$
for almost every $x$.
Since a $\T$-valued 2-cocycle of $\R$ is a coboundary,
$\al^x$ is cocycle conjugate to the trivial flow $\id_{\cM_x}$.
The previous result implies $\al\sim\id$.
\end{proof}

\subsection{Perturbation by continuous unitary path}
Let $\varphi$ be a normal state on a von Neumann algebra
$\cM$.
We need the following basic inequalities:
\[
\|x\varphi\|\leq \|x\|_\varphi,
\
\|\varphi x\|\leq \|x^*\|_\varphi,
\
\|x\varphi\|+\|\varphi x\|
\leq
\|x\|_\varphi+\|x^*\|_\varphi
\leq
2\|x\|_\varphi^{\sharp}.
\]
\[
\|x\|_\varphi^2\leq \|x\varphi\|\|x\|,
\
\|x^*\|_\varphi^2\leq \|\varphi x\|\|x\|,
\
\|x\|_\varphi^{\sharp2}
\leq \left(\|x\varphi\|
+\|\varphi x\|\right)\|x\|/2.
\]

For $\psi\in \cM_*$,
let $\psi=w_l|\psi|$ and $\psi=|\psi^*|w_r$ be
the left and right polar decompositions,
respectively.
Then 
\[
\|\psi x \|=\||\psi| x\|\leq \|x^*\|_{|\psi|},
\
\|x\psi \|=\|x|\psi^*|\|\leq \|x\|_{|\psi^*|},
\]

In what follows, we assume that
$\cM=\bigotimes_{k=1}^\infty(L_k,\rho_k)$
and
$\varphi_0:=\bigotimes_{k=1}^\infty\rho_k$  is 
lacunary, i.e., 
1 is isolated in $\mathrm{Sp}(\Delta_{\varphi_0})$, 
where $L_k$ is a finite
dimensional type I factor,
and $\rho_k$ is a faithful normal state on
$L_k$.
Denote $\hat{L}_k:=L_1\otimes \cdots \otimes L_k$.
Note that
any injective type II and III$_\la$ factors
with $0<\la<1$
have such form
(see \cite{AW}).
Then we can strengthen Lemma \ref{lem:adBorel} as follows.

\begin{prop}\label{prop:appro}
Let $\cM$ be an ITPFI factor as above.
Let $\alpha$ and $\beta$ be flow
on
$\cM$
with $\mathrm{mod}(\alpha_t)=\mathrm{mod}(\beta_t)$
for all $t\in\R$.
Then for any $T>0$, $\varepsilon>0$
and a finite set $\Phi\subset \cM_*$,
there exists a continuous unitary path
$\{u(t)\}_{|t|\leq T}$ such that
\[
\|\Ad u(t)\circ\alpha_t(\varphi)
-\beta_t(\varphi)\|<\varepsilon,
\quad
\mbox{for all }
\varphi\in \Phi,
\
t\in [-T,T].
\]
\end{prop}

We first recall some results proved in \cite{Co-almost}.
\begin{lem}[{\cite[Proposition 3.2, Lemma 2.7]{Co-almost}}]
\label{lem:centralizer}
The following statements hold:
\begin{enumerate}
\item 
There exists a universal constant
$C_0>0$ such that for any von Neumann
algebra $\mathcal{M}$
and its faithful normal state $\varphi$, 
\[
\|\sigma_t^\varphi(x)-x\|_\varphi^\sharp
<C_0(1+|t|)\|[x,\varphi]\|^{\frac{1}{2}},
\quad
t\in\R.
\] 
\item
Let $\mathcal{M}$ be a von Neumann algebra
and $\varphi\in\cM_*$ a faithful lacunary state.
Let $E_\varphi$ be
the $\varphi$-preserving conditional expectation from
$\mathcal{M}$ onto $\mathcal{M}_\varphi$.
Then there exists a constant $C_\varphi$,
which depends only on $\varphi$,
such that for all $x\in\cM$,
\[
\|E_{\varphi}(x)-x\|_\varphi^\sharp
<C_\varphi\|[x,\varphi]\|^{\frac{1}{2}}.
\]
\end{enumerate}
\end{lem}
\begin{proof}
(1). See \cite{Co-almost}. 
(2). Choose a positive $f\in L^1(\mathbb{R})$ as follows;
\[
\int_{\mathbb{R}}f(t)dt=1,
\
\int_{\mathbb{R}}|t||f(t)|dt<\infty,
\
\mathrm{supp}(\hat{f})\cap \mathrm{Sp}(\Delta_{\varphi})=\{1\},
\]
where $\hat{f}(\lambda)=\int_{\mathbb{R}}\lambda^{it}f(t)dt$.
Then $\sigma_{f}^\varphi(x)=E_\varphi(x)$.
Set $C_\varphi:=\int_{\mathbb{R}}C_0(1+|t|)|f(t)|dt$
and we are done.
\end{proof}

The following result can be
found in \cite[Lemma 3.2.1]{J-act}.

\begin{lem}\label{lem:approunitary}
Let $\mathcal{M}$ be a finite von Neumann algebra
with a normal tracial state $\tau$,
and $a\in \mathcal{M}$
such that $\|a^*a-1\|_\tau<\delta$.
Then there exists a unitary
$v\in \mathcal{M}$ such that
$\|a-v\|_\tau<(3+\|a\|)\delta$.
\end{lem}

\begin{lem}\label{lem:unitarypath}
Let $\Phi\subset \cM_*$ be a finite set,
and assume $v\in U(\cM)$ satisfies 
\[
\|\varphi\cdot (v-1)\|<\varepsilon,
\
\|(v-1)\varphi\|<\varepsilon,
\quad
\varphi\in \Phi.
\] 
Then there exists a continuous unitary path
$v(t)$, $0\leq t\leq 1$,
such that 
$v(0)=1$, $v(1)=v$
and
\[
\|\varphi\cdot(v(t)-1)\|<\sqrt{2\varepsilon},
\
\|(v(t)-1)\varphi\|<\sqrt{2\varepsilon}, 
\quad
\varphi\in \Phi.
\] 
\end{lem}
\begin{proof}
Let $v=\int_{-\pi}^\pi e^{i\lambda}\,de_\lambda$
be the spectral decomposition of $v$.
We set
$v(t):=\int_{-\pi}^\pi e^{it\lambda}\,de_\lambda$.
Then for all
$\varphi\in \Phi$, we have
\[
\int_{-\pi}^{\pi}|e^{-i\lambda}-1|^2
\,d|\varphi|(e_\lambda)
=
\|v^*-1\|_{|\varphi|}^2
\leq
\||\varphi|\cdot(v-1)\|\|v-1\|
\leq 2\varepsilon.
\]
Thus 
\[
\|v(t)^*-1\|_{|\varphi|}^2
=\int_{-\pi}^{\pi}|e^{-it\lambda}-1|^2
\,d|\varphi|(e_\lambda)
\leq \int_{-\pi}^{\pi}|e^{-i\lambda}-1|^2
\,d|\varphi|(e_\lambda)
\leq 2\varepsilon
\]
holds for $0\leq t\leq 1$.
Hence
\[
\|\varphi\cdot (v(t)-1)\|
\leq 
\||\varphi|\cdot (v(t)-1)\|
\leq \|v(t)^*-1\|_{|\varphi|}
\leq \sqrt{2\varepsilon}.
\]
If we replace $|\varphi|$ with $|\varphi^*|$ above,
then
we get 
\[
\|(v(t)-1)\varphi\|\leq \sqrt{2\varepsilon}.
\]
\end{proof}

\begin{lem}\label{lem:appro}
Let $\alpha$ and $\beta$
be as in Proposition \ref{prop:appro}.
For any $k\in \mathbb{N}$,  a finite set $F\subset \hat{L}_k$,
$T>0$, and $\varepsilon>0$, 
 there exists a continuous unitary path
$\{u(t)\}_{|t|\leq T}$ such that 
\[
\|\Ad u(t)\circ\alpha_t(\varphi_0a)
-\beta_t(\varphi_0a)\|<\varepsilon,
\quad
a\in F,
\
t\in [-T,T].
\]
\end{lem}
\begin{proof}
Let $m:=\big{(}\dim\hat{L}_k\big{)}^{1/2}$
and
$\{e_{ij}\}_{1\leq i,j\leq m}$ a system
of matrix units of $\hat{L}_k$.
We may assume that 
$\{1\}\cup\{e_{ij}\}_{i,j}\subset F\subs \cM_1$.
Set
\[
\Phi_0
:=
\{\varphi_0a\mid a\in F\},
\
\Phi
:=
\{a\varphi_0\mid a\in F\}
\cup\{\varphi_0 a\mid a\in F\}.
\]

Let $\vep'>0$.
Fix $N\in \mathbb{N}$ with $N\geq m$
so that
if $|t|\leq T/N$ and $\varphi\in \Phi$,
\begin{equation}
\label{eq:altvph}
\|\alpha_t(\varphi)-\beta_t(\varphi)\|
<
\varepsilon'/4m,
\
\|\alpha_t(\varphi)-\varphi\|<\varepsilon'/4m,
\
\|\beta_t(\varphi)-\varphi\|<\varepsilon'/4m.
\end{equation}

Set $t_0:=T/N$.
Since $\al_t\be_t^{-1}\in\oInt(\cM)$ for all $t\in\R$,
we can take an $\alpha_{t_0}$-cocycle $\{w_n\}_{n\in\Z}$
such that
\begin{equation}
\label{eq:wnal}
\|\Ad w_{n}\circ \alpha_{nt_0}(\varphi)
-\beta_{nt_0}(\varphi)\|
\leq
\varepsilon'/4m,
\quad
|n|\leq N,
\
\varphi\in \Phi.
\end{equation}

Set $w:=w_1$, 
$B:=\alpha_{t_0}(\hat{L}_k)$,
$\psi_0:=\alpha_{t_0}(\varphi_0)$,
$\Psi:=\alpha_{t_0}(\Phi)$
and $f_{ij}:=\alpha_{t_0}(e_{ij})$. 
We will find a continuous unitary path $w(t)$,
$0\leq t\leq t_0$
so that
$w(t)$ connects $1$ and $w$,
and $\Ad w(t)\circ \alpha_t$ approximates $\beta_t$ on $F$.
Note
$\mathrm{Sp}(\Delta_{\psi_0})
=\mathrm{Sp}(\Delta_{\varphi_0})$.
Thus we can
assume $C_{\psi_0}=C_{\varphi_0}$
in Lemma \ref{lem:centralizer}.
Using (\ref{eq:wnal}),
we have 
\[
\|\Ad w(\psi)-\beta_{t_0}\alpha_{t_0}^{-1}(\psi)\|
\leq
\varepsilon'/4m,
\quad
\psi\in \Psi.
\]
Thus for $\psi=\alpha_{t_0}(\varphi)\in \Psi$,
we have
\begin{align}
\|[w,\psi]\|
&=\|\Ad w(\psi)-\psi\|
\notag\\
&\leq 
\|\Ad w(\psi)-\beta_{t_0}\alpha_{t_0}^{-1}(\psi)\|
+
\|\beta_{t_0}(\varphi)-\alpha_{t_0}(\varphi)\|
\notag\\
&\leq
\varepsilon'/2m.
\label{eq:wpsial}
\end{align}

Since $\psi_0, \psi_0 f_{ij}\in \Psi$,
we have
\begin{align*}
\|\psi_0\cdot(wf_{ij}-f_{ij}w)\|
&\leq
\|\psi_0 wf_{ij}-w \psi_0 f_{ij}\|
+\|[w,\psi_0 f_{ij}]\|\\
&\leq
\|\psi_0 w-w \psi_0\|
+\|[w,\psi_0 f_{ij}]\|\\
&\leq
\varepsilon'/m.
\end{align*}
In the same way,
we have
$\|(wf_{ij}-f_{ij}w)\psi_0\|<\varepsilon'/m$.

Let
$E(x)
=m^{-1}
\sum_{i,j}
f_{ij}xf_{ji}$.
Then $E$
is a conditional expectation
from $\cM$ onto $B'\cap \cM$.
Set
$c:=E(w)$,
and then
\[
\|\psi_0\cdot (c-w)\|
\leq \frac{1}{m}
\sum_{i,j}
\|\psi_0\cdot (f_{ij}wf_{ji}-wf_{ij}f_{ji})\|
\leq \frac{1}{m}
\sum_{i,j}
\|\psi_0 [w,f_{ij}]\|<\varepsilon'.
\]
We also have $\|(c-w) \psi_0\|<\varepsilon'$.
Using (\ref{eq:wpsial}),
we obtain
\[
\|[c,\psi_0]\|
\leq \|[w,\psi_0]\|+
\|\psi_0\cdot(c-w)\|+
\|(c-w) \psi_0\|
<3\varepsilon'.
\]
Since
$\|w-c\|_{\psi_0}<\varepsilon'$ and
$\|(w-c)^*\|_{\psi_0}<\varepsilon'$,
we have
\begin{equation}
\label{wcsharp}
\|w-c\|_{\psi_0}^{\sharp}<\sqrt{\varepsilon'}.
\end{equation}

Set $d:=E_{\psi_0}(c)$.
Note that $d\in (B'\cap \cM)_{\psi_0}$, 
and
$\mathrm{Sp}(\Delta_{\psi_0|B'\cap \cM})
\subset \mathrm{Sp}(\Delta_{\varphi_0})$.
Thus we can assume $C_{\psi_0|B'\cap \cM}=C_{\varphi_0}$
by the definition
of $C_{\varphi_0}$ in Lemma \ref{lem:centralizer}.
By the lemma,
\[
\|d-c\|_{\psi_0}^\sharp
\leq C_{\varphi_0}\|[c,\psi_0]\|^{\frac{1}{2}}
<C_{\varphi_0} \sqrt{3\varepsilon'}.
\]
Hence by (\ref{wcsharp}),
we get
\[
\|w-d\|_{\psi_0}^{\sharp}
<\sqrt{3\varepsilon'}(1+C_{\varphi_0}).
\]

We should note that 
$[x,\psi]=0$ for $x\in (B'\cap \cM)_{\psi_0}$
and $\psi\in \Psi$
since
$\Psi
=\{b\psi_0, \psi_0 b\mid b\in \alpha_{t_0}(F)\}$
and
$\alpha_{t_0}(F)\subset \alpha_{t_0}(\hat{L_k})=B$.
We have 
\begin{align*}
\|d^*d-1\|_{\psi_0}
&\leq
\|d^*d-w^*d\|_{\psi_0}+\|w^*d-w^*w\|_{\psi_0}
\\
&\leq
\|d^*-w^*\|_{\psi_0}+\|d-w\|_{\psi_0}\\
&\leq2\sqrt{3\varepsilon'}(1+C_{\varphi_0}).
\end{align*}
Thus by Lemma \ref{lem:approunitary},
we can take a unitary $v\in (B'\cap \cM)_{\psi_0}$
such that 
\[
\|d-v\|_{\psi_0}
=
\|d-v\|_{\psi_0}^{\sharp}
< 
8\sqrt{3\varepsilon'}(1+C_{\varphi_0}).
\]
Then
setting $\varepsilon'':=
\varepsilon'+8\sqrt{3\varepsilon'}(1+C_{\varphi_0})$,
we obtain
\begin{equation}
\label{eq:wvpsi}
\|w-v\|_{\psi_0}^{\sharp}
<\vep''.
\end{equation}

Let
$v(t)\in (B'\cap \cM)_{\psi_0}$,
$t_0/2\leq t\leq t_0$, 
be a continuous path of
unitaries with $v(t_0/2)=v$
and $v(t_0)=1$.
Set $u(t):=wv(t)^*$.
Note that
$(B'\cap \cM)_{\psi_0}
\subset
\cM_{\psi_0}$. 
Since $[v(t),\psi]=0$ for all $\psi\in \Psi$,
for all $\varphi\in \Phi$,
we have
\begin{equation}
\label{eq:ut02}
\|\Ad u(t)\circ \alpha_{t_0}(\varphi)
-\beta_{t_0}(\varphi)\|
=
\|\Ad w\circ\alpha_{t_0}(\varphi)
-\beta_{t_0}(\varphi)\|
<
\varepsilon'/4m.
\end{equation}

We will find a path which connects
$1$ and $wv^*$.
For $b\in \alpha_{t_0}(F)$,
we obtain
\begin{align*}
\|(\psi_0 b)\cdot (wv^*-1)\|
&=
\|\psi_0 bw-\psi_0 bv\|
\\
&\leq\|[\psi_0 b, w]\|+
\|w\psi_0 b-\psi_0 vb\|
\\
&\leq\|[\psi_0 b, w]\|+
\|w \psi_0 -\psi_0 v\|
\\
&\leq\|[\psi_0 b, w]\|+\|[w,\psi_0]\|+
\|\psi_0 w -\psi_0 v\|
\\
&\leq\varepsilon'+2\varepsilon''
\quad
\mbox{by }
(\ref{eq:wnal}),
\
(\ref{eq:wpsial}),
\
(\ref{eq:wvpsi}),
\end{align*}
and by (\ref{eq:wvpsi}),
\[
\|(wv^*-1) \psi_0 b\|
=
\|(w-v) \psi_0 vb\|
\leq
\|(w-v) \psi_0\|\leq 2\varepsilon''.
\]
Hence by Lemma
\ref{lem:unitarypath},
there exists a continuous unitary path $u(t)$,
$\displaystyle{0\leq t \leq t_0/2}$
such that
$u(0)=1$
,
$u(t_0/2)=wv^*$
and 
\[
\|(u(t)-1) \psi_0 a\|
\leq
\sqrt{2(\varepsilon'+2\varepsilon'')},
\
\|(\psi_0 a)\cdot (u(t)-1)\|
\leq
\sqrt{2(\varepsilon'+2\varepsilon'')},
\quad a\in \alpha_{t_0}(F).
\]
Summarizing the inequalities
(\ref{eq:altvph})
and
(\ref{eq:ut02}),
we get a
continuous unitary path $u(t)$,
$0\leq t\leq t_0$, 
such that $u(0)=1$, $u(t_0)=w$,
and
\[
\|\Ad u(t)\circ \alpha_{t_0}(\varphi)
-\beta_{t_0}(\varphi)\|
<
2\sqrt{2(\varepsilon'+2\varepsilon'')}
+\varepsilon',
\quad
\varphi\in \Phi_0,
\
0\leq t\leq t_0.
\]

Then we have
\[
\|\Ad u(t)\circ \alpha_{t}(\varphi)
-\beta_{t}(\varphi)\|
<
2\sqrt{2(\varepsilon'+2\varepsilon'')}
+2\varepsilon',
\quad
\varphi\in \Phi_0,
\
0\leq t\leq t_0.
\]

For $t\in [-T,T]$,
let $t=nt_0+s$, $0\leq s< t_0$,
and define
$u(t):=w_n\alpha_{nt_0}(u(s))$.
Then $u(t)$ is a continuous unitary path
and for $\varphi\in \Phi_0$,
\begin{align*}
&\|\Ad u(t)\circ \alpha_t(\varphi)-\beta_t(\varphi)\|
\\
&=
\|\Ad w_n\circ \alpha_{nt_0}
\circ \Ad u(s)\circ\alpha_s(\varphi)
-
\beta_{nt_0}\circ \beta_s(\varphi)\|
\\
&\leq
\|\Ad w_n\circ \alpha_{nt_0}
\circ \Ad u(s)\circ\alpha_s(\varphi)
-
\Ad w_n\circ \alpha_{nt_0}\circ \beta_s(\varphi)
\|\\
&
\quad+
\|\Ad w_n\circ \alpha_{nt_0}\circ \beta_s(\varphi)
-
\beta_{nt_0}\circ \beta_s(\varphi)\|\\
&\leq
\|\Ad u(s)\circ\alpha_s(\varphi)-\beta_s(\varphi)\|
+
\varepsilon'/m
+
\|\Ad w_n\circ \alpha_{nt_0}(\varphi)
-
\beta_{nt_0}(x)\|\\
&< 2\sqrt{2(\varepsilon'+2\varepsilon'')}
+3\varepsilon'
\quad
\mbox{by }
(\ref{eq:wnal}).
\end{align*}
For any given $\varepsilon>0$,
we choose $\varepsilon'$ as 
$2\sqrt{2(\varepsilon'+2\varepsilon'')}
+3\varepsilon'<\varepsilon$,
and we get the conclusion.
\end{proof}

\begin{proof}[Proof of Proposition \ref{prop:appro}.]
Let $\Phi:=\{\psi_i\}_{i=1}^n$
be a finite set of $\cM_*$.
Since the set
$\{\varphi_0 a
\mid a\in \bigcup_{k=1}^\infty \hat{L}_k\}$
is dense in $\cM_*$, 
there exist $k>0$
and $\{a_i\}_{i=1}^{n}\subset \hat{L}_k$
such that
$\|\psi_i-\varphi_0 a_i\|<\vep/3$. 
By Lemma \ref{lem:appro},
there exists a continuous path
of unitaries $u(t)$, $|t|\leq T$,
such that 
\[
\|\Ad u(t)\circ \alpha_t(\varphi_0  a_i)
-\beta_t(\varphi_0 a_i)\|<\vep/3,
\quad
1\leq i\leq n,
\
|t|\leq T.
\]  
Then we obtain
\[
\|\Ad u(t)\circ \alpha_t(\psi_i)-\beta_t(\psi_i)\|
<\varepsilon,
\quad
1\leq i\leq n,
\
|t|\leq T.
\]  
\end{proof}

\end{document}